\documentclass[12pt, a4paper]{report}
\usepackage{amsmath}
\usepackage{amsfonts}
\usepackage{amssymb}
\usepackage{amsthm}
\usepackage{chngcntr}
\usepackage{enumitem}
\usepackage{hyperref}
\usepackage{mathabx}
\usepackage{mathtools}
\usepackage{titlesec}

\titleformat{\chapter}{\normalfont\huge}{\thechapter.}{20pt}{\huge\bf}

\DeclareMathOperator*{\Argmin}{Argmin}
\DeclareMathOperator*{\ba}{ba}
\DeclareMathOperator*{\capa}{cap}
\DeclareMathOperator*{\CL}{CL}
\DeclareMathOperator*{\cl}{cl}
\DeclareMathOperator*{\clco}{\overline{co}}

\DeclareMathOperator*{\co}{co}
\DeclareMathOperator*{\dist}{dist}
\DeclareMathOperator*{\DIV}{div}
\DeclareMathOperator*{\dom}{dom}
\DeclareMathOperator*{\esscap}{ess-\bigcap}
\DeclareMathOperator*{\essinf}{ess-inf}
\DeclareMathOperator*{\esssup}{ess-sup}
\DeclareMathOperator*{\epi}{epi}
\DeclareMathOperator*{\graph}{graph}
\DeclareMathOperator*{\interior}{int}
\DeclareMathOperator*{\lev}{lev}
\DeclareMathOperator*{\lin}{lin}
\DeclareMathOperator*{\LS}{LS}
\DeclareMathOperator*{\supp}{supp}

\def\N{{\mathbb N}}
\def\Q{{\mathbb Q}}
\def\R{{\mathbb R}}
\def\AA{{\mathcal A}}
\def\BB{{\mathcal B}}
\def\DD{{\mathcal D}}
\def\EE{{\mathcal E}}
\def\FF{{\mathcal F}}
\def\GG{{\mathcal G}}
\def\HH{{\mathcal H}}
\def\II{{\mathcal I}}
\def\JJ{{\mathcal J}}
\def\KK{{\mathcal K}}
\def\LL{{\mathcal L}}
\def\MM{{\mathcal M}}
\def\PP{{\mathcal P}}
\def\SS{{\mathcal S}}
\def\TT{{\mathcal T}}

\def\XX{{\mathcal X}}
\def\YY{{\mathcal Y}}
\def\Om{\Omega}
\def\e{\varepsilon}
\def\i{\infty}
\def\p{\partial}
\def\dd{\mathfrak{d}}
\def\ee{\mathfrak{e}}
\def\mm{\mathfrak{m}}
\def\ss{\mathfrak{s}}
\def\tt{\mathfrak{t}}
\def\om{\omega}
\def\st{\, \middle| \,}
\def\ddto{\overset{\dd}{\to}}
\def\dddto{\overset{\dd^*}{\to}}
\def\eeto{\overset{\ee}{\to}}

\def\ssto{\overset{\ss}{\to}}
\def\ttto{\overset{\tt}{\to}}
\def\weak{\rightharpoonup}
\def\weakast{\overset{\ast}{\rightharpoonup}}

\newtheorem{corollary}{Corollary}
\newtheorem{definition}{Definition}
\newtheorem{lemma}{Lemma}
\newtheorem{proposition}{Proposition}
\newtheorem{theorem}{Theorem}

\counterwithin{corollary}{chapter}
\counterwithin{definition}{chapter}
\counterwithin{lemma}{chapter}
\counterwithin{proposition}{chapter}
\counterwithin{theorem}{chapter}

\def\Xint#1{\mathchoice
	{\XXint\displaystyle\textstyle{#1}}%
	{\XXint\textstyle\scriptstyle{#1}}%
	{\XXint\scriptstyle\scriptscriptstyle{#1}}%
	{\XXint\scriptscriptstyle\scriptscriptstyle{#1}}%
	\!\int}
\def\XXint#1#2#3{{\setbox0=\hbox{$#1{#2#3}{\int}$} 
		\vcenter{\hbox{$#2#3$}}\kern-.5\wd0}} 

\def\dashint{\Xint-}

\title{Orlicz-type Function Spaces and Generalized Gradient Flows with Degenerate Dissipation Potentials in Non-Reflexive Banach Spaces: Theory and Application}

\author{Thomas Daniel Ruf}

\begin{document}
	
	\begin{titlepage}
		\centering

		\LARGE\textbf{Orlicz-type Function Spaces and Generalized Gradient Flows with Degenerate Dissipation Potentials in Non-Reflexive Banach Spaces: Theory and Application}
		
		\vspace{1.0cm}
		\Large\textbf{Dissertation}
		
		\vspace{1cm}
		\large\textit{Eingereicht zur Erlangung des Grades}
		
		\vspace{0.5cm}
		\large\textbf{Dr. rer. nat.}
		
		\vspace{1.0cm}
		\large\textbf{Mathematisch-Naturwissenschaftlich-Technische Fakultät}
		
		\vspace{0.5cm}
		\large\textbf{Universität Augsburg}
		
		\vspace{1.5cm}
		\large\textbf{Vorgelegt von}
		
		\vspace{0.5cm}
		\large\textbf{Thomas Daniel Ruf}
		
		\vspace{2cm}
		\large\textbf{\today}
		
	\end{titlepage}
	
	\tableofcontents
	
	\newpage
	
	\begin{abstract}
		This thesis explores two important areas in the mathematical analysis of nonlinear partial differential equations: Generalized gradient flows and vector valued Orlicz spaces. The first part deals with the existence of strong solutions for generalized gradient flows, overcoming challenges such as non-coercive and infinity-valued dissipation potentials and non-monotone subdifferential operators on non-reflexive Banach spaces. The second part focuses on the study of Banach-valued Orlicz spaces, a flexible class of Banach spaces for quantifying the growth of nonlinear functions. Besides improving many known results by imposing minimal assumptions, we extend the theory by handling infinity-valued Orlicz integrands and arbitrary Banach-values in the duality theory. The combination of these results offers a powerful tool for analyzing differential equations involving functions of arbitrary growth rates and leads to a significant improvement over previous results, demonstrated through the existence of weak solutions for a doubly nonlinear initial-boundary value problem of Allen-Cahn-Gurtin type.
	\end{abstract}
	
	\begin{abstract}
		Diese Arbeit untersucht zwei wichtige Bereiche der mathematischen Analysis nichtlinearer partieller Differentialgleichungen: Generalisierte Gradientenflüsse und vektorwertige Orlicz-Räume. Der erste Teil befasst sich mit der Existenz starker Lösungen für generalisierte Gradientenflüsse und bietet eine Lösung für Herausforderungen wie nicht-koerzive und unendlichwertige Dissipationspotentiale und nicht-monotonen Subdifferentialoperatoren auf nicht-reflexiven Banachräumen. Der zweite Teil konzentriert sich auf die Untersuchung von Banach-wertigen Orlicz-Räumen, einer flexiblen Klasse von Banachräumen zur Quantifizierung des Wachstums nichtlinearer Funktionen. Neben der Verbesserung vieler bekannter Ergebnisse durch minimale Annahmen erweitern wir die Theorie, indem wir uns mit unendlichwertigen Orlicz-Integranden und beliebigen Banach-Werten in der Dualitätstheorie befassen. Die Kombination dieser Ergebnisse bietet ein leistungsstarkes Werkzeug zur Analyse von Differentialgleichungen, die Funktionen beliebiger Wachstumsraten beinhalten, und führt zu einer signifikanten Verbesserung gegenüber früheren Ergebnissen, wie sie durch die Existenz schwacher Lösungen für ein doppelt nichtlineares Anfangs-Randwertproblem vom Allen-Cahn-Gurtin-Typ gezeigt wird.
	\end{abstract}
	
	\chapter{Introduction}
	
	\section{Subject matters}
	
	This thesis consists of two parts that may be read separately. The first part establishes existence and stability of global solutions to the doubly nonlinear evolution inclusion
	\begin{equation} \label{eq: DNI}
		\p \phi_{t, u(t) } \left( u'(t) \right) + \p \EE_t \left( u(t) \right) \ni 0 \text{ for a.e. } t \in (0, T), \quad u(0) = u_0
	\end{equation}
	on a Banach space $V$. Here, $\phi \colon \left(0, T \right) \times \dom \left( \EE_t \right) \times V \to \left[0 , + \i \right]$ is a time-and-state-dependent family of convex dissipation potentials with superlinear growth and $\EE \colon \left( 0, T \right) \times V \to \left(- \i, + \i \right]$ is a time-dependent family of lower semicontinuous energies. The subdifferential for $\phi$ is that of convex analysis, whereas we allow a broad class of subdifferentials for the nonconvex energy $\EE$.
	
	We focus on dissipation potentials such that $\phi$ or $\phi^*$ are \emph{degenerate} in various ways, e.g., by assuming infinite values, by depending on time or state in a non-uniform matter, or by lacking compact sublevel sets. Let us illustrate these phenomena with examples.
	
	\begin{enumerate}
		
		\item In \cite{BoPe} the authors consider the ordinary differential equation
		\begin{equation} \label{eq: BP ODE}
			x'(t) = - \sinh \left( \nabla E_t \left( x(t) \right) \right),
		\end{equation}
		which may be written in the form (\ref{eq: DNI}) by inverting the hyperbolic sine. This gives the dissipation potentials
		\begin{equation} \label{eq: BP dissi}
			\psi(v) = v \log \left( v + \sqrt{v^2 + 1} \right) - \sqrt{v^2 + 1} + 1 \text{ and } \psi^*(\xi) = \cosh(\xi) - 1.
		\end{equation}
		If one wishes to extend \cite{BoPe} to an infinite dimensional state space and thus to partial differential equations, one arrives at the analogous dissipation potentials
		\begin{equation} \label{eq: BP dissi PDE}
			\phi_0(v) = \int_\Om \psi \left( v(x) \right) \, dx \text{ and } \phi_0^*(\xi) = \int_\Om \psi^* \left( \xi(x) \right) \, dx, \quad \Om \subset \R^d \text{ open.}
		\end{equation}
		The natural state space for $\phi$ is the Orlicz space $L_\psi(\Om)$. However, since $\psi^* \not \in \Delta_2$, the conjugate $\phi^*$ is not finite on the dual space $L_{\psi^*}(\Om)$. This observation actually applies to a much broader class of dissipation potentials. If $\varphi \colon \Om \times \R^d \to \left[ 0, + \i \right]$ is an Orlicz integrand such that $\varphi \not \in \Delta_2$ or $\varphi^* \not \in \Delta_2$, then one of the dissipation potentials
		\begin{equation} \label{eq: Orl dissi}
			\phi_1(v) = \int_\Om \varphi(x, v(x) ) \, dx \text{ and } \phi_1^*(\xi) = \int_\Om \varphi^*(x, \xi(x) ) \, dx
		\end{equation}
		will be unbounded on its generalized Orlicz space $L_\varphi(\Om)$ or $L_{\varphi^*}(\Om)$. Particular choices for (\ref{eq: Orl dissi}) have been studied, for example, in \cite{Ak, ASch} under assumption requiring $\varphi, \varphi^* \in \Delta_2$. However, more general choices for $\varphi$ and $\varphi^*$ continue to be physically plausible, cf. \cite{Gu}. Therefore, treating rapid growth integrands such as (\ref{eq: BP dissi PDE}) or (\ref{eq: Orl dissi}) motivates us to handle infinity valued dissipation potentials.
		
		\item Irreversible evolution inclusions such as in damage models \cite{MiRo} lead to dissipation potentials such as
		\begin{equation} \label{eq: dmg dissi}
			\begin{aligned}
				\phi_2(v) & =
				\begin{cases}
					\frac{1}{2} \int_\Om \left| v(x) \right|^2 \, dx & \text{ if } v \le 0 \text{ a.e. on } \Om, \\
					+ \i & \text{ else;}
				\end{cases}
				\\
				\phi_2^*(\xi) & =
				\frac{1}{2} \int_\Om \left| \max\{ - \xi(x), 0 \} \right|^2 \, dx.
			\end{aligned}
		\end{equation}
		Notably, the dual dissipation potential $\phi_2^*$ provides no control on the positive part of $\xi$, resulting in a lack of weak coercivity for $\phi^*$ on $L_2(\Om)$. However, if $1 < p < + \i$ and
		\begin{equation} \label{eq: dmg toy en}
			\EE(u) =
			\begin{cases}
				\frac{1}{p} \int_\Om \left| \nabla u(x) \right|^p \, dx & \text{ if } u \in W^{1, p} \left( \Om \right) \cap L_2(\Om), \\
				+ \i & \text{ else,}
			\end{cases}
		\end{equation}
		then $D \EE(u) = - \Delta_p u$ so that if $\sup_n \EE(u_n) + \phi_2^* \left( - D \EE(u_n) \right) < + \i$, we have a bound on $\Delta_p u$ in $W^{1, p}(\Om)^*$ and control the positive part of $\Delta_p u$ in $L_2(\Om)$. Combining this information and using that a signed distribution is Radon regular, we can bound $\Delta_p u_n$ as a Radon measure in the total variation norm, which may provide sufficient compactness. Therefore, treating unilateral constraints such as (\ref{eq: dmg dissi}) motivates us to study dissipation potentials where the coercivity of $\phi^*$ fails unless it is conditioned on the energy and its subdifferential. As the example (\ref{eq: dmg dissi}) shows, this may lead to a primal dissipation potential $\phi$ lacking even a single point of continuity on the natural dissipation space, which here is $L_2(\Om)$.
		
		\item In \cite[§2.3.2]{LMPR}, a nonlinear reaction-diffusion system is discussed. Let $k^f_r$ and $k^b_r$ be the forward and backward reaction coefficients and $\alpha^r_i, \beta^r_i$ non-negative integers, $i \in \left\{1, \dots, N \right\}$, $r \in \left\{ 1, \dots, R \right\}$ such that there holds the detailed balance condition
		$$
		\exists w = \left( w_i \right)_i \colon w_i > 0 \langle k^f_r w^{\alpha^r} = k^b_r w^{\beta^r} \text{ for } r \in \left\{ 1, \dots, R \right\}.
		$$
		They consider (\ref{eq: DNI}) with the dual dissipation potential
		\begin{equation} \label{eq: LMPR dua dissi}
			\begin{gathered}
				\phi^*_u(\xi) = \int_\Om \frac{1}{2} \sum_{i = 1}^N \left| u_i(x) \right| \| \nabla \xi_i(x) \|^2 + \sum_{r = 1}^R H^r(u) \left| \left( \alpha^r - \beta^r \right) \cdot \xi(x) \right|^2 \, dx
				\\
				\text{ with } H^r(u) = \frac{k^b_r u^{\beta^r} - k^f_r u^{\alpha^r} }{\log \left( k^b_r u^{\beta^r} \right) - \log \left( k^f_r u^{\alpha^r} \right) }
			\end{gathered}
		\end{equation}
		and the energy
		\begin{equation} \label{eq: LMPR en}
			\EE(u) = \int_\Om \sum_{i = 1}^N \lambda_B \left( u_i(x) / w_i \right) w_i \, dx \text{ with } \lambda_B{z} = z \log z - z + 1.
		\end{equation}
		Again, the potential $\phi^*_u(\xi)$ fails to provide coercivity unless we insert the relation $\xi_i = - D \EE(u)_i = \log \left( u_i / w_i \right)$. More importantly now, the presence of the factor $\left| u_i \right|$ in the first term of $\phi^*$ will involve the integrand $\left| u_i \right|^{-1} \left| v_i \right|^2$ in the primal dissipation potential, which arises as the infimal convolution of the addends of $\phi^*$. Thus, as we cannot rule out that $\left| u_i \right|$ might be large, no uniform control on $v_i$ will be available as $u_i$ varies, leading to problems in checking the chain rule (in)equality in the abstract existence theory of, for example, \cite{MRS}. This motivates us to treat Ljapunov solutions to (\ref{eq: DNI}).
		
	\end{enumerate}
	
	In the second part of the thesis we complement the abstract existence theory for (\ref{eq: DNI}) by studying Orlicz spaces whose function elements take values in an arbitrary Banach space. Precisely, we let $\left( \Om, \AA, \mu \right)$ be a measure space, $X$ be a Banach space and $\varphi \colon \Om \times X \to \left[ 0, \i \right]$ be an even, convex integrand such that $\lim_{x \to 0} \varphi(\om, x) = 0$ and $\lim_{| x | \to \i} \varphi(\om, x) = \i$ for a.e. $\om \in \Om$. Then our Orlicz space $L_\varphi(\mu)$ is defined as those strongly measurable mappings $u \colon \Om \to X$ for which the norm
	$$
	\| u \|_\varphi = \inf \left\{ \alpha > 0 \st \int_\Om \varphi(\om, \alpha^{-1} u(\om) ) \, d \mu(\om) \le 1 \right\}
	$$
	is finite. By being the smallest Banach subspace of the strongly measurable functions $\LL_0 \left( \Om ; X \right)$ on which the integral functional
	$$
	I_\varphi \colon u \mapsto \int_\Om \varphi(\om, u(\om) ) \, d \mu(\om)
	$$
	has a point of continuity (at the origin), the spaces $L_\varphi(\mu)$ serve as a powerful tool to facilitate subdifferential calculus involving $I_\varphi$. Thus, the key idea behind Orlicz spaces is to let the tuple $\left( \varphi, \mu \right)$ generate a suiting function space for studying $I_\varphi$ instead of imposing artificial growth assumptions on $\varphi$.
	
	Our initial interest in $L_\varphi(\mu)$ arose due to the so-called Weighted Energy-Dissipation (WED) principle, which constructs solutions to evolution equations by studying the limits of minimizers to certain integral functionals, whose Euler-Lagrange equation represents an elliptic-in-time regularization of the target equation, cf., e.g., \cite{AkSt}. Though we eventually found minimizing movements more easily applicable in our existence theory for (\ref{eq: DNI}), the generalized Orlicz spaces nevertheless proved an extremely useful tool for applying the results obtained in the first part of the thesis. This is due to two factors: Firstly, many dissipation potentials $\phi_t$ in concrete instances of (\ref{eq: DNI}) happen to be or are close to being Orlicz integrands; Secondly, the subdifferential calculus of energies $\EE_t$ involving complicated growth conditions is greatly facilitated by Orlicz spaces. While an application of Orlicz spaces with an infinite-dimensional underlying Banach space $X$ does not come up in the present thesis, we still consider this an important problem in light of future applications. To give an example beyond the WED principle, we remark that Theorem \ref{thm: conjugate A} in our opinion very strongly suggests a suitable alternative form of the energy (in)equality
	\begin{equation*}
		\begin{gathered}
			\int_s^t \phi_{r, u(r) } \left( u(r) \right) + \phi^*_{r, u(r) } \left( - \xi(r) \right) \, dr + \EE_t \left( u(t) \right)
			\\
			= \EE_s \left( u(s) \right) + \int_s^t \p_t \EE_r \left( u(r) \right) \, dr \quad \forall s, t \in \cl \left[ 0, T \right)
		\end{gathered}
	\end{equation*}
	that solutions to (\ref{eq: DNI}) formally enjoy, by including an additional non-integral term that would unify the treatment of rate-dependent and rate-independent instances of (\ref{eq: DNI}), namely, we conjecture that the present theory of the first part of the thesis could be extended to the case of $\phi$ growing linearly. The time derivative of such a solution to (\ref{eq: DNI}) would no longer be a Radon measure, i.e., a curve taking values in the dual of a space of continuous vector-valued functions, but a curve taking values in the dual of a generalized vector-valued Orlicz space.
	
	What open problems in the theory of generalized Orlicz spaces do we address and how do we solve them? Firstly, fundamental measurability problems are known to arise if $X$ is not separable. For example, no effective representation for the convex conjugate of an integral functional on $L_\varphi(\mu)$ was known for this case in general, though some sufficient conditions for this situation to match the separable case exist in specific constellations with additional topological structure, see, e.g., \cite[Ch. VII, §3]{CaKa}. We circumvent these difficulties by employing measure theoretical operations of an essential nature, e.g., the essential infimum of a family of functions and the essential intersection of a family of sets, instead of the set theoretical equivalents. If we index such an essential infimum or intersection over the uncountable family of all separable subspaces of $X$, then the result will be measurable, which may fail for the set theoretical version. Using this technique, we prove in particular that $L_\varphi(\mu)$ is in a certain measure theoretical sense almost a subspace of $L_1 \left( \mu; X \right)$ and almost a superspace of $L_\i \left( \mu; X \right)$ if $\mu$ is finite. This innocuous result will allow us to systematically derive many important properties of $L_\varphi(\mu)$ from the better understood Bochner-Lebesgue spaces. For example, this idea plays an important role in solving the second major open problem we address: The duality theory was limited to various restrictive special constellations in previous works. For example, \cite{Ko1, Ko2} and \cite{Gi1} treated rather general Orlicz integrands if $X$ had a separable dual, whereas \cite{Scha} considered reflexive $X$ for autonomous integrands. By locally sandwiching $L_\varphi(\mu)$ between $L_1$ and $L_\i$ and patching the result together in a globally well-defined way, we manage to do away with any restrictions on $X$ and need no additional assumptions on $\varphi$ to represent the function component of the dual. Our duality theory is not limited to real-valued Orlicz integrands. The central ingredient in avoiding this assumption is already contained in \cite{Gi1}, which was kindly pointed out to us by E. Giner. Most regrettably, his idea seems to be absent from the recent literature on Orlicz spaces and we hope that the present work can contribute to righting this wrong. He observed that in order to decompose a functional $\ell \in L_\varphi(\mu)^*$ into addends for which a representation can be obtained, it is most useful to study the set functions $A \mapsto \ell \left( \chi_A u \right)$ for an arbitrary, fixed $u \in L_\varphi(\mu)$. These set functions can be decomposed as a direct sum of set functions having distinct properties, e.g., being $\sigma$-additive or strictly finitely additive, and this decomposition carries over to every $\ell$, e.g., there exist unique $\ell_\sigma$ and $\ell_f$ such that $\ell = \ell_\sigma + \ell_f$ and, for every $u \in L_\varphi(\mu)$, the set function $A \mapsto \ell_\sigma \left( \chi_A u \right)$ is $\sigma$-additive, the set function $A \mapsto \ell_f \left( \chi_A u \right)$ is strictly finitely additive. An essential gain of this decomposition is that we obtain a class of functionals that may be reasonably expected to be identical with the function component of the dual space, making its representation a meaningful undertaking in the first place. Other works on the duality of vector-valued Orlicz spaces were frequently limited to real-valued integrands, cf., e.g. \cite{Ko1, Ko2, Scha}.
	
	Third, three subspaces of the Orlicz space are generally considered to be important: The maximal linear subspace $M_\varphi(\mu)$ of $\dom I_\varphi$,
	the linear subspace $C_\varphi(\mu)$ of elements having absolutely continuous norm in $L_\varphi(\mu)$, i.e.,
	\begin{equation} \label{eq: abs cnt nrm}
		u \in C_\varphi(\mu) \iff u \in L_\varphi(\mu) \land \lim_{A \to \emptyset} \| u \chi_A \|_\varphi = 0,
	\end{equation}
	and $E_\varphi(\mu)$, the closure of simple functions in $L_\varphi(\mu)$. Strikingly, not all these subspaces are considered equally important by all authors. For example, $C_\varphi(\mu)$ does not explicitly show up in \cite{Ko1, Ko2, Mu}, where frequently conditions on $\varphi$ are invoked such as the local integrability
	\begin{equation} \label{eq: loc int}
		\mu(A) < + \i \implies \int_A \varphi(\om, x) \, d \mu(\om) < + \i \quad \forall x \in X,
	\end{equation}
	which guarantees that $C_\varphi(\mu)$ coincides with its superspace $E_\varphi(\mu)$. On the contrary, we found it more natural to let $C_\varphi(\mu)$ take center stage among the prominent subspaces, because firstly, we found it to coincide with $M_\varphi(\mu)$ if $\varphi$ is real on atoms of finite measure and, secondly, it allows us to avoid uniformity assumptions such as (\ref{eq: loc int}) all-together. A frequent key tool of ours in implementing the latter observation is the nice interplay between (\ref{eq: abs cnt nrm}) and the Egorov theorem. Avoiding uniformity assumptions on $\varphi$ has several advantages, starting with the aesthetic and convenient aspect of not having to check a superfluous assumption in all subsequent applications of the theory. Also, for state-dependent Orlicz integrands $\varphi_v$, we found that imposing a seemingly harmless uniformity assumption such as (\ref{eq: loc int}) on every function $x \mapsto \varphi(\om, v(\om), x)$ is no longer a negligible restriction. Consider the innocuous situation where $v \colon \Om \to \left( 0, + \i \right)$ is positive and
	$$
	\varphi(\om, v, u) = v \left| u \right|^2.
	$$
	In the same spirit, we remain frugal with our conditions on $\varphi$ throughout, thereby achieving many results under minimal assumptions that improve even the known case $X = \R^d$. To give two popular examples, we never impose superlinear vanishing at the origin or superlinear growth at infinity on $\varphi$. Also, the measure $\mu$ is kept rather general throughout the investigation, the only restriction being that we sometimes rule out atoms of infinite measure. From a perspective of pure mathematics, such a disciplined approach has the advantage of compelling to paint a truer picture of the subject and its essential requirements.
	
	Fourth and finally, for many important properties such as separability or reflexivity, we could find in the literature conditions in special cases of $L_\varphi(\mu)$ only, for example, if $X = \R^d$ \cite{ChGSW}, if $\Om \subset \R^d$ is a Borel set carrying the Lebesgue measure \cite{ChGSW}, if the structure of a scalar-valued Banach function space is present \cite{BeSh}. Often, the expositions were limited to proving sufficiency of these conditions. We rectify this by giving characterizations that equivalently describe many basic properties of a generalized vector-valued Orlicz space, including separability, reflexivity, or the Asplund property. Thereby, we not only provide criteria to check these properties in applications, but also demonstrate the limitations implicit in assuming them.
	
	As an application, we synthesize both theories to establish a new existence result for weak solutions to doubly nonlinear equations of Allen-Cahn-Gurtin type. In a number of ways, we surpass previous results in terms of admissible growth conditions in Orlicz spaces. The result provides a glimpse at the potential of our theory without exhausting its limits, as alluded to by the various potential applications discussed above.
	
	\paragraph{Open and follow-up questions.}
	
	An extension of the theory to even more general non-reflexive Banach spaces is essentially hindered by open measurability issues arising if the state space lacks the Suslin property. Can these problems be circumvented by employing a generalized integration based on Hahn-Banach extensions of the Lebesgue integral to non-measurable functions? Theorem \ref{thm: conjugate B} makes a strong suggestion how the energy (in)equality should look if the dissipation potential $\phi$ in (\ref{eq: DNI}) has merely linear growth from below. Can this case be included in the theory, unifying the treatment of rate-dependent and rate-independent doubly nonlinear inclusions?
	
	We assume the Orlicz integrand $\varphi$ to vanish continuously at the origin. However, the Orlicz space remains a Banach space if this assumption is dropped. What then is the dual space of $L_\varphi(\mu)$? At least for $\varphi(\om, x) = \varphi(x)$ we know that a discontinuity at the origin can always be avoided by adapting the range space, but in general we cannot say. The assumption of an even Orlicz integrand restricts. Can an analogue theory be developed for Orlicz cones generated by a non-even convex Orlicz integrand $\phi \colon \Om \times X \to \left[ 0, \i \right]$? Can one devise an associated subdifferential calculus for functionals on locally convex cones that, together with the Orlicz cones $L_\phi(\mu)$, provides a systematic, more powerful approach to treating unilateral constraints such as obstacle conditions on spaces of smooth functions or other strong asymmetries of the integrand?
	
	\section{Structure of the thesis}
	
	\paragraph{Generalized Gradient Flows}
	
	In section \ref{sec: Setting and Assumptions}, we introduce and discuss the setting in which we prove existence and stability for the generalized gradient flow
	\begin{equation} \label{eq: ggf}
		\p \phi_{t, u(t) } \left( u'(t) \right) + \p \EE_t(u(t) ) \ni 0 \quad \text{ for a.e. } t \in (0, T).
	\end{equation}
	In section \ref{sec: Approximation}, we set up the minimizing movement scheme to approximate (\ref{eq: ggf}) and prove a priori estimates that eventually lead to a proof of our first main result Theorem \ref{thm: main 1} on existence. Though general, this result excludes the most degenerate dissipation potentials. These are included in our second main Theorem \ref{thm: main 2} on existence and stability. Building upon Theorem \ref{thm: main 1}, it approximates generalized gradient flows within their own class on the time-continuous level, to handle systems whose degeneration is too delicate for an approximation via minimizing movements.
	
	In section \ref{sec: Application}, we apply Theorem \ref{thm: main 1} to obtain weak solutions to a doubly nonlinear inclusion of Allen-Cahn-Gurtin type. Drawing upon our theory of generalized Orlicz spaces, we treat nonlinearities with non-standard growth, encompassing non-homogeneity, full anisotropy, and infinite values. We show how to sidestep the need for an Orlicz-Sobolev inequality in the existence theory by using the locality of integral functionals. During our analysis, we expose a new criterion for an Orlicz-Sobolev space to have a sequentially compact ball in a suitable topology, transcending the need for uniformly superlinear growth in the gradient variable.
	
	\paragraph{Generalized Orlicz Spaces}
	
	In section \ref{sec: inf-int}, we settle technical preliminaries regarding convex conjugacy of integral functionals on spaces of non-separably vector-valued functions. This includes an interchange criterion between the infimum of an integral functional over such a space and taking the pointwise infimum of its integrand. As a by-product, a new subdifferential representation for integral functionals on a general class of non-separably vector-valued function spaces arises.
	
	In section \ref{sec: L}, we define our notion of an Orlicz integrand $\varphi \colon \Om \times X \to \left[ 0, + \i \right]$ on a measure space $\left( \Om, \AA, \mu \right)$ and a Banach space $X$, and introduce the pertaining Orlicz space $L_\varphi(\mu)$. We demonstrate some of its basic properties such as completeness. A central result of this section is that, for a finite measure $\mu$, the space $L_\varphi(\mu)$ continues to be an intermediate space of $L_1(\mu)$ and $L_\i(\mu)$ once we exclude a set of arbitrarily small measure. A large portion of our theory hinges on this inconspicuous property, which remarkably is valid without any uniformity assumption on the functions $x \mapsto \varphi(\om, x)$ as $\om$ ranges over $\Om$.
	
	In section \ref{sec: E}, we introduce the closure of simple functions in $L_\varphi(\mu)$ and demonstrate how it may be used to approximate elements of $L_\varphi(\mu)$ in a sense weaker than norm convergence. This is a technical preparation for the duality theory.
	
	In section \ref{sec: C}, we study the subspace $C_\varphi(\mu)$ of those elements having absolutely continuous norm and expose its relation with the maximal linear subspace in the domain of the modular functional
	$$
	I_\varphi(u) = \int_\Om \varphi(\om, u(\om) ) \, d \mu(\om).
	$$
	We equivalently describe when $C_\varphi(\mu)$ coincides with $L_\varphi(\mu)$ in terms of $\varphi$, a matter which we expose to be closely related to the reflexivity of $L_\varphi(\mu)$. We characterize separability of $C_\varphi(\mu)$ and reduce the matter of when $L_\varphi(\mu)$ is separable to the separability of $C_\varphi(\mu)$.
	
	In Section \ref{sec: duality}, we decompose the dual of $L_\varphi(\mu)$ into a direct sum of three fundamentally different types of functionals. We use this decomposition to describe the dual of both $C_\varphi(\mu)$ and $L_\varphi(\mu)$, identifying $C_\varphi(\mu)^*$ with the function component of $L_\varphi(\mu)^*$. Neither the Asplund property nor the separability of $X$ are needed for our most general representation of the function component. We equivalently describe the reflexivity of $L_\varphi(\mu)$ in terms of $\varphi$. Notably, our duality theory does not need $\varphi$ to be real-valued, a feature that is absent form the published literature even if $\dim X < \i$. The findings of Section \ref{sec: inf-int} - \ref{sec: duality} have been published in \cite{Ru}.
	
	In Section \ref{sec: weak compactness}, we equivalently describe compact subsets of $L_\varphi(\mu)$ and its dual function component in a class of weak topologies.
	
	\paragraph{Appendix}
	
	The appendix \ref{prt: Appendix} contains technical results on multimaps, integrands, hyperspace topologies, Young measures, subdifferential calculus, and inequalities. While many of these results are similar to known ones, we did not find their desired form in the literature directly, therefore proofs are given.
	
	\section{Notation and conventions}
	
	The following will be in force throughout the thesis. $\left( \Om, \AA, \mu \right)$ is a measure space with a non-trivial positive measure. $\AA_\mu$ is the completion of $\AA$ w.r.t. $\mu$ and $\bar{\mu}$ is the completion of $\mu$. The ring of sets having finite measure is $\AA_f$, the $\sigma$-ring of sets having $\sigma$-finite measure is $\AA_\sigma$, the $\sigma$-ring of countable unions of atoms and $\sigma$-finite sets is $\AA_{a\sigma}$. If $A \subset \Om$, then $\AA(A)$ denotes the trace $\sigma$-algebra of $\AA$ on $A$. The indicator of a set $S$ is $\chi_S$ with $\chi_S(a) = 1$ if $a \in S$ and $\chi_S(a) = 0$ otherwise. The indicator of a set $S$ in the sense of convex analysis is $I_S$ with $I_S(a) = 0$ if $a \in S$ and $I_S(a) = + \i$ otherwise. The power set of $S$ is $\PP(S)$. For a sequence $A_n \in \AA$, we write $\lim_n A_n = A$ to mean $\chi_{A_n} \to \chi_{A}$ almost everywhere. $\TT$ is a topological space, $\left( M, d \right)$ a metric space, and $X$ is a Banach space with dual space $X^*$. We write $\cl S$ to denote the closure of $S \subset \TT$. If $\TT$ is equipped with another topology $\dd$, we write $\TT_\dd$ to denote $\TT$ in the $\dd$-topology. We write $x_i \ddto x$ to denote convergence of a net $x_i$ to a limit $x$ in $\dd$. The Borel $\sigma$-algebra generated by the closed (or open) sets is denoted by $\BB(\TT)$. A ball in $M$ with centre $x$ and radius $r > 0$ is denoted by $B_r(x)$. If no radius is specified, then $r = 1$. If no centre is specified, then $x = 0$ if $M = X$. We write $B_W = B \cap W$ for $W \subset M$. We also combine both notations such as $B_{r, W}(x) = B_r(x) \cap W$. The system $\SS(\TT)$ are the separable subsets of $\TT$, $\CL \left( \TT \right)$ are the closed subsets of $\TT$, $\LS(\TT)$ are the lower semicontinuous, proper functions $f \colon \TT \to \left( - \i, \i \right]$ and $\Gamma \left( X \right)$ are the closed, convex, proper functions $g \colon X \to \left( -\i, \i \right]$. A function $\Om \to M$ is called simple if it is measurable and takes finitely many values. A simple function valued in a normed space is integrable iff it vanishes outside a set of finite measure. $\LL_0 \left( \Om ; M \right)$ are the strongly measurable functions $u \colon \Om \to M$. The function $u$ is strongly measurable iff there exists a sequence of simple functions $u_n \colon \Om \to M$ such that $u_n \to u$ pointwise as $n \to + \i$. This is equivalent to $u$ being the uniform limit of a sequence of measurable functions taking countably many values; or to $u$ being measurable and having a separable range \cite[Prop. 1.9]{CTV}. A function $u \colon \Om \to X^*$ is weak* measurable if, for any $x \in X$, the function $\om \mapsto \langle u(\om), x \rangle$ is measurable. For two locally convex spaces $\mathcal{V}$ and $\mathcal{W}$ in duality, we write $\ss \left( \mathcal{V}, \mathcal{W} \right)$ for the weak topology induced on $\mathcal{V}$ by $\mathcal{W}$ via the duality and likewise $\mm \left( \mathcal{V}, \mathcal{W} \right)$ for the Mackey topology. Let $T \in \left( 0, \i \right]$. We write $I = \cl \left[ 0, T \right)$ and denote by $\LL_I$ the $\sigma$-algebra of Lebesgue measurable subsets of $I$. If $L, M, N$ are sets and $f \colon L \times M \to N$ is a mapping, then for a fixed $\ell \in L$, we denote by $f_\ell$ the partial map $M \to N \colon m \mapsto f_\ell(m)$. For $S \subset L \times M$, we denote by $f_S$ the restriction of $f$ to $S$.
	
	We denote by $\langle \cdot, \cdot \rangle$ various dual pairings between locally convex spaces without specifying the duality if it is clear from the context. An analogous remark applies to the norm symbol $\| \cdot \|$. We use the symbols $C, c$ to denote various positive constants depending only on known quantities and whose concrete value may change from line to line. When in doubt, all vector spaces in this thesis are real.
	
	\newpage
	
	\chapter{Generalized gradient flows} \label{ch: GGF}
	
	\section{Setting and assumptions} \label{sec: Setting and Assumptions}
	
	In this section, we introduce assumptions that will be in force throughout the rest of Chapter \ref{ch: GGF} without further mention.
	
	\subsection{The spaces}
	
	Let $V, W, X, Y$ be Banach spaces such that $(V, W)$, $(V, X)$, and $(X, Y)$ form dual pairs where $W$ separates $V$, i.e.,
	$$
	\forall v \in V \, \exists w \in W \colon \langle v, w \rangle_{V, W} \ne 0,
	$$
	and $Y$ separates $X$ (but not necessarily vice versa), and $W, Y$ are separable in norm topology. Employing the notation $\ss$ for the weak topology of a locally convex dual pair, we call $\dd \coloneqq \ss(V, W)$ the dissipation topology and call $\dd^* \coloneqq \ss(X, Y)$ the dual dissipation topology. The Banach spaces $V, X$ are required to carry norms satisfying
	\begin{equation} \label{eq: dual nrms}
		\| v \|_V = \sup_{w \in B_W} \langle v, w \rangle_{V, W};
		\quad
		\| x \|_X = \sup_{y \in B_Y} \langle x, y \rangle_{X, Y}.
	\end{equation}
	Therefore, by means of the isometric embedding $V \to W^* \colon v \mapsto \langle v, \cdot \rangle_{V, W}$, the Banach space $V$ is a closed subspace of $W^*$. It is clear that $\dd$ agrees with the relative weak* topology under this identification. We require $V$ to be a (sequentially) weak*-closed subspace this way, which implies that $V$ is dual to the quotient space $W / V_\bot$ with $V_\bot = \left\{ w \in W \st \langle v, w \rangle_{V, W} = 0 \quad \forall v \in V \right\}$. In particular, $V_\dd$ can be identified with the dual of a separable Banach space in its weak* topology. Denoting by $w_n \in S_W$ a sequence that is dense in the unit sphere $S_W$, there hold the following properties that are essential to our analysis:
	\begin{subequations}
		\begin{equation} \label{eq: nrm V lsc}
			v_i \ddto v \implies \| v \|_V \le \liminf \| v_i \|_V;
		\end{equation}
		\begin{equation} \label{eq: B V cpt}
			\text{The unit ball } B_V \text{ is sequentially compact in } \dd;
		\end{equation}
		\begin{equation} \label{eq: dd metric}
			\text{The metric } d_\dd(v_1, v_2) = \sum_{n = 1}^\i 2^{- n} \left| \langle w_n, v_1 - v_2 \rangle \right|^2 \text{ induces $\dd$ on } B_V;
		\end{equation}
		\begin{equation} \label{eq: dd Suslin}
			\text{The space } V_\dd \text{ is a Suslin locally convex vector space.}
		\end{equation}
	\end{subequations}
	Let $V_1 \left( 0, T; V_\dd \right)$ be the normed vector space of functions $v \colon (0, T) \to V_\dd$ that are Lebesgue-$\BB(V_\dd)$-measurable and satisfy
	\begin{equation} \label{eq: fnt nrm}
		\| v \|_{V_1(0, T; V_\dd) } \coloneqq \int_0^T \| v(t) \|_V \, dt < \i.
	\end{equation}
	The measurability of $v$ is equivalent to each function $t \mapsto \langle w, v(t) \rangle$ being Lebesgue-$\BB(0, T)$-measurable since $W$ is separable. Analogous considerations and notations apply to $X_{\dd^*}$. We say that
	\begin{equation}
		v_{n_k} \to v \text{ in the biting sense of } V_1 \left( \mu; V_\dd \right)
	\end{equation}
	if there exists a non-increasing sequence of measurable sets $A_j$ with $\mu \left( \Om \setminus A_j \right) \to 0$ such that
	$$
	v_{n_k} \to v \text{ in } \ss \left( V_1 \left( \Om \setminus A_j; V_\dd \right) ; L_\i \left( \Om \setminus A_j ; W \right) \right) \text{ for every } j \in \N.
	$$
	As our last assumption on the spaces, we require that
	\begin{equation} \label{eq: V X mb}
		\text{The pairing } \langle \cdot, \cdot \rangle \colon V \times X \to \R \text{ is } \BB \left( V_\dd \times X_{\dd^*} \right) \text{-measurable.}
	\end{equation}
	
	For the energy functional $\EE \colon \cl \left[ 0, T \right) \times V \to \left( - \i, + \i \right]$, we assume that there exists $D \subset V$ such that
	\begin{equation} \label{eq: E_0}
		\begin{gathered}
			\dom (\EE) = \cl \left[ 0, T \right) \times D, \\
			(t, u) \mapsto \EE_t(u) \text{ is } \LL_I \otimes \BB \left( V_\dd \right) \text{-measurable}, \\
			\exists C_0 > 0 \colon \EE_t(u) \ge C_0 \quad \forall (t, u) \in \cl \left[ 0, T \right) \times D.
		\end{gathered}
	\end{equation}
	
	Throughout the paper, for $T_0 \in \cl \left[ 0, T \right)$ and $u \in D$, we use the notation
	\begin{equation} \label{eq: GG def}
		\GG_{T_0}(u) = \sup_{t \in \left[ 0, T_0 \right] } \EE_t(u).
	\end{equation}
	We consider on $D$ an auxiliary Hausdorff \emph{energy topology} $\ee$ that is sequentially concurrent with $\dd$, i.e.,
	\begin{equation} \label{eq: rh lph cncrrnt}
		u_n \eeto u, \, u_n \ddto v \implies u = v.
	\end{equation}
	
	We denote by $F \colon \cl \left[ 0, T \right) \times D \rightrightarrows X$ a time-dependent family of multimaps such that, for $t \in \cl \left[ 0, T \right)$, the set $F_t(u)$ is (in a suitable sense) a subdifferential of the functional $\EE_t$ at $u$ with respect to the dual pair $(V, X)$. We employ the usual multimap notations
	\begin{align*}
		\dom(F) = \left\{ (t, u) \in \cl \left[ 0, T \right) \times D \st F_t(u) \not = \emptyset \right\}, \\
		\graph(F) = \left\{ (t, u, \xi) \in \cl \left[ 0, T \right) \times D \times X \st \xi \in F_t(u) \right\}
	\end{align*}
	for the domain and the graph of the multimap $F$. We require that the set $\graph(F)$ belong to the Lebesgue-Borel $\sigma$-algebra $\LL_I \otimes \BB \left( D_\dd \times X_{\dd^*} \right)$.
	
	Regarding the dissipation potentials $\phi$ and $\phi^*$, let $\phi \colon \cl \left[ 0, T \right) \times D \times V \to \left[ 0, + \i \right]$ be a function such that
	\begin{equation} \label{eq: phi mb}
		\begin{gathered}
			\phi \text{ is measurable with respect to } \LL_I \otimes \BB \left( D_\dd \times V_\dd \right); \\
			\forall (t, u) \in \cl \left[ 0, T \right) \times D \quad v \mapsto \phi_{t, u}(v) \text{ is convex and } \phi_{t, u}(0) = 0.
		\end{gathered}
	\end{equation}
	We denote by $\phi^*_W \colon \cl \left[ 0, T \right) \times D \times W \to \left[ 0, + \i \right]$ the partial convex conjugate of $\phi$ in the last variable with respect to the pair $(V, W)$ and denote by $\phi^* \colon \cl \left[ 0, T \right) \times D \times X_{\dd^*} \to \left[ 0, + \i \right]$ the partial convex conjugate in the last variable with respect to the pair $(V, X)$. We require
	\begin{equation} \label{eq: phi* mb}
		\begin{gathered}
			\phi^* \text{ is measurable with respect to } \LL_I \otimes \BB \left( D_\dd \times X_{\dd^*} \right); \\
			\forall (t, u) \in \cl \left[ 0, T \right) \times D \quad \xi \mapsto \phi^*_{t, u}(\xi) \text{ is convex and } \phi^*_{t, u}(0) = 0.
		\end{gathered}
	\end{equation}
	
	\paragraph{Remark and examples of setting.}
	
	\begin{enumerate}
		
		\item The second of (\ref{eq: phi* mb}) is redundant because it follows from (\ref{eq: phi mb}). We record it for reference.
		
		\item If $u \colon \cl \left[ 0, T \right) \to D_\dd$ is a measurable curve, then $\left( t, v \right) \mapsto \phi_{t, u(t) }(v)$ is $\LL_I \otimes \BB \left( V_\dd \right)$-measurable. Hence, for every such $u$, the function $(t, w) \mapsto \phi^*_{W, t, u(t) }$ is $\LL_I \otimes \BB \left( W \right)$-measurable by \cite[Cor. VII.2]{CaVa}.
		
		\item If $V$ and $W$ are separable Banach spaces such that $W^* = V$, our assumptions on the pair $\left( V, W \right)$ are satisfied. Setting $X = V^*$ and $Y = V$, the canonical dual pairings fulfill all requirements. In particular, since $\BB(V) = \BB \left( V_\dd \right)$ in this situation, (\ref{eq: V X mb}) is true. This includes the case when $V$ is reflexive so that we may take $W = V^*$.
		
		\item A concrete example is given if $\Om \subset \R^d$ is an open set equipped with the Lebesgue measure $\lambda$ and $\varphi \colon \Om \times \R^d \to \left[ 0, \i \right]$ is a real-valued Orlicz integrand such that the conjugate integrand $\varphi^*$ is also real-valued. Then the Orlicz class $C_\varphi(\Om)$ is a separable Banach space whose dual is given by the Orlicz space $L_{\varphi^*}(\Om)$. The latter is separable if and only if $\varphi^* \in \Delta_2$ in the sense of Definition \ref{def: Delta2 cond}. To see this, combine Theorems \ref{thm: linear domain 2}, \ref{thm: C sep}, and \ref{thm: sep implies L = C}. We may then take $V = L_{\varphi^*}(\Om)$ and $W = C_\varphi(\Om)$ in the previous example.
		
		\item One could try to consider also the above situation when $\varphi^* \not \in \Delta_2$ with $V = L_{\varphi^*}(\Om)$, $W = C_\varphi(\Om)$, $X = L_\varphi(\Om)$, and $Y = C_{\varphi^*}(\Om)$. This leads to a (separated) locally convex duality between $V$ and $X$ via the standard integral pairing
		$$
		\langle \cdot, \cdot \rangle_{X, V} \colon L_\varphi(\Om) \times L_{\varphi^*}(\Om) \to \R \colon (x, v) \mapsto \int_\Om \langle x(\om), v(\om) \rangle \, d \lambda(\om).
		$$
		This example would fit our setting once (\ref{eq: V X mb}) were verified.
		
	\end{enumerate}
	
	\subsection{The dissipation potentials}
	
	We call $\phi$ an admissible dissipation potential if for every fixed $t \in \cl \left[ 0, T \right), u \in D, T_0 \in \cl \left[ 0, T \right), R > 0$, there hold the following properties:
	\begin{subequations}
		\begin{equation} \label{eq: phi.1}
			\phi_{t, u} \colon V \to \left[ 0, + \i \right] \text{ is lower semicontinuous and convex;}
		\end{equation}
		
		\begin{equation} \label{eq: phi.2.1}
			\phi_{t, u}(0) = 0, \quad \lim_{\| v \| \uparrow \i} \phi_{t, u}(v) = +\i;
		\end{equation}
		
		Uniform lower linear growth conditioned on the sublevels:
		\begin{alignat}{2}
			& \liminf_{\| v \|_V \to \i} \inf_{0 \le t \le T_0} \inf_{\GG_{T_0}(u) \le R } & \frac{\phi_{t, u}(v)}{\| v \|} > 0; \label{eq: phi.2.2.1} \\
			& \liminf_{\| \xi \|_X \to \i} \inf_{0 \le t \le T_0} \inf_{\GG_{T_0}(u) \le R } & \frac{\phi^*_{t, u}(\xi)}{\| \xi \|} > 0. \label{eq: phi.2.2.2}
		\end{alignat}
		
		Radial superlinearity:
		\begin{equation} \label{eq: phi.2.3}
			\begin{gathered}
				t_n \to t, \, u_n \eeto u, \, \GG_{T_0}(u_n) \le R, \, v_n \eeto v, \, \GG_{T_0}(v_n) \le R \implies \\
				\lim_{\theta \downarrow 0} \liminf_{n \to \i} \theta \phi_{t_n, u_n} \left( \tfrac{v_n}{\theta} \right) =
				\begin{cases}
					+ \i & \text{ if } v \in V \setminus \left\{ 0 \right\}, \\
					0 & \text{ if } v = 0;
				\end{cases}
			\end{gathered}
		\end{equation}
		
		\begin{equation} \label{eq: phi.2.4}
			\int_0^{T_0} \sup_{\GG_{T_0}(u) \le R} \phi^*_{W, t, u}(w) \, dt < +\i \quad \forall w \in W;
		\end{equation}
		
		\begin{equation} \label{eq: phi.3}
			\xi_1, \xi_2 \in \p \phi_{t, u}(v) \implies \phi^*_{t, u}(\xi_1) = \phi^*_{t, u}(\xi_2);
		\end{equation}
		
		For all $t \in \left[ 0, T_0 \right]$, we have
		\begin{align}
			\begin{gathered}
				u_n \eeto u, \, \| u_n \|_V \le R, \, \GG_{T_0}(u_n) \le R, \, v_n \ddto v \implies
				\\
				\liminf_n \phi_{t, u_n}(v_n) \ge \phi_{t, u}(v);
			\end{gathered}
			\label{eq: phi.4.1}
			\\
			\begin{gathered}
				u_n \eeto u, \, \| u_n \|_V \le R, \, \GG_{T_0}(u_n) \le R, \, \xi_n \dddto \xi \implies
				\\
				\liminf_n \phi^*_{t, u_n}(\xi_n) \ge \phi^*_{t, u}(\xi).
			\end{gathered}
			\label{eq: phi.4.2}
		\end{align}
	\end{subequations}
	\\
	
	\paragraph{Remark on assumptions.}
	
	\begin{enumerate}
		
		\item The second of (\ref{eq: phi.2.1}) is redundant since it is strictly weaker than (\ref{eq: phi.2.2.1}).
		
		\item As $\phi_{t, u}(0) = 0$ and $\phi_{t, u} \ge 0$, there holds $\phi^*_{t, u} \ge 0$ and $\phi^*_{t, u}(0) = 0$.
		
		\item Neither $\phi_{t, u}$ nor $\phi^*_{t, u}$ is required to be real-valued. In particular, the weakly superlinear growth (\ref{eq: phi.2.3}) does not force $\phi^*_{t, u}$ to be continuous. However, by (\ref{eq: phi.2.2.1}), the dual dissipation functional $\phi^*_{t, u}$ is continuous at the origin due to \cite[Prop. 3.2, Lem. 3.51(a)]{Pe}.
		
		\item Even though rate-independent dissipation is of course excluded by (\ref{eq: phi.2.3}), it may happen that the limit in (\ref{eq: phi.2.2.1}) is nevertheless finite, so that our analysis includes some dissipation functionals exhibiting linear lower growth.
		
		\item Assumption (\ref{eq: phi.2.3}) holds if
		$$
		\liminf_{n \to \i} \theta \phi_{t_n, u_n} \left( \tfrac{v_n}{\theta} \right) \ge \theta \phi_{t, u} \left( \tfrac{v}{\theta} \right) \text{ and } w \mapsto \phi^*_{W, t, u}(w) \text{ is real-valued}.
		$$
		For, if $\bar{w} \in W$ is fixed and $\theta > 0$, then
		$$
		\theta \phi_{t, u} \left( \tfrac{v}{\theta} \right)
		= \sup_{w \in W} \langle v, w \rangle - \theta \phi^*_{W, t, u}(w)
		\ge \langle v, \bar{w} \rangle - \theta \phi^*_{W, t, u} (\bar{w}) \to \langle v, \bar{w} \rangle
		$$
		as $\theta \to 0$. Now conclude (\ref{eq: phi.2.3}) by (\ref{eq: nrm V lsc}).
		
		\item As already observed in \cite[Rem. 2.1]{MRS}, the condition (\ref{eq: phi.3}) equivalently says that
		\begin{equation} \label{eq: smooth at one}
			\text{If } \p \phi_{t, u}(v) \not = \emptyset \text{, then } r \mapsto \phi_{t, u} \left( r v \right) \text{ is differentiable at } r = 1.
		\end{equation}
		
		\item We will remove (\ref{eq: phi.3}) and weaken (\ref{eq: phi.2.2.2}), (\ref{eq: phi.4.2}) in our second main existence and stability result Theorem \ref{thm: main 2}.
		
		\item We believe that the merely measurable dependence on the time variable for the dissipation potential has not been considered before in the context of (\ref{eq: DNI}).
		
		\item The integral in (\ref{eq: phi.2.4}) is meaningful since the integrand is measurable with respect to $\LL_I$. Alas, this is surprisingly complicated to check. First, $\phi$ is measurable by \ref{eq: phi mb}. Therefore, it remains measurable when restricted to the Suslin set $I \times \left\{ \GG_{T_0}(u) \le R \right\} \times V_\dd$ for any $R > 0$. Using \cite[Lem. VII.1]{CaVa}, we deduce from this that the epigraphical multimap of the restricted $\phi$ has a measurable graph. Arguing as for the implication (b) $\implies$ (c) of \cite[Thm. 14.8]{RocW}, using the Suslin projection theorem \cite[Thm. III.23]{CaVa}, this implies that pre-images remain measurable under the epigraphical multimap. Combining this with Lemma \ref{lem: inf meas normality}, we see that the restricted $\phi$ is infimally measurable, hence the claimed measurability of the integrand in (\ref{eq: phi.2.4}) follows directly from the definition of a convex conjugate.
		
	\end{enumerate}
	
	\subsection{The energy functional} \label{ss:enf}
	
	Besides (\ref{eq: E_0}), we assume for every $t \in \cl \left[ 0, T \right)$
	\\
	
	\begin{subequations} \label{eq: assu en only}
		\textbf{Lower semicontinuity}: The energy $\EE_t$ is sequentially lower $\ee$-semicontinuous on norm bounded sets:
		\begin{equation} \label{eq: lsc}
			\sup_n \left| u_n \right|_V < \i, \, u_n \eeto u \implies \EE_t(u) \le \liminf_n \EE_t \left( u_n \right);
		\end{equation}
		
		\textbf{Compactness}: Every norm bounded set contained in a sublevel set of $\EE_t$ is sequentially relatively $\ee$-compact:
		\begin{equation} \label{eq: crc}
			\begin{gathered}
				\sup_n \left| u_n \right|_V < \i, \, \sup_n \EE_t(u_n) < \i \implies \\ \left( u_n \right) \text{ has an } \ee \text{-convergent subsequence.}
			\end{gathered}
		\end{equation}
		By (\ref{eq: phi.2.1}), (\ref{eq: lsc}), and (\ref{eq: crc}), there exists $r^* > 0$ such that, for $R = \min\{ T - t, r^* \}$ and all $r \in \left[ 0, R \right)$, all $\theta \in \left( 0, r^* \right)$, all $t \in \cl \left[ 0, T \right)$ and every $u_0 \in D$, the map
		\begin{equation} \label{eq: coerc}
			u \mapsto \theta \phi_{t, u_0} \left( \frac{u - u_0}{\theta} \right) + \EE_{t + r}(u) \text{ is sequentially } \ee \text{-inf-compact.}
		\end{equation}
		A function is called inf-compact if its sublevel sets are compact.
		
		\textbf{Time dependence}: The energy $\EE_t$ is absolutely continuous in time:
		\begin{equation} \label{eq: absolute  continuity}
			\begin{gathered}
				\exists f_1 \in L^1_{\text{loc} }\left( \cl \left[ 0, T \right) \right) \colon \forall u \in D \, \forall s, t \in \cl \left[ 0, T \right) \colon
				\\
				\left| \EE_s(u) - \EE_t(u) \right| \le\EE_t(u) \int_{\left[ s, t \right] } f_1(r) \, dr.
			\end{gathered}
		\end{equation}
		
		Let there exist an $\LL_I \otimes \BB \left( D_\dd \times X_{\dd^*} \right)$-measurable function $P \colon \graph(F) \to \R$ and, for every $T_0 \in \cl \left[ 0, T \right)$, let there exists a locally integrable function $f_2 \in L^1_{\text{loc} } \left( \cl \left[ 0, T \right) \right)$ such that
		\begin{equation} \label{eq: radially time differentiable}
			\begin{gathered}
				\forall (t, u, \xi) \in \graph(F) \colon
				\\
				\liminf_{h \downarrow 0} \frac{\EE_{t + h}(u) - \EE_t(u)}{h} \le P_t(u, \xi) \le \GG_{T_0}(u) f_2(t).
			\end{gathered}
		\end{equation}
	\end{subequations}
	
	We record that (\ref{eq: absolute  continuity}) implies
	\begin{equation} \label{eq: Gron cons}
		\EE_t(u) \le \EE_s(u) \exp \left( \int_{\left[ s, t \right] } f_1(r) \, dr \right)
	\end{equation}
	by the Gronwall inequality. In particular,
	\begin{equation} \label{eq: G controlled}
		\forall T_0 > 0 \, \exists C_3 > 0 \, \forall u \in D \colon \GG_{T_0}(u) \le C_3 \inf_{0 \le t \le T_0} \EE_t(u).
	\end{equation}
	
	\textbf{Sum rule}: If, for $u_0 \in D$, $r \in \left[ 0, R \right)$, $\theta \in \left( 0, r^* \right)$, and $t \in \cl \left[ 0, T \right)$, the point $\bar{u}$ minimizes the function in (\ref{eq: coerc}), then let $\bar{u}$ satisfy the Euler-Lagrange inclusion
	\begin{equation} \label{eq: sum rule}
		\exists \xi \in F_{t + r} \left( \bar{u} \right) \colon - \xi \in \p \phi_{t, u_0} \left( \frac{\bar{u} - u_0}{\theta} \right).
	\end{equation}
	
	\textbf{Chain rule inequality} for the tuple $\left( V_\dd, X_{\dd^*},\EE, \phi, \langle \cdot, \cdot \rangle_{X, V}, F, P \right)$: For every $S \in \left( 0, T \right)$ and every weakly differentiable curve $u \colon \left( 0, S \right) \to V_\dd$ with $u' \in V_1 \left( 0, S; V_\dd \right)$ and for every curve $\xi$ such that $\xi \in V_1 \left( 0, S; X_{\dd^*} \right)$, let there hold the following implication:
	\begin{subequations}
		\begin{equation} \label{eq: ch ru}
			\begin{aligned}
				&
				\begin{gathered}
					\sup_{0 \le t \le S} \EE_t(u(t) ) < + \i, \quad \xi(t) \in F_t(u(t) ) \quad \text{ for a.e. } t \in \left( 0, S \right), \\
					\int_0^S \phi_{t, u(t) } \left( u'(t) \right) + \phi^*_{t, u(t) } \left( -\xi(t) \right) \, dt < + \i \implies
				\end{gathered}
				\\
				&
				\begin{gathered}
					\frac{d}{dt} \EE_t(u(t) ) \ge \langle \xi(t), u'(t) \rangle_{X, V} + P_t(u(t), \xi(t) ) \quad \text{ in } \DD' \left( 0, S \right).
				\end{gathered}
			\end{aligned}
		\end{equation}
		For later reference, we also record the weaker chain rule
		\begin{equation} \label{eq: ch ru wkr}
			\begin{aligned}
				&
				\begin{gathered}
					\sup_{0 \le t \le S} \EE_t(u(t) ) < + \i, \quad \xi(t) \in F_t(u(t) ) \quad \text{ for a.e. } t \in \left( 0, S \right), \\
					\int_0^S \phi_{t, u(t) } \left( u'(t) \right) + \phi^*_{t, u(t) } \left( -\xi(t) \right) \, dt < + \i \implies
				\end{gathered}
				\\
				&
				\begin{gathered}
					\text{The function } t \mapsto \EE_t \left( u(t) \right) \text{ is of locally bounded variation and } \\
					\frac{d}{dt} \EE_t(u(t) ) \ge \langle \xi(t), u'(t) \rangle_{X, V} + P_t(u(t), \xi(t) ) \quad \text{ for a.e. } t \in (0, S).
				\end{gathered}
			\end{aligned}
		\end{equation}
	\end{subequations}
	
	\textbf{Closedness implication}: Given $t \in \cl \left[ 0, T \right)$ and sequences $u_n \in V$, $\xi_n \in F_t(u_n)$, $p_n = P_t(u_n, \xi_n)$ such that
	\begin{subequations} \label{seq: cl cnd}
		\begin{equation} \label{eq: cc ssmptn}
			u_n \eeto u , \quad \xi_n \dddto \xi, \quad p_n \to p \text{ in } \R,
			\quad \limsup_n | u_n | + \EE_t(u_n)
		\end{equation}
		let there hold
		\begin{equation} \label{eq: cc cnclsn}
			(t, u) \in \dom(F), \quad \xi \in F_t(u), \quad p \le P_t(u, \xi).
		\end{equation}
		For later reference, we also record the stronger closedness implication
		\begin{equation} \label{eq: cl str impl}
			\begin{gathered}
				u_n \eeto u , \quad \xi_n \dddto \xi, \quad p_n \to p \text{ in } \R,
				\quad \limsup_n | u_n | + \EE_t(u_n) < \i \implies \\
				(t, u) \in \dom(F), \quad \xi \in F_t(u), \quad p \le P_t(u, \xi), \quad \lim_n \EE_t(u_n) = \EE_t(u).
			\end{gathered}
		\end{equation}
		We call (\ref{eq: cl str impl}) the continuous closedness implication.
	\end{subequations}
	\\
	
	\paragraph{Remark on assumptions}
	
	\begin{enumerate}
		
		\item Our assumptions on the energy $\EE$ constitute a time-dependent variant of the lower semicontinuity and compactness assumptions imposed on energies in the theory of metric gradient flows, see \cite{AGS}. The time-dependence we permit for $\EE$ is similar to that in \cite{MRS}.
		
		\item If $F_t$ is given by an operator $\p \colon \left( - \i, \i \right]^V \times V \rightrightarrows X$ such that $\p$ agrees with the Fenchel-Moreau subdifferential of convex analysis on $\Gamma \left( V_\dd \right)$ and $\p f \left( \bar{v} \right) \ni 0$ whenever $f \colon V \to \left( - \i, \i \right]$ has a minimizer at $\bar{v} \in V$, then a sufficient condition for (\ref{eq: sum rule}) is
		\begin{equation} \label{eq: sm rl sf cnd}
			\begin{gathered}
				\p \left( \theta \phi_{t, u_0} \left( \frac{\cdot - u_0}{\theta} \right) + \EE_t(\cdot) \right) \left( \bar{u} \right)
				\\
				\subset \cl_{\text{seq } \dd^*} \left( \p \phi_{t, u_0} \left( \frac{\bar{u} - u_0}{\theta} \right) + \p \EE_t \left( \bar{u} \right) \right).
			\end{gathered}
		\end{equation}
		The claim is immediate if we forego the sequential $\dd^*$-closure on the right side of (\ref{eq: sm rl sf cnd}). Otherwise, it follows by the closedness implication (\ref{seq: cl cnd}). An equivalent description of (\ref{eq: sm rl sf cnd}) for the radial subdifferential of locally convex\footnote{A function $f$ is called locally convex at a point $x_0 \in \dom(f)$ if the radial derivative $\lim_{h \downarrow 0} \frac{f(x_0 + hv) - f(x_0)}{h}$ exists for every direction $v$ and is subadditive in $v$.} functions can be found in \cite[Thm. 3.3]{Ro}. It can be extended to energies with subadditive Dini subderivative, see Theorem \ref{thm: sum rule}. Note in this regard that the sequential closure in (\ref{eq: sm rl sf cnd}) coincides with the topological one by the Krein-Shmulyan Theorem \cite[Thm. 1.11]{Pe} or \cite[Thm. VIII.3.15]{W} if the sum of subdifferentials is convex. This always happens for the radial subdifferential because it is convex and Minkowski sums remain convex.
		
		\item A generalization of \cite[Prop. 4.2]{MRS} continues to hold in our setting. The following lemma crucially uses the assumption of a globally finite dissipation potential, which contrasts with the sum rule of convex analysis, where it suffices if one addend has a point of continuity in the joint domain to yield a global sum rule.
		\begin{lemma}
			Let $V$ be a Frechet or Hadamard smooth Banach space in the sense of \cite[§4]{Pe} and denote by $\p$ the Frechet or Hadamard subdifferential, respectively. Let $X = V^*$ and $\dd^* = \sigma \left( V, V^* \right)$. Let the function $v \mapsto \phi_{t, u}(v)$ take real values for every $(t, u) \in \cl \left[ 0, T \right) \times D$, and $\EE \colon \cl \left[ 0, T \right) \times V \to \left( - \i, \i \right]$ be such that $v \mapsto \EE_t(v)$ is lower semicontinuous for all $t$ and $\left( \EE, F \right)$ complies with (\ref{seq: cl cnd}) in the weak form (\ref{eq: cc cnclsn}). Moreover, suppose that
			$$
			\forall (t, u) \in \cl \left[ 0, T \right) \times D \quad \p \EE_t(u) \subset F_t(u).
			$$
			Then (\ref{eq: sum rule}) is true.
		\end{lemma}
		
		\begin{proof}
			The proof is a straight-forward adaption of \cite[Prop. 4.2]{MRS} upon replacing \cite[Lem. 2.32]{Mo} in their argument by \cite[Cor. 4.64]{Pe}.
		\end{proof}
		
		\item We offer a sufficient condition for the sum rule if $\phi_{t, u_0}$ is a dissipation potential as above that may be unbounded but continuous at the origin. Suppose that $\EE_t$ has a convex domain $D$, and for every $t \in \cl \left[ 0, T \right)$, the Dini-subderivative $\left( \EE_t \right)'_D(\bar{u}; \cdot)$ is subadditive. Moreover, let $\EE_t$ be $\| \cdot \|_V$-$\om$-Dini-subderivable at $\bar{u}$ in the sense introduced in Theorem \ref{thm: smcnvx} for a time-dependent family of moduli $\om_t \colon V \times V \to \left[ 0, + \i \right]$ such that
		\begin{equation} \label{eq: om cnd}
			\liminf_{h \downarrow 0} \om_t \left( \bar{u} + h v, \bar{u} \right) = 0 \quad \forall (t, v) \in \cl \left[ 0, T \right) \times V.
		\end{equation}
		Then (\ref{eq: sum rule}) is true if $F_t(\bar{u})$ contains the Dini subdifferential. This follows by Corollary \ref{cor: sum rule}.
		
		\item Lemma \ref{lem: ch rl} provides a sufficient condition for the chain rule inequality (\ref{eq: ch ru}) to hold. Further such conditions are contained in \cite{MR, RMS, RS} and may be extended to the present setting. Another chain rule is proposed in \cite[Prop. 2.4]{MRS}. However, we assess that the proof has a gap that currently prevents it from being true in the generality claimed. Cf. our discussion in Subsection \ref{ssec: chain rule} for details. We thank R. Rossi for interesting discussions on this topic.
		
		\item The conclusion in (\ref{eq: ch ru}) implies that the function $t \mapsto \EE_t \left( u(t) \right)$ is of bounded variation. To see this, remember that a signed distribution is a Radon measure. The function need not be absolutely continuous.
		
		\item Lemma \ref{lem: sd clsdnss} gives a sufficient condition for the continuous version of (\ref{seq: cl cnd}).
		
		\item By (\ref{eq: crc}), a sequence $u_n$ such that $\sup_n \left| u_n \right|_V < \i$ and $\sup_n \EE_t(u_n) < \i$ converges in $\ee$ iff it converges in $\dd$ by (\ref{eq: B V cpt}). In particular, the sequential $\ee$-convergence is induced by a Suslin topology on bounded subsets of the sublevel sets of $\EE_t$ and $\GG$ due to (\ref{eq: dd Suslin}). \label{en: rem en 8}
		
	\end{enumerate}
	
	\subsection{Notions of solution}
	
	Our notions of solutions for (\ref{eq: DNI}) are contained in the following
	
	\begin{definition} \label{def: sol}
		Let $\EE \colon \cl \left[ 0, T \right) \times V \to \left( - \i, + \i \right]$ be an energy functional satisfying (\ref{eq: E_0}) and $\phi \colon \cl \left[ 0, T \right) \times D \times V \to \left[0 , + \i \right]$ be a dissipation potential fulfilling (\ref{eq: phi mb}), (\ref{eq: phi* mb}). If the chain rule (\ref{eq: ch ru}) holds, then a triple
		$$
		\left( u, u', \xi \right) \in \bigcap_{T_0 \in \cl \left[ 0, T \right) } L_\i \left( 0, T_0; V \right) \times V_1 \left( 0, T_0; V_\dd \right) \times V_1 \left( 0, T_0; V^* \right)
		$$
		is said to be an energy solution to the generalized gradient system
		$$
		\left( V_\dd, X_{\dd^*}, \EE, \phi, \langle \cdot, \cdot \rangle_{X, V}, F, P \right)
		$$
		if
		\begin{equation} \label{eq: DNL inc}
			\xi(t) \in F_t \left( u(t) \right), \quad \p \phi_{t, u(t) } \left( u'(t) \right) + \xi(t) \ni 0 \quad \text{for a.e. } t \in \left( 0, T \right)
		\end{equation}
		and
		\begin{equation} \label{eq: en id 2}
			\begin{aligned}
				\int_s^t \phi_{r, u(r) } \left( u'(r) \right) + \phi^*_{r, u(r) } \left( - \xi(r) \right) \, dr + \EE_t \left( u(t) \right) \\
				= \EE_s \left( u(s) \right) + \int_s^t P_r \left( u(r), \xi(r) \right) \, dr \quad \forall s, t \in \cl \left[ 0, T \right).
			\end{aligned}
		\end{equation}
		If the chain rule (\ref{eq: ch ru wkr}) holds, then the triple $\left( u, u', \xi \right)$ is said to be a Lyapunov solution to $\left( V_\dd, X_{\dd^*}, \EE, \phi, \langle \cdot, \cdot \rangle_{X, V}, F, P \right)$ if (\ref{eq: DNL inc}) and
		\begin{equation} \label{eq: en id 3}
			\begin{aligned}
				\int_s^t \phi_{r, u(r) } \left( u'(r) \right) + \phi^*_{r, u(r) } \left( - \xi(r) \right) \, dr + \EE_t \left( u(t) \right) \\
				\le \EE_s \left( u(s) \right) + \int_s^t P_r \left( u(r), \xi(r) \right) \, dr \quad \forall s, t \in \cl \left[ 0, T \right) \colon s \le t.
			\end{aligned}
		\end{equation}
		We shortly say that $u$ or $(u, \xi)$ is an energy solution or a Lyapunov solution if there exist $u'$ and $\xi$ as above such that $\left( u, u', \xi \right)$ solves
		$$
		\left( V_\dd, X_{\dd^*}, \EE, \phi, \langle \cdot, \cdot \rangle_{X, V}, F, P \right)
		$$
		in the appropriate sense.
	\end{definition}
	
	\paragraph{Remark on definition.} The notion of a Lyapunov solution shows up in \cite{MR} for gradient flows, where the idea is attributed to considerations by Luckhaus \cite{Lu} for a particular model. The possibility of introducing the concept to generalized gradient flows is mentioned in \cite{MRS}.
	
	\subsection{Example for the chain rule assumption} \label{ssec: chain rule}
	
	In this subsection, we present a class of examples that satisfy the assumption of the chain rule inequality. The following Lemma \ref{lem: ch rl} generalizes a result by A. Mielke, R. Rossi, and G. Savaré \cite[Prop. 2.4]{MRS}. We also address a gap in the original proof, which made use of an analogy to the metric theory in \cite[Thm. 1.2.5]{AGS}. While this argument is valid when the dissipation potential is homogeneous, we propose a fix that adheres more closely to the original assumption. This fix allows us to recover the entire chain rule identity \cite[Prop. 2.4]{MRS} in the case of an even dissipation potential. In general, our assumptions guarantee only a one-sided control, which may lead to a genuine inequality. However, this poses no issue when applying the result to generalized gradient flows. In this section, we consider a (separable) Banach space $U$ as our underlying framework. To begin our investigation, we present two preliminary results.
	
	\begin{proposition} \label{pr: gnrlzd Schff}
		If $\mu$ is a measure and $f_n, f$ are non-negative measurable functions such that $f \le \liminf_n f_n$ a.e. and $\lim_n \| f_n \|_{L_1} = \| f \|_{L_1} < \i$, then $\lim_n \| f - f_n \|_{L_1} = 0$.
	\end{proposition}
	
	\begin{proof}
		By the Fatou lemma,
		\begin{equation*}
			2 \int f \, d \mu
			\le \liminf_n \int f + f_n - \left| f - f_n \right| \, d \mu
			= 2 \int f \, d \mu - \limsup_n \| f - f_n \|_{L_1}. \qedhere
		\end{equation*}
	\end{proof}
	
	\begin{proposition} \label{pr: dq prprts}
		Let $a, b \in \R$, $\XX$ be a Banach space and $u \in W^{1, 1} \left( a, b; \XX \right)$. Suppose there exists a convex, lower semicontinuous function $\psi \colon \XX \mapsto \left[0, \i \right]$ such that $\psi(0) = 0$ and
		$$
		\int_a^b \psi \left( u'(t) \right) \, d t < \i.
		$$
		Consider $u'$ extended to all of $\R$ by zero. Then, for all $h \in \R$, there holds
		\begin{equation} \label{eq: dq dmntn}
			\int_a^b \psi \left( \tfrac{u(t + h) - u(t)}{h} \right) \, dt \le \int_a^b \psi \left( u'(t) \right) \, dt.
		\end{equation}
		In particular, the integrand on the left-hand side of (\ref{eq: dq dmntn}) converges to the integrand on the right-hand side strongly in $L_1 \left(a, b \right)$ as $h \downarrow 0$.
	\end{proposition}
	
	\begin{proof}
		By the Jensen inequality,
		$$
		\psi \left( \tfrac{u(t + h) - u(t)}{h} \right) \le \dashint_t^{t + h} \psi \left( u'(s) \right) \, ds = \dashint_0^h \psi(u'(s + t)) \, ds.
		$$
		Integrating over $\left( a, b \right)$ and invoking the Fubini theorem yields
		$$
		\int_a^b \psi \left( \tfrac{u(t + h) - u(t)}{h} \right) \, dt \le \int_a^b \psi \left( u'(t) \right) \, dt
		$$
		because $u'$ vanishes outside of $\left( a, b \right)$. As the difference quotient of $u$ converges to $u'$ strongly in $L_1 \left( a, b; \XX \right)$, we conclude the addendum by lower semicontinuity of $\psi$ and Proposition \ref{pr: gnrlzd Schff}.
	\end{proof}
	
	Let $\phi \in \Gamma(U)$ be a non-negative function satisfying
	\begin{equation} \label{eq: cnt crcv}
		\lim_{\| u \| \downarrow 0} \phi(u) = 0; \qquad \lim_{\| u \| \uparrow \i} \phi(u) = + \i.
	\end{equation}
	The conjugate $\phi^* \in \Gamma \left( U^* \right)$ also possesses these properties by Lemma \ref{lem: gen iff conj}. We introduce the functionals
	\begin{align*}
		P_\phi(u) & \coloneqq \inf \left\{ \alpha > 0 \st \int_0^T \phi \left( \alpha^{-1} u(t) \right) \, dt \le 1 \right\}, \\
		Q_{\phi^*}(v) & \coloneqq \inf \left\{ \alpha > 0 \st \int_0^T \phi^* \left( \alpha^{-1} v(t) \right) \, dt \le 1 \right\},
	\end{align*}
	for strongly measurable curves $u \colon \left( 0, T \right) \to U$ and weak* measurable ones $v \colon \left( 0, T \right) \to U^*$. The function $P_\phi$ is the Minkowski functional of the $1$-sublevel set of the integral functional
	$$
	u \mapsto \int_0^T \phi \left( u(t) \right) \, dt, \quad u \in \LL_0 \left( 0, T; U \right).
	$$
	Therefore, $P_\phi(u) < + \i$ if and only if there exists $\alpha > 0$ such that
	\begin{equation} \label{eq: fnt}
		\int_0^T \phi \left( \alpha^{-1} u(t) \right) \, dt < + \i.
	\end{equation}
	In particular, if $P_\phi(u) < + \i$, then (\ref{eq: fnt}) is true whenever $\alpha > 0$ is sufficiently large. An analogous comment applies for the function $Q_{\phi^*}$.
	
	\begin{lemma} \label{lem: ch rl}
		Let $\phi \colon U \to \left[ 0, \i \right]$ and $\phi^* \colon U^* \to \left[ 0, \i \right]$ be as described above. Given for $T > 0$ an energy $\EE_t \colon \left[ 0, T \right] \times U \to \left[ 0, + \i \right]$ as in Section \ref{sec: Setting and Assumptions}, let $u \in W^{1, 1} \left( 0, T ; U \right)$ and $\xi \in V_1 \left( 0, T; U^* \right)$ satisfy
		\begin{gather}
			\sup_{0 \le t \le T} \EE_t(u(t) ) < \i, \label{eq: nrm fnt 1} \\
			P_\phi \left( u' \right) + Q_{\phi^*} \left( - \xi \right) < \i. \label{eq: nrm fnt 2}
		\end{gather}
		Asssume that, for every $R > 0$ and $D_R = \left\{ \GG_T(u) \le R \right\}$, there exists a map $\om_R \colon \left[ 0, T \right] \times D_R \times D_R \to \left[ 0, \i \right]$ and a null set $N \subset \left( 0, T \right)$ such that, for all $t \in \left( 0, T \right) \setminus N$, there holds
		\begin{equation} \label{eq: lwr tylr xpsn}
			\begin{gathered}
				\left( t, u, v \right) \mapsto \om_{t, R} \left( u, v \right) \text{ is locally bounded above on } \left\{ u = v \right\}; \\
				\EE_t(v) - \EE_t(u) -\langle \xi(t), v - u \rangle \ge - \om_{t, R}(u, v) \| v - u \| \quad \forall u, v \in D_R.
			\end{gathered}
		\end{equation}
		Then, the distributional derivative $\frac{d}{dt} \EE_t \left( u(t) \right)$ is a Radon measure on $\left( 0, T \right)$, whose negative part has an integrable Lebesgue density. Furthermore, if $\xi(t)$ belongs to the Dini-Hadamard subdifferential $\p_H \EE_t(u(t) )$ for a.e. $t \in \left( 0, T\right)$, then
		\begin{equation} \label{eq: ch ru ineq}
			\frac{d}{dt} \EE_t \left( u(t) \right) \ge \langle \xi(t), u'(t) \rangle + \p_t \EE_t \left( u(t) \right) \text{ in } \DD' \left( 0, T \right).
		\end{equation}
		Moreover, if $\phi$ or equivalently $\phi^*$ is even, then $\frac{d}{dt} \EE_t \left( u(t) \right)$ belongs to $L_1 \left( 0, T \right)$. If, in addition, $\xi(t)$ belongs to the Fréchet subdifferential $\p_F \EE_t(u(t) )$ for a.e. $t \in \left( 0, T\right)$, then (\ref{eq: ch ru ineq}) holds with equality.
	\end{lemma}
	
	\begin{proof}
		Let $0 \le s, t \le T$. By (\ref{eq: G controlled}) and (\ref{eq: nrm fnt 1}), there exists $R > 0$ such that $\GG(u(t) ) \le R$ for all $0 \le t \le T$. To prove that the function $t \mapsto \EE_t(u(t) )$ has a Radon measure derivative with the claimed properties, observe first by (\ref{eq: absolute  continuity}), (\ref{eq: G controlled}) and (\ref{eq: nrm fnt 1}) that
		\begin{equation} \label{eq: obs 1}
			\begin{gathered}
				\EE_t(u(t) ) - \EE_s(u(t) ) = \int_s^t \p_t \EE_r (u(t) ) \, dr, \\
				\left| \int_s^t \p_t \EE_r (u(t) ) \, dr \right| \le \GG(u(t) ) \exp \left( \int_{\left[ s, t \right] } f_1(r) \, dr \right)
				\\
				\le R \exp \left( \int_{\left[ s, t \right] } f_1(r) \, dr \right).
			\end{gathered}
		\end{equation}
		Second, in view of (\ref{eq: lwr tylr xpsn}), for $\e > 0$, there holds
		\begin{equation} \label{eq: obs 2}
			\begin{aligned}
				\tfrac{\EE_s(u(t) ) - \EE_s(u(s) )}{\left| t - s \right|}
				& \ge \left\langle \xi(s), \tfrac{u(t) - u(s)}{\left| t - s \right|} \right\rangle - \om_{s, R} \left( u(s), u(t) \right) \left| \tfrac{u(t) - u(s)}{\left| t - s \right|} \right| \\
				& \ge - \tfrac{1}{\e^2} \left\{ \phi^* \left( - \e \xi(s) \right) + \phi \left( \e \tfrac{u(t) - u(s)}{\left| t - s \right|} \right) \right\} \\
				& - \om_{s, R} \left( u(s), u(t) \right) \left| \tfrac{u(t) - u(s)}{\left| t - s \right|} \right|.
			\end{aligned}
		\end{equation}
		Putting (\ref{eq: obs 1}), (\ref{eq: obs 2}) together yields for $h > 0$
		\begin{equation} \label{eq: ntgrbl mnrnt}
			\begin{aligned}
				\tfrac{\EE_{t + h}(u(t + h) ) - \EE_t(u(t) )}{h}
				& \ge \dashint_t^{t + h} \p_t \EE_r (u(t) ) \, dr \\
				& - \tfrac{1}{\e^2} \left\{ \phi^* \left( - \e \xi(t) \right) + \phi \left( \e \tfrac{u(t + h) - u(t)}{h} \right) \right\} \\
				& - \om_{t, R} \left( u(t), u(t + h) \right) \left| \tfrac{u(t + h) - u(t)}{h} \right|.
			\end{aligned}
		\end{equation}
		We claim that, choosing $\e > 0$ sufficiently small, the first difference quotient in (\ref{eq: ntgrbl mnrnt}) has an $L_1(0, T)$-convergent minorant. First, the average integral has an $L_1(0, T)$-convergent majorant by the second line of (\ref{eq: obs 1}). Next, invoking Proposition \ref{pr: dq prprts} and (\ref{eq: nrm fnt 2}), we see that the third term in (\ref{eq: ntgrbl mnrnt}) converges in $L_1(0, T)$. Furthermore, $\om_{t, R} \left( u(t), u(t + h) \right) \le C$ for all $\left| h \right|$ sufficiently small by the first of (\ref{eq: lwr tylr xpsn}), so that the fourth term has an $L_1(0, T)$-convergent minorant. Consequently, extracting a subsequence $h_n \downarrow 0$ along which there exists a weak limit
		$$
		\lim_n \left( \tfrac{\EE_{t + h_n}(u(t + h_n) ) - \EE_t(u(t) )}{h_n} \right)^- \eqqcolon E'_{-}(t) \in L_1 \left( 0, T \right),
		$$
		we find
		\begin{equation} \label{eq: nng dstr}
			\frac{d}{dt} \EE_t(u(t) ) + E'_{-}(t) \ge 0 \quad \text{in } \DD' \left( 0, T \right).
		\end{equation}
		Therefore, $\frac{d}{dt} \EE_t(u(t) )$ is a signed distribution hence a Radon measure by \cite[pp. 28-29]{LSch}. As taking the positive part of measures is order preserving, (\ref{eq: nng dstr}) implies that
		$$
		E'_{-}(t) \ge \left\{ \frac{d}{dt} \EE_t(u(t) ) \right\}^- \ge 0 \quad \text{in } \DD' \left( 0, T \right),
		$$
		whence the negative part has an integrable Lebesgue density by the Radon-Nikodym theorem \cite[Thm. 1.101]{FoLe}.
		
		(\ref{eq: ch ru ineq}): Let $\varphi \in C_c \left( 0, T \right)$ with $\varphi \ge 0$ and observe that, by a suitable variant of the Fatou lemma with a convergent minorant of the a.e. converging integrand and by the definition of $\p_H$, there holds
		\begin{align*}
			\lim_{h \downarrow 0} \int_0^T \tfrac{\EE_t(u(t + h) ) - \EE_t(u(t) )}{h} \cdot \varphi(t) \, dt
			& \ge \int_0^T \liminf_{h \downarrow 0} \tfrac{\EE_{t + h}(u(t + h) ) - \EE_t(u(t) )}{h} \cdot \varphi(t) \, dt \\
			& \ge \int_0^T \left( \EE_t \right)'_H \left( u(t); u'(t) \right) \cdot \varphi(t) \, dt \\
			& \ge \int_0^T \langle \xi(t), u'(t) \rangle \cdot \varphi(t) \, dt,
		\end{align*}
		where we used for the second inequality that $u$ is weakly differentiable so that $u(t + h) = u(t) + u'(t) h + o(h)$ for a.e. $t \in \left( 0, T \right)$. We obtain (\ref{eq: ch ru ineq}) from this upon noting that
		$$
		\lim_{h \downarrow 0} \int_0^T \dashint_t^{t + h} \p_t \EE_r \left( u(t) \right) \cdot \varphi(t) \, dr \, dt = \int_0^T \p_t \EE_t \left( u(t) \right) \cdot \varphi(t) \, dt.
		$$
		
		Addendum if $\phi$ is even: In this case, (\ref{eq: obs 2}) implies that, for $h < 0$, we have
		\begin{equation} \label{eq: rv nqlty}
			\begin{aligned}
				\tfrac{\EE_{t + h}(u(t + h) ) - \EE_t(u(t) )}{h}
				& \le \dashint_{t + h}^t \p_t \EE_r (u(r) ) \, dr + \langle \xi(t), \tfrac{u(t + h) - u(t)}{h} \rangle \\
				& \le \dashint_{t + h}^t \p_t \EE_r (u(r) ) \, dr \\
				& + \tfrac{1}{\e^2} \left\{ \phi^* \left( \e \xi(t) \right) + \phi \left( \e \tfrac{u(t + h) - u(t)}{h} \right) \right\} \\
				& + \om_{t, R} \left( u(t), u(t + h) \right) \left| \tfrac{u(t + h) - u(t)}{h} \right|
			\end{aligned}
		\end{equation}
		so that the difference quotient has an integrable majorant hence converges strongly in $L_1 \left( 0, T \right)$ to the distributional derivative $\frac{d}{dt} \EE_t(u(t) )$. If moreover $\xi(t) \in \p_F \EE_t(u(t) )$ for a.e. $t \in \left( 0, T \right)$, then, by the first line of (\ref{eq: rv nqlty}), the difference quotient is majorized by a function that converges a.e. to the right-hand side of (\ref{eq: obs 2}), hence equality obtains.
	\end{proof}
	
	\paragraph{Remark on Lemma \ref{lem: ch rl}}
	
	\begin{enumerate}
		
		\item Besides extending \cite[Prop. 2.4]{MRS} to non-reflexive state spaces and more general subdifferentials, Lemma \ref{lem: ch rl} relaxes the condition on the error term of subdifferentiability, which need not vanish uniformly. Our application in Section \ref{sec: Application} provides an example of how the latter is a more natural assumption than continuous differentiability if $V$ is an Orlicz space that is not reflexive but dual to a separable Orlicz class.
		
		\item By considering (\ref{eq: ch ru ineq}) in $\DD'(0, T)$ instead of in an almost everywhere sense, we manage to obtain energy solutions in the sense of Definition \ref{def: sol} even if the function $t \mapsto \EE_t \left( u(t) \right)$ is only known to be of bounded variation initially. Of course, once an energy solution has been found, equality obtains in (\ref{eq: ch ru ineq}) so that $t \mapsto \EE_t \left( u(t) \right)$ will be absolutely continuous.
		
		\item In the case where $u'$ in Lemma \ref{lem: ch rl} exists only as a weak* measurable $U^{**}$-valued function and $\xi$ is strongly measurable so that $t \mapsto \langle \xi(t), u'(t) \rangle$ is measurable, a suitable adaptation of the lemma still holds. In this scenario, we can establish that $u(t + h) = u(t) + u'(t)h + o(h)$ for almost every $t \in \left( 0, T \right)$ remains valid when interpreting $o(h)$ in the Mackey topology $\mm \left( U^{**}, U^* \right)$.
		Hence, if we consider the energy $\EE$ extended to $U^{**}$ by setting it equal to $+ \i$ outside of $U$, the natural choice for the subdifferential to replace $\p_H$ in this situation is the one associated with the bornology of Mackey compact sets. For more information on the notion of a bornological subdifferential, we refer the interested reader to \cite[§4.1.6]{Pe}.
		
	\end{enumerate}
	
	\section{Approximation} \label{sec: Approximation}
	
	\subsection{Approximation scheme}
	
	In this section, we set up the minimizing movements scheme to approximate solutions of (\ref{eq: DNI}). Particular care is required in handling the dissipation potentials that are merely measurable with respect to the time variable, which we overcome by means of a uniform Lusin result guaranteeing joint lower semicontinuity on a sequence of exhausting sets. Let $\tau > 0$. By Lemma \ref{lem: unif Lus}, we find an increasing sequence of closed sets $A_n \subset \cl \left[ 0, T \right)$ such that $\lim_n \lambda \left( \cl \left[ 0, T \right) \setminus A_n \right) = 0$ and the dissipation functionals $\phi$ and $\phi^*$ are sequentially lower semicontinuous when restricted to any $A_n$ in their first component and to any bounded subset of any $\left\{ \GG_{T_0}(u)  \le R \right\}$ for $R > 0$ in their second component. Here, we used Remark \ref{en: rem en 8} of Subsection \ref{ss:enf}. We define the partition
	$$
	\PP_\tau = \left\{ 0 = t_0 < t_1 < t_2 < \dots \right\}
	$$
	as follows: Let $n$ be the smallest natural number such that
	\begin{equation} \label{eq: smll rst}
		\lambda \left( \cl \left[ 0, T \right) \setminus A_n \right) < \tau.
	\end{equation}
	For $1 \le k \le n$, we define the partition $B_k = \left\{ 0 = t^k_0 < t^k_1 < \dots \right\}$ of $\cl \left[ 0, T \right)$ by defining inductively $t^k_\ell$ as the smallest element of the closed set
	$$
	A_k \setminus \left[ 0, t^k_{\ell - 1} + \tau \right).
	$$
	Finally, we set
	$$
	\PP_\tau = \bigcup_{k = 1}^n B_k.
	$$
	By construction, there holds
	\begin{equation} \label{eq: tau grd}
		\forall t \in \cl \left[ 0, T \right) \quad \dist \left( t, \PP_\tau \right) < 2 \tau
	\end{equation}
	by (\ref{eq: smll rst}). We introduce the constant interpolants $\overline{t}_\tau$ and $\underline{t}_\tau$ associated with the partition $\PP_\tau$, that is,
	\begin{align*}
		& \overline{t}_\tau(0) = \underline{t}_\tau(0) = 0; \\
		& \underline{t}_\tau(t) =
		\begin{cases}
			t^k_i & \text{if } t \in \left[ t^k_i, t^k_{i + 1} \right) \text{for } i = \Argmin \left\{ j \in \N \st t \in A_j \right\}, \\
			t^n_i & \text{if } t \in \left[ t^n_i, t^n_{i + 1} \right) \text{if } \left\{ j \in \N \st t \in A_j \right\} = \emptyset;
		\end{cases}
		\\
		& \overline{t}_\tau(t) =
		\begin{cases}
			t^k_{i + 1} & \text{if } t \in \left[ t^k_i, t^k_{i + 1} \right) \text{for } i = \Argmin \left\{ j \in \N \st t \in A_j \right\}, \\
			t^n_{i + 1} & \text{if } t \in \left[ t^n_i, t^n_{i + 1} \right) \text{if } \left\{ j \in \N \st t \in A_j \right\} = \emptyset.
		\end{cases}
	\end{align*}
	We have $\sup_{t \in \cl \left[ 0, T \right) } \left| \underline{t}_\tau(t) - t \right| < 2 \tau$ by (\ref{eq: tau grd}). For a given initial datum $u_0 \in D$ and a time step $\tau > 0$, we consider the partition $\PP_\tau = \left\{ 0 = t_0 < t_1 < \dots \right\}$ and, setting $U^0_\tau \coloneqq u_0$, we construct a sequence $\left( U^n_\tau \right)_{n = 1}^N$ by recursively choosing
	\begin{equation} \label{eq: min mov}
		U^n_\tau \in \Argmin_{U \in D} \left\{ \tau \phi_{t_n, U^{n - 1}_\tau} \left( \frac{U - U^{n - 1}_\tau}{\tau} \right) + \EE_{t_n}(U) \right\}, \quad 1 \le n \le N.
	\end{equation}
	Given $U^{n - 1}_\tau$, there exists by (\ref{eq: coerc}) at least one solution $U^n_\tau$ to the minimization problem (\ref{eq: min mov}) for all $\tau \in \left( 0, r^* \right)$. We define $\overline{U}_\tau$ and $\underline{U}_\tau$, respectively, to be the left-continuous and right-continuous constant interpolants of the values $\left( U^n_\tau \right)$, namely
	\begin{equation} \label{eq: const interpol}
		\begin{gathered}
			\overline{U}_\tau(t) = U^n_\tau, \quad t_{n - 1} < t \le t_n;
			\\
			\underline{U}_\tau(t) = U^{n - 1}_\tau, \quad t_{n - 1} \le t < t_n, \quad n \ge 1.
		\end{gathered}
	\end{equation}
	Also, we introduce the affine interpolant
	\begin{equation} \label{eq: aff interpol}
		U_\tau(t) = \frac{t - t_{n - 1}}{\tau} U^n_\tau + \frac{t_n - t}{\tau} U^{n - 1}_\tau, \quad t_{n - 1} \le t < t_n, \quad n \ge 1.
	\end{equation}
	Due to (\ref{eq: sum rule}), there exists
	$$
	\xi^n_\tau \in - \p \phi_{t_n, U^{n - 1}_\tau} \left( \frac{U^n_\tau - U^{n - 1}_\tau}{\tau} \right) \cap F_{t_n} \left( U^n_\tau \right).
	$$
	Again, $\overline{\xi}_\tau$ denotes the left-continuous constant interpolant of the values $\left( \xi^n_\tau \right)$. Moreover, we consider the variational interpolant $\tilde{U}_\tau$ of the values $U^n_\tau$. It is defined as any Lebesgue measurable function $\tilde{U}_\tau \colon \left[ 0, T \right) \to V_\dd$ satisfying $\tilde{U}_\tau(0) = u_0$ and, for $t = t_{n - 1} + \theta \in \left( t_{n - 1}, t_n \right]$,
	\begin{equation} \label{eq: var interpol}
		\tilde{U}_\tau(t) \in \Argmin_{U \in D} \left\{ \theta \phi_{t, U^{n - 1}_\tau} \left( \frac{U - U^{n - 1}_\tau}{\theta} \right) + \EE_t(U) \right\}.
	\end{equation}
	The existence of such a measurable selection follows by \cite[Lem.III.39]{CaVa} since the minimand in (\ref{eq: var interpol}) is $\LL_I \otimes \BB \left( V_\dd \right)$-measurable by (\ref{eq: E_0}) and has non-empty values for every $t \in \cl \left[ 0, T \right)$ by (\ref{eq: coerc}). For $t = t_n$, the minimization problems (\ref{eq: min mov}) and (\ref{eq: var interpol}) coincide, whence we may arrange
	\begin{equation} \label{eq: coinc on grid}
		\overline{U}_\tau(t_n) = \underline{U}_\tau(t_n) = U_\tau(t_n) = \tilde{U}_\tau(t_n), \quad n \ge 1.
	\end{equation}
	Thus, the variational interpolant $\tilde{U}_\tau$ contains all information of the other interpolants. Moreover, by (\ref{eq: sum rule}), there exists a measurable function $\tilde{\xi}_\tau \colon (0, T) \to V^*$ such that
	\begin{equation} \label{eq: var interpol xi}
		\begin{gathered}
			\tilde{\xi}_\tau(t) \in - \p \phi_{t, \underline{U}_\tau(t) } \left( \frac{\tilde{U}_\tau(t) - \underline{U}_\tau(t)}{t - t_{n -1}} \right) \cap F_t \left( \tilde{U}_\tau(t) \right),
			\\
			t_{n - 1} \le t < t_n, \quad n \ge 1.
		\end{gathered}
	\end{equation}
	To see this, invoke \cite[Thm. III.22]{CaVa} as such a $\tilde{\xi}_\tau$ is a measurable selection of the multimap that arises by intersecting the subdifferential multimaps. As these have measurable graphs, their intersection retains such a graph. Note in this regard that the multimap
	$$
	t \mapsto \p \phi_{t, \underline{U}_\tau(t) } \left( \frac{\tilde{U}_\tau(t) - \underline{U}_\tau(t)}{t - t_{n -1}} \right)
	$$
	has a measurable graph since it arises by intersecting the graph of measurable functions with the graph of the multimap
	$$
	\left( t, u, v \right) \mapsto \p \phi_{t, u} \left( v \right) = \left\{ \xi \in V^* \st \phi_{t, u} \left( v \right) + \phi^*_{t, u} \left( \xi \right) = \langle \xi, v \rangle \right\},
	$$
	whose graph belongs to $\LL_I \otimes \BB \left( D_\dd \times V_\dd \times X_{\dd^*} \right)$ because $\phi$ and $\phi^*$ are measurable. Arguing similarly, we see that the multimap involving $F_t$ in (\ref{eq: var interpol xi}) has a measurable graph. We set
	\begin{equation} \label{eq: def JJ}
		\begin{gathered}
			\JJ_{t, r}(u) = \inf_{v \in D} \left\{ r \phi_{t, u} \left( \frac{v - u}{r} \right) + \EE_{t + r}(v) \right\} \\
			\text{ for } t \in \cl \left[ 0, T \right), u \in D, 0 < r < T - t.
		\end{gathered}
	\end{equation}
	
	\subsection{A priori estimates}
	
	The following lemma generalizes \cite[Lem. 6.1]{MRS} to our setting. In order to connect it to a broader picture and phrase it concisely, we first introduce the following definition.
	
	\begin{definition} \label{def: KP limit}
		Let $\left( \TT, \tau \right)$ be a Hausdorff space and let $A, A_n \subset \TT$ be a sequence of subsets. Then we say that $A$ is the sequential Kuratowski-Painlevé limit of $A_n$ and write $K$-$\tau$-$\lim A_n = A$ if the following two conditions are satisfied:
		\begin{enumerate}
			\item For each $a \in A$, there exists $N \in \N$ and points $a_n \in A_n$ such that $a_n \to a$;
			\item Whenever $n_k$ is a strictly increasing subsequence such that $x_{n_k} \in A_{n_k}$ for each $k \in \N$ and $x_{n_k} \to x$, then $x \in A$.
		\end{enumerate}
	\end{definition}
	If $\left( \TT, \tau \right)$ is a first countable space, then Definition \ref{def: KP limit} agrees with the usual notion of Kuratowski-Painlevé convergence for sets by \cite[Prop. 5.29]{B} so that our terminology is justified.
	
	\begin{lemma}
		Assume $(\ref{eq: phi.1})$-$(\ref{eq: phi.3})$ and $(\ref{eq: coerc})$-$(\ref{eq: radially time differentiable})$. Then, for every $t \in \cl \left[ 0, T \right)$ and $R < \min\{ T - t, r^* \}$, the multimap $A_t \colon D \times \left( 0, R \right) \to D$ defined by 
		\begin{equation} \label{eq: Argmin}
			(u, r) \mapsto \Argmin_{v \in D} \left\{ r \phi_{t, u} \left( \frac{v - u}{r} \right) + \EE_{t + r}(v) \right\} \text{ has non-empty values;}
		\end{equation}
		For every $t \in \cl \left[ 0, T \right)$ and every $u \in D$, there exists a measurable selection $r \mapsto u_r \in A_{t, r}(u)$ such that
		\begin{equation} \label{eq: SD vanishes}
			0 \in \p \phi_{t, u} \left( \frac{u_r - u}{r} \right) + F_{t + r}(u_r);
		\end{equation}
		For every $T_0 \in \cl \left[ 0, T \right)$, there holds
		\begin{equation} \label{eq: mes sel ctr}
			\forall t \in \cl \left[ 0, T \right), \, u \in D, \, r \in (0, T_0 - t), \, u_r \in A_{t, r}(u) \colon \GG_{T_0}(u_r) \le C_3 \GG_{T_0}(u),
		\end{equation}
		\begin{equation} \label{eq: selections converge}
			K \text{-} \ee \text{-} \lim_{r \downarrow 0} A_{t, r}(u) = u, \quad \lim_{r \downarrow 0} \JJ_{t, r}(u) = \EE_t(u) \quad \forall t \in \cl \left[ 0, T \right), u \in D
		\end{equation}
		with $C_3$ the constant in (\ref{eq: G controlled}); The map
		\begin{equation} \label{eq: differentiable a.e}
			\left( 0, T - t \right) \ni r \mapsto \JJ_{t, r}(u) \text{ is differentiable a.e. in } \left( 0, T - t\right)
		\end{equation}
		and for each $r_0 \in \left( 0, T - t \right)$ and every measurable selection $\left( 0, r_0 \right] \to D \colon r \mapsto u_r \in A_{t, r}(u)$, there holds
		\begin{equation} \label{eq: energy id 1}
			\begin{gathered}
				\int_0^{r_0} \phi_{t, u} \left( \frac{u_{r_0} - u}{r_0} \right) + \phi^*_{t, u} \left( - \xi_r \right) \, dr + \EE_{t + r_0} (u_{r_0}) \\
				\le \EE_t(u) + \int_0^{r_0} P_{t + r} \left( u_r , \xi_r \right) \, dr,
			\end{gathered}
		\end{equation}
		where $\xi_r$ is any measurable selection in $- \p \phi_{t + r, u} \left( \frac{u_r - u}{r} \right) \cap F_{t + r}(u_r)$.
	\end{lemma}
	
	\begin{proof}
		We extend \cite[Lem. 6.1]{MRS} to the present setting, overcoming challenges related to the potential unboundedness and slow growth of the dissipation primitive $\phi$. To address these issues, we employ a crucial new tool, namely the approximate subdifferential $\p_\e$ presented in \cite{HMSV}. It has two key properties for our proof. First, elements $x^* \in \p_\e f(x)$ are characterized by satisfying the Fenchel-Young identity up to an error $\e > 0$ according to \cite[Prop. 1.1]{HMSV}, i.e.
		\begin{equation} \label{eq: pprxmt FY}
			x^* \in \p_\e f(x) \iff \langle x^*, x \rangle \ge f(x) + f^* \left( x^* \right) - \e \quad \forall x \in \dom(f).
		\end{equation}
		Second, the remark below \cite[Prop. 1.1]{HMSV} yields
		\begin{equation} \label{eq: pprxmt dom}
			\forall \e > 0 \quad \dom \left( f \right)= \dom \left( \p_\e f \right) \text{ if } f \text{ is a convex function.}
		\end{equation}
		
		We obtain (\ref{eq: Argmin}) by the coercivity assumption (\ref{eq: coerc}). By (\ref{eq: dd Suslin}) and \cite[Lem. III.22, III.39]{CaVa}, there exists a $\LL_I \otimes \BB \left( V_\dd \right)$-measurable selection $\left( 0, \i \right) \to D \colon r \mapsto u_r \in A_{t, r}(u)$, which satisfies (\ref{eq: SD vanishes}) by (\ref{eq: sum rule}).
		
		We obtain (\ref{eq: mes sel ctr}) because, for every $T_0 \in \cl \left[ 0, T \right)$, by the first of (\ref{eq: phi.2.1}), by the minimality of $u_r$ and since $\phi \ge 0$, there hold the inequalities
		\begin{equation} \label{eq: ineq chain}
			\begin{aligned}
				\GG_{T_0}(u) \ge \EE_{t + r}(u) \ge r \phi_{t, u} \left( \frac{u_r - u}{r} \right) + \EE_{t + r}(u_r)
				& \ge \EE_{t + r}(u_r) \\
				& \ge C_3^{-1} \GG_{T_0}(u_r).
			\end{aligned}
		\end{equation}
		We used (\ref{eq: G controlled}) in the last step. Regarding (\ref{eq: selections converge}), we note
		$$
		\JJ_{t, r}(u) \le \EE_{t + r}(u) \quad \forall r \in \left( 0, T - t \right)
		$$
		by choosing the admissible competitor $v = u$ and using the first of (\ref{eq: phi.2.1}). As
		$$
		\lim_{r \downarrow 0} \EE_{t + r}(u) = \EE_t(u) \quad \forall u \in V
		$$
		due to (\ref{eq: absolute  continuity}), we find
		\begin{equation} \label{eq: fn ulmt}
			\limsup_{r \downarrow 0} \JJ_{t, r}(u) \le \EE_t(u) < \i \quad \forall u \in D.
		\end{equation}
		By (\ref{eq: coerc}), every sequence $v_n \in A_{t, r_n}(u)$ with $r_n \downarrow 0$ has an $\ee$-convergent subsequence. Selecting such a sequence (not relabeled), we are to prove that $\lim_n v_n = u$. Since the function $\theta \mapsto \theta \phi_{t, u} \left( \frac{v}{\theta} \right)$ for every $v \in V$ is non-increasing by convexity and since the functions $\left(t, u \right) \mapsto \EE_t(u)$ and $v \mapsto \phi_{t, u}(v)$ are lower semicontinuous by (\ref{eq: absolute  continuity}) and (\ref{eq: phi.1}), we have for every $\theta > 0$ that
		\begin{equation} \label{eq: J E dmnnc}
			\begin{aligned}
				\liminf_n \JJ_{t, r_n}(u) & \ge \liminf_n r_n \phi_{t, u} \left( \frac{v_n - u}{r_n} \right) + \EE_{t + r_n}(v_n) \\
				& \ge \liminf_n \theta \phi_{t, u} \left( \frac{v_n - u}{\theta} \right) + \EE_t(v).
			\end{aligned}
		\end{equation}
		Sending $\theta \downarrow 0$, we conclude $\lim_n v_n = u$ by (\ref{eq: fn ulmt}) and (\ref{eq: phi.2.3}). This proves the first of (\ref{eq: selections converge}). The second follows upon combining (\ref{eq: fn ulmt}) and (\ref{eq: phi.2.3}). The sequences $v_n \in A_{t, r_n}(u)$ and $r_n \downarrow 0$ having been arbitrary, the first convergence in (\ref{eq: selections converge}) also follows by Definition \ref{def: KP limit}.
		
		Now, to prove (\ref{eq: differentiable a.e}), we fix $0 < r_1 < r_2$, $\e > 0$ and an element
		$$
		w_\e \in \p_\e \phi_{t, u} \left( \frac{u_{r_1} - u}{r_2} \right)
		$$
		of the approximate subdifferential $\p_\e$. This is possible by (\ref{eq: pprxmt dom}). Observe that
		\begin{equation} \label{eq: nqlty}
			\begin{aligned}
				& \JJ_{t, r_2}(u) - \JJ_{t, r_1}(u) - \EE_{t + r_2}(u_{r_1}) + \EE_{t + r_1}(u_{r_1}) \\
				& \le r_2 \phi_{t, u} \left( \frac{u_{r_1} - u}{r_2} \right) - r_1 \phi_{t, u} \left( \frac{u_{r_1} - u}{r_1} \right)
				\\
				& = \left( r_2 - r_1 \right) \phi_{t, u} \left( \frac{u_{r_1} - u}{r_2} \right) + r_1 \left\{ \phi_{t, u} \left( \frac{u_{r_1} - u}{r_2} \right) - \phi_{t, u} \left( \frac{u_{r_1} - u}{r_1} \right) \right\}
				\\
				& \le \left( r_2 - r_1 \right) \left\{ \phi_{t, u} \left( \frac{u_{r_1} - u}{r_2} \right) - \langle w_\e , \frac{u_{r_1} - u}{r_2} \rangle \right\} + r_1 \e
				\\
				& \le \left( r_2 - r_1 \right) \left\{ - \phi^*_{t, u} \left( w_\e \right) + \e \right\} + r_1 \e
				\le r_2 \e.
			\end{aligned}
		\end{equation}
		We used the definition of $\JJ_{t, r}(u)$ in the first inequality, the definition of $w_\e$ in the second one, the approximate Fenchel-Young inequality (\ref{eq: pprxmt FY}) in the third one and $\phi^*_{t, u} \ge 0$ by (\ref{eq: phi.2.1}) in the last step. Choosing $\e > 0$ arbitrarily small, we may replace $r_2 \e$ by $0$ in the last step of (\ref{eq: nqlty}). Consequently,
		\begin{align*}
			\JJ_{t, r_2}(u)
			& \le \JJ_{t, r_1}(u) + \EE_{t + r_2}(u_{r_1}) - \EE_{t + r_1}(u_{r_1})
			\\
			& \le \JJ_{t, r_1}(u) + \exp \left( \int_{r_1}^{r_2} f_1(r) \, dr \right) \GG_{T_0}(u_{r_1})
			\\
			& \le \JJ_{t, r_1}(u) + C_3 \GG_{T_0}(u) \exp \left( \int_{r_1}^{r_2} f_1(r) \, dr \right)
		\end{align*}
		with the second inequality by (\ref{eq: absolute  continuity}) and the third one by (\ref{eq: mes sel ctr}). Therefore, the function $r \mapsto \JJ_{t, r}(u)$ is the sum of the non-increasing function
		$$
		r \mapsto \JJ_{t, r}(u) - C_3 \GG_{T_0}(u) \exp \left( \int_{r_1}^{r_2} f_1(r) \, dr \right)
		$$
		and the absolutely continuous function
		$$
		r \mapsto C_3 \GG_{T_0}(u) \exp \left( \int_{r_1}^{r_2} f_1(r) \, dr \right)
		$$
		so that (\ref{eq: differentiable a.e}) follows. To conclude (\ref{eq: energy id 1}), we fix a point $r \in \left(0, T - t \right)$ where the map $r \mapsto \JJ_{t, r}(u)$ is differentiable, an element $\xi_r \in - \p \phi_{t, u} \left( \frac{u_r - u}{r} \right) \cap F_{t + r}(u)$ and $v \in V$. We also fix a sequence $h_k \downarrow 0$ such that
		\begin{equation} \label{eq: recov seq}
			\liminf_{h_k \downarrow 0} \frac{\EE_{t + r + h_k}(u_r) - \EE_{t + r}(u_r)}{h_k} = \liminf_{h \downarrow 0} \frac{\EE_{t + r + h}(u_r) - \EE_{t + r}(u_r)}{h}.
		\end{equation}
		Then
		\begin{align*}
			& \JJ_{t, r_2}(u) - \JJ_{t, r_1}(u) - \EE_{t + r_2}(u_{r_1}) + \EE_{t + r_1}(u_{r_1}) \\
			& \le \left( r_2 - r_1 \right) \phi_{t, u} \left( \frac{u_{r_1} - u}{r_2} \right) + r_1 \left\{ \phi_{t, u} \left( \frac{u_{r_1} - u}{r_2} \right) - \phi_{t, u} \left( \frac{u_{r_1} - u}{r_r} \right) \right\} \\
			& \le \left( r_2 - r_1 \right) \left\{ \phi_{t, u} \left( \frac{u_{r_1} - u}{r_2} \right) - \phi'_{t, u} \left( \frac{u_{r_1} - u}{r_2} ; \frac{u_{r_1} - u}{r_2} \right) \right\}.
		\end{align*}
		We used the convexity of $\phi_{t, u}$ in the second step. Consequently, for $h > 0$
		\begin{align*}
			& \frac{\JJ_{t, r + h}(u) - \JJ_{t, r}(u)}{h} + \phi'_{t, u} \left( \frac{u_r - u}{r + h} ; \frac{u_r - u}{r + h} \right) - \phi_{t, u} \left( \frac{u_r - u}{r + h} \right) \\
			& \le \frac{\EE_{t + r + h}(u) - \EE_{t + r}(u)}{h},
		\end{align*}
		hence by (\ref{eq: smooth at one}) and the Fenchel-Young identity in the form
		$$
		\phi^*_{t, u} \left( - \xi_r \right) = \phi'_{t, u} \left( \frac{u_r - u}{r} ; \frac{u_r - u}{r} \right) - \phi_{t, u} \left( \frac{u_r - u}{r} \right)
		$$
		by (\ref{eq: phi.3}), we may use the continuous differentiability of differentiable convex functions on the real line to find
		\begin{align*}
			& \frac{d}{dr} \JJ_{t, r}(u) + \phi^*_{t, u} \left( - \xi_r \right) \\
			& = \lim_{h \downarrow 0} \left\{ \frac{\JJ_{t, r + h}(u) - \JJ_{t, r}(u)}{h} + \phi'_{t, u} \left( \frac{u_r - u}{r + h} ; \frac{u_r - u}{r + h} \right) - \phi_{t, u} \left( \frac{u_r - u}{r + h} \right) \right\}
			\\
			& \le \liminf_{h \downarrow 0} \frac{\EE_{t + r + h}(u) - \EE_{t + r}(u)}{h} \le P_{t + r} \left( u_r, \xi_r \right).
		\end{align*}
		We used that
		$$
		\lim_{h \downarrow 0} \phi_{t, u} \left( \frac{u_r - u}{r + h} \right) = \phi_{t, u} \left( \frac{u_r - u}{r} \right)
		$$
		by lower semicontinuity and since the function $r \to \phi \left( r v \right)$ is non-increasing for any $v \in V$. The last inequality is due to (\ref{eq: recov seq}) and (\ref{eq: radially time differentiable}). Since $r$ is any point of differentiability, we find
		\begin{equation} \label{eq: differentiated energy id}
			\frac{d}{dr} \JJ_{t, r}(u) + \phi^*_{t, u} \left( - \xi_r \right) \le P_{t + r} \left( u_r, \xi_r \right) \quad \text{ for a.e. } r \in \left( 0, T - t \right).
		\end{equation}
		As $\JJ_{t, r}$ is the sum of a non-increasing function and an absolutely continuous one, we may by \cite[Cor. 1.25]{Le} integrate (\ref{eq: differentiated energy id}) over $\left( 0, r_0 \right)$ to find (\ref{eq: energy id 1})	by the second of (\ref{eq: selections converge}).
	\end{proof}
	
	\begin{proposition} \label{pr: ntrplnts tgthr}
		Assume $(\ref{eq: phi.1})$-$(\ref{eq: phi.3})$ and $(\ref{eq: assu en only})$. Let $\overline{U}_\tau, \underline{U}_\tau, U_\tau, \tilde{U}_\tau$ and $\tilde{\xi}_\tau$ be the interpolants defined in (\ref{eq: const interpol}), (\ref{eq: aff interpol}), (\ref{eq: var interpol}) and (\ref{eq: var interpol xi}). There holds the discrete energy estimate
		\begin{equation} \label{eq: energy id 2}
			\begin{gathered}
				\int_{\overline{t}_\tau(s)}^{\overline{t}_\tau(t)} \phi_{\overline{t}_\tau(r), \underline{U}_\tau(r) } ( U'_\tau(r) )
				+ \phi^*_{\overline{t}_\tau(r), \underline{U}_\tau(r) } ( - \tilde{\xi}_\tau(r) ) \, dr + \EE_{\overline{t}_\tau(t)} ( \overline{U}_\tau(t) ) \\
				\le \EE_{\overline{t}_\tau(s)} ( \overline{U}_\tau(s) ) + \int_{\overline{t}_\tau(s)}^{\overline{t}_\tau(t)} P_r ( \tilde{U}_\tau(r), \tilde{\xi}_\tau(r) ) \, dr
			\end{gathered}
		\end{equation}
		for all $0 \le s \le t \le T_0$ and $T_0 \in (0, T)$. Moreover, there exists for every $T_0$ a $C > 0$ such that the following estimates hold for every $\tau > 0$:
		\begin{gather} \label{eq: E/P unf bds}
			\sup_{0 \le t \le T_0} \EE_t ( \overline{U}_\tau(t) ) \le C, \quad \sup_{0 \le t \le T_0} \EE_t ( \tilde{U}_\tau(t) ) \le C, \\
			\left| P_t ( \tilde{U}_\tau(t), \tilde{\xi}_\tau(t) ) \right| \le C + f_1(t) + f_2(t) \text{ for a.e. } t \in \left( 0, T_0 \right),
		\end{gather}
		and
		\begin{equation} \label{eq: phi/phi* int bds}
			\int_0^{T_0} \phi_{\overline{t}_\tau(r), \underline{U}_\tau(r) } ( U'_\tau(r) ) + \phi^*_{\overline{t}_\tau(r), \underline{U}_\tau(r) } ( - \tilde{\xi}_\tau(r) ) \, dr \le C;
		\end{equation}
		The families
		\begin{equation} \label{eq: U' nd xi bds}
			U'_\tau \in V_1\left(0, T_0; V_\dd \right) \text{ and } \tilde{\xi}_\tau \in V_1\left(0, T_0; X_{\dd^*} \right) \text{ are bounded.}
		\end{equation}
		Finally, let $E_\i \left( 0, T_0; W \right)$ be the closure of simple functions in $L_\i \left( 0, T_0; W \right)$. From every subsequence $\tau_k \downarrow 0$, we may extract a further subsequence (not relabeled) for which there exists a measurable curve $u \colon \cl \left[ 0, T \right) \to V_\dd$ possessing a measurable weak derivative $u' \colon \cl \left[ 0, T \right) \to V_\dd$ such that there hold the convergences
		\begin{gather}
			t_k \to t \text{ in } \cl \left[ 0, T \right) \implies \overline{U}_{\tau_k}(t_k), \underline{U}_{\tau_k}(t_k), U_{\tau_k}(t_k), \tilde{U}_{\tau_k}(t_k) \eeto u(t), \label{eq: ipol cvg} \\
			U'_{\tau_k} \to u' \text{ in } \ss \left( V_1 \left( 0, T_0 ; V_\dd \right) ; E_\i \left( 0, T_0; W \right) \right).
			\label{eq: ipol drv cvg}
		\end{gather}
	\end{proposition}
	
	\begin{proof}
		The first part follows \cite[Prop. 6.3]{MRS} up to (\ref{eq: U' bound}). We present it for the sake of completeness. Let $t_{n - 1}$ and $t_n$ be two neighboring nodes of the partition $\PP_\tau$. Apply inequality (\ref{eq: energy id 1}) with the assignments $t = t_{n - 1}$, $u = U^{n - 1}_\tau$, $r_0 = t - t_{n - 1}$, $u_{r_0} = \tilde{U}_\tau(t)$, $u_r = \tilde{U}_\tau(r)$, and $\xi_r = \tilde{\xi}_\tau(r)$ for $t_{n - 1} < r < t$ to find
		\begin{equation} \label{eq: anothr en id}
			\begin{gathered}
				\left( t - t_{n - 1} \right) \phi_{t, \underline{U}_\tau(r) } \left( \tfrac{\tilde{U}_\tau(t) - \underline{U}_\tau(t)}{t - t_{n - 1}} \right) + \int_{t_{n - 1}}^t \phi^*_{t, \underline{U}_\tau(r) } ( - \tilde{\xi}_\tau(r) ) \, dr
				+ \EE_t ( \tilde{U}_\tau(t) ) \\
				\le \EE_{t_{n - 1}} ( \overline{U}_\tau(t_{n - 1}) ) + \int_{t_{n - 1}}^t P_r ( \tilde{U}_\tau(r), \tilde{\xi}_\tau(r) ) \, dr.
			\end{gathered}
		\end{equation}
		For $t = t_n$, this yields
		\begin{equation} \label{eq: an en id}
			\begin{gathered}
				\int_{t_{n - 1}}^{t_n} \phi_{\overline{t}_\tau(r), \underline{U}_\tau(r) } ( U'_\tau(r) ) + \phi^*_{\overline{t}_\tau(r), \underline{U}_\tau(r)} ( - \tilde{\xi}_\tau(r) ) \, dr
				+ \EE_{t_n} ( \overline{U}_\tau(t_n) ) \\
				\le \EE_{t_{n - 1}} ( \overline{U}_\tau(t_{n - 1}) ) + \int_{t_{n - 1}}^{t_n} P_r ( \tilde{U}_\tau(r), \tilde{\xi}_\tau(r) ) \, dr.
			\end{gathered}
		\end{equation}
		Summing over the subintervals of the partition results in (\ref{eq: energy id 2}). By (\ref{eq: radially time differentiable}) and (\ref{eq: Gron cons}), we further estimate the last term in (\ref{eq: an en id}) to find it is bounded by
		\begin{align*}
			\int_{t_{n - 1}}^{t_n} P_r ( \tilde{U}_\tau(r), \tilde{\xi}_\tau(r) ) \, dr
			& \le \int_{t_{n - 1}}^{t_n} G_{T_0} ( \tilde{U}_\tau(r) ) f_2(r) \, dr \\
			& \le C_3 \exp \left( \int_{t_{n - 1} }^{t_n} f_1(r) \, dr ) \right) \int_{t_{n - 1}}^{t_n} G_{T_0} ( \underline{U}_\tau(r) ) f_2(r) \, dr.
		\end{align*}
		On the other hand, by (\ref{eq: G controlled}), $\EE_{t_n} ( \overline{U}_\tau(t_n) ) \ge C^{-1}_3 \GG_{T_0} ( \overline{U}_\tau(t_n) )$. Summing over the intervals of the partition, as $\phi \ge 0$ and $\phi^* \ge 0$, we obtain from (\ref{eq: an en id}) the inequality
		\begin{equation} \label{eq: rdy 4 Grnwll}
			\begin{aligned}
				\GG_{T_0} ( \overline{U}_\tau (t_k) )
				& \le C \left( \EE_0(u_0) + \int_0^{t_k} \GG_{T_0} ( \overline{U}_\tau(r) ) f_2(r) \, dr + 1 \right) \\
				& \le C \left( \EE_0(u_0) + 2 \tau \int_0^{t_k} \left| f_2(r) \right| \, dr \sum_{i = 1}^k \GG_{T_0} ( \overline{U}_\tau(t_i) ) \, dr + 1 \right),
			\end{aligned}
		\end{equation}
		where we used (\ref{eq: tau grd}) in the last step. The first estimate in (\ref{eq: E/P unf bds}) follows by invoking (\ref{eq: rdy 4 Grnwll}) and the discrete Gronwall inequality \cite[Lem. 4.5]{RS}. The second in (\ref{eq: E/P unf bds}) follows from (\ref{eq: mes sel ctr}) and the third from the second, (\ref{eq: absolute  continuity}), and (\ref{eq: radially time differentiable}) by an argument involving Lebesgue points. In total, the right-hand side in (\ref{eq: energy id 2}) is bounded on every compact subinterval of $\cl \left[ 0, T \right)$, hence (\ref{eq: phi/phi* int bds}) and (\ref{eq: U' nd xi bds}) follow by (\ref{eq: phi.2.2.2}). By (\ref{eq: anothr en id}), we have
		\begin{equation} \label{eq: U' bound}
			\sup_{0 \le t \le T_0} \left( t - \underline{t}_\tau(t) \right) \phi_{t, \underline{U}_\tau(t) } \left( \tfrac{\tilde{U}_\tau(t) - \underline{U}_\tau(t)}{t - \underline{t}_\tau(t)} \right) \le C_{T_0}.
		\end{equation}
		We claim the existence of a subsequence $\tau_k$ such that
		\begin{equation} \label{eq: w cpt cvg}
			U_{\tau_k}(t_k) \ddto u(t) \text{ whenever } t_k \to t \text{ in } \cl \left[ 0, T \right) \text{ as } k \uparrow \i.
		\end{equation}
		Let $w \in W$ and $T_0 \in \cl \left[ 0, T \right)$ be fixed. To prove (\ref{eq: w cpt cvg}), we first show that the function
		\begin{equation} \label{eq: u' q-ntgrbl}
			t \mapsto \langle w, U'_{\tau_k}(t) \rangle \text{ is equi-integrable on } \left( 0, T_0 \right).
		\end{equation}
		To do this, we use (\ref{eq: E/P unf bds}) and (\ref{eq: G controlled}) to find a constant $R > 0$ such that
		$$
		\GG_{T_0} \left( \underline{U}_\tau(t) \right) \le R \quad \forall t \in \left[ 0, T_0 \right].
		$$
		We then use this bound and the Fenchel-Young inequality to estimate for any Lebesgue set $L \in \LL \left( 0, T_0 \right)$ and any $ \kappa > 0$ that
		\begin{align*}
			\int_L \langle w, U'_{\tau_k}(r) \rangle \, dr
			& \le  \kappa \int_L \phi^*_{W, \overline{t}_{\tau_k}(r), \underline{U}_{\tau_k}(r) } \left( \tfrac{w}{ \kappa} \right) + \phi_{\overline{t}_{\tau_k}(r), \underline{U}_{\tau_k}(r) } ( U'_{\tau_k}(r) ) \, dr \\
			& \le  \kappa \int_L \sup_{\GG_{T_0}(u) \le R}
			\phi^*_{W, \overline{t}_{\tau_k}(r), u } \left( \tfrac{w}{ \kappa} \right) \, dr + S  \kappa.
		\end{align*}
		Hence, taking $ \kappa > 0$ sufficiently small, we find (\ref{eq: u' q-ntgrbl}) by (\ref{eq: phi.2.4}). Next, we show that $U_{\tau_k}(t)$ is weakly equi-continuous and uniformly norm bounded on compact intervals. To do this, we first combine (\ref{eq: u' q-ntgrbl}) with a Poincaré inequality and the initial condition $U_{\tau_k}(0) = u_0$ to find
		\begin{equation} \label{eq: w-eqi cont}
			t \mapsto \langle w, U_{\tau_k}(t) \rangle \text{ is equi-continuous, uniformly bounded on compact sets.}
		\end{equation}
		We then use the Banach-Steinhaus uniform boundedness principle to conclude the uniform norm boundedness. In doing this, we have identified $V_\dd$ with a closed subspace of $W^*$ as explained below (\ref{eq: dual nrms}). In conclusion, we have (\ref{eq: ipol drv cvg}) by (\ref{eq: u' q-ntgrbl}) and a standard diagonal argument. In particular, (\ref{eq: w cpt cvg}) follows. To conclude that
		\begin{equation} \label{eq: cpt cvg}
			U_{\tau_k}(t_k) \eeto u(t) \text{ whenever } t_k \to t \text{ in } \cl \left[ 0, T \right) \text{ as } k \uparrow \i,
		\end{equation}
		we use that the convergence is true in $\dd$ by (\ref{eq: w cpt cvg}) and then appeal to (\ref{eq: crc}) and the pointwise sequential $\ee$-compactness of $U_{\tau_k}$. To show the convergence of $\tilde{U}_{\tau_k}$, we first prove that if $t_k \to t$ in $\cl \left[ 0, T \right)$, then
		\begin{equation} \label{eq: tld cpt}
			\text{the set } \tilde{U}_{\tau_k}(t_k) \text{ is sequentially relatively } \ee \text{-compact}.
		\end{equation}
		This is established using minimality of $\tilde{U}_\tau(t)$ and (\ref{eq: phi.2.2.1}) in combination with (\ref{eq: E/P unf bds}). In view of sending $\tau_k \to 0^+$, we may assume $\tau < \tfrac{1}{2}$ so that the mesh size of the partition $\PP_\tau$ is less than $1$ by (\ref{eq: tau grd}). Setting $t = t_{n - 1} + \theta \in \left( t_{n - 1}, t_n \right]$ and using (\ref{eq: G controlled}), we have
		\begin{align*}
			C_{T_0} & \ge \EE_t \left( \underline{U}_\tau(t) \right) = \EE_t \left( U^{n - 1}_\tau \right) \\
			& \ge \EE_t \left( \tilde{U}_\tau(t) \right) + \theta \phi_{t, U^{n - 1}_\tau} \left( \tfrac{\tilde{U}_\tau(t) - U^{n - 1}_\tau}{\theta} \right) \\
			& \ge \EE_t \left( \tilde{U}_\tau(t) \right) + \phi_{t, U^{n - 1}_\tau} \left( \tilde{U}_\tau(t) - U^{n - 1}_\tau \right) \\
			& \ge C \GG_{T_0} \left( \tilde{U}_\tau(t) \right) + C \| \tilde{U}_\tau(t) - U^{n - 1}_\tau \|.
		\end{align*}
		Then, invoking (\ref{eq: tld cpt}), we extract a subsequence such that $\tilde{U}_{\tau_k}$ converges to a limit $U_\i$. Finally, using (\ref{eq: U' bound}), we conclude the bound
		\begin{align*}
			C_{T_0}
			& \ge \left( t_k - \underline{t}_\tau(t_k) \right) \phi_{t_k, \underline{U}_\tau(t_k) } \left( \tfrac{\tilde{U}_\tau(t_k) - \underline{U}_\tau(t_k)}{t_k - \underline{t}_\tau(t_k)} \right) \\
			& \ge \theta \phi_{t_k, \underline{U}_\tau(t_k) } \left( \tfrac{\tilde{U}_\tau(t_k) - \underline{U}_\tau(t_k)}{\theta} \right).
		\end{align*}
		This implies
		$$
		C_{T_0} \ge \lim_{\theta \to 0} \liminf_{k \to \i} \theta \phi_{t_k, \underline{U}_\tau(t_k) } \left( \tfrac{\tilde{U}_\tau(t_k) - \underline{U}_\tau(t_k)}{\theta} \right),
		$$
		whence we conclude $U_\i = u(t)$ by (\ref{eq: phi.2.3}).
	\end{proof}
	
	\begin{proposition} \label{pr: Y ms cnstr}
		Assume $(\ref{eq: phi.1})$-$(\ref{eq: phi.2.3}), (\ref{eq: phi.3})$-$(\ref{eq: phi.4.2})$, and $(\ref{eq: assu en only})$. For every sequence $\tau_k \downarrow 0$ of time steps, there exist a subsequence (not relabeled), a weakly differentiable curve $u \colon \cl \left[ 0, T \right) \to V_\dd$ with an $\LL_I \otimes \BB \left( V_\dd \right)$-measurable and locally integrable derivative $u' \colon \cl \left[ 0, T \right) \to V_\dd$, a function $E \colon \cl \left[ 0, T \right) \to \R$ of locally bounded variation, and a time-dependent Young measure
		$$
		\mu = \left( \mu_t \right)_{t \in \cl \left[ 0, T \right) } \in \YY \left( \cl \left[ 0, T \right); V_\dd \times X_{\dd^*} \times \R \right)
		$$
		associated with the sequence of tuples
		$$
		\left( U'_{\tau_k}, \tilde{\xi}_{\tau_k}, P_{\tau_k} \right)
		$$
		in the space $V_\dd \times X_{\dd^*} \times \R$	such that, as $k \uparrow \i$, there hold the following relations and convergences:
		\begin{align}
			\begin{cases}
				\EE_t \left( \overline{U}_{\tau_k}(t) \right) \to E(t) \quad & \forall t \in \cl \left[ 0, T \right), \quad E(0) = \EE_0(u_0),		\\
				E(t) \ge \EE_t (u(t) ) \quad & \forall t \in \cl \left[ 0, T \right); \\
			\end{cases}
			\label{eq: E cvg}
		\end{align}
		If (\ref{eq: cl str impl}), then
		\begin{equation} \label{eq: E cvg II}
			E(t) = \EE_t(u(t) ) \quad \text{for a.e. } t \in \left( 0, T \right);
		\end{equation}
		Setting
		\begin{align}
			\tilde{\xi}(t) & = \int \eta \, d \mu_t(v, \eta, p) \quad \text{ for a.e. } t \in \left( 0, T \right), \label{eq: xi exp} \\
			P(t) & = \int p \, d \mu_t(v, \eta, p) \quad \text{ for a.e. } t \in \left( 0, T \right),
		\end{align}
		we have
		\begin{align}
			& u'(t) = \int v \, d \mu_t(v, \eta, p) \quad \text{ for a.e. } t \in \left( 0, T \right), \label{eq: t drv exp} \\
			& \tilde{\xi}_{\tau_k} \to \tilde{\xi} \text{ in the biting sense of } V_1 \left( 0, T_0 ; X_{\dd^*} \right) \quad \forall T_0 \in \cl \left[ 0, T \right), \\
			&
			\begin{gathered} \label{eq: P exp}
				P_{\tau_k} \weak P \text{ in } L_1 \left( 0, T_0 \right) \quad \forall T_0 \in \cl \left[ 0, T \right) \text{ with } \\
				P(t) \le \int P_t \left( u(t), \eta \right) \, d \mu_t(v, \eta, p) \quad \text{ for a.e. } t \in \left( 0, T \right);
			\end{gathered}
		\end{align}
		For any $T_0 \in \cl \left[ 0, T \right)$ and $0 \le s \le t \le T_0$,
		\begin{equation} \label{eq: en ineq}
			\begin{gathered}
				\int_s^t \int \phi_{r, u(r) } (v) + \phi^*_{r, u(r) } (- \eta) \, d \mu_r(v, \eta, p) \, dr + E(t) \\
				\le E(s) + \int_s^t P(r) \, d r \le E(s) + \int_s^t \int P_r(u(r), \eta ) \, d \mu_r(v, \eta, p) \, dr.
			\end{gathered}
		\end{equation}
	\end{proposition}
	
	\begin{proof}
		The proof of (\ref{eq: E cvg}), (\ref{eq: E cvg II}) follows closely the proof of \cite[Prop. 6.4]{MRS}. We present the relevant parts here for the sake of completeness. By Proposition \ref{pr: ntrplnts tgthr}, we may extract a subsequence $\tau_k$ such that (\ref{eq: ipol cvg}) and (\ref{eq: ipol drv cvg}) hold. From the locally integrable bound on $P$ in (\ref{eq: E/P unf bds}), we find a further diagonal subsequence $\tau_k$ (not relabeled) such that
		\begin{equation} \label{eq: P precpt in Linf}
			P_{\tau_k} \weak P \text{ in } L_1 \left( 0, T_0 \right) \quad \forall T_0 \in \cl \left[ 0, T \right).
		\end{equation}
		To prove (\ref{eq: E cvg}), we observe first that
		\begin{equation} \label{eq: mono decr}
			t \mapsto \eta_\tau(t) = \EE_{\overline{t}_\tau(t) } \left( \overline{U}_\tau(t) \right) - \int_0^{\overline{t}_\tau(t) } P_\tau(r) \, dr \text{ decreases on } \cl \left[ 0, T \right).
		\end{equation}
		Hence, by the Helly selection theorem, there exists a non-increasing function $\eta \colon \cl \left[ 0, T \right) \to \R$ and a diagonal subsequence $\tau_k$ such that $\eta_{\tau_k}(t) \to \eta(t)$ for all $t \in \cl \left[ 0, T \right)$. By (\ref{eq: P precpt in Linf}),
		\begin{equation} \label{eq: E cvg 2}
			\lim_k \EE_{\overline{t}_{\tau_k}(t) } \left( \overline{U}_{\tau_k}(t) \right) = E(t) = \eta(t) + \int_0^t P(r) \, dr \quad \forall t \in \cl \left[ 0, T \right).
		\end{equation}
		Thereby, we find the first of (\ref{eq: P exp}) because (\ref{eq: absolute  continuity}) and (\ref{eq: E/P unf bds}) imply
		\begin{equation} \label{eq: E cvg 3}
			\begin{gathered}
				\left| \EE_{\overline{t}_{\tau_k}(t) } \left( \overline{U}_\tau(t) \right) - \EE_t \left( \overline{U}_\tau(t) \right) \right|
				\le \GG \left( \overline{U}_\tau(t) \right) \exp \left( \int_{\left[ t, \overline{t}_{\tau_k}(t) \right] } f_1(r) \, dr \right) \\
				\le S \exp \left( \int_{\left[ t, \overline{t}_{\tau_k}(t) \right] } f_1(r) \, dr \right) \to 0 \text{ as } \tau \to 0.
			\end{gathered}
		\end{equation}
		Now, the second of (\ref{eq: E cvg}) follows by the lower semicontinuity of $\EE_t$ and (\ref{eq: ipol cvg}), while (\ref{eq: E cvg II}) follows by the continuous closedness implication (\ref{eq: cl str impl}). Note in this regard that
		$$
		\liminf_{k \uparrow \i} \| \tilde{\xi}_{\tau_k}(t) \|_X < + \i \quad \text{ for a.e. } t \in \left( 0, T \right)
		$$
		by the Fatou lemma and the second of (\ref{eq: U' nd xi bds}). We now adapt \cite[Prop. 6.4]{MRS} to time-dependent dissipation potentials. By (\ref{eq: U' nd xi bds}) and the last of (\ref{eq: E/P unf bds}), we may apply the Young measure result Theorem \ref{thm: Y meas} to the sequence
		$$
		\left( U'_{\tau_k}, \tilde{\xi}_{\tau_k}, P_{\tau_k} \right)
		$$
		to find a subsequence (not relabeled) and a limit Young measure $\mu$ such that, for a.e. $t \in (0, T)$, there holds
		\begin{equation} \label{eq: Y meas conc}
			\begin{gathered}
				\mu_t \text{ is concentrated on } L(t), \text{ the set of limit points of } \left( U'_{\tau_k}, \tilde{\xi}_{\tau_k}, P_{\tau_k} \right)(t) \\
				\text{ with respect to the topology } \dd \times \dd^* \times \R,
			\end{gathered}
		\end{equation}
		and for $\mu$ we have (\ref{eq: G-liminf}) and (\ref{eq: exp b cvg 2}). The latter relations imply that the right side of (\ref{eq: t drv exp}) is indeed a weak* time derivative of $u$ by (\ref{eq: ipol drv cvg}) and that the marginal expectation in (\ref{eq: xi exp}) is a biting limit of $\tilde{\xi}_{\tau_k}$. Since $V_\dd$ and $X_{\dd^*}$ are locally convex Hausdorff spaces, we may apply the Jensen inequality to find
		\begin{equation} \label{eq: Jensen}
			\begin{cases}
				& \int \phi_{t, u(t) } (v) \, \mu_t(v, \eta, p) \ge \phi_{t, u(t) } \left( u'(t) \right), \\
				& \int \phi^*_{t, u(t) } (-\eta) \, \mu_t(v, \eta, p) \ge \phi_{t, u(t) } \left( - \tilde{\xi}(t) \right)
			\end{cases}
			\quad \text{ for a.e. } t \in (0, T).
		\end{equation}
		Invoking (\ref{eq: ch ru}), we pass to the limit in the Euler equation (\ref{eq: var interpol xi})	to find by (\ref{eq: ipol cvg}) and the first of (\ref{eq: E cvg}) that, for a.e. $t \in (0, T)$, there holds
		\begin{equation} \label{eq: L property}
			\forall (v, \eta, p) \in L(t) \quad \eta \in F_t(u(t) ), \quad p \le P_t(u(t), \eta).
		\end{equation}
		From this, we obtain the inequality in (\ref{eq: P exp}) by the convergence of expectations (\ref{eq: exp b cvg 2}). Applying the $\Gamma$-$\liminf$ inequality (\ref{eq: G-liminf}) to the integrands
		$$
		f_k(t, v, \eta, p) = \phi_{\overline{t}_{\tau_k}(t), \underline{U}_{\tau_k}(t) } (v)
		$$
		fulfilling (\ref{eq: G-liminf assu}) by (\ref{eq: phi.4.1}), (\ref{eq: E/P unf bds}) and (\ref{eq: ipol cvg}) on each set $A_n$ where $\phi$ and $\phi^*$ are jointly lower semicontinuous. Consequently, for all $0 \le s \le t \le T_0$ and $T_0 \in \cl \left[ 0, T \right)$, there holds
		\begin{equation} \label{eq: Y meas ineq}
			\begin{gathered}
				\liminf_{k \uparrow \i} \int_{\overline{t}_{\tau_k}(s) }^{\overline{t}_{\tau_k}(t) } \phi_{\overline{t}_{\tau_k}(r), \overline{U}_{\tau_k}(t) } \left( U'_{\tau_k}(r) \right) \, dr
				\\
				\ge \int_{A_n \cap \left( s, t \right)} \int \phi_{r, u(r) } (v) \, d \mu_r(v, \eta, p) \, dr.
			\end{gathered}
		\end{equation}
		Considering the integrands
		$$
		f_k(t, v, \eta, p) = \phi^*_{\overline{t}_{\tau_k}(t), \underline{U}_{\tau_k}(t) } (\eta),
		$$
		which on $A_n$ fulfill (\ref{eq: G-liminf assu}) by (\ref{eq: phi.4.2}), (\ref{eq: E/P unf bds}) and (\ref{eq: ipol cvg}), yields
		\begin{equation} \label{eq: Y meas ineq 2}
			\begin{gathered}
				\liminf_{k \uparrow \i} \int_{\overline{t}_{\tau_k}(s) }^{\overline{t}_{\tau_k}(t) } \phi^*_{\underline{t}_{\tau_k}(r), \underline{U}_{\tau_k}(t) } \left( - \tilde{\xi}_{\tau_k}(r) \right) \, dr
				\\
				\ge \int_{A_n \cap \left( s, t \right)} \int \phi^*_{r, u(r) } ( - \eta) \, d \mu_r(v, \eta, p) \, dr.
			\end{gathered}
		\end{equation}
		Taking the supremum with respect to $n \in \N$, we may replace $A_n \cap \left( s, t \right)$ by $\left( s, t \right)$ in (\ref{eq: Y meas ineq}) and (\ref{eq: Y meas ineq 2}) due to (\ref{eq: smll rst}). In total, we may pass to the limit in (\ref{eq: energy id 2}): Employing (\ref{eq: E cvg}), (\ref{eq: P exp}), (\ref{eq: P precpt in Linf}), (\ref{eq: Y meas ineq}) and (\ref{eq: Y meas ineq 2}), we find (\ref{eq: en ineq}).
	\end{proof}
	
	\section{The main theorems}
	
	\begin{theorem} \label{thm: main 1}
		Let the assumptions of Proposition \ref{pr: Y ms cnstr} and either (\ref{eq: ch ru}) or (\ref{eq: ch ru wkr}) hold. Then, for every $u_0 \in D$, there exist a measurable curve $u \colon \cl \left[ 0, T \right) \to V_\dd$ having a locally integrable, $\LL_I \otimes \BB \left( V_\dd \right)$-measurable time derivative $u' \colon \left[ 0, T \right) \to V_\dd$ and a locally integrable, $\LL_I \otimes \BB \left( X_{\dd^*} \right)$-measurable function $\xi \colon \cl \left[ 0, T \right) \to X_{\dd^*}$ such that
		\begin{enumerate}
			
			\item the tuple $(u, \xi)$ is an energy solution to the generalized gradient system $\left( V_\delta, X_{\dd^*}, \EE, \phi, \langle \cdot, \cdot \rangle_{X, V}, F, P \right)$ in the sense of Definition \ref{def: sol} if (\ref{eq: ch ru ineq}) holds;
			
			\item the tuple $(u, \xi)$ is a Lyapunov solution in the sense of Definition \ref{def: sol} if (\ref{eq: ch ru wkr}) and (\ref{eq: cl str impl}) hold.
			
		\end{enumerate}
		Moreover, if (\ref{eq: ch ru ineq}), then for any family of approximate solutions
		$$
		\left( \tilde{U}_\tau, \tilde{\xi}_\tau \right)_{\tau > 0},
		$$
		there exists a sequence $\tau_k \downarrow 0$ such that, as $k \uparrow \i$, there hold the convergences
		\begin{equation} \label{eq: cvgs}
			\begin{gathered}
				\EE_{\overline{t}_{\tau_k}(t) } \left( \overline{U}_{\tau_k}(t) \right) \to \EE_t \left( u(t) \right) \quad \forall t \in \cl \left[ 0, T \right),
				\\
				\int_{\overline{t}_{\tau_k}(s) }^{\overline{t}_{\tau_k}(t) } \phi_{\overline{t}_\tau(r), \underline{U}_\tau(r) } \left( U'_{\tau_k}(r) \right) \, dr \to \int_s^t \phi_{r, u(t) } \left( u'(r) \right) \, dr,
				\\
				\int_{\overline{t}_{\tau_k}(s) }^{\overline{t}_{\tau_k}(t) } \phi^*_{\overline{t}_\tau(r), \underline{U}_\tau(r) } \left( - \tilde{\xi}_{\tau_k}(r) \right) \, dr \to \int_s^t \phi_{r, u(t) } \left( - \xi(r) \right) \, dr,
			\end{gathered}
		\end{equation}
		for all $s, t \in \cl \left[ 0, T \right)$.
	\end{theorem}
	
	\begin{proof}
		We generalize \cite[Thm. 4.4]{MRS} to our setting.
		
		Step 1: By (\ref{eq: en ineq}) and the second of (\ref{eq: E cvg}), there exists a function of locally bounded variation $E \colon \cl \left[ 0 , T \right) \to \R$ such that
		\begin{equation} \label{eq: E prprty}
			E(0) = \EE_0(u_0), \quad E(t) \ge \EE_t \left( u(t) \right) \quad \forall t \in \cl \left[ 0, T \right)
		\end{equation}
		and
		\begin{equation} \label{eq: en ineq'}
			\begin{gathered}
				\int_s^t \int \phi_{r, u(r) } \left( u'(r) \right) + \phi^*_{r, u(r) }(- \eta) \, d \mu_r(v, \eta, p) \, dr + E(t) \\
				\le \int_s^t \int \phi_{r, u(r) } (v) + \phi^*_{r, u(r) }(- \eta) \, d \mu_r(v, \eta, p) \, dr + E(t) \\
				\le E(s) + \int_s^t P(r) \, dr \quad \forall s, t \in (0, T).
			\end{gathered}
		\end{equation}
		We used the Jensen inequality in the first estimate of (\ref{eq: en ineq'}). By the estimates (\ref{eq: E/P unf bds}), (\ref{eq: phi/phi* int bds}) and from (\ref{eq: E cvg}), (\ref{eq: Y meas conc}) - (\ref{eq: Y meas ineq}), the curve $u \colon \cl \left[ 0, T \right) \to V$ and the Young measure $\mu$ satisfy assumptions (\ref{eq: assu 1}) - (\ref{eq: YM ch rl ass 2}) of Lemma \ref{lem: YM chain-rule ineq}. Therefore, the map $t \to \EE_t (u(t) )$ is of bounded variation and (\ref{eq: YM chain-rule ineq}) is true either in the strong form
		\begin{equation}\label{eq: YM ch ru str}
			\frac{d}{dt} \EE_t \left( u(t) \right) \ge \int \langle \eta, u'(t) \rangle_{X, V} + p \, d \mu_t(v, \eta, p) \quad \text{ in } \DD'(0, T)
		\end{equation}
		or the weak form 
		\begin{equation} \label{eq: YM ch ru wk}
			\frac{d}{dt} \EE_t \left( u(t) \right) \ge \int \langle \eta, u'(t) \rangle_{X, V} + p \, d \mu_t(v, \eta, p) \quad \text{ for a.e. } t \in (0, T)
		\end{equation}
		
		Step 1a: If (\ref{eq: YM ch ru str}), then there hold the inequalities
		\begin{equation} \label{eq: YM ineq}
			\begin{gathered}
				\int_0^t \int \phi_{r, u(r) } (v) + \phi^*_{r, u(r) }(- \eta) \, d \mu_r(v, \eta, p) \, dr + \EE_t(u(t) ) \\
				\le \EE_0(u_0) + \int_0^t P(r) \, dr
				\le \EE_t(u(t) ) + \int_0^t \int \langle - \eta , u'(r) \rangle \, d \mu_r(v, \eta, p) \, dr.
			\end{gathered}
		\end{equation}
		The first inequality follows from (\ref{eq: en ineq'}) with the choice $0 = s < t$ by (\ref{eq: E prprty}). The second inequality is due to (\ref{eq: YM ch ru str}), which by \cite[Cor. 1.25]{Le} implies
		\begin{align*}
			\EE_t(u(t) ) - \EE_0(u_0)
			& \ge \int_0^t \int \langle \eta, u'(r) \rangle + p \, \mu_r(v, \eta, p) \, dr \\
			& = \int_0^t \int \langle \eta, u'(r) \rangle \, \mu_r(v, \eta, p) + P(r) \, dr.
		\end{align*}
		By the Jensen inequality, we conclude
		\begin{equation} \label{eq: res Jens}
			\int_0^t \int \phi_{r, u(r) } \left( u'(r) \right) + \phi^*_{r, u(r) }(- \eta) - \langle - \eta, u'(r) \rangle \, d \mu_r(v, \eta, p) \, dr \le 0.
		\end{equation}
		By the Fenchel-Young inequality, the integrand in (\ref{eq: res Jens}) is non-negative, hence
		\begin{equation} \label{eq: res Jens 2}
			\begin{gathered}
				\int \phi_{t, u(t) } \left( u'(t) \right) + \phi^*_{t, u(t) }(- \eta) - \langle - \eta, u'(t) \rangle \, d \mu_t(v, \eta, p) = 0 \\ \text{ for a.e. } t \in (0, T).
			\end{gathered}
		\end{equation}
		Consequently, all inequalities in (\ref{eq: YM ineq}) become equalities. Invoking again the chain-rule inequality (\ref{eq: YM chain-rule ineq}), for a.e. $t \in (0, T)$, we find
		\begin{gather*}
			\int \phi_{t, u(t) } \left( u'(t) \right) + \phi^*_{t, u(t) }(- \eta) - p \, d \mu_t(v, \eta, p)
			\\
			= \int \langle - \eta, u'(t) \rangle - p \, d \mu_t(v, \eta, p) = - \frac{d}{dt} \EE_t(u(t) ).
		\end{gather*}
		Consequently, for every $0 \le s \le t \le T_0$ and $T_0 \in \cl \left[ 0, T \right)$, there holds the energy identity
		\begin{equation} \label{eq: en id}
			\int_s^t \int \phi^*_{r, u(r) }(v) + \phi^*_{t, u(r) }(- \eta) - p \, d \mu_r(v, \eta, p) \, dr = \EE_s(u(s) ) - \EE_t(u(t) ).
		\end{equation}
		
		Step 1b: If (\ref{eq: YM ch ru wk}) and (\ref{eq: cl str impl}), then $E(t) = \EE_t \left( u(t) \right)$ for a.e. $t \in (0, T)$ by (\ref{eq: E cvg 2}) so that (\ref{eq: en ineq'}) and the Lebesgue differentiation theorem imply
		\begin{equation} \label{eq: Leb res}
			\begin{gathered}
				\int \phi_{t, u(t) } \left( u'(t) \right) + \phi^*_{t, u(t) } \left( - \eta \right) \, \mu_t(v, \eta, p)
				\le - \frac{d}{dt} \EE_t(u(t) ) + \int p \, \mu_t(v, \eta, p) \\
				\le \int \langle - \eta, u'(t) \rangle_{X, V}\, \mu_t(v, \eta, p)
				\quad \text{ for a.e. } t \in (0, T),
			\end{gathered}
		\end{equation}
		where we used (\ref{eq: YM ch ru wk}) for the second inequality. From (\ref{eq: Leb res}), we deduce (\ref{eq: res Jens 2}) again. Plugging (\ref{eq: res Jens 2}) into (\ref{eq: Leb res}) yields
		$$
		\frac{d}{dt} \EE_t(u(t) ) = \int \phi_{t, u(t) } \left( u'(t) \right) + \phi^*_{t, u(t) } \left( - \eta \right) - p \, \mu_t(v, \eta, p) \quad \text{ for a.e. } t \in (0, T)
		$$
		so that \cite[Cor. 1.25]{Le} for every $0 \le s \le t \le T_0$ and $T_0 \in \cl \left[ 0, T \right)$ implies the energy inequality
		\begin{equation} \label{eq: en unid}
			\begin{gathered}
				\int_s^t \int \phi^*_{r, u(r) }(v) + \phi^*_{t, u(r) }(- \eta) - p \, d \mu_r(v, \eta, p) \, dr \le \EE_s(u(s) ) - \EE_t(u(t) )
				\\
				\forall s, t \in \cl \left[ 0, T \right).
			\end{gathered}
		\end{equation}
		
		Step 2: The marginal measure $\nu_t = \left( \pi_{2, 3} \right)_\sharp \left( \mu_t \right)$ of $\mu$ with respect to the $(\eta, p)$-component is given by
		$$
		\nu_t(B) = \mu_t \left( \pi^{-1}_{2, 3} (B) \right) \quad \forall B \in \BB \left( X_{\dd^*} \times \R \right).
		$$
		By (\ref{eq: L property}) and (\ref{eq: res Jens 2}), the measure $\nu_t$ is concentrated on the set
		\begin{gather*}
			\SS_t = \SS \left( t, u(t), u'(t) \right) \coloneqq \\
			\left\{ (\eta, p) \in X_{\dd^*} \times \R \st \eta \in - \p \phi_{t, u(t) } \left( u'(t) \right) \cap F_t(u(t) ), \, p \le P_t( u(t), \xi(t) ) ) \right\}
		\end{gather*}
		for a.e. $t \in (0, T)$. A fortiori, almost all of the sets $\SS_t$ are non-empty. Thus, by Lemma \ref{lem: mes sel}, there exists a measurable selection $\left( \xi , p \right)$ of the multimap $t \mapsto \SS_t$ such that
		\begin{equation} \label{eq: meas sel min}
			\phi^*_{t, u(t) } \left( - \xi(t) \right) - p(t) = \min_{(\eta, p) \in \SS_t } \phi^*_{t, u(t) }(- \eta) - p \quad \text{ for a.e. } t \in (0, T).
		\end{equation}
		In particular, the doubly nonlinear inclusion (\ref{eq: DNL inc}) is satisfied by $\xi$, hence $u$ solves the initial value problem of (\ref{eq: DNI}). In fact,
		\begin{equation} \label{eq: xi fin en}
			\int_0^{T_0} \phi^*_{t, u(t) } \left( - \xi(t) \right) \, dt < + \i \quad \forall T_0 \in \cl \left[ 0, T \right),
		\end{equation}
		whence $\xi$ is locally integrable by (\ref{eq: phi.2.2.2}). We check (\ref{eq: xi fin en}):
		\begin{align*}
			& \int_0^{T_0} \phi^*_{t, u(t) } \left( - \xi(t) \right) \, dt \\
			\le & \int_0^{T_0} \int_{V^* \times \R} \phi^*_{r, u(r) } (- \eta) - p \, d \nu_r(\eta, p) \, dr + \int_0^{T_0} P_r(u(r), \xi(r) ) \, dr \\
			\le & \int_0^{T_0} \int \phi^*_{r, u(r) } (- \eta) - p \, d \mu_r(v, \eta, p) \, dr + \int_0^{T_0} \GG_{T_0}(u(r) ) f_2(r) \, dr < + \i
		\end{align*}
		with the first inequality since for a.e. time, the measure $\nu_t$ is concentrated on $\SS \left(t, u(t), u'(t) \right)$ and $p(t) \le P_t(u(t), \xi(t) )$. The second inequality is from (\ref{eq: radially time differentiable}) and the final one from (\ref{eq: en id}) and $\sup_{t \in \left( 0, T_0 \right) } \GG(u(t) ) < + \i$.
		
		Step 3: Now we prove the energy (in)equalities (\ref{eq: en id 2}) and (\ref{eq: en id 3}). For all $0 \le s \le t \le T_0$ and $T_0 \in \cl \left[ 0, T \right)$, there holds
		\begin{equation} \label{eq: en ineq 2}
			\begin{gathered}
				\int_s^t \phi_{r, u(r) } \left( u'(r) \right) + \phi^*_{r, u(r) } \left( - \xi(r) \right) - P_r(u(r), \xi(r) ) \, dr
				\\
				\le \int_s^t \phi_{r, u(r) } \left( u'(r) \right) + \phi^*_{r, u(r) } \left( - \xi(r) \right) - p(r) \, dr
				\le \EE_s(u(s) ) - \EE_t(u(t) )
			\end{gathered}
		\end{equation}
		by (\ref{eq: en id}) and since $\left( \xi, p \right)$ is a measurable selection of $\SS(t, u(t), u'(t) )$. We find (\ref{eq: en id 3}) upon rearranging the terms of (\ref{eq: en ineq 2}).	Moreover, the chain rule inequality (\ref{eq: ch ru}) applied to the pair $(u, \xi)$ yields
		\begin{equation} \label{eq: integr chain ineq}
			\begin{gathered}
				\EE_t(u(t) ) - \EE_s(u(s) )
				\ge \int_s^t \langle \xi(r), u'(r) \rangle + P_r(u(r), \xi(r) ) \, dr \\
				\text{ for all } 0 \le s \le t \le T_0 \text{ and } T_0 \in \cl \left[ 0, T \right).
			\end{gathered}
		\end{equation}
		Combining (\ref{eq: en ineq 2}) with (\ref{eq: integr chain ineq}) and arguing as before for (\ref{eq: res Jens}) and (\ref{eq: res Jens 2}), we recognize all inequalities in (\ref{eq: en ineq 2}) as equalities if (\ref{eq: ch ru}) holds. In particular, $p(t) = P_t(u(t), \xi(t) )$ for a.e. $t \in (0, T)$. We have therefore proved that $(u, \xi)$ satisfies (\ref{eq: energy id 2}). Comparing (\ref{eq: energy id 2}) and (\ref{eq: en id}), we find for a.e. $t \in (0, T)$ that
		\begin{equation} \label{eq: mn slc}
			\begin{gathered}
				\phi_{t, u(t) } \left( u'(t) \right)
				= \int \phi_{t, u(t) }(v) \, d \mu_t(v, \eta, p), \\
				\phi^*_{t, u(t) } (- \xi(t) ) - P_t(u(t), \xi(t) )
				= \min_{(\eta, p) \in \SS_t } \phi^*_{t, u(t) } (- \eta) - p \\
				= \phi^*_{t, u(t) } (-\eta) - p
				\text{ for } \nu_t \text{-a.e. } (\eta, p) \in X \times \R.
			\end{gathered}
		\end{equation}

		Step 4: Now, to conclude the convergences (\ref{eq: cvgs}), we set
		\begin{align*}
			A_k(t) & = \int_0^{\overline{t}_{\tau_k}(t) } \phi_{\overline{t}_\tau(r), \underline{U}_\tau(r) } \left( U'_{\tau_k}(r) \right) \, dr, \\
			B_k(t) & = \int_0^{\overline{t}_{\tau_k}(t) } \phi^*_{\overline{t}_\tau(r), \underline{U}_\tau(r) } \left( - \tilde{\xi}_{\tau_k}(r) \right) \, dr, \\
			C_k(t) & = \EE_{\overline{t}_{\tau_k}(t) } \left( \overline{U}_{\tau_k}(t) \right), \\
			P_k(t) & = P_t \left( \tilde{U}_{\tau_k}(t), \tilde{\xi}_{\tau_k}(t) \right)
		\end{align*}
		and estimate for $t \in \cl \left[ 0, T \right)$
		\begin{equation} \label{eq: nqlt chain}
			\begin{aligned}
				& \int_0^t \phi_{r, u(t) } \left( u'(r) \right) + \phi_{r, u(t) } \left( - \xi(r) \right) \, dr + \EE_t \left( u(t) \right) \\
				& = \int_0^t \phi_{r, u(t) } \left( u'(r) \right) + \int \phi_{r, u(t) } \left( - \eta \right) \, \mu_r(v, \eta, p) \, dr + \EE_t \left( u(t) \right) \\
				& \le \liminf_k A_k(t) + \liminf_k B_k(t) + \liminf_k C_k(t) \\
				& \le \limsup_k \left( A_k + B_k + C_k \right)(t) \\
				& \le \limsup_k \EE_0 \left( \overline{U}_{\tau_k}(0) \right) + \lim_k \int_0^{\overline{t}_{\tau_k}(t) } P_k(r) \, dr \\
				& \le \EE_0 \left( u_0 \right) + \int_0^t P_r \left( u(r), \xi(r) \right) \, dr \\
				& \le \int_0^t \phi_{r, u(t) } \left( u'(r) \right) + \phi_{r, u(t) } \left( - \xi(r) \right) \, dr + \EE_t \left( u(t) \right),
			\end{aligned}
		\end{equation}
		where the first step follows from (\ref{eq: mn slc}), the second step from (\ref{eq: E cvg}), (\ref{eq: E cvg 2}) - (\ref{eq: E cvg 3}) and (\ref{eq: Y meas ineq}) - (\ref{eq: Y meas ineq 2}), the fourth step from (\ref{eq: energy id 2}), the fifth step from (\ref{eq: P exp}) and the sixth step from (\ref{eq: en id 2}). In conclusion, the inequality chain (\ref{eq: nqlt chain}) turns into an identity chain that implies (\ref{eq: cvgs}).
	\end{proof}
	
	We now come to our second main theorem on existence and stability for (\ref{eq: DNI}). We provide a sufficient condition under which the solutions to a sequence of generalized gradient systems
	$$
	\left( V_\dd, X_{\dd^*}, \EE^n, \phi^n, \langle \cdot, \cdot \rangle_{X, V}, F^n, P^n \right)
	$$
	converges to the solution of a limiting gradient system
	$$
	\left( V_\dd, X_{\dd^*}, \EE, \phi, \langle \cdot, \cdot \rangle_{X, V}, F, P \right).
	$$
	We impose the following assumptions on the energies $\EE^n$ and the dissipation potentials $\phi^n$:
	
	\paragraph{Energy functionals.} Let there be given a sequence
	$$
	\EE^n \colon \cl \left[ 0, T \right) \times V \to \left(- \i, \i \right]
	$$
	of $\LL_I \otimes \BB \left( V_\dd \right)$-measurable energy functionals with domains
	$$
	\dom \left( \EE^n \right) = \cl \left[ 0, T \right) \times D_n
	$$
	for $D_n \subset V$ independent of time and with subdifferential multimaps $F^n \colon \cl \left[ 0, T \right) \times D_n \rightrightarrows X$. Setting
	$$
	\GG_{T_0}^n(u) = \sup_{0 \le t \le T_0} \EE_t^n(u),
	$$
	we assume uniform versions of (\ref{eq: E_0}), (\ref{eq: crc}), (\ref{eq: absolute  continuity}), and (\ref{eq: radially time differentiable}):
	\\
	
	\begin{subequations} \label{seq: H1}
		\textbf{Lower Bound}:
		\begin{equation} \label{eq: H1.1}
			\exists C_0 > 0 \colon \EE_t^n(u) \ge C_0 \quad \forall (t, u) \in \cl \left[ 0, T \right)\times D_n.
		\end{equation}
		
		\textbf{Compactness}:
		\begin{equation} \label{eq: H1.2}
			\begin{gathered}
				\sup_n \left| u_n \right|_V < \i, \, \sup_n \EE_t^n(u_n) < \i \implies \\
				\left( u_n \right) \text{ has an } \ee \text{-convergent subsequence.}
			\end{gathered}
		\end{equation}
		
		\textbf{Absolute continuity}:
		\begin{equation} \label{eq: H1.3}
			\begin{gathered}
				\exists f_1 \in L^1_{\text{loc} } \left( \cl \left[ 0, T \right) \right) \, \forall u \in D_n \colon s, t \in \cl \left[ 0, T \right) \implies \\
				\left| \EE_s^n(u) - \EE_t^n(u) \right| \le \EE_t^n(u) \exp \left( \int_{\left[ s, t \right] } f_1(r) \, dr \right).
			\end{gathered}
		\end{equation}
		
		\textbf{Time subderivative control}: Let there exist $\LL_I \otimes \BB \left( D_\dd \times X_{\dd^*} \right)$-measurable functions $P^n \colon \graph(F^n) \to \R$ and for every $T_0 \in \cl \left[ 0, T \right)$ let there exists a locally integrable function $f_2 \in L^1_{\text{loc} } \left( \cl \left[ 0, T \right) \right)$ such that
		\begin{equation} \label{eq: H1.4}
			\begin{gathered}
				\forall (t, u, \xi) \in \graph \left( F^n \right) \colon \\
				\liminf_{h \downarrow 0} \frac{\EE_{t + h}^n(u) - \EE_t^n(u)}{h} \le P_t^n(u, \xi) \le \GG_{T_0}^n(u) f_2(t).
			\end{gathered}
		\end{equation}
	\end{subequations}
	Moreover, we require that there be a generalized gradient system
	$$
	\left( V_\dd, X_{\dd^*}, \EE, \phi, \langle \cdot, \cdot \rangle_{X, V}, F, P \right)
	$$
	such that the energy $\EE \colon \cl \left[ 0, T \right)\times V \to \left( - \i, \i \right]$ complies with (\ref{eq: E_0}) and the sequence $\EE^n$ converges to $\EE$ in the sense that, for every $t \in \cl \left[ 0, T \right)$, for every pair of sequences $u_n \in V$ and $\xi_n \in F_t^n(u_n)$, we have the implication
	\begin{subequations}
		\begin{equation} \label{eq: seq cc cnclsn}
			\begin{gathered}
				u_n \eeto u , \quad \xi_n \dddto \xi, \quad P_t^n \left( u_n, \xi_n \right) \to p \text{ in } \R, \\
				\limsup_n | u_n | + \EE_t^n(u_n) < \i \implies \\
				(t, u) \in \dom(F), \quad \xi \in F_t(u), \quad p \le P_t \left( u, \xi \right), \quad \liminf_n \EE_t^n(u_n) \ge \EE_t(u).
			\end{gathered}
		\end{equation}
		or the stronger implication
		\begin{equation} \label{eq: seq cc cnclsn str}
			\begin{gathered}
				u_n \eeto u , \quad \xi_n \dddto \xi, \quad P_t^n \left( u_n, \xi_n \right) \to p \text{ in } \R, \quad \EE_t^n(u_n) \to E \text{ in } \R, \\
				\limsup_n | u_n | + \EE_t^n(u_n) < \i \implies \\
				(t, u) \in \dom(F), \quad \xi \in F_t(u), \quad p \le P_t \left( u, \xi \right), \quad E = \EE_t(u).
			\end{gathered}
		\end{equation}
	\end{subequations}
	
	\paragraph{Dissipation potentials}
	
	Let
	$$
	\phi^n \colon \cl \left[ 0, T \right) \times D_n \times V \to \left[ 0, + \i \right]
	$$
	and its partial conjugate
	$$
	\phi^{*, n} \colon \cl \left[ 0, T \right) \times D_n \times X \to \left[ 0, + \i \right]
	$$
	with respect to the pair $(V, X)$ be a sequence of functions satisfying (\ref{eq: phi mb}), (\ref{eq: phi* mb}) and such that, for every $R > 0$ and $T_0 \in \cl \left[ 0, T \right)$, there hold the following assumptions:
	\begin{subequations} \label{seq: H2}
		\begin{alignat}{2}
			& \liminf_{\| v \|_V \to + \i} \inf_{n \ge 1} \inf_{0 \le t \le T_0} \inf_{\GG_{T_0}^n(u) \le R} & \frac{\phi_{t, u, n}(v)}{\| v \|_V} > 0; \label{eq: H2.1.1} \\
			& \liminf_{m \to + \i} \inf_{n \ge 1} \inf_{0 \le t \le T_0} \inf_{\GG_{T_0}^n(u) \le R} \inf_{\xi \in F_t^n(u) \colon \| \xi \| \ge m} & \frac{\phi^*_{t, u, n}(- \xi)}{\| \xi \|_X} > 0; \label{eq: H2.1.2}
		\end{alignat}
		For every $t \in \cl \left[ 0, T \right)$, we have
		\begin{equation} \label{eq: H2.2}
			\begin{gathered}
				u_n \eeto u, \, \GG_{T_0}^n(u_n) \le R, \, v_n \ddto v \implies
				\liminf_n \phi_{t, u_n, n}(v_n) \ge \phi_{t, u}(v); \\
				u_n \eeto u, \, \GG_{T_0}^n(u_n) \le R, \, \xi_n \dddto \xi, \xi_n \in F_t^n(u_n) \implies \\
				\liminf_n \phi^*_{t, u_n, n}(- \xi_n) \ge \phi^*_{t, u}(- \xi);
			\end{gathered}
		\end{equation}
		For every $w \in W$, the sequence of functions
		\begin{equation} \label{eq: H2.3}
			t \mapsto \sup_{\GG_{T_0}^n(u) \le R} \phi^{*, n}_{W, t, u}(w) 
			\text{ is uniformly integrable in } L_1 \left( 0, T_0 \right).
		\end{equation}
	\end{subequations}
	
	\paragraph{Remark on assumptions.}
	
	\begin{enumerate}
		
		\item The above assumptions are weaker than those for Theorem \ref{thm: main 1} even if $\phi^n = \phi$ is a constant sequence. Most importantly, (\ref{eq: H2.1.2}) and (\ref{eq: H2.2}) are conditioned on the subdifferential operator $F$, which improves (\ref{eq: phi.2.2.2}) and (\ref{eq: phi.4.2}). Implicitly, this also alleviates the continuity condition on $\phi$ at the origin. No analogue of (\ref{eq: phi.2.3}) and (\ref{eq: phi.3}) is required.
		
		\item While it is apparent that Theorem \ref{thm: main 2} concerns stability, it might be less clear to the reader why we also call it an existence result. This is related to the possibility of using it to approximate a generalized gradient flow satisfying the above assumptions (as a constant sequence of generalized gradient flow tuples) by flows that satisfy the assumptions of Theorem \ref{thm: main 1}. A simple method to do this is infimal convolution of dissipation potentials, e.g., setting for $r > 0$
		$$
		\phi_{t, u, r}(v) = \inf_{w \in B_r(v) } \phi_{t, u} \left( v - w \right) = \inf_{w \in V} \phi_{t, u} \left( v - w \right) + I_{B_r}(v),
		$$
		we have
		$$
		\phi^*_{t, u, r}(\xi) = \phi_{t, u} \left( \xi \right) + r \| \xi \|_X
		$$
		so that the regularized dual potential satisfies (\ref{eq: phi.2.2.2}). Different choices of convolution kernels or methods of regularization could be used to remove (\ref{eq: phi.2.3}) or (\ref{eq: phi.3}). For example, to remove (\ref{eq: phi.3}), one could consider the limit $r \to 0^+$ for
		\begin{equation} \label{eq: rdl rglrztn}
			\begin{gathered}
				\phi_{t, u, r}(v) = \inf_{x > 0} \phi_{t, u} \left( f \left( \tfrac{\lambda}{x} \right) v \right)
				\\
				\text{ with } f(r) =
				\begin{cases}
					\frac{1}{\sqrt{1 - (r - 1)^2}} & \text{if } r \in (-1, 1), \\
					+\infty & \text{else}.
				\end{cases}
			\end{gathered}
		\end{equation}
		The concrete choice of $f$ is not important, any differentiable convex function with bounded open domain and global minimum $f(1) = 1$ will guarantee (\ref{eq: phi.3}) by rendering the function $\lambda \mapsto \phi_{t, u, r}( \lambda v)$ differentiable. Note that by regularizing only along rays in (\ref{eq: rdl rglrztn}), we avoid assuming that $V$ allows a smooth convolution kernel, which is a non-trivial assumption for a general Banach space. One could try to combine several methods of regularization.
		
	\end{enumerate}
	
	\begin{theorem} \label{thm: main 2}
		Let the sequence $\left( V_\dd, X_{\dd^*}, \EE^n, \phi^n, \langle \cdot, \cdot \rangle_{X, V}, F^n, P^n \right)$ of generalized gradient flows comply with (\ref{seq: H1}), (\ref{seq: H2}) and either (\ref{eq: seq cc cnclsn}) or (\ref{eq: seq cc cnclsn str}). Let $u_{0, n} \in D_n$ be a sequence of initial values such that
		\begin{equation} \label{eq: init cvg}
			u_{0, n} \ddto u_0; \quad \EE^n_0 \left( u_{0, n} \right) \to \EE_0(u_0),
		\end{equation}
		and let the tuples $\left( u_n, \xi_n \right)$ be Lyapunov solutions to the Cauchy problem
		\begin{equation} \label{eq: Cau prbl}
			\p \phi^n_{t, u_n(t) } \left( u'_n(t) \right) + F_t^n \left( u_n(t) \right) \ni 0 \quad \text{for a.e. } t \in (0, T); \quad u_n(0) = u_{0, n}
		\end{equation}
		in the sense of Definition \ref{def: sol}, fulfilling in particular the energy inequality (\ref{eq: en id 3}). Moreover, let $\left( V_\dd, X_{\dd^*}, \EE, \phi, \langle \cdot, \cdot \rangle_{X, V}, F, P \right)$ comply with either the strong chain rule (\ref{eq: ch ru}) or the weak chain rule (\ref{eq: ch ru wkr}) and the continuous closedness implication (\ref{eq: cl str impl}). Then there exist a subsequence $n_k$ and a tuple
		$$
		(u, \xi) \in \bigcap_{T_0 \in \cl \left[ 0, T \right)} V_1 \left( 0, T_0 ; V_\delta \right) \times V_1 \left( 0, T_0 ; X_{\dd^*} \right),
		$$
		such that
		\begin{enumerate}
			
			\item if (\ref{eq: ch ru}), then $(u, \xi)$ is an energy solution to the Cauchy problem (\ref{eq: DNL inc}) in the sense of Definition \ref{def: sol} and the following convergences and relations hold as $k \to + \i$
			\begin{subequations}
				\begin{equation} \label{eq: stb cvg 1.n1}
					t_k \to t \text{ in } \cl \left[ 0, T \right) \implies u_{n_k}(t_k) \ddto u(t) \text{ and } u_{n_k}(t_k) \eeto u(t);
				\end{equation}
				\begin{equation} \label{eq: stb cvg 1.n2}
					\begin{gathered}
						u'_{n_k} \to u' \text{ in } \ss \left( V_1 \left( 0, T_0 ; V_\dd \right) ; E_\i \left( 0, T_0; W \right) \right)
						\\
						\forall T_0 \in \cl \left[ 0, T \right);
					\end{gathered}
				\end{equation}
				\begin{equation} \label{eq: stb cvg 1.n3}
					P^{n_k} \weakast p \text{ in } L_\i \left( 0, T_0 \right) \quad \forall T_0 \in \cl \left[ 0, T \right);
				\end{equation}
				\begin{equation} \label{eq: stb cvg 1.n5}
					\begin{gathered}
						\lim_k \int_s^t \phi_{r, u_{n_k}(r), n_k} \left( u'_{n_k}(r) \right) + \phi^*_{r, u_{n_k}(r), n_k} \left( - \xi_{n_k}(r) \right) \, dr \\
						= \int_s^t \phi_{r, u(r) } \left( u'(r) \right) + \phi^*_{r, u(r) } \left( - \xi(r) \right) \, dr \quad \forall s, t \in \cl \left[ 0, T \right).
					\end{gathered}
				\end{equation}
				\begin{equation} \label{eq: stb cvg 2.n1}
					\lim_k \EE_t^{n_k} \left( u_{n_k}(t) \right) = \EE_t (u(t) ) \quad \forall t \in \cl \left[ 0, T \right);
				\end{equation}
				\begin{equation} \label{eq: stb cvg 2.n3}
					p(t) = P_t \left( u(t), \xi(t) \right) \quad \text{ for a.e. } t \in (0, T).
				\end{equation}
			\end{subequations}
			
			\item If (\ref{eq: ch ru wkr}) and (\ref{eq: cl str impl}), then $(u, \xi)$ is a Lyapunov solution to the Cauchy problem (\ref{eq: DNL inc}) in the sense of Definition \ref{def: sol} and the following convergences and relations hold as $k \to + \i$:
			\begin{subequations}
				We have (\ref{eq: stb cvg 1.n1}), (\ref{eq: stb cvg 1.n2}), (\ref{eq: stb cvg 1.n3}), (\ref{eq: stb cvg 2.n1}) and
				\begin{equation} \label{eq: stb cvg 3.n1}
					p(t) \le P_t \left( u(t), \xi(t) \right) \quad \text{ for a.e. } t \in (0, T).
				\end{equation}
			\end{subequations}
		\end{enumerate}
	\end{theorem}
	
	\begin{proof}
		We prove both cases at once. For every $n \ge 1$, the solution tuple $\left( u_n, \xi_n \right)$ satisfies (\ref{eq: en id 3}), i.e.,
		\begin{equation} \label{eq: en id 2.1}
			\begin{gathered}
				\int_0^t \phi_{r, u_n(r), n} \left( u'_n(r) \right) + \phi^{*, n}_{r, u_n(r)} \left( - \xi(r) \right) \, dr + \EE_t^n \left( u_n(t) \right) \\
				\le \EE_0^n \left( u_{0, n} \right) + \int_0^t P_r^n \left( u_n(r), \xi_n(r) \right) \, dr \quad \forall t \in \cl \left[ 0, T \right).
			\end{gathered}
		\end{equation}
		Since the integral on the left of (\ref{eq: en id 2.1}) is non-negative, we may by (\ref{eq: H1.4}) estimate
		\begin{equation} \label{eq: tb gronwalled 1}
			\EE_t^n \left( u_n(t) \right) \le \EE_0^n \left( u_{0, n} \right) + \int_0^t \GG_{T_0}^n \left( u_n(r) \right) f_2(r) \, dr.
		\end{equation}
		By the Gronwall Lemma and (\ref{eq: H1.2}), we have
		\begin{equation} \label{eq: E lwr ctrl}
			\forall T_0 > 0 \, \exists C_3 > 0 \colon \forall u \in D_n \quad \GG_{T_0}^n(u) \le C_3 \inf_{0 \le t \le T_0} \EE_t(u)
		\end{equation}
		with a constant $C_3$ independent of $n$ so that (\ref{eq: tb gronwalled 1}) yields
		\begin{equation} \label{eq: tb gronwalled 2}
			\GG_{T_0}^n \left( u_n(t) \right) \le C \left[ \EE^n_0 \left( u_{0, n} \right) + \int_0^t \GG_{T_0}^n \left( u_n(r) \right) f_2(r) \, dr \right] \quad \forall n \ge 1.
		\end{equation}
		Again invoking the Gronwall Lemma, we find by the second of (\ref{eq: init cvg}) that
		\begin{equation} \label{eq: E_n bnd}
			\forall T_0 \in \cl \left[ 0, T \right) \, \exists S > 0 \colon \sup_{0 \le t \le T_0} \EE_t^n \left( u_n(t) \right) \le S.
		\end{equation}
		Combining (\ref{eq: E lwr ctrl}) and (\ref{eq: H1.4}), the estimate (\ref{eq: E_n bnd}) together with an argument involving Lebesgue points implies
		\begin{equation} \label{eq: P^n bnd}
			\exists S > 0 \colon \left| P_t^n \left( u_n(t), \xi_n(t) \right) \right| \le S + f_1(t) + f_2(t) \text{ for a.e. } t \in \left( 0, T\right),
		\end{equation}
		which in turn shows that
		\begin{equation} \label{eq: phi_n bnd}
			\forall T_0 \in \cl \left[ 0, T \right) \, \exists S > 0 \colon \int_0^t \phi_{r, u_n(r), n} \left( u'_n(r) \right) + \phi^{*, n}_{r, u_n(r)} \left( - \xi(r) \right) \, dr \le S
		\end{equation}
		by (\ref{eq: en id 2.1}). In particular,
		\begin{equation} \label{eq: lnr bnd}
			\forall T_0 \in \cl \left[ 0, T \right) \, \exists S > 0 \colon \int_0^t \| u'_n(r) \|_V + \| \xi(r) \|_X \, dr \le S
		\end{equation}
		by the uniform coercivity of the potentials (\ref{eq: H2.1.1}). Invoking (\ref{eq: phi_n bnd}), (\ref{eq: H1.2}), (\ref{eq: E_n bnd}), and (\ref{eq: H2.3}), we may argue as in the part stretching from (\ref{eq: w cpt cvg}) to (\ref{eq: cpt cvg}) in the proof of Proposition \ref{pr: ntrplnts tgthr} to conclude (\ref{eq: stb cvg 1.n1}). By analogy to the closing remark of the same proof, we find (\ref{eq: stb cvg 1.n2}). By (\ref{eq: P^n bnd}), we have (\ref{eq: stb cvg 1.n3}).
		
		Arguing as in the proof of Proposition \ref{pr: Y ms cnstr} on the time-continuous level, we invoke (\ref{eq: seq cc cnclsn}) together with (\ref{eq: stb cvg 1.n1}), (\ref{eq: stb cvg 1.n2}), (\ref{eq: lnr bnd}) to find a function $E \colon \cl \left[ 0, T \right) \to \R$ of locally bounded variation and a time-dependent Young measure
		$$
		\mu \in \YY \left( \cl \left[ 0, T \right) ; V_\dd \times X_{\dd^*} \times \R \right)
		$$
		such that, as $k \to +\i$, there hold the convergences and relations
		\begin{subequations}
			\begin{equation} \label{eq: E lmt}
				\begin{cases}
					\EE_t^{n_k} \left( u_{n_k}(t) \right) \to E(t) \quad & \forall t \in \cl \left[ 0, T \right), \quad E(0) = \EE_0(u_0),	\\
					E(t) \ge \EE_t (u(t) ) \quad & \forall t \in \cl \left[ 0, T \right), \\
					E(t) = \EE_t(u(t) ) \quad & \forall t \in \cl \left[ 0, T \right) \text{ if } (\ref{eq: seq cc cnclsn str});
				\end{cases}
			\end{equation}
			\begin{equation}
				u'(t) = \int v \, d \mu_t(v, \eta, p) \quad \text{ for a.e. } t \in \left( 0, T \right);
			\end{equation}
			\begin{equation}
				p(t) = \int p \, d \mu_t(v, \eta, p) \quad \text{ for a.e. } t \in \left( 0, T \right);
			\end{equation}
			\begin{equation}
				p(t) \le \int P_t \left( u(t), \eta \right) \, d \mu_t(v, \eta, p) \quad \text{ for a.e. } t \in \left( 0, T \right);
			\end{equation}
			For any $T_0 \in \cl \left[ 0, T \right)$ and $0 \le s \le t \le T_0$,
			\begin{equation} \label{eq: en ineq 3}
				\begin{aligned}
					\int_s^t \int \phi_{r, u(r) } (v) + \phi^*_{r, u(r) } (- \eta) \, d \mu_r(v, \eta, p) \, dr + E(t) \\
					\le E(s) + \int_s^t P(r) \, d r \le E(s) + \int_s^t \int P_r(u(r), \eta ) \, d \mu_r(v, \eta, p) \, dr.
				\end{aligned}
			\end{equation}
		\end{subequations}
		If (\ref{eq: ch ru wkr}) and (\ref{eq: seq cc cnclsn str}), then the third of (\ref{eq: E lmt}) yields the Lyapunov inequality (\ref{eq: en id 3}). If (\ref{eq: ch ru}), then we may use (\ref{eq: E lmt}) - (\ref{eq: en ineq 3}), as we used Proposition \ref{pr: Y ms cnstr} in the proof of Theorem \ref{thm: main 1}, to find (\ref{eq: stb cvg 2.n1}), (\ref{eq: stb cvg 1.n5}), (\ref{eq: stb cvg 2.n3}), (\ref{eq: DNL inc}), and (\ref{eq: en id 2}).
	\end{proof}
	
	\section{Application} \label{sec: Application}
	
	In this section, we apply our abstract existence result to non-smooth Allen-Cahn-Gurtin type inclusions on a time-space cylinder $(0, T) \times \Om$ with $T \in \left( 0, + \i \right]$ and an open subset $\Om \subset \R^d$. For integrands $\alpha$ and $\epsilon$, we study the inclusion
	\begin{equation} \label{eq: AC}
		\begin{cases*}
			\p \varphi(x, u') + \epsilon_u(x, u, \nabla u) - \DIV \p_z \epsilon(x, u, \nabla u) \ni 0; \\
			u(0) = u_0; \quad \left. \p_{\nu} \epsilon(x, u, \nabla u) \right|_{\p \Om} = 0.
		\end{cases*}
	\end{equation}
	For our existence result to hold, we require growth conditions of Orlicz type on the integrands as well as conditions on their integral functionals
	$$
	u \mapsto \int_\Om \varphi(x, u(x) ) \, dx; \qquad
	(u, v) \mapsto \int_\Om \epsilon(x, u(x), v(x) ) \, dx.
	$$
	Our result differs from previous contributions by neither requiring finiteness of $\epsilon$ nor a superlinear lower growth in the gradient variable that is uniform with respect to $x \in \Om$, which seems new even for parabolic equations where the time derivative enters linearly. Compare the recent result \cite{CGZa} on parabolic equations with very general Orlicz growth, where neither feature is present. Our result may also be viewed as an improvement over earlier work on local doubly nonlinear equations with variational structure, cf. \cite{ASch}, where the particular case of variable exponent growth is studied under the assumption of a compact Sobolev embedding from the energy topology into the strong topology of the dissipation space $V$. Besides allowing more general growth of the potentials, we have no need for such strong compactness assumptions. Instead, we employ our flexible closedness implication in connection with the locality of the primitives involved. We also discuss concrete sufficient conditions on $\alpha$ and $\epsilon$ for the assumptions on their integral functionals to hold.
	
	\subsection{Set-up and concrete assumptions}
	
	Before we state the main result of the section, we introduce the maps with the help of which we phrase it and prescribe them the properties needed for our analysis. Let $\left( \Om, \AA \right)$ be a complete measure space. A normal integrand on $\Om \times \R^n$ is a map $\Om \times \R^n \to (-\i, \i]$ that is measurable for the product $\sigma$-algebra $\AA \otimes \BB \left( \R^n \right)$ and is lower semicontinuous. Here and in the following, properties of an integrand requiring more structure than a measure space, such as being lower semicontinuous, convex, differentiable, etc., refer to that property being present in the second component for almost all $x \in \Om$. Whenever we formulate an inequality between integrands, this is to be understood as holding for all $u \in \R^n$ for almost all $x \in \Om$, i.e., there exists a null set $N \subset \Om$ such that the inequality holds for all $(x, u) \in \left( \Om \setminus N \right) \times \R^n$. From now on, let $\Om \subset \R^d$ be an open set of finite Lebesgue measure. We consider $\Om$ and $\cl \left[ 0, T \right) \times \Om$ as measure spaces carrying the restricted Lebesgue measures. Let $\alpha \colon \Om \times \R \to [0, + \i]$ and $\beta \colon \Om \times \R^{1 + d} \to [0, + \i]$ be Orlicz integrands as in Definition \ref{def: Orlicz integrand}. Let $m \in L^1(\Om; \R^+ )$ and $E \subset \Om$ a closed exceptional set such that $\HH^{d - 1}(E) = 0$, where $\HH^s$ is the $s$-dimensional Hausdorff measure. For $\alpha$ and the conjugate integrand $\alpha^*$ with respect to the last variable, we assume that there exists $C > 0$ such that
	\begin{subequations}
		\begin{equation} \label{eq: lph1}
			\alpha(x, 2u) \le C \alpha(x, u) + m(x);
		\end{equation}
		\begin{equation} \label{eq: lph2}
			u^* \mapsto \alpha^*(x, u^*) \text{ is real-valued for a.e. } x;
		\end{equation}
		\begin{equation} \label{eq: lph3}
			\alpha(\cdot, u) \in L_1(\Om) \quad \forall u > 0.
		\end{equation}
	\end{subequations}
	We introduce the Orlicz space $V = L_\alpha(\Om)$ with the dual space given by $V^* = L_{\alpha^*}(\Om)$ according to (\ref{eq: lph1}) and Corollary \ref{cor: C*}. For the Orlicz integrand $\beta$ and its conjugate integrand $\beta^*$, we assume that there exists $c > 0$ such that
	\begin{subequations}
		\begin{equation} \label{eq: bt ssmptns1}
			\sup_{|\xi| \le r} \beta^*(\cdot, 0, \xi) \in L^1_{\text{loc} }(\Om \setminus E) \quad \forall r > 0;
		\end{equation}
		\begin{equation} \label{eq: bt ssmptns3}
			c \| z^* \| \le \beta^*(x, v^*, z^*) + m(x);
		\end{equation}
		\begin{equation} \label{eq: bt ssmptns2}
			c \| z \| \le \beta(x, v, z) + m(x).
		\end{equation}
	\end{subequations}
	We set $\gamma(x, v, z) = \alpha(x, v) + \beta(x, v, z)$ and introduce the Orlicz-Sobolev space
	$$
	W = W^1 L_\gamma(\Om) = \left\{ u \in W^{1, 1}_{\text{loc} }(\Om) \st (u, \nabla u) \in L_\gamma(\Om) \right\}
	$$
	equipped with the inherited Luxemburg norm
	$$
	\| u \|_W = \| \left( u, \nabla u \right) \|_{L_\gamma}
	$$
	Let $\varphi \colon \cl \left[ 0, T \right) \times \Om \times \R^2 \to \left[ 0, +\i \right]$ be an $\LL \otimes \BB$-measurable, lower semicontinuous integrand such that
	\begin{subequations}
		\begin{equation} \label{eq: vrph ass1}
			\varphi(t, x, u, 0) = 0 \text{ for all } u \text{ and a.e. } (t, x);
		\end{equation}
		\begin{equation}
			\varphi(t, x, u, \cdot) \text{ is convex for all } u \text{ and a.e. } (t, x);
		\end{equation}
		\begin{equation} \label{eq: vrph ass3}
			\exists C > 0 \qquad \alpha(x, v) \le \varphi(t, x, u, v) \le C \alpha(x, Cv) + m(x);
		\end{equation}
		\begin{equation} \label{eq: vrph ass2}
			\begin{gathered}
				\lambda \mapsto \varphi (t, x, u, \lambda v) \text{ is differentiable at } \lambda = 1
				\\
				\text{ for all } (u, v) \text{ and a.e. } (t, x).
			\end{gathered}
		\end{equation}
	\end{subequations}
	We define the time-dependent family of dissipation potentials
	\begin{equation} \label{eq: df 1}
		\phi \colon \cl \left[ 0, T \right) \times V \times V \to \left[ 0, \i \right)
		\colon (t, u, v) \mapsto \int_\Om \varphi(t, x, u(x), v(x) ) \, dx.
	\end{equation}
	Let $\varphi^*(t, x, u, \xi)$ denote the conjugate integrand of $\varphi$ in the variable $v$. The partial convex conjugate of $\phi$ with respect to the last variable is given by
	\begin{equation} \label{eq: df 2}
		\phi^* \colon \cl \left[ 0, T \right) \times V \times V^* \to \left[ 0, \i \right]
		\colon (t, u, \xi) \mapsto \int_\Om \varphi^*(t, x, u(x), \xi(x) ) \, dx
	\end{equation}
	by Theorem \ref{thm: conjugate B}. The space $V$ is dual to the Orlicz class $C_{\alpha^*}(\Om)$ by Corollary \ref{cor: C*}. If we restrict $\phi^*$ to $C_{\alpha^*}(\Om)$ in the last component, then the restriction has $\phi$ as a convex conjugate with respect to the pairing of $C_{\alpha^*}(\Om)$ and $V$ by Theorem \ref{thm: conjugate A}. In particular, $\phi$ is weak* lower semicontinuous on $V$. We take the weak* topology of $V$ as the dissipation topology $\dd$ for $\phi$.
	
	Let $\epsilon \colon \Om \times \R \times \R \times \R^d \to \left[ 0, +\i \right] \colon (x, u, v, z) \mapsto \epsilon(x, u, v, z)$ be an $\LL \otimes \BB$-measurable and lower semicontinuous integrand. We suppose that there exist functions $\bar{u}_1, \bar{u}_2 \in W$ such that
	\begin{subequations}
		\begin{equation}
			\dom \epsilon(x, u, \cdot) \text{ is independent of } u \text{ for a.e. } x \in \Om;
		\end{equation}
		\begin{equation}
			\epsilon(x, u, \cdot) \text{ is convex and proper for all } u \text{ and a.e. } x;
		\end{equation}
		\begin{equation}
			\epsilon(x, \cdot, v, z) \text{ is differentiable for all } (v, z) \in \dom \epsilon(x, u, \cdot) \text{ and a.e. } x;
		\end{equation}
		\begin{equation} \label{eq: E bsc ass2}
			\begin{gathered}
				\text{The partial derivative } \epsilon_u(x, \cdot)
				\\
				\text{ is radially continuous on } \dom \epsilon(x, \cdot);
			\end{gathered}
		\end{equation}
		\begin{equation} \label{eq: E bsc ass3}
			\begin{gathered}
				\epsilon(x, u, v, z)
				\\
				\ge \beta(x, v - \bar{u}_1(x), z - \nabla \bar{u}_1(x) ) - c \left[ \alpha(x, u) + \alpha(x, v) \right] - m(x);
			\end{gathered}
		\end{equation}
		\begin{equation} \label{eq: E bsc ass4}
			\begin{gathered}
				\epsilon(x, u, v, z)
				\\
				\le \beta(x, v - \bar{u}_2(x), z - \nabla \bar{u}_2(x) ) + C \left[ \alpha(x, u) + \alpha(x, v) \right] + m(x).
			\end{gathered}
		\end{equation}
	\end{subequations}
	We introduce the functionals
	\begin{equation} \label{eq: energy}
		\begin{aligned}
			&\II_\epsilon \colon V \times L_\kappa(\Om) \times \BB(\Om) \to [0, \i] \colon (v, w; A) \mapsto \int_A \epsilon(x, v(x), w(x) ) \, dx; \\
			&\EE_W \colon V \times W \times \BB(\Om) \to \left[ 0, \i \right] \colon (v, w; A) \mapsto \II_\epsilon \left[ v, (w, \nabla w) ; A \right]; \\
			&\EE_V(v_1, v_2; A) =
			\begin{cases}
				\EE_W(v_1, v_2; A) & \text{ if } v_2 \in W; \\
				+ \i & \text{ if } v_2 \in V \setminus W;
			\end{cases}
			\\
			&\EE(u; A) = \EE_W(u, u; A); \qquad \EE(u) = \EE(u; \Om).
		\end{aligned}
	\end{equation}
	The energy of our generalized gradient flow will be given by the functional $\EE$. As the energy topology $\ee$ on the sublevels of $\EE$, we take the weak* topology $\ss \left( V; C_{\alpha^*}(\Om) \right)$. Let $\EE_W^* \colon V \times W^* \times \BB(\Om) \to \left( - \i, \i \right]$ be the partial convex conjugate of $\EE_W$ and
	$$
	\EE^*(u, \xi; A) =
	\begin{cases}
		\EE^*(u, \left. \xi \right|_W; A) & \text{ if } \xi \in V^*; \\
		+ \i & \text{ else.}
	\end{cases}
	$$
	Here and in the following, if $\xi \in W^*$, we mean by $\xi \in V^*$ that $\xi$ arises by restricting an element of $V^*$ to $W$. We set $\EE^*(u, \xi) = \EE^*(u, \xi; \Om)$. Let there be a mapping
	$$
	\Sigma \colon \left\{ (v, w) \in V \times W \st \EE_W(v, w; \Om) < \i \right\} \to L_1(\Om)
	$$
	that is bounded on every sublevel set of the functional $(v, w) \mapsto \EE_W(v, w; \Om)$. We assume the following growth condition:
	\begin{equation} \label{eq: E prtl grwth}
		\begin{gathered}
			\forall a > 0 \, \exists c > 0 \colon \left\{ \EE_W(v, w; \Om) \le a \right\} \implies \\
			c \alpha^* \left[ x, c \epsilon_u (x, v, w, \nabla w) \right] \le c^{-1} \alpha(x, v) + c^{-1} \alpha(x, w) + \Sigma(v, w)(x)
			\\
			\text{ for a.e. } x \in \Om.
		\end{gathered}
	\end{equation}
	We also assume that
	\begin{equation} \label{eq: eu a.e. cvg}
		\begin{gathered}
			u_n \weakast u \text{ in } V \text{ and } \EE(u_n; A) \to \EE(u; A) \text{ in } \R \implies \\
			\epsilon_u(x, u_n(x), u_n(x), \nabla u_n(x) ) \to \epsilon_u(x, u(x), u(x), \nabla u(x) )
			\\
			\text{ for a.e. } x \in A.
		\end{gathered}
	\end{equation}
	Here, $\epsilon_u$ is the derivative of $u \mapsto \epsilon(x, u, v, z)$. Finally, we assume that either
	\begin{subequations}
		\begin{equation} \label{eq: cpt sbls}
			\begin{gathered}
				u_n \in V \colon \sup_n \| u_n \|_V + \EE(u_n) < \i \implies
				\\
				\exists n_k \colon u_{n_k} \to u \text{ in } \mm(V, V^*);
			\end{gathered}
		\end{equation}
		or
		\begin{equation} \label{eq: emb grwth}
			\forall r > 0 \quad \sup_{| v | \le r} \gamma(x, v, 0) \in L_1(\Om).
		\end{equation}
	\end{subequations}
	In view of our other assumptions, we could equivalently ask norm bounded subsets of the sublevel sets of $\EE$ to be equi-integrable in $V$ instead of (\ref{eq: cpt sbls}). Our precise notion of solution to (\ref{eq: AC}) is encapsulated in the following existence result.
	
	\begin{theorem} \label{thm: rslt}
		There exist functions $(u, \xi)$ such that, for every real number $S \le T$, there holds
		$$
		u \in W^1 L_\phi \left( 0, S; L_\alpha(\Om) \right) \cap L_\i \left( 0, S; W^1 L_{\beta}(\Om) \right), \quad
		\xi \in V_{\phi^*} \left( 0, S; L_{\alpha^*}(\Om) \right)
		$$
		and the tuple $(u, \xi)$ is an energy solution to $\left( V_\dd, X_{\dd^*}, \EE, \phi, \langle \cdot, \cdot \rangle_{X, V}, F, P \right)$, i.e., it fulfills
		\begin{enumerate}[label=(\roman*)]
			\item $u(0) = u_0$ in $L_\alpha(\Om)$;
			\item $\xi(t) \in - \p \phi_t \left( u'(t) \right) \cap \p_D \EE \left( u(t) \right)$ for a.e. $t \in (0, S)$. \label{it: rslt 2}
		\end{enumerate}
		As the energy $\EE$ is finite on all of the space $W^1 L_{\beta}(\Om)$, \ref{it: rslt 2} also implies that $u$ satisfies the Neumann boundary conditions
		$$
		\p_\nu \epsilon(x, u(t, x), u(t, x), \nabla u(t, x) ) = 0 \text{ for a.e. } (x, t) \in \p \Om \times (0, T)
		$$
		in a generalized sense. In particular, the tuple $(u, \xi)$ satisfies the energy identity
		\begin{equation} \label{eq: app en id}
			\forall 0 \le s < t < T \qquad \int_s^t \phi_r \left( u(r) \right) + \phi^*_r \left( - \xi(r) \right) \, dr + \EE \left( u(t) \right) = \EE \left( u(s) \right).
		\end{equation}
	\end{theorem}
	
	\begin{proof}
		The statement will follow once we have checked the assumptions of Theorem \ref{thm: main 1}, which will be done below. Note that then (\ref{eq: app en id}) will imply $u \in W^1 L_\phi \left( 0, S; L_\alpha(\Om) \right) \cap L_\i \left( 0, S; W^1 L_{\beta}(\Om) \right)$ and $\xi \in V_{\phi^*} \left( 0, S; L_{\alpha^*}(\Om) \right)$ by (\ref{eq: vrph ass3}), (\ref{eq: eps 3}), and the conjugate estimate relating $\alpha^*$ and $\varphi^*$ obtained from (\ref{eq: vrph ass3}).
	\end{proof}

	\subsection{Verification of abstract assumptions}
	
	We check that the present setting falls within the scope of Theorem \ref{thm: main 1}.
	\\
	
	\emph{Spaces.} The Orlicz space $V = C_\alpha(\Om) = L_\alpha(\Om)$ will be the state space with $W = C_{\alpha^*}(\Om)$. The dual state space is given by the dual Banach space $X = L_{\alpha^*}(\Om)$ with $Y = C_\alpha(\Om)$. The pairing with respect to which we do subdifferential calculus will be the standard pairing $\langle \cdot, \cdot \rangle_{X, V}$.
	\\
	
	\emph{Dissipation potentials}: We check our assumptions on an abstract dissipation potential for the concrete case of $\phi$ given by (\ref{eq: df 1}). It clearly is lower semicontinuous. We have $\phi_{t, u}(0) = 0$ and $\phi \ge 0$ by the first of (\ref{eq: vrph ass1}) and $\varphi \ge 0$. For all $t \in \cl \left[ 0, T \right)$ and $u \in W$, the map $v \mapsto \phi_{t, u}(v)$ is convex since $\varphi$ is. Moreover
	$$
	\| u \|_\alpha > 1 \implies \| u \|_\alpha \le \int_\Om \alpha \left(x, v(x) \right) \, dx \le \phi_{t, u}(v)
	$$
	by (\ref{eq: vrph ass3}) and \cite[Cor. 2.1.15(b)]{DHHR}. Hence, the functional $\phi$ satisfies (\ref{eq: phi.2.2.1}). We have (\ref{eq: phi.2.3} by (\ref{eq: df 2}) and the pertaining explanation there. In particular
	$$
	\p \phi_{t, u}(v) = \left\{ \xi \in V^* \st \xi(x) \in \varphi_{t, u(x) }(v(x) ) \text{ for a.e. } x \in \Om \right\}.
	$$
	so that (\ref{eq: phi.3}) is inherited by the integral functional from the integrand by (\ref{eq: vrph ass2}) and the remark on (\ref{eq: phi.3}). Finally, we obtain (\ref{eq: phi.4.1}), (\ref{eq: phi.4.2}) as follows: Given $v_n, w_n \in V$ such that
	$$
	\limsup_n \EE(v_n) < \i
	$$
	and $(v_n, w_n) \to (u, v)$ in $\ss(V; C_{\alpha^*}(\Om))$, we are to select a subsequence (not relabeled) such that $\liminf_n \phi_{t, v_n}(w_n) \ge \phi_{t, u}(v)$ for any $t \in \cl [0, T)$. Combining (\ref{eq: bt ssmptns2}) and (\ref{eq: E bsc ass3}) with the Rellich-Kondrachov compact embedding, we may assume $v_n \to u$ in $L^1_{\text{loc} }(\Om)$. Invoking \cite[Lem. 2.31]{FoLe} and a standard diagonal argument, we find an increasing sequence of measurable sets $A_m \uparrow \Om$ and a subsequence of $w_n$ (not relabeled) that is equi-integrable and bounded in $L_1(A_m)$ for every $A_m$. Hence, it converges to $v$ weakly in $L_1(A_m)$ so that \cite[Thm. 1]{Io} yields
	\begin{align*}
		\liminf_n \phi_{t, v_n}(w_n)
		& \ge \sup_m \liminf_n \int_{A_m} \varphi(t, x, v_n(x), w_n(x) ) \, dx \\
		& \ge \sup_m \int_{A_m} \varphi(t, x, v(x), w(x) ) \, dx \\
		& = \phi_{t, u}(v)
	\end{align*}
	and (\ref{eq: phi.4.1}) has been checked. Regarding (\ref{eq: phi.4.2}), the proof that if $\xi_n \in V^*$ such that $\xi_n \weakast \xi$, then $\liminf_n \phi^*_{t, v_n}(\xi_n) \ge \phi^*_{t, u}(\xi)$ is analogous to the argument for $\phi$.
	\\
	
	\emph{Energy}: We check our assumptions on an abstract energy for the concrete case of $\EE$ given in (\ref{eq: energy}).
	
	\begin{lemma}[Lower semicontinuity and coercivity] \label{lem: crcvty}
		Let $A, A_n \in \BB(\Om)$ be measurable sets, $v, v_n \in V$ and $w, w_n \in W$ such that
		$$
		(\chi_{A_n}, v_n) \to (\chi_A, v) \text{ a.e. and } w_n \eeto w \text{ with } \sup_n \| w_n \|_W < \i.
		$$
		Then
		\begin{equation} \label{eq: crcvt chck 1.0}
			w_n \weak w \text{ in } W^{1, 1}_{\text{loc} } \left( \Om \setminus E \right) \text{ and } \EE_W(v, w; A) \le \liminf_n \EE_W(v_n, w_n; A_n).
		\end{equation}
		In particular, the function $u \mapsto \EE(u)$ satisfies (\ref{eq: lsc}) and (\ref{eq: crc}) for $\ee$ the weak* topology of $V$.
	\end{lemma}
	
	\begin{proof}
		Let $\delta(x, z) = v \beta(x, 0, z)$. By (\ref{eq: E bsc ass4}), we may extract a subsequence such that
		\begin{equation} \label{eq: crcvt chck 1.1}
			w_n \weakast w \text{ in } V, \quad \nabla w_n \weakast y \text{ in } L_\delta(\Om), \quad \nabla w_n \weak y \text{ in } L^1_{\text{loc} } \left( \Om \setminus E \right)
		\end{equation}
		where we used (\ref{eq: bt ssmptns1}) to obtain (i) the second convergence since bounded sets in $L_\delta(\Om) = C_{\delta^*}(\Om)^*$ are sequentially weak* compact by separability of the predual space due to Theorem \ref{thm: C sep} (ii) the third convergence since (\ref{eq: bt ssmptns1}) guarantees that $L_\i(K) \subset C_{\delta^*}(\Om)$ if $K \subset \subset \Om \setminus E$. Now, to conclude that $y = \nabla w$ by Lemma \ref{lem: nglgbl sts}, it remains to check that
		\begin{equation} \label{eq: lc cvg}
			w_n \to w \text{ in } L^1_{\text{loc} }(\Om)
		\end{equation}
		to conclude $y = \nabla w$. The convergence in (\ref{eq: lc cvg}) obtains by the Rellich-Kondrachov compact embedding theorem together with Corollary \ref{cor: cncrt pncr}. In total, we have found (\ref{eq: crc}). Let now $B_m \subset \subset \Om \setminus E$ be an increasing sequence of measurable sets such that $\left| \Om \setminus \bigcup_m B_m \right| = 0$ and $w_n \to w$ in $L_1(B_m)$ for every $m \in \N$. We set $A_{n, m} = A_n \cap B_m$ and $A_m = A \cap B_m$. Then, by \cite[Thm. 1]{Io} together with the third of (\ref{eq: crcvt chck 1.1}) and (\ref{eq: lc cvg}), we conclude
		\begin{align*}
			\liminf_n \EE_W(v_n, w_n, A_n)
			& \ge \liminf_n \int_{A_{n, m} } \epsilon \left(x, v_n(x), w_n(x), \nabla w_n(x) \right) \, dx \\
			& \ge \int_{A_m} \epsilon \left(x, v(x), w(x), \nabla w(x) \right) \, dx
		\end{align*}
		so that sending $m \to +\i$ obtains $\liminf_n \EE_W(v_n, w_n; A_n) \ge \EE_W(v, w; A)$ as remained to be shown. Regarding the addendum, bounded subsets of sublevel sets of $\EE$ are bounded in $W$ by (\ref{eq: E bsc ass3}) since $V$ agrees with its Orlicz class $C_\alpha(\Om)$ by (\ref{eq: lph1}). In particular, every sequence $u_n$ such that $u_n \to u$ in $\ee$ and $\sup_n \EE(u_n) < \i$ converges locally in measure by the Rellich-Kondrachov theorem and Corollary \ref{cor: cncrt pncr} so that the addendum follows from the main statement.
	\end{proof}
	
	We come to representing the Dini-Hadamard subdifferential $\p_H \EE$. We first obtain an upper estimate by computing the Dini subdifferential $\p_D \EE$, which is never smaller than $\p_H \EE$. Then, we show that $\p_H \EE = \p_D \EE$ by proving stability properties of $\p_D \EE$ that guarantee it cannot be larger than $\p_H \EE$.
	
	\begin{proposition}[subdifferential, convex part] \label{pr: F cnvx SD}
		Let $v \in V$, $A \in \BB(\Om)$, and $L_\kappa(\Om) \to [0, \i] \colon w \mapsto I(w) = \II_\epsilon(v, w; A)$. Setting
		\begin{align*}
			&\p_a I(w) = \biggl\{ \ell_a \in L_{\kappa^*}(\Om) \colon \ell_a(x) \in \chi_A(x) \p \epsilon(x, v(x), w(x) ) \text{ for a.e. } x \in \Om \biggr\}, \\
			&\p_f I(w) = \biggl\{ \ell_f \in I_{\kappa^*}(\Om) \colon \langle \ell_f, w \rangle = \max_{y \in \dom I} \langle \ell_f, y \rangle \biggr\},
		\end{align*}
		we have the direct sum decomposition
		\begin{equation} \label{eq: SD I}
			\p I(w) = \p_a I(w) \oplus \p_f I(w) \quad \forall w \in \dom I.
		\end{equation}
		Moreover, let $W \to [0, \i] \colon w \mapsto \II(w) = \EE(v, w; A)$. There holds
		\begin{equation} \label{eq: SD II}
			\p \II(w) = (1, \nabla)^* I(w, \nabla w) \quad \forall w \in \dom \II.
		\end{equation}
	\end{proposition}
	
	\begin{proof}
		We have (\ref{eq: SD I}) by Theorem \ref{thm: conjugate B} and Corollary \ref{cor: L a decomp}. (\ref{eq: SD II}): The definition of the Luxemburg norm shows that the functional $I$ is bounded above around $(\bar{u}_2, \nabla \bar{u}_2 )$ on $L_\kappa(\Om)$ by (\ref{eq: E bsc ass4}). Thus, it is continuous there by \cite[§3.2, Thm. 1]{IT}, whence the chain rule \cite[§4.2, Thm. 2]{IT} yields (\ref{eq: SD II}).
	\end{proof}
	
	\begin{proposition}[subdifferential, non-convex part] \label{pr: F G-dffrntbl}
		If $\EE_W(v, w; \Om) < \i$, then the partial Gateaux derivative $\p_v \EE_W(v, w; \Om)$ exists and is given by
		\begin{equation} \label{eq: prtl G-dffrntbl}
			\langle \p_v \EE_W(v, w; \Om), y \rangle = \int_\Om \epsilon_u \left(x, v(x), w(x), \nabla w(x) \right) y(x) \, dx \quad \forall y \in V.
		\end{equation}
		Moreover, for any $R > 0$, the derivative maps the set
		$$
		M_R = \left\{ (v, w) \in V \times W \st \| v \|_\alpha + \| w \|_\alpha + \EE_W(v, w; \Om) \le R \right\}
		$$
		into a bounded subset of $V^*$. In particular, we have the implication
		\begin{equation} \label{eq: cvg impl}
			\begin{gathered}
				\forall A \in \BB(\Om)\quad v_n \weakast v \text{ in } V \text{ and } \lim_n \EE(v_n; A) = \EE(v; A) \implies \\
				\p_v \EE_W(v_n, v_n; \Om) \weakast \p_v \EE_W(v, v; \Om) \text{ in } L_{\alpha^*}(A).
			\end{gathered}
		\end{equation}
	\end{proposition}
	
	\begin{proof}
		For $h > 0$, the mean value theorem yields a (measurable) function $\kappa \colon \Om \to (0, h)$ such that
		$$
		\chi_A \epsilon(\cdot, v + hy, v, \nabla v) - \chi_A \epsilon(\cdot, v, w, \nabla w) = \chi_A \epsilon_u(v + \kappa y, w, \nabla w) h y.
		$$
		Combining this with (\ref{eq: E prtl grwth}) we find by $\EE(w) < \i$ and $V = C_\alpha(\Om)$ together with Theorem \ref{thm: C max subspace} that $f_u(v + \kappa y, w, \nabla w)$ is bounded in $V^*$. Thus, the Young inequality shows that the difference quotient is bounded and equi-integrable in $L_1$, hence it converges strongly to its a.e. limit by the Vitali convergence theorem. Moreover, given sequences $v_n, w_n \in V$ such that $\| v_n \|_\alpha + \| w_n \|_\alpha + \EE(w_n) \le \alpha < \i$, we obtain a bound for $\epsilon_u(\cdot, v_n, w_n, \nabla w_n)$ in $V^*$ by (\ref{eq: E prtl grwth}) as before, so that the first addendum follows. Finally, we may invoke (\ref{eq: eu a.e. cvg}) to conclude the last claim. This uses that an a.e. limit must agree with a weak* limit in $L_{\alpha^*}(A)$, which is elementary to check.
	\end{proof}
	
	\begin{lemma} \label{lem: F SD}
		The energy $\EE \colon V \to \left( - \i, \i \right]$ is $\| \cdot \|_V$-$\om$-Dini-subdifferentiable at every $u \in \dom \EE$ in the sense defined in Theorem \ref{thm: smcnvx} with $\om \colon V \times V \to \left[ 0, \i \right]$ a locally bounded function such that
		\begin{equation} \label{eq: om vnshs}
			\forall u \in \dom(\EE) \quad \forall w \in V \quad
			\limsup_{h \downarrow 0, v \to w} \om(u + hv, u) = 0.
		\end{equation}
		Moreover, we have $\p_D \EE(u) = \p_H \EE(u)$ for all $u \in \dom \EE$. A functional $\xi \in V^*$ belongs to $\p_D \EE(u)$ if and only if the restriction of $\xi$ to $W$ belongs to the subspace Dini-subdifferential $\p_D \EE_W(u)$ of the restricted energy $\EE_W(u) = \EE_W(u, u)$. Finally, there holds
		\begin{equation} \label{eq: SD}
			\p_D \EE_W(u) = \p_v \EE_W(u, u) + \p \EE_W(u, \cdot)(u)
		\end{equation}
		with
		\begin{equation} \label{eq: SD 2}
			\begin{gathered}
				\p \EE_W(u, \cdot)(u) = \\
				\left\{ \eta_f \in F_{\gamma^*}(\Om) \st \forall y \in \dom \EE_W(u, \cdot) \quad \langle \eta_f; u, \nabla u \rangle \ge \langle \eta_f; y, \nabla y \rangle \right\} \\
				+ \left\{ \eta_a \in L_{\gamma^*}(\Om) \st \eta_a(x) \in \p_{v, z} \epsilon(x, u(x), u(x), \nabla u(x) ) \text{ for a.e. } x \in \Om \right\}.
			\end{gathered}
		\end{equation}
	\end{lemma}
	
	\begin{proof}
		We start by proving an upper estimate on $\p_D \EE(u)$. We know that $\xi \in \p_D \EE(u)$ iff the restriction of $\xi$ to $W$ belongs to $\p_D \EE_W(u)$ by Lemma \ref{lem: smplst cr} because $\dom \EE \subset W$ by (\ref{eq: E bsc ass3}). We claim that
		\begin{equation} \label{eq: clm 1}
			\p_D \EE_W(u) \subset \p_D \EE_W(\cdot, u)(u) + \p_D \EE_W(u, \cdot)(u).
		\end{equation}
		To show this, it suffices to prove that
		$$
		\EE'_{W, D}(u; y) \le \EE'_{W, D} (u, u; y, 0) + \EE'_{W, D} (u, u; 0, y)
		$$
		for any $y \in V$ such that the right-hand side is not $+ \i$. Hence, we may assume that
		$$
		\EE_V(u, u + hy) < \i
		$$
		for all step-widths $h > 0$ that are sufficiently small. But then Proposition \ref{pr: F G-dffrntbl} implies that $v \mapsto \EE_V(v, u + h y)$ is Gateaux differentiable hence radially continuous. Let $\theta \ge h > 0$. We combine the Gateaux differentiability with the fact that the difference quotient of a convex function is non-decreasing in the step-width to estimate
		\begin{equation*}
			\begin{aligned}
				\delta_h \EE_W(u; y)
				& = \delta_h \EE_W(u, u; y, 0) + \delta_h \EE_W(u + hy, u; 0, y) \\
				& \le \delta_h \EE_W(u, u; y, 0) + \delta_\theta \EE_W(u + hy, u; 0, y) \\
				& \to \EE'_{W, D}(u, u; y, 0) + \delta_\theta \EE_W(u, u; 0, y) \\
				& \to \EE'_{W, D}(u, u; y, 0) + \EE'_{W, D}(u, u; 0, y)
			\end{aligned}
		\end{equation*}
		as $h \to 0^+$ and $\theta \to 0^+$, consecutively. Having shown (\ref{eq: clm 1}), we turn to proving the lower estimate
		\begin{equation} \label{eq: clm 2}
			\p_H \EE_W(u) \supset \p_D \EE_W(u, u) + \p \EE_W( u, \cdot)(u)
		\end{equation}
		on the subdifferential. For this, we prove first that $\EE$ is $\| \cdot \|_V$-$\om$-Dini-subdifferentiable for an $\om$ satisfying (\ref{eq: om vnshs}). Let $\xi_1 = \p_v \EE_V(u, u) = \p_v \EE_W(u, u)$ and $\xi_2 \in \p \EE_W(u, \cdot)(u)$. We have
		\begin{align*}
			\EE(y) - \EE(u) - \langle \xi_1 + \xi_2, y - u \rangle
			& = \EE(y) - \EE_W(u, y) - \langle \xi_1, y - u \rangle \\
			& + \EE_W(u, y) - \EE(u) - \langle \xi_2, y - u \rangle \\
			& \ge \EE(y) - \EE_W(u, y) - \langle \xi_1, y - u \rangle \\
			& \eqqcolon \om(y, u) \| y - u \|_V.
		\end{align*}
		Proposition \ref{pr: F G-dffrntbl} renders the function $\om$ locally bounded. The mean value theorem yields $\kappa \in (u, y)$ such that
		$$
		\EE(y) - \EE_W(u, y) - \langle \xi_1, y - u \rangle = \langle \p_v \EE_W(\kappa, y) - \p_v \EE_W(u, u), y - u \rangle.
		$$
		Consequently, if $y = u + h z$, then (\ref{eq: E bsc ass2}) and Proposition \ref{pr: F G-dffrntbl} together with the fact that $\EE_W = \EE_V$ on $\dom \EE_V$ implies $\EE_W(\kappa, y) \weakast \xi_1$ in $V^*$ as $h \downarrow 0$ so that
		\begin{equation} \label{eq: rmndr}
			\forall \bar{z} \in V \colon \EE'_{V, H}(u; \bar{z} ) < \i \implies \limsup_{h \to 0^+ 0, z \to \bar{z} } \om(u + h z, u) = 0.
		\end{equation}
		Here, we used that a finite Dini-Hadamard subderivative in some direction entails eventual finiteness of the lower limit of the difference quotient in that direction by lower semi-continuity of $\EE$. In particular, $\xi = \xi_1 + \xi_2 \in \p_H \EE_W(u)$ so that (\ref{eq: clm 2}) has been shown. Putting together (\ref{eq: clm 1}) and (\ref{eq: clm 2}), we have proven (\ref{eq: SD}). Consequently, we find $\EE$ to be $\| \cdot \|_V$-$\om$-Dini-subdifferentiable as claimed since $\left. \p_D \EE(u) \right|_W = \p_D \EE_W(u)$. This in turn implies $\p_D \EE(u) = \p_H \EE(u)$ due to (\ref{eq: rmndr}). We know (\ref{eq: SD 2}) by Proposition \ref{pr: F cnvx SD}.
	\end{proof}
	
	\begin{corollary}[chain rule]
		The subdifferential $(t, u) \mapsto F_t(u) = \p_D \EE(u) = \p_H \EE(u)$ satisfies the chain rule assumption (\ref{eq: ch ru}).
	\end{corollary}
	
	\begin{proof}
		Combining Lemma \ref{lem: ch rl} with the subdifferentiability provided by Lemma \ref{lem: F SD}, we conclude that (\ref{eq: ch ru ineq}) holds for our choice of $\EE$ and $\phi$ so that (\ref{eq: ch ru}) is satisfied for $\p_H \EE(u)$.
	\end{proof}
	
	\emph{Closedness implication}: We start with a preparatory lemma of general type.
	
	\begin{lemma} \label{lem: G-lim SD clsdnss}
		Let $\XX, \YY$ be a dual pair of locally convex Hausdorff spaces and $F_\alpha \in \Gamma(\XX)$, be a net of functions such that $\Gamma$-$\lim F_\alpha = F \in \Gamma(\XX)$ with recovery sequences that converge in $\ss(\XX, \YY)$, i.e.
		\begin{equation} \label{eq: rcv sqc}
			\forall x \in \dom(F) \, \exists x_\alpha \in \XX \colon x_\alpha \to x \text{ in } \ss(\XX, \YY) \text{ and } F_\alpha(x_\alpha) \to F(x).
		\end{equation}
		Moreover, let the convex conjugate $F^*$ be lower semicontinuous, for example, let $\YY$ carry the weak topology $\ss(\YY, \XX)$. Then we have the implication
		\begin{equation} \label{eq: w assu}
			\begin{gathered}
				z_\alpha \to z \text{ in } \XX, \quad y_\alpha \to y \text{ in } \YY \implies \\
				\liminf F_\alpha(z_\alpha) + F^*_\alpha(y_\alpha) \ge F(z) + F^*(y).
			\end{gathered}
		\end{equation}
		Moreover, if the implication (\ref{eq: w assu}) is true and if $z_\alpha \to z$ in $\XX$ and $y_\alpha \to y$ in $\YY$ are nets such that $y_\alpha \in \p F_\alpha(z_\alpha)$ and $\langle z_\alpha, y_\alpha \rangle \to \langle z, y \rangle$, then there holds
		\begin{equation}
			\lim F_\alpha(z_\alpha) = F(z), \quad \lim F^*_\alpha(y_\alpha) = F^*(y), \quad y \in \p F(z).
		\end{equation}
	\end{lemma}
	
	\begin{proof}
		The lower semicontinuity of $F^*$ implies
		\begin{equation} \label{eq: eps 1}
			\forall \e > 0 \, \exists \alpha' \in I \colon \alpha \ge \alpha' \implies F^*(y) \le F^*(y_\alpha) + \frac{\e}{3}.
		\end{equation}
		By definition of the convex conjugate,
		\begin{equation} \label{eq: eps 2}
			\forall \e > 0 \, \forall \alpha \in I \, \exists u_\alpha \in \XX \colon F^*(y_\alpha) \le \langle u_\alpha, y_\alpha \rangle - F(u_\alpha) + \frac{\e}{3}.
		\end{equation}
		By (\ref{eq: rcv sqc}),
		\begin{equation} \label{eq: eps 3}
			\forall \e > 0 \, \forall \alpha \in I \, \exists \beta \in I \, \exists v \in \XX \colon \langle u_\alpha, y_\alpha \rangle - F(u_\alpha) \le \langle v, y_\alpha \rangle - F_\beta(v) + \frac{\e}{3}.
		\end{equation}
		Putting together (\ref{eq: eps 1}), (\ref{eq: eps 2}), and (\ref{eq: eps 3}) obtains
		\begin{align*}
			F^*(y)
			\le F^*(y_\alpha) + \frac{\e}{3}
			\le \langle u_\alpha, y_\alpha \rangle - F(u_\alpha) + \frac{2\e}{3}
			& \le \langle v_\beta, y_\alpha \rangle - F_\beta(v_\beta) + \e \\
			& \le F^*_\beta(y_\alpha) + \e
		\end{align*}
		so that $\liminf_{\alpha, \beta} F^*_\beta(y_\alpha) \ge F^*(y)$, from which (\ref{eq: w assu}) follows. Finally, if (\ref{eq: w assu}) is true, then Fenchel-Young implies
		\begin{equation*}
			\langle z, y \rangle
			= \lim \langle z_\alpha, y_\alpha \rangle
			= \lim F_\alpha(z_\alpha) + F^*_\alpha(y_\alpha)
			\ge F(z) + F^*(y)
			\ge \langle z, y \rangle. \qedhere
		\end{equation*}
	\end{proof}
	
	\begin{proposition} \label{pr: equi-crc*}
		Let $\XX$ be a Banach space and $f_i \in \Gamma(\XX)$ a family such that
		$$
		\exists \e > 0 \, \exists \delta > 0 \colon | x | < \e \implies \sup f_i(x) < \delta.
		$$
		Then there holds
		$$
		\inf f^*_i(x') \ge \e | x' | - \delta \quad \forall x' \in \XX^*.
		$$
	\end{proposition}
	
	\begin{proof}
		By definition of the convex conjugate, there holds
		$$
		\left( \inf f^*_i \right)^* = \left( \clco \inf f^*_i \right)^* = \sup f_i^{**} = \sup f_i,
		$$
		hence $\left( \sup f_i \right)^* = \clco \inf f^*_i$. Therefore, the matter reduces to considering a family with a single member $f$. By definition of the convex conjugate, we have
		\begin{equation*}
			f^*(x') \ge \sup_{| x | < \e} \langle x', x \rangle - f(x) \ge \e | x' | - \delta. \qedhere
		\end{equation*}
	\end{proof}
	
	\begin{lemma} \label{lem: E SD clsdnss}
		Let $u_n \eeto u$ and $\xi_n \weakast \xi$ in $V^*$ be sequences such that $\sup_n \EE(u_n) < \i$ and $\xi_n \in \p_D \EE(u_n)$. Then $\xi \in \p_D \EE(u)$.
	\end{lemma}
	
	\begin{proof}
		We are to consider the cases when either (\ref{eq: cpt sbls}) or (\ref{eq: emb grwth}) holds. Let (\ref{eq: cpt sbls}) be true. Then, the claim follows by Lemma \ref{lem: sd clsdnss} for $V$ in its strong topology $\mm(V, V^*)$. Note in this regard that (\ref{eq: f cnt cvg}) obtains by Lemma \ref{lem: F SD}, where $\| \cdot \|_V$-$\om$-subdifferentiability was established for an $\om$ that is locally bounded. Given (\ref{eq: f cnt cvg}), we deduce (\ref{eq: cl cnd 1.2}) by another invocation of $\| \cdot \|_V$-$\om$-subdifferentiability.
		
		Let (\ref{eq: emb grwth}) be true. We decompose $\xi_n = \p_v \EE_V(u_n, u_n; \Om) + \psi_n$ with $\psi_n \in \p \EE_V \left( u_n, \cdot \right)(u_n)$ and $\psi_n = \eta_{1, n} + \nabla^* \eta_{2, n}$ by Lemma \ref{lem: F SD}. Passing to a subsequence (not relabeled), we may arrange $\psi_n \weakast \psi$ in $V^*$ as $\p_v \EE_V(u_n, u_n)$ is bounded in $V^*$ by Proposition \ref{pr: F G-dffrntbl}. We claim that
		\begin{equation} \label{eq: 1st clm}
			\psi \in \p \EE_V(u, \cdot)(u) \text{ or equivalently } \psi \in \p \EE_W(u, \cdot) \text{ and } \psi \in V^*.
		\end{equation}
		Because we know $\psi \in V^*$, we may consider the claim with $\EE_W$. Towards this, we want to invoke Lemma \ref{lem: G-lim SD clsdnss} for the dual pair consisting of $W$ in the Mackey topology $\mm(W, V^*)$ and $W^*$ in the weak* topology $\ss(W^*, W)$. Note that $\mm(W, V^*)$ is the same as the strong topology of $V$ restricted to $W$. We set $u_{n, k} = \max \{ k, \min\{ - k, u_n \} \}$ for $k > 0$ and introduce the measurable sets $A_{n, k} = \left\{ \left| u_n \right| < k \right\}$. Note that $u_{n, k} \to u_k$ as $n \to + \i$ strongly in $V$ if $k$ is kept fixed and $u_k \to u$ as $k \to + \i$ strongly in $V$ by (\ref{eq: lph3}) so that there exists a sequence $\ell = \ell(n) \to + \i$ such that $\lim_n u_{n, \ell} = u$ strongly in $V$. We consider the net of functions $\EE_W(u_n, \cdot; A_{n, \ell} )$, whose convex conjugates $\EE^*_W(u_n, \cdot; A_{n, \ell})$ are lower semicontinuous since we equipped $W^*$ with its weak* topology. We check the assumptions of Lemma \ref{lem: G-lim SD clsdnss}, starting with the first auxiliary claim that
		\begin{equation} \label{eq: 1st aclm}
			\Gamma \text{-} \lim_n \EE_W(u_n, \cdot; A_{n, \ell} ) = \EE_W(u, \cdot) \text{ and pointwise on } \dom \EE_W(u, \cdot).
		\end{equation}
		The $\Gamma$-$\liminf$ inequality follows from Lemma \ref{lem: crcvty} so that it remains to prove the pointwise convergence. Let $w \in W$ with $\EE_W(u, w) < \i$. Then Proposition \ref{pr: F G-dffrntbl} yields the local Lipschitz continuity of $v \mapsto \EE_W(v, w)$ on $V$ so that 
		$$
		\EE_W(u_n, w; A_{n, \ell} ) \le \EE_W(u_{n, \ell}, w) \le \EE_W(u, w) + \left| \EE_W(u_{n, \ell}, w ) - \EE_W(u, w) \right| \to 0.
		$$
		Hence, $\lim_n \EE_W(u_n, w; A_{n, \ell} ) \to \EE_W(u, w)$ by the Fatou Lemma. Setting $\eta_{n, \ell} = \chi_{A_{n, \ell} } \eta_n$, Proposition \ref{pr: F cnvx SD} shows that
		\begin{equation} \label{eq: 2nd aclm}
			(1, \nabla)^* \eta_{n, \ell} \in \p \EE_W(u_n, \cdot; A_{n_\ell} )(u_{n, \ell} ).
		\end{equation}
		Moreover, there exists $\eta \in L_\gamma(\Om)^*$ such that $\psi = (1, \nabla^*) \eta$ on $W$. Our next auxiliary claim is that
		\begin{equation} \label{eq: 3rd aclm}
			\langle (1, \nabla)^* \eta_{n, \ell}, u_{n, \ell} \rangle \to \langle (1, \nabla^*) \eta, u \rangle \text{ as } n \to + \i.
		\end{equation}
		Regarding $W$ as a closed subspace of $L_\gamma(\Om)$ via the canonical identification $u \leftrightarrow (u, \nabla u)$, we have $W^* \cong L_\gamma(\Om)^* / W^\perp$. Therefore, by definition of the quotient norm, the sequence $(1, \nabla)^* \eta_n$ is equivalent in $W^*$ to a sequence $(1, \nabla)^* \bar{\eta}_n$ with $\bar{\eta}_n$ bounded. Hence, it is no restriction to assume that $\eta_n$ itself is bounded. Setting $\delta(x, v) = \beta(x, v, 0)$, we know by \ref{eq: emb grwth} that the component $\eta_{1, n}$ is bounded in $L_{\delta^*}(\Om) = C_\delta(\Om)^*$. Therefore, since $C_\delta(\Om)$ is separable by Theorem \ref{thm: C sep}, we may pass to a weak* convergent subsequence (not relabeled). Because $u_{n, k} \in C_\delta(\Om)$ for every $k > 0$, there holds
		\begin{equation} \label{eq: n/k cvg}
			\begin{aligned}
				\langle (1, \nabla)^* \eta_{n, k}, u_{n, k} \rangle
				& = \langle \eta_n , (1, \nabla) u_{n, k} \rangle - \langle \eta_{1, n}, \chi_{\Om \setminus A_{n, k} } u_{n, k} \rangle \\
				& = \langle \xi_n, u_{n, k} \rangle - \langle \eta_{1, n}, \chi_{\Om \setminus A_{n, k} } u_{n, k} \rangle \\
				& \to \langle \xi, u_k \rangle - \langle \eta_1, \chi_{\Om \setminus A_k} u_k \rangle \text{ as } n \to +\i \\
				& \to \langle \xi, u \rangle \text{ as } k \to +\i \\
				& = \langle (1, \nabla)^* \eta, u \rangle.
			\end{aligned}
		\end{equation}
		Here, we used that $\eta^f_1$ vanishes on $C_\delta(\Om)$ by Corollary \ref{cor: C*} so that only the absolutely continuous part $\eta^a_1$ acts on $u_k \in C_\delta(\Om)$, which allows passing to the limit as $k \to +\i$ by absolute continuity of the integral even though $u_k \to u$ need not hold strongly in $L_\delta(\Om)$. By possibly decreasing the speed at which $\ell = \ell(n)$ tends to infinity, we may arrange (\ref{eq: 3rd aclm}) by (\ref{eq: n/k cvg}). Taken together (\ref{eq: 1st aclm}), (\ref{eq: 2nd aclm}) and (\ref{eq: 3rd aclm}) enable to invoke Lemma \ref{lem: G-lim SD clsdnss} so that
		\begin{equation} \label{eq: SD cnclsn}
			\begin{gathered}
				\lim_n \EE(u_{n, \ell}; A_{n, \ell} ) = \EE(u), \qquad \lim_n \EE^*_W(u_{n, \ell}, \psi_n; A_{n, \ell} ) = \EE^*(u, \psi), \\ \psi = (1, \nabla) \eta \in \p \EE(u, \cdot)(u).
			\end{gathered}
		\end{equation}
		In particular, since the sum of two lower semi-continuous functions is continuous only if each addend is continuous, we deduce that $\lim_n \EE(u_{n, \ell}; A) = \EE(u; A)$ for every measurable set $A$ that eventually is contained in $A_{n, \ell}$. Therefore, since $\p_v \EE_W(u_n, u_n)$ is bounded in $V^*$, we have for every such $A \in \BB(\Om)$ that
		$$
		\chi_A \p_v \EE_W(u_{n, \ell}, u_{n, \ell} ) \weakast \chi_A \p_v \EE_W(u, u) \text{ in } V^*
		$$
		by (\ref{eq: cvg impl}) in Proposition \ref{pr: F G-dffrntbl}. Since the sequence $A_{n, \ell}$ converges to a set of full measure as $n, \ell \to +\i$, we conclude
		$$
		\p_v \EE_W(u_{n, \ell}, u_{n, \ell} ) \weakast \p_v \EE_W(u, u) \text{ in } V^*.
		$$
		In total, $\xi = \p_v \EE_W(u, u) + \psi \in \p_D \EE(u)$ by Lemma \ref{lem: F SD}.
	\end{proof}
	
	\emph{Sum rule}: We invoke \cite[Thm. 4.101]{Pe} for the Dini-Hadamard subdifferential $\p_H$ on the Banach space $V$, which is separable hence Hadamard-smooth by \cite[Thm. 3.95(a)]{Pe}. The function $\phi$ is soft in the sense of \cite[Def. 4.99]{Pe} by convexity, while $\EE$ is soft by Lemma \ref{lem: E SD clsdnss}. In total, the application of \cite[Thm. 4.101]{Pe} has been justified.
	
	\newpage
	
	\chapter{Generalized Orlicz spaces with Banach-values} \label{prt: GOS}
	
	In this part of the thesis, we begin a theory of non-separably vector valued Orlicz spaces $L_\varphi(\mu)$ generated by an even convex integrand $\varphi \colon \Om \times X \to \left[ 0, \i \right]$ when the range Banach space $X$ is arbitrary. Requiring $\varphi$ to satisfy
	$$
	\lim_{x \to 0} \varphi(\om, x) = 0; \quad \lim_{\| x \| \to \i} \varphi(\om, x) = \i \quad \text{ for } \mu\text{-a.e. } \om \in \Om,
	$$
	our $L_\varphi(\mu)$ are those strongly measurable functions with finite Luxemburg norm
	
	\begin{equation} \label{eq: Minkowksi functional}
		\| u \|_\varphi = \inf \left\{ \alpha > 0 \st \int \varphi \left[ \om, \alpha^{-1} u \left( \om \right) \right] \, d \mu \left( \om \right) \le 1 \right\}.
	\end{equation}
	
	What is new in our approach is that we overcome the need for separability of $X$ and for the Radon-Nikodym property of the dual space $X^*$ while yet allowing a wide class of possibly $\om$-dependent Orlicz integrands.
	
	\section{Convex conjugacy of integral functionals} \label{sec: inf-int}
	
	We prove in this section an interchange criterion between infimum and integral and compute with it the convex conjugate of a general integral functional $I_f$ on a space of merely measurable functions. Besides representing the subdifferential, we conclude from the conjugate formula a characterization of those integrands for which integration and convex conjugacy continue to commutate as if $X$ were separable. To make our criterion applicable, we propose two sufficient conditions, cf. Lemmas \ref{lem: brl sff} and \ref{lem: dualizable suff cond}. Even if $X$ is separable, our result is more general than previous ones since the measure $\mu$ may be arbitrary. The criterion could be further generalized by working with the notion of an integrand decomposable relatively to a function space instead of the function space itself being decomposable, cf., e.g., \cite{Gi2} for this idea. We shall briefly relate our result to similar criteria after the proof.
	
	\subsection{Interchange criterion}
	
	Before we can state and prove our interchange criterion, we define necessary notions and provide measure theoretic background material. We work with a metric range space $M$ as this adds no complications.
	
	\begin{definition}[almost decomposable space] \label{def: almost decomposable}
		A space $S$ of (strongly) measurable functions $u \colon \Omega \to M$ is called almost decomposable with respect to $\mu$ if for every $u_0 \in S$, every $F \in \AA_f$, every $\e > 0$ and every bounded (strongly) measurable function $u_1 \colon F \to M$ there exists $F_\e \subset F$ with $\mu \left( F \setminus F_\e \right) < \e$ such that the function
		\begin{equation} \label{eq: almost decomposable}
			u(\om) =
			\begin{cases}
				u_0(\om) & \text{ for } \om \in \Om \setminus F_\e, \\
				u_1(\om) & \text{ for } \om \in F_\e
			\end{cases}
		\end{equation}
		belongs to $S$. The space $S$ is decomposable if $F_\e = F$ may be chosen. $S$ is called weakly (almost) decomposable if only $u_1 \in S$ are allowed.
	\end{definition}
	
	Equivalently, $u_1$ may be unbounded in the definition of almost decomposability. However, the same is not possible for decomposability. If two function spaces defined over the same measure space $\Om$ and the same range space $M$ are almost decomposable and weakly decomposable, then their intersection retains both properties. If $S$ is a weakly decomposable vector space of $X$-valued functions, then its weak decomposability is equivalent to closedness under multiplication by indicators of sets having finite or co-finite measure.
	
	As we aim to prove our interchange criterion for general measures, we need a proposition about divergent integrals.
	
	\begin{proposition} \label{pr: divergent subintegral}
		Let $\alpha \colon \Om \to \left[ 0, \i \right]$ be a measurable function with $\int \alpha \, d \mu = \i$. There either exists $A \in \AA_\sigma$ or an atom $A$ with $\mu \left( A \right) = \i$ such that $\int_A \alpha \, d \mu = \i$.
	\end{proposition}
	
	\begin{proof}
		Employing \cite[Prop. 1.22]{FoLe} and its terminology we find a pair of measures $\mu_i$ with $\mu_1$ purely atomic, $\mu_2$ non-atomic and $\mu = \mu_1 + \mu_2$. Setting $d \nu_i= \alpha d \mu_i$, either $\nu_1$ or $\nu_2$ is infinite. If $\nu_1$ is infinite, let $\AA_a$ be the system of countable unions of atoms and $A_n \in \AA_a$ an isotonic sequence with $\lim_n \nu_1\left(A_n\right) = \sup_{A \in \AA_a} \nu_1\left(A\right)$. There is nothing left to prove if the supremum is infinite. Otherwise, we set $A = \lim_n A_n$. As $\nu_1 \left( A^c \right) = \i$ and $\mu_1$ is purely atomic, the set $A^c \cap \left\{ \alpha > 0 \right\}$ has positive measure	whence it contains an atom $A_\i$. In particular $\nu_1 \left( A_\i \right) > 0$ so that $\nu_1 \left( A_\i \cup A \right)$ surpasses the supremum, a contradiction; the countable union of atoms $A$ either is $\sigma$-finite or contains an atom of infinite measure. In the remaining case if $\nu_2$ is infinite, we consider an isotonic sequence $B_n \in \AA_\sigma$ with $\lim_n \nu_2\left(B_n\right) = \sup_{A \in \AA_\sigma} \nu_2 \left( A \right)$. Again, we are finished if this supremum is infinite. Otherwise, set $B = \lim_n B_n$. Since $\nu_2 \left( B^c \right) = \i$ and $\mu_2$ is non-atomic, there exists $F \subset B^c \cap \left\{ \alpha > 0 \right\}$ with $0 < \mu_2 \left( F \right) < \i$. We may assume $\nu_1\left(F\right) < \i$ without loss of generality so that $F \in \AA_\sigma$. But then $B \cup F \in \AA_\sigma$ and $\nu_2\left( B \cup F \right)$ surpasses the supremum, yielding a contradiction.
	\end{proof}
	
	Proposition \ref{pr: divergent subintegral} prompts us to define a notion of integral that will be apt for stating our interchange criterion concisely. Denoting by $\AA_{a\sigma}$ the $\sigma$-ring of sets arising as a union of countably many $\mu$-atoms and a $\sigma$-finite set, we call a function $\alpha \colon \Om \to \left[ -\i, \i \right]$ such that the restriction of $\alpha$ to any $A \in \AA_{a\sigma}$ is measurable \emph{integrally measurable}. Similarly, we shall say that some measurability property holds integrally if it holds on any atom and every $\sigma$-finite set. In particular, we consider integrally negligible sets, which are defined as sets whose intersection with any atom or $\sigma$-finite set is null. We say that a property holds integrally almost everywhere if it holds except on an integral null set and abbreviate this by i.a.e. A moment's reflection together with \cite[Prop. 1.22]{FoLe} shows that a measurable set is integrally null if and only if it is null.
	We define the integral of the integrally measurable positive part $\alpha^+$ as
	$$
	\int \alpha^+ \, d \mu = \sup_{A \in \AA_{a\sigma} } \int_A \alpha^+ \, d \mu.
	$$
	As usual, we then define $\int \alpha \, d \mu = \int \alpha^+ \, d \mu - \int \alpha^- \, d \mu$ if one of these integrals is finite. Finally, we set $\int \alpha \, d \mu = +\i$ if neither the positive part $\alpha^+$ nor the negative part $\alpha^-$ is thus integrable. If $\alpha$ is measurable, this corresponds to the convention of interpreting $\int \alpha \, d \mu$ as an (extended) Lebesgue integral if $\alpha^+$ or $\alpha^-$ is integrable and setting $+ \i$ if both parts fail to be so. If $\int \alpha \, d \mu$ is finite, then the integrally measurable function $\alpha$ equals a measurable function a.e. since it vanishes outside of a $\sigma$-finite set, on which it is measurable. We call this an \emph{exhausting integral}. This integral is monotone, i.e., if $\alpha \le \beta$ i.a.e. then $\int \alpha \, d \mu \le \int \beta \, d \mu$. Let $\alpha_n$ be a sequence of integrally measurable functions converging to a limit function $\alpha$ locally in $\mu$ and a.e. on every atom. Then there holds the Fatou lemma
	$$
	\int \alpha^+ \, d \mu \le \liminf_n \int \alpha^+_n \, d \mu.
	$$
	Indeed, if $A$ is an atom or a set of finite measure, then
	$$
	\int_A \alpha^+ \, d \mu \le \liminf_n \int_A \alpha^+_n \, d \mu \le \liminf_n \int \alpha^+_n \, d \mu
	$$
	by the classical Fatou lemma. Taking the supremum over all such $A$ on the left-hand side then yields the claim. More generally, let $\alpha_A \colon A \mapsto \left[-\i, \i\right]$ be a family of measurable functions indexed by $A \in \AA_{a\sigma}$ such that $\alpha_A = \alpha_B$ a.e. on $A \cap B$. Then we define the exhausting integral of the family $\alpha_A$ by means of
	$$
	\int \alpha^+_A \, d \mu = \sup_{A \in \AA_{a\sigma} } \int_A \alpha^+_A \, d \mu.
	$$
	This renders the integral of an essential infimum function of an arbitrary family $v_i \colon \Om \to \left[-\i, \i\right]$ of measurable functions meaningful, even though it need only exist on any $\sigma$-finite set by \cite[Lem. 1.108]{FoLe} and on any atom by an elementary consideration. Indeed, in the last case, since any extended real-valued function is constant a.e. on an atom, we may define the essential infimum function as the infimum of these constants. If the integral of such a family is finite, then it derives from a $\bar{\mu}$-integrable function $\alpha$ by $\alpha_A = \alpha$ a.e. on each $A \in \AA_{a\sigma}$. To see this, pick $A \in \AA_\sigma$ where the supremum of the exhausting integral is obtained and argue by contradiction that any member of $\alpha_A$ vanishes a.e. outside $A$ as in the proof of Proposition \ref{pr: divergent subintegral}. Monotonicity and the Fatou lemma continue to hold for this type of integral. When we consider integral functionals in the following, we interpret all integrals in this sense. It is worth mentioning that this reduces to the extended Lebesgue integral if $\Om$ is $\sigma$-finite.
	
	We briefly recapitulate technical background on the measurability of integrands. A set-valued function $\Gamma \colon \Omega \to \mathcal{P} \left( \TT \right)$ is (Effros) measurable if for every open set $O \subset \TT$ the set $\Gamma^- \left( O \right) = \left\{ \omega \in \Omega \st \Gamma \left( \omega \right) \cap O \ne \emptyset \right\}$ is measurable. A pre-normal integrand is defined to be a function $f \colon \Om \times \TT \to \left[ - \i, \i \right]$ such that the epigraphical mapping $S_f \colon \Om \to \PP \left( \TT \right) \colon \om \mapsto \epi f_\om$ is (Effros) measurable. A pre-normal integrand is normal iff $S_f$ is closed-valued. By Lemma \ref{lem: joint measurability}, the normality of an integrand $f \colon \Om \times M \to \left( -\i, \i \right]$ on a separable metric space $M$ is equivalent to lower semicontinuity in the second component and $\AA \otimes \BB(M)$-measurability if the measure $\mu$ is complete. In the following, a subscript $f_W$ denotes the restriction of an integrand $f \colon \Om \times M \to \left[ -\i, \i \right]$ in its second component to a subset $W \subset M$.
	
	\begin{definition}[separable measurability] \label{def: separably measurable}
		An integrand $f \colon \Om \times M \to \left[ - \i, \i \right]$ is said to be \emph{separably measurable} if for any $W \in \SS (M)$ the restriction $f_W$ is $\AA \otimes \BB \left( W \right)$-measurable.
	\end{definition}
	
	It is equivalent to require that, for all $W_0 \in \SS (M)$, there should exist $W \in \SS (M)$ with $W_0 \subset W$ such that $f_{W}$ is $\AA \otimes \BB \left( W \right)$-measurable, since
	$$
	\AA \otimes \BB \left( W \right) = \left. \AA \otimes \BB \left( X \right) \right|_{\Om \times W}
	$$
	by \cite[Satz III.5.2]{El}.
	In particular, separable measurability reduces to the ordinary one if $M$ is separable. The composition of a separably measurable integrand with a strongly measurable (hence separably valued) function $u \colon \Om \to M$ is measurable as a composition of measurable functions.

	\begin{theorem} \label{thm: inf int}
		Let $M$ be complete and $R$ a space of integrally strongly measurable functions $u \colon \Omega \to M$ that is almost decomposable with respect to $\mu$. Let $f \colon \Om \times M \to \left[ - \i, \i \right]$ be an integrally separably measurable integrand. Suppose that, for any atom $A \in \AA$ with $\mu \left( A \right) = \i$ and every $W \in \SS (M)$, there holds $\inf_W f_\om \ge 0$ for a.e. $\om \in A$. Then, if
		$$
		I_f \not \equiv \i \text{ on } R \text{ where } I_f(u) = \int f \left[ \om, u(\om) \right] \, d \mu(\omega),
		$$
		one has
		\begin{equation} \label{eq: inf-int exchange}
			\inf_{u \in R} I_f(u) = \int \essinf_{W \in \SS (M) } \inf_{x \in W} f(\om, x) \, d \bar{\mu}(\omega).
		\end{equation}
		Moreover, if the common value in (\ref{eq: inf-int exchange}) is not $-\i$, then the essential infimum function $\bar{m}$ exists on all of $\Om$ and is attained by a $W \in \SS(M)$. In this case, for $\bar{u} \in R$, one has
		\begin{equation} \label{eq: minimizer pointwise characterization}
			\bar{u} \in \Argmin_{u \in R} I_f \left( u \right) \iff f \left[ \om, \bar{u} (\om) \right] = \bar{m}(\om) \quad \mu \text{-a.e.}
		\end{equation}
	\end{theorem}
	
	We consider in the following essential infimum functions for families of functions $v_i \colon \Om \to \left[ -\i, \i \right]$, $i \in I$ an index, such that there exists a family of measurable functions $u_i$ with $u_i = v_i$ a.e. for any $i \in I$. It is elementary to check by \cite[Def. 1.106]{FoLe} that the essential infimum functions of the families $v_i$ and $u_i$ agree in this situation. Any such family $v_i$ admits an essential infimum function on any $A \in \AA_{a\sigma}$ as explained before. The essential infimum function in (\ref{eq: inf-int exchange}) reduces to the pointwise infimum $\inf_M f_\om$ if $M$ itself is separable. We may take all integrals in the ordinary extended Lebesgue sense obeying the convention $+\i - \i = +\i$ if $\mu$ is $\sigma$-finite so that Theorem \ref{thm: inf int} is a genuine generalization of the classical infimum-integral interchange criterion \cite[Thm. 14.60]{RocW} from $\sigma$-finite $\mu$ and $M = \R^d$ to arbitrary measures and non-separable metric range spaces.
	
	\begin{proof}
		Generalizing \cite[Thm. 14.60]{RocW}, we follow its basic strategy of proof wherever no adaption is necessary. For $A \in \AA_{a\sigma}$, we set
		$$
		m_A(\om) = \essinf_{W \in \SS (M) } \inf_{x \in W} f(\om, x), \quad \om \in A.
		$$
		For any $u \in R$ with $I_f (u) < \i$, we may apply Proposition \ref{pr: divergent subintegral} to find $A_0 \in \AA_\sigma$ for which
		\begin{equation} \label{eq: exh assumed}
			\int_{\Om \setminus A_0} f(u)^+ \, d \mu = 0; \quad \int_{A_0} f(u)^- \, d \mu = \int f(u)^- \, d \mu.
		\end{equation}
		We used the assumption $\inf_W f_\om \ge 0$ on atoms of infinite measure together with $f(u)^- \le \left( \inf_W f_\om \right)^-$ i.a.e. for $W \in \SS(M)$ containing the range of $u$. Thus, for any sequence $v_n \in \dom I_f(v_n) \cap R$ with $\inf_R I_f = \lim_n I_f(v_n)$ there exists $A_0$ satisfying (\ref{eq: exh assumed}) simultaneously for all $v_n$, hence $\lim_n I_f(v_n) = \lim_n \int_A f(v_n) \, d \mu \ge \int_A m_A \, d \mu$ whenever $A \in \AA_\sigma$ with $A_0 \subset A$. Taking the supremum over $A \in \AA_\sigma$, we find $\inf_R I_f \ge \int m_A \, d \mu$ with the last integral being the exhausting one of the family $m_A$.
		
		It remains to prove the opposite inequality when $\inf_R I_f > - \i$. Since $f(u)^+ \ge \left( \inf_W f_\om \right)^+$ i.a.e. for any $W \in \SS(M)$ containing the range of $u$, it suffices to show that, for any $A \in \AA_\sigma$ and $\alpha > \int_A m_A \, d \mu$, there exists $u \in R$ with $I_f(u) < \alpha$. To simplify notation, we write $m$ instead of $m_A$. We may enlarge the subspace $W$ so that $m = \inf_W f$ a.e. on $A$ and $W$ is closed. We restrict our consideration to the subspace of $W$-valued functions in $R$, so that we may assume $M$ itself to be separable. Since $\int_A m \, d \mu < \i$, the positive part $m^+$ is integrable on $A$ so that
		$$
		m_\e \left( \om \right) = \max \left\{ m \left( \om \right), - \e^{-1} \right\},
		\quad \lim_{\e \downarrow 0} \int_A m_\e \, d \mu = \int_A m \, d \mu
		$$
		by monotone convergence. The set $A$ being $\sigma$-finite, there exists a non-negative integrable function $p \colon \Om \to \R^+$ that is positive on $A$. Setting $q_\e \left( \om \right) = \e p \left( \om \right) + m_\e \left( \om \right)$, we have $\int_A q_\e \, d \mu \to \int_A m \, d \mu < \alpha$ as $\e \downarrow 0$. Since $q_\e > m$ on $A$, the sets
		$$
		L_\e \left( \om \right) = \left\{ x \in M \st f_\om (x) < q_\e \left( \om \right) \right\}, \quad \om \in A
		$$
		are non-empty. Choose $\e$ small enough that $\int_A q_\e \, d \mu < \alpha$. Let $\AA'$ be the trace $\sigma$-algebra of $\AA$ on $A$. By assumption, the integrand $g := f - q_\e$ is $\AA' \otimes \BB \left( M \right)$-measurable so that the separably valued multimap $L_\e \colon A \to \PP \left( M \right) \setminus \left\{ \emptyset \right\}$ has the measurable graph $\graph L_\e = \left\{ \left( \om, x \right) \in A \times M \st g_\om \left( x \right) < 0 \right\}$, whence there exists a $\AA'_\mu$-measurable selection $u_1$ by \cite[Thm. 6.10]{FoLe}: an $\AA'_\mu$-measurable function $u_1 \colon A \to M$ with $u_1 \left( \om \right) \in L_\e \left( \om \right)$ for all $\om \in A$, i.e.
		$$
		f \left[ \om, u_1 \left( \om \right) \right] < q_\e \left( \om \right), \quad \om \in A.
		$$
		As $M$ is separable, Lemma \ref{lem: strong mb completion} yields a strongly $\AA'$-measurable function $u_2$ with $u_1 = u_2$ a.e. We have $\int_A f \left[ \om, u_2 \left( \om \right) \right] \, d \mu \left( \om \right) < \alpha$. The set $A$ being $\sigma$-finite, we can express $A$ as a union of an isotonic sequence of sets $\Om_n$ with $\mu \left( \Om_n \right) < \i$ for every $n \ge 1$. Fix $x \in M$ and let $A_n = \left\{ \om \in \Om_n \st d \left[ x, u_2 \left( \om \right) \right] \le n \right\} \in \AA$. Note that $A_n \uparrow A$. The space $R$ being almost decomposable, there exists an isotonic sequence $A'_n \subset A_n$ with $\mu \left( A_n \setminus A'_n \right) < n^{-1}$ such that the function defined by
		\begin{equation*}
			w_n(\om) =
			\begin{cases}
				v_0 & \text{ if } \om \in \Om \setminus A'_n, \\
				u_2 & \text{ if } \om \in A'_n
			\end{cases}
		\end{equation*}
		belongs to $R$. Since $A'_n \uparrow A$, we have
		\begin{equation} \label{eq: convergent integrals}
			\int_{A \setminus A'_n} f \left[ \om, v_0 \left( \om \right) \right] \, d \mu \to 0; \quad \int_{A'_n} f \left[ \om, u_2 \left( \om \right) \right] \, d \mu \to \int_A f \left[ \om, u_2 \left( \om \right) \right] \, d \mu
		\end{equation}
		as $n \to \i$ by the theorems of dominated and monotone convergence. Since
		$$
		I_f \left( w_n \right) = \int_{A \setminus A'_n} f \left[ \om, v_0 \left( \om \right) \right] \, d \mu \left( \om \right) + \int_{A'_n} f \left[ \om, u_2 \left( \om \right) \right] \, d \mu \left( \om \right),
		$$
		we have $I_f \left( w_n \right) \to \int_A f \left[ \om, u_2 \left( \om \right) \right] \, d \mu \left( \om \right) < \alpha$ by (\ref{eq: convergent integrals}), hence $I_f \left( w_n \right) < \alpha$ if $n$ is sufficiently large.
		
		Regarding the second part of the claim, we start by showing that the $\AA_\mu$-measurable function $m$ induced by the integrable family $m_A$ indeed defines the essential infimum function in (\ref{eq: inf-int exchange}) on $\Om$. Otherwise there were $W \in \SS(M)$ such that the set $\left\{ m > \inf_W f \right\}$ is not contained in a negligible set.
		
		Assume first that $\left[ m - \inf_W f \right]^+$ is $\AA_\mu$-measurable so that not being contained in a null set is equivalent to having positive measure. No atom $A$ with $\mu \left( A \right) = \i$ may contribute to the positive measure since $m^+_A$ is integrable as the common value (\ref{eq: inf-int exchange}) is not $+\i$. Here, we have used the assumption $\inf_W f \ge 0$ a.e. on atoms of infinite measure. Hence, some $A \in \AA_\sigma$ contributes to the positive measure by \cite[Prop. 1.22]{FoLe}. But then $m_A > \inf_W f$ a.e. on $A$ is contradictory.
		
		If second the function $\left[ m - \inf_W f \right]^+$ is only known to be integrally measurable, attempt its integration w.r.t. the completion $\bar{\mu}$ in the exhausting sense. If the integral is finite, then $\left[ m - \inf_W f \right]^+$ is integrable and integrally measurable hence equals an $\AA$-measurable function a.e. We are back to first the case. If the integral is not finite, then the subintegral over an atom of infinite measure or a $\sigma$-finite set is infinite,	on which $\left[ m - \inf_W f \right]^+$ is $\AA_\mu$-measurable. Proceed as in the first case, arriving at a contradiction; The subintegral hence the integral is finite. We are back to the integrable second case. We have proved that $\inf_W f \ge m$ a.e. for any $W \in \SS (M)$. It remains to prove that any further measurable function fulfilling this inequality is dominated by $m$ a.e. Let $r \colon \Om \to \left[ -\i, \i \right]$ be such a function and suppose that the set $\left\{ r > m \right\}$ has positive measure. If an atom $A$ with $\mu \left( A \right) = \i$ contributes to the positive measure, we may by $I_f \not \equiv \i$ pick $W \in \SS (M)$ such that $\inf_W f_\om = 0$ a.e. on $A$, hence $0 > m$ a.e. on $A$ so that the contradiction $\int_A m^- \, \mu = \i$ obtains. Therefore, some $A \in \AA_\sigma$ contributes to the positive measure. Setting $\bar{r} = \max\{ r,, m \}$ yields the contradiction
		$$
		\inf_{u \in R} I_f (u) = \lim_n I_f (v_n) \ge \int_{A_0 \cup A} \bar{r} \, d \mu > \int_{A_0 \cup A} m \, d \mu = \int_{A_0} m_{A_0} \, d \mu = \inf_{u \in R} I_f (u).
		$$
		To see that the essential infimum function $\bar{m}$ is attained by some $W \in \SS(M)$ if it is integrable, consider again the sequence $v_n$ with $\int_R I_f = \lim_n I_f(v_n)$. Choose $W \in \SS(M)$ containing the range of $v_n$ and observe that $W$ provides the desired subspace as
		$$
		\int \bar{m} \, d \mu = \lim_n I_f(v_n) \ge \int \inf_W f \, d \mu \ge \int \bar{m} \, d \mu.
		$$
		The addendum (\ref{eq: minimizer pointwise characterization}) is equivalent to $\mu \left( \left\{ \om \st f \left[ \om, \bar{u} \left( \om \right) \right] > \bar{m} \left( \om \right) \right\} \right) = 0$ if $\int \bar{m} \, d \mu$ is finite, whence it follows.
	\end{proof}
	
	\paragraph{Remark.}
	
	\begin{enumerate}
		
		\item We know of no previous interchange result for a function space $S$ with a non-separable range space except \cite[Thm. 6.1]{L}. There it is proved in the particular case of convex conjugacy that if the function space $S$ is weakly decomposable and $M = X$, then the infimum may be computed by taking the $L_1$-infimum under the integral sign. While this formulation appeals by its elegance, it does not satisfy our need to relate the infimum function under the integral sign to the pointwise infimum of the integrand. Under the mere assumption of weak decomposability, no analogue of our result can be expected in this respect, a property like our almost decomposability is indispensable for it. Our criterion could be generalized to the effect that one could compute the infimum function under the integral in $L_1$ on an (almost) weakly decomposable function space $S$ and then derive our representation of this infimum function in the special case when the space has the stronger property of being almost decomposable.
		
		\item More recently, interchange criteria for finite dimensional range spaces were discussed in \cite{Gi2}, including an overview of previous results. Much of this work would carry over to separable range spaces with little effort. We note that, at least for $\sigma$-finite measures, an alternative proof of Theorem \ref{thm: inf int} could be devised by appealing to (slight extensions of) results in \cite{Gi2}. However, since we are interested in bringing the pointwise infimum of the integrand into play, no generalization would result directly from this, even though \cite{Gi2} provides conditions that are both necessary and sufficient for essential infima to be interchanged with an integral.
		
		\item Drawing upon the ideas of \cite{Gi2}, one could try to extend the above result in a way that would allow their application also on some spaces of continuous or even smooth functions, which could be of interest in extending our results on generalized gradient flows to the rate-independent case, where an additional term must be included in the dissipation term of the energy inequality to guarantee lower semicontinuity if the primal dissipation potential grows merely linearly. Interpreting the time derivative as belonging to a space of measures that is dual to a space of continuous functions, one might be able to unify the theory of linear and superlinear primal dissipation potentials.
		
	\end{enumerate}
	
	\subsection{Convex conjugacy}
	
	We can now represent the convex conjugate of a general integral functional on a space of strongly measurable functions in duality with a space of weak* measurable ones. Though this result will not apply directly to all Orlicz spaces, as their dual space may contain elements that are no functions, it is fundamental in representing the convex conjugate of an integral functional on the function component of the dual.
	
	\begin{theorem} \label{thm: conjugate A}
		Let $R$ be a linear space of integrally strongly measurable functions $u \colon \Om \to X$ that is almost decomposable with respect to $\mu$. Let $S$ be a linear space of integrally weak* measurable functions $u' \colon \Om \to X^*$ such that the bilinear form
		\begin{equation} \label{eq: duality pairing}
			R \times S \to \R \colon \left( u, u' \right) \mapsto \int \langle u' \left( \om \right), u \left( \om \right) \rangle \, d \mu \left( \om \right)
		\end{equation}
		is well-defined. Let $f \colon \Om \times X \to \left[ - \i, \i \right]$ be an integrally separably measurable integrand. Suppose that, for $v \in S$, any atom $A \in \AA$ with $\mu \left( A \right) = \i$ and any $W \in \SS(X)$ there holds $\sup_{x \in W} \langle v(\om), x \rangle - f_\om(x) \le 0$ for a.e. $\om \in A$. Then, if
		$$
		I_f \not \equiv \i \text{ on } R \text{ where } I_f(u) = \int f \left[ \om, u(\om) \right] \, d \mu(\om),
		$$
		the convex conjugate $I^*_f$ of $I_f$ at $v$ with respect to the pairing (\ref{eq: duality pairing}) is given by
		\begin{equation} \label{eq: conjugate representation}
			I^*_f \left( v \right) = \int \esssup_{W \in \SS \left( X \right) } \sup_{x \in W} \langle v(\om), x \rangle - f_\om (x) \, d \bar{\mu}(\om).
		\end{equation}
		Denoting by $\SS_u(X)$ the separable subsets almost containing the range of $u$, the Fenchel-Moreau subdifferential of $I_f$ on $\dom I_f$ is given by
		\begin{equation} \label{eq: subdifferential}
			\p I_f(u) = \bigcap_{W \in \SS_u(X) } \left\{ v \in S \st v^*_W(\om) \in \p f_W \left[ \om, u(\om) \right] \text{ a.e.} \right\}.
		\end{equation}
		Moreover, if $v \in \dom I^*_f$, then the following two are equivalent:
		\begin{enumerate}
			\item The mapping $\om \mapsto f^* \left[ \om, v(\om) \right]$ is $\AA_\mu$-measurable and there holds
			\begin{equation} \label{eq: conjugate under integral}
				I^*_f (v) = I_{f^*} (v) = \int f^* \left[ \om, v(\om) \right] \, d \bar{\mu}(\om);
			\end{equation}
			\item There exists $W \in \SS (X)$ such that
			\begin{equation} \label{eq: ess-inf = inf}
				f^* \left[ \om, v(\om) \right] = \sup_{x \in W} \langle v(\om), x \rangle - f_\om (x) \quad \mu \text{-a.e.}
			\end{equation}
		\end{enumerate}
		In either case, the intersection in (\ref{eq: subdifferential}) over $W \in \SS_u(X)$ may be replaced by $W = X$.
	\end{theorem}
	
	\begin{proof}
		Invoking Theorem \ref{thm: inf int}, we find (\ref{eq: conjugate representation}) once we show that the tilted integrand
		$$
		\Om \times X \to \left[ -\i, \i \right] \colon (\om, x) \mapsto f(\om, x) - \langle v(\om), x \rangle
		$$
		is integrally separably measurable (ism.). This obtains since $f$ is ism. by assumption and since the tilt is integrally separably Carathéodory hence ism. so that the difference is ism.
		
		If $X$ is separable, then the Fenchel-Young identity together with (\ref{eq: conjugate representation}) shows that
		\begin{equation} \label{eq: sep sd}
			\p I_f(u) = \left\{ v \in S \st v^*(\om) \in \p f \left[ \om, u(\om) \right] \text{ a.e.} \right\}.
		\end{equation}
		Applying the case of separable $X$ then yields (\ref{eq: subdifferential}): It is obvious that $v \in \p I_f(u)$ must belong to the subdifferential of $I_f$ when the functional is restricted to the subspace $S_W$ consisting of those functions in $S$ taking values in a separable subspace $W \in \SS(X)$. As $S_W$ satisfies the same assumptions as $S$, we have (\ref{eq: sep sd}) on $S_W$, whence the function $v$ belongs to the right-hand side in (\ref{eq: subdifferential}). Conversely, if $v$ belongs to that right-hand side, then obviously
		$$
		I_f(w) \ge I_f(u) + \int \langle v(\om), u(\om) - w(\om) \rangle \, d \mu(\om) \quad \forall w \in S
		$$
		as $u, w \in S$ are almost separably valued. Consequently, $v \in \p I_f(u)$.
		
		Regarding the addendum on the conjugate, observe that
		$$
		f^* \left[ \om, u(\om) \right] \ge \esssup_{W \in \SS \left( X \right) } \sup_{x \in W} \langle v(\om), x \rangle - f_\om(x) \ge \sup_{x \in W} \langle v(\om), x \rangle - f_\om(x) \quad \mu \text{-a.e.}
		$$
		Consequently, if $f^*_\om(v)$ is $\AA_\mu$-measurable and (\ref{eq: conjugate under integral}) holds as an identity of real numbers, then (\ref{eq: ess-inf = inf}) obtains since Theorem \ref{thm: inf int} guarantees attainment of the essential supremum function. Conversely, if (\ref{eq: ess-inf = inf}) holds, then the integrals in (\ref{eq: conjugate under integral}) and (\ref{eq: conjugate representation}) agree. The function $\om \mapsto f^* \left( \om, v(\om) \right)$ then equals an $\AA_\mu$-measurable function a.e. so that it is $\AA_\mu$-measurable.
		
		The addendum on the subdifferential follows by the Fenchel-Young identity as in the case of (\ref{eq: subdifferential}) when $I^*_f(v) = I_{f^*}(v)$.
	\end{proof}
	
	Theorem \ref{thm: conjugate A} suggests to introduce the following notion:
	
	\begin{definition}[dualizable integrand]
		An integrand $f \colon \Om \times X \to \left[ -\i, \i \right]$ that is separably measurable and such that, for a weak* measurable function $v \colon \Om \to X^*$, there exists $W \in \SS(X)$ with
		\begin{equation} \label{eq: dualizability}
			f^* \left[ \om, v(\om) \right] = \sup_{x \in W} \langle v(\om), x \rangle - f_\om(x) \quad \forall \om \in \Om
		\end{equation}
		is called \emph{dualizable} at $v$. We say that $f$ is dualizable for a space $S$ of such functions if it is dualizable at each $v \in S$.
	\end{definition}
	
	We shall also consider integrands that are dualizable a.e. or i.a.e. This is meaningful if $f$ and $v$ are merely integrally measurable.
	
	If $f$ is dualizable for $v$ and $W \in \SS(X)$, then the integrand $f_\om(x) - \langle v(\om), x \rangle - f^* \left[ \om, v(\om) \right]$ is $\AA_\mu \otimes \BB(W)$-measurable and thus an $\AA_\mu$-pre-normal integrand on $W$ by Lemma \ref{lem: joint measurability}. As such it is infimally measurable by Lemma \ref{lem: inf meas normality}, its strict sublevel multimaps are measurable by Lemma \ref{lem: equivalence infimal measurability} and non-empty for positive level values. Hence, we find from them (strongly) $\AA_\mu$-measurable selections by the Aumann theorem \cite[Thm. 6.10]{FoLe} if $W$ is closed. Conversely, if the integrand $f_\om(x) - \langle v(\om), x \rangle - f^* \left[ \om, v(\om) \right]$ admits such selections, then it is obvious that it dualizable for $v$. We apply this characterizing observation to discuss our first of two sufficient conditions for dualizability at all strongly measurable functions.
	
	\begin{lemma} \label{lem: brl sff}
		Let $\Om$ be a separable metric Borel space, $X$ a reflexive Banach space and $f \colon \Om \times X \to \left( -\i, \i \right]$ a normal convex integrand. Then $f$ is dualizable for any strongly measurable function $v \colon \Om \to X^*$.
	\end{lemma}
	
	\begin{proof}
		Any Borel measurable map on $\Om$ into another metric space has a separable range by \cite[Prop. 1.11]{CTV}. The integrand $f_\om(x) - \langle v(\om), x \rangle$ is infimally measurable in the sense of Definition \ref{def: infimal measurability} by Lemma \ref{lem: normal sums} and an easy limiting argument that approximates $v$ pointwise by a sequence of simple functions. Thus, the function $f^* \left[ \om, v(\om) \right] = - \inf_{x \in X} f_\om(x) - \langle v(\om), x \rangle$ is measurable by Lemma \ref{lem: equivalence infimal measurability}, whereby we recognize $f_\om(x) - \langle v(\om), x \rangle + f^* \left[ \om, v(\om) \right]$ as infimally measurable. Consequently, its (strict) sublevel multimaps are measurable by Lemma \ref{lem: equivalence infimal measurability} and non-empty for positive level values. We now want to apply \cite[Cor. 5.19]{CKR} to obtain Borel-measurable selections from the sublevels and thus conclude dualizability by our initial comment and the observation before this lemma. Note in this regard that $X$ is locally uniformly rotund by reflexivity \cite{Tr}. Literally, the result \cite[Cor. 5.19]{CKR} requires a finite measure space and weakly compact convex values of the epigraphical multimap. However, an extended inspection of the proof reveals that the statement holds on any measurable space and if only the intersection of any value of $\epi \left( f_\om - \langle v(\om), \cdot \rangle + f^* \left[ \om, v(\om) \right] \right)$ with any closed ball centred at the origin are weakly compact and convex. To see this, check in \cite[Lem. 5.3]{CKR} that the cardinality $\gamma$ may be countably infinite on any measurable space and observe in \cite[Lem. 5.11]{CKR} that the proof still works if the sublevel sets of the function $g$ therein have compact intersections with the values of the multimap therein. Finally, by a limiting argument approximating any bounded closed convex (hence weakly compact) set $K$ by the open sets $K_\e = K + B_\e$ for $\e > 0$ and then approximating any closed convex set by bounded closed convex sets, it is easy to check that the $\mathcal{M}^{cc}$-measurability required in \cite[Cor. 5.19]{CKR} is implied by Effros measurability in a reflexive space so that in total our adapted application of \cite[Cor. 5.19]{CKR} has been warranted and the proof is complete.
	\end{proof}
	
	Assuming the continuum hypothesis, the above argument still works for any $\sigma$-algebra $\AA$ whose cardinality is at most $\left| \R \right|$, cf. the remarks after \cite[Prop. 1.11]{CTV}. In particular, this covers the case of any countably generated $\sigma$-algebra $\AA$, see \cite{CTV}.
	
	We now present our second sufficient condition for dualizability. The proof and formulation of this condition requires some background information and technical results about hyperspace topologies on the lower semicontinuous proper functions $\LS(X)$. We defer the definitions and technical results to the appendix but repeat the basic definitions here. For proofs and further information, we refer to \cite{B}. The facts about the Wijsman topology will be needed later. The Attouch-Wets topology $\tau_{AW}$ on the closed subsets $\CL (M)$ of the metric space $\left( M, d \right)$ is obtained by identifying $A \in \CL(M)$ with the distance function $d_x(A) = \inf_{a \in A} d(x, a)$ and considering the topology of their uniform convergence on bounded sets, i.e.
	\begin{equation} \label{eq: AW topology}
		\tau_{AW} \text{-} \lim_\alpha A_\alpha \iff \lim_\alpha \sup_{x \in B} \left| d_x(A_\alpha) - d_x(A) \right| = 0 \quad \forall B \subset M \text{ bounded}.
	\end{equation}
	Similarly, the Wijsman topology $\tau_W$ is defined by considering pointwise convergence in (\ref{eq: AW topology}). We omit the dependence of $\tau_{AW}$ and $\tau_W$ on the metric $d$ to ease notation. One defines on $\LS \left( M \right)$ the Attouch-Wets and Wijsman topologies, denoted again by $\tau_{AW}$ and $\tau_W$, via the identification of $f$ with $\epi f \subset M \times \R$, where $M \times \R$ carries the box metric $\rho \left[ \left( x_0, \alpha_0 \right), \left( x_1, \alpha_1 \right) \right] = \max \left\{ d \left( x_0, x_1 \right), \left| \alpha_0 - \alpha_1 \right| \right\}$. The topology $\tau_{AW}$ is metrizable and complete w.r.t. the metric
	\begin{equation} \label{eq: AW metric}
		d_{AW} \left( A, B \right) = \sum_{n = 1}^\i 2^{-n} \min\{ 1, \sup_{x \in B_n(x_0) } \left| d_x(A) - d_x(B) \right| \}
	\end{equation}
	if $\left( M, d \right)$ is complete. The topology $\tau_W$ is metrizable and separable if and only if $M$ is separable. An important feature of $\tau_W$ is that, for separable $M$, an integrand $f$ is normal if and only if it is $\tau_W$-measurable as a $\LS(M)$-valued mapping by the Hess theorem \cite[Thm. 6.5.14]{B}. As $\tau_{AW}$ is metrizable, we may consider strongly $\tau_{AW}$-measurable integrands $f$, i.e., those for which there exists a sequence of simple integrand mappings $f_n$ with $\tau_{AW}$-$\lim_n f_{\om, n} = f_\om$ for $\om \in \Om$.
	
	We now prove our second sufficient condition for dualizability.
	
	\begin{lemma} \label{lem: dualizable suff cond}
		Let $f \colon \Om \times X \to \left( - \i, \i \right]$ be an integrand identified with the mapping $f \colon \Om \to \LS(X)$. If $f$ is strongly measurable in the Attouch-Wets topology on $\LS(X)$, then it is dualizable for any strongly measurable function $v \colon \Om \to X^*$. In particular, any autonomous integrand $f \in \LS(X)$ is thus dualizable.
	\end{lemma}
	
	\begin{proof}
		By Lemma \ref{lem: strongly tau_AW sep subspace}, it suffices to show that if $v$ is strongly measurable, then the mapping $f_\om - \langle v(\om), \cdot \rangle$ is strongly $\tau_{AW}$-measurable. Let $v_n$ be a sequence of simple functions and $f_n$ a sequence of simple integrands with
		$$
		\lim_n v_n(\om) = v(\om), \quad \tau_{AW}\text{-}\lim_n f_{\om, n} = f_\om \quad \forall \om \in \Om.
		$$
		Then Lemma \ref{lem: AW affine stable} implies
		\begin{equation*}
			\tau_{AW} \text{-} \lim_n f_{\om, n} - \langle v_n(\om), \cdot \rangle = f_\om - \langle v(\om), \cdot \rangle \quad \forall \om \in \Om. \qedhere
		\end{equation*}
	\end{proof}
	
	One might try coarser hyperspace topologies to establish more general criteria similar to Lemma \ref{lem: dualizable suff cond}, thereby placing dualizable integrands within the theoretical framework of measurable multifunctions in a deeper way. The slice topology seems apt.
	
	\section{Orlicz spaces $L_\varphi(\mu)$} \label{sec: L}
	
	In this section, we define the notion of an Orlicz integrand $\varphi$ and show how it induces the Banach spaces $L_\varphi(\mu)$ of vector-valued functions called Orlicz spaces, whose basic properties like completeness, decomposability and class-internal embedding properties we study. As $L_\varphi(\mu)$ enjoys better properties when each of its elements vanishes outside a $\sigma$-finite set, we characterize this behaviour in terms of the Orlicz integrand. Similar spaces can be found in the literature under various names, such as Fenchel-Orlicz, generalized Orlicz, or Musielak-Orlicz spaces.
	
	\subsection{Orlicz integrands}
	
	As mentioned in the introduction, we never impose any kind of uniform behaviour w.r.t. $\om \in \Om$ on the Orlicz integrand. Instead
	
	\begin{definition}[Orlicz integrand] \label{def: Orlicz integrand}
		An even function $\varphi \in \Gamma(X)$ satisfying
		$$
		\lim_{x \to 0} \varphi(x) = 0 \text{ and } \lim_{\| x \| \to +\i} \varphi(x) = +\i
		$$
		is an \emph{Orlicz function}. A map $\varphi \colon \Om \times X \to \left[ 0, \i \right]$ is an Orlicz integrand if
		\begin{enumerate}[label = \textnormal{\alph*)'}]
			\item the function $x \mapsto \varphi \left( \om, x \right)$ is Orlicz for a.e. $\om \in \Om$; \label{en. it. Orlicz integrand a.e.}
			\item the integrand $\varphi$ is integrally separably measurable. \label{en. it. Orlicz integrand int sep norm}
		\end{enumerate}
	\end{definition}
	
	By convexity and evenness, an Orlicz integrand assumes a global minimum at the origin hence is non-negative.
	
	\begin{proposition} \label{pr: coercivity}
		For a convex function $\phi \colon X \to \left( -\i, \i \right]$ with $\phi(0) = 0$, there holds
		\begin{equation}\label{eq: coercivity}
			\lim_{\| x \| \to \i} \phi \left( x \right) = \i \iff \, \exists r > 0 \colon \inf_{\| x \| = r} \phi \left( x \right) > 0 \iff \liminf_{\| x \| \to \i} \frac{\phi \left( x \right)}{\| x \|} > 0.
		\end{equation}
	\end{proposition}
	
	\begin{proof}
		The first statement implies the second, the second implies the third as convexity renders the quotient non-decreasing in $\| x \|$, and the third implies the first.
	\end{proof}
	
	The next lemma reveals why our notion of an Orlicz integrand is apt for duality theory.
	
	\begin{lemma} \label{lem: gen iff conj}
		$\varphi \in \Gamma \left( X \right)$ is an Orlicz function iff $\varphi^* \in \Gamma \left( X^* \right)$ is one.
	\end{lemma}
	
	\begin{proof}
		It suffices to prove that $\varphi^*$ is an Orlicz function if $\varphi$ is one since $\varphi$ is conjugate to $\varphi^*$ for the duality between $X$ and $X^*$. The function $\varphi^*$ is even. As $\varphi$ has bounded sublevel sets, we see that $\varphi^*$ vanishes continuously at the origin and since $\varphi$ vanishes continuously at the origin, we see that $\varphi^*$ has bounded sublevel sets. More precisely
		$$
		\exists r, s > 0 \colon \| x \| > r \implies \varphi(x) > s \| x \|
		$$
		by Proposition \ref{pr: coercivity}, hence there holds
		$$
		\| x' \| < s \implies \varphi^* \left( x' \right) \le \sup_{\| x \| < r} \langle x', x \rangle \le r \| x' \|.
		$$
		For the second claim, note
		$$
		\forall \e > 0 \, \exists \delta > 0 \colon \| x \| < \delta \implies \varphi(x) < \e.
		$$
		Therefore
		$$
		\varphi^* \left( x' \right) \ge \sup_{\| x \| < \delta} \langle x', x \rangle - \sup_{\| x \| < \delta} \varphi (x) \ge \delta \| x' \| - \e
		$$
		so that $\varphi^*$ has bounded sublevel sets.
	\end{proof}
	
	Lemma \ref{lem: gen iff conj} implies that the conjugate integrand
	$$
	\varphi^* \left( \om, x' \right) = \sup_{x \in X} \langle x', x \rangle - \varphi \left( \om, x \right)
	$$
	remains Orlicz iff it is integrally separably measurable. This happens for a dualizable Orlicz integrand:
	
	\begin{lemma} \label{lem: conj sep norm}
		Let the Orlicz integrand $\varphi$ be dualizable for a decomposable space $S$ of strongly measurable functions	and let $\mu$ have no atom of infinite measure. Then the integrand $\varphi^*$ is integrally separably measurable. If $\mu$ is complete, then it suffices if $S$ is almost decomposable.
	\end{lemma}
	
	\begin{proof}
		Since $\mu$ has no atom of infinite measure, we are left to demonstrate that the restriction of $\varphi^*$ to a $\sigma$-finite set in the first component and a separable set $V \in \SS\left( X^* \right)$ in the other component is measurable. It suffices therefore to assume that $\mu$ is $\sigma$-finite. Given $F_n \in \AA_f$ with $\Om = \bigcup_n F_n$, it suffices if given $\e > 0$, we obtain $F_\e \subset F$ with $\mu\left( F \setminus F_\e \right) < \e$ and such that $\varphi^*$ is $\AA\left( F_\e \right) \otimes \BB(V)$-measurable if $\mu$ is complete. If $\mu$ is incomplete, it suffices if the same holds with $F_\e = F$. Because then $\varphi^*$ equals a measurable function $\mu$-a.e. in the first case and everywhere in the second case hence is measurable.
		
		Let $v_n \in V$ be a dense sequence. Using the almost decomposability of $S$, we find
		$$
		G_n \in \AA(F);
		\quad v_n \chi_{G_n} \in S;
		\quad \mu\left( F \setminus G_n \right) \le 2^{-n} \e.
		$$
		If $S$ is decomposable, we may pick $G_n = F$ instead. We find $W_n \in \SS(X)$ with
		$$
		\varphi^*_\om(v_n) = \sup_{x \in W_n} \langle v_n, x \rangle - \varphi_\om(x) \quad \text{for a.e. } \om \in G_n.
		$$
		Consequently, there holds for $W = \bigcup_{n \ge 1} W_n \in \SS(X)$ and $E_\e = \bigcup_{n \ge 1} G_n$ that
		$$
		\mu \left( F \setminus E_{\e_n} \right) \le \e; \quad \varphi^*_\om(v_n) = \sup_{x \in W} \langle v_n, x \rangle - \varphi_\om(x) \quad \text{for a.e. } \om \in E_\e \quad \forall n \in \N.
		$$
		Setting $g(v) = \sup_{x \in W} \langle v, x \rangle - \varphi_\om(x)$ for $v \in V$ we have $\varphi^*_\om \ge g_\om$ and $\varphi^*_\om(v_n) = g_\om(v_n)$ for $n \in \N$, hence $\varphi^*_\om$ and $g_\om$ agree on $\interior \dom g_\om \ne \emptyset$ for all $\om \in E_\e$ by convex continuity in the interior of the domain, cf. \cite[§3.2, Thm. 1]{IT}.	Using lower semicontinuity and the fact that both $r \mapsto \varphi_\om \left( r v \right)$ and $r \mapsto g_\om \left( r v \right)$ for $r \ge 0$ are non-decreasing for every fixed $v \in V$, we conclude $\varphi = g$. As $g$ is $\AA( E ) \otimes \BB(V)$-measurable, our claim obtains.
	\end{proof}
	
	For later reference, we record another simple observation about Orlicz functions.
	
	\begin{proposition} \label{pr: Orlicz function and convergence}
		Let $\varphi \in \Gamma(X)$ be an Orlicz function. A sequence $x_n \in X$ converges iff
		$$
		\forall k \in \N \, \exists N \in \N \colon m, n \ge N \implies \varphi \left[ k \left( x_m - x_n \right) \right] < k^{-1}.
		$$
	\end{proposition}
	
	We close this section remarking that in the literature we find divergent names and definitions for Orlicz functions, which are sometimes equivalent to Definition \ref{def: Orlicz integrand} or whose apparently greater generality is to some extent spurious. For example, Orlicz functions that are discontinuous or lack bounded sublevel sets may be adapted to match our definition without essentially altering their Orlicz space by deriving from them a Banach space where they are continuous at the origin.\footnote{Details available at request.}
	
	\subsection{Definition and basic properties} \label{ssec: def Orlicz}
	
	For an Orlicz integrand $\varphi$ and a strongly measurable function $u \in \mathcal{L}_0 \left( \Om ; X \right)$, we set
	$$
	I_\varphi (u) = \int \varphi \left[ \om, u \left( \om \right) \right] \, d \mu \left( \om \right); \quad \| u \|_\varphi = \inf \left\{ \alpha > 0 \st I_\varphi \left( \alpha^{-1} u \right) \le 1 \right\}.
	$$
	The Minkowski functional $\| \cdot \|_\varphi$ is a seminorm on its domain $\mathcal{L}_\varphi(\mu)$. By Proposition \ref{pr: Orlicz function and convergence}, the kernel of $\| \cdot \|_\varphi$ is characterized as the functions that vanish a.e. Factoring out the kernel, we arrive at the \emph{Orlicz space} $L_\varphi(\mu)$ on which $\| \cdot \|_\varphi$ is called the \emph{Luxemburg norm}. An equivalent norm is given by the \emph{Amemiya norm}
	$$
	\vvvert u \vvvert_\varphi = \inf_{\alpha > 0} \alpha^{-1} \left[ 1 + I_\varphi \left( \alpha u \right) \right]
	$$
	with
	$$
	\| u \|_\varphi \le \vvvert u \vvvert_\varphi \le 2 \| u \|_\varphi \quad \forall u \in L_\varphi(\mu)
	$$
	according to \cite[Thm. 1.10]{Mu}. Similarly, we define the \emph{dual Luxemburg norm} and the \emph{dual Amemiya norm} on $L_\varphi(\mu)^*$ by
	$$
	\| v \|^*_\varphi = \inf \left\{ \alpha > 0 \st I^*_\varphi \left( \alpha^{-1} v \right) \le 1 \right\}, \quad \vvvert v \vvvert^*_\varphi = \inf_{\alpha > 0} \alpha \left[ 1 + I^*_\varphi \left( \alpha^{-1} v \right) \right].
	$$
	As before
	$$
	\| v \|^*_\varphi \le \vvvert v \vvvert^*_\varphi \le 2 \| v \|^*_\varphi \quad \forall v \in L_\varphi(\mu)^*.
	$$
	One may easily check that the dual Amemiya norm agrees with the canonical operator norm induced by the Luxemburg norm. In the same way, the Amemiya norm agrees with the operator norm that $L_\varphi(\mu)$ carries as a subset of its bidual space. For frequent, sometimes tacit, later use, we record the following useful inequalities relating in particular $I_\varphi$ and $I^*_\varphi$ with their Luxemburg norms.
	
	\begin{lemma} \label{lem: modular-norm}
		Let $V$ be a real vector space, $f \colon V \to \left[ 0, \i \right]$ a convex function with $f(0) = 0$ and left-continuous, i.e., $\lim_{\lambda \uparrow 1} f\left(\lambda v\right) = f(v)$ for $v \in V$; Let $p \colon V \to \left[ 0, \i \right]$ be the Minkowski functional of the sublevel set $\left\{ f \le 1 \right\}$. Then there hold the following inequalities:
		\begin{enumerate}
			\item $p(v) \le 1 \implies f(v) \le p(v)$,
			\item $1 < p(v) \implies f(v) \ge p(v)$,
			\item $p(v) \le 1 + f(v)$.
		\end{enumerate}
	\end{lemma}
	
	\begin{proof}
		This follows from the proof of \cite[Cor. 2.1.15]{DHHR}, where the same assertion is made for $f$ a semimodular, but only the assumptions above are actually used.
	\end{proof}
	
	To prove that $L_\varphi(\mu)$ is complete, we first record a simple observation that will frequently be used to reduce considerations for $\sigma$-finite measures to finite ones.
	
	\begin{proposition} \label{pr: equivalent finite measure}
		Let $\mu$ be $\sigma$-finite and $f \colon \Om \to \R$ a positive integrable function. For the finite measure $d \nu = f d \mu$ and the Orlicz integrand $\phi = \frac{\varphi}{f}$, there holds $L_\varphi(\mu) = L_\phi \left( \nu \right)$.
	\end{proposition}
	
	The existence of such a function $f$ is equivalent to the $\sigma$-finiteness of $\mu$. We need to get one last measure theoretic generality out of our way: the space $L_\varphi(\mu)$ does not change if $\mu$ is replaced by its completion $\bar{\mu}$. More precisely, the total set of a.e. equivalence classes of strongly measurable functions w.r.t. $\mu$ does not change under completion as can be seen by appealing to Lemma \ref{lem: strong mb completion}. In this sense, there exists a canonical isometric isomorphism between $L_\varphi(\mu)$ and $L_\varphi \left( \bar{\mu} \right)$.
	
	We now prove the completeness of $L_\varphi(\mu)$ for an arbitrary measure $\mu$. The adaptation of the usual proof for Lebesgue spaces is not completely trivial in the case of a non-$\sigma$-finite measure due to the $\om$-dependence of the Orlicz integrand $\varphi$.
	
	\begin{theorem} \label{thm: L complete}
		$L_\varphi(\mu)$ in the Luxemburg-norm $\| \cdot \|_\varphi$ is a Banach space. Each convergent sequence in $L_\varphi (\mu)$ has a subsequence that converges a.e. to its limit.
	\end{theorem}
	
	The following proof remains valid if the Orlicz integrand $\varphi$ has no point of continuity on a set of positive measure.
	
	\begin{proof}
		Since completeness is preserved under isometry, we may assume $\mu$ complete without loss of generality. It is standard to check that $L_\varphi (\mu)$ is a normed linear space. We extend the completeness proof of \cite[Thm. 2.4]{Ko1} to the non-$\sigma$-finite case. It suffices to prove that any Cauchy sequence $u_n$ has a norm convergent subsequence that converges a.e. We claim that it is enough to supply a subsequence $u_{n_k}$ that converges a.e. For then, the a.e. limit of $u_{n_k}$ agrees with some strongly measurable function $u$ a.e. so that the Fatou lemma implies
		\begin{align*}
			\int \varphi \left[ \om, \lambda \left( u - u_{n_\ell} \right) \right] \, d \mu
			& = \int \liminf_k \varphi \left[ \om, \lambda \left( u_{n_k} - u_{n_\ell} \right) \right] \, d \mu \\
			& \le \liminf_k \int \varphi \left[ \om, \lambda \left( u_{n_k} - u_{n_\ell} \right) \right] \, d \mu,
		\end{align*}
		whence $\| u - u_{n_\ell} \|_\varphi \le \liminf_k \| u_{n_k} - u_{n_\ell} \|_\varphi$ follows. Consequently, $u \in L_\varphi$ and $u_n \to u$, whereby completeness obtains. Since each member of the sequence $u_n$ is almost separably valued, we may assume $X$ to be separable without loss of generality. As $\mu$ is complete, we may then also assume that $\varphi$ is an integrally normal integrand on $X$ by Lemma \ref{lem: joint measurability}.
		$$
		\forall m \in \N \, \exists M \in \N \colon n_1, n_2 \ge M \implies \| u_{n_1} - u_{n_2} \| < \frac{1}{m}.
		$$
		The set
		$$
		A = \bigcup_{m \ge 1} \bigcup_{n_1, n_2 \ge M(m)} \left\{ \varphi \left( \om, m \left[ u_{n_1}(\om) - u_{n_2}(\om) \right] \right) > 0 \right\}
		$$
		is a countable union of sets, on each of which a positive, integrable function exists. Hence, it is $\sigma$-finite. We claim that
		\begin{equation} \label{eq: a.e. on complement}
			\exists \lim_n u_n(\om) \in X \quad \text{for a.e. } \om \in \Om \setminus A.
		\end{equation}
		Indeed
		$$
		n_1, n_2 \ge M(m) \implies \varphi \left( \om, m \left[ u_{n_1}(\om) - u_{n_2}(\om) \right] \right) = 0 \quad \forall \om \in \Om \setminus A.
		$$
		Therefore, (\ref{eq: a.e. on complement}) follows from Proposition \ref{pr: Orlicz function and convergence} by Definition \ref{def: Orlicz integrand}. We have reduced to the problem of extracting from $u_n$ a subsequence that converges a.e. on the $\sigma$-finite set $A$. Thus, we may from now on assume that $\mu$ is $\sigma$-finite without loss of generality, hence we may take $\mu$ finite by possibly modifying the integrand and measure as in Proposition \ref{pr: equivalent finite measure}. We argue by contradiction that $u_n$ converges in measure: suppose that there exists $\e > 0$ and $\delta > 0$ such that, for any subsequence of $n$, there exists a subsubsequence $n_k$ with
		$$
		C_\e := \left\{ \| u_{n_k} - u_{n_\ell} \| > \e \right\}, \quad \mu \left( C_\e \right) > \delta.
		$$
		Note that the sets
		$$
		\left\{ \inf_{\| x \| > \e} \frac{\varphi \left( \om, r x\right)}{\| x \|} > 1 \right\}
		$$
		are measurable by normality of $\varphi$ and Lemma \ref{lem: equivalence infimal measurability}. The measure $\mu$ being finite, we find $r > 0$ so large that
		$$
		\mu \left( \om \in C_\e \st \inf_{\| x \| > \e} \frac{\varphi \left( \om, r x\right)}{\| x \|} > 1 \right) > \mu \left( C_\e \right) - \frac{\delta}{2},
		$$
		hence the Definition of $C_\e$ implies
		$$
		\mu \left( \om \in C_\e \st \varphi \left( \om, r \left[ u_{n_k} - u_{n_\ell} \right] \right) > \e \right) > \frac{\delta}{2}.
		$$
		However, by the Markov inequality, we have
		$$
		\mu \left( \om \in C_\e \st \varphi \left( \om, r \left[ u_{n_k} - u_{n_\ell} \right] \right) > \e \right) \le \frac{1}{\e} \int \varphi \left( \om, r \left[ u_{n_k} - u_{n_\ell} \right] \right) \, d \mu \xrightarrow{k, \ell \to \i} 0.
		$$
		We have arrived at a contradiction; $u_n$ converges in measure hence admits an a.e. convergent subsequence on $A$ thus on $\Om$.
	\end{proof}
	
	An important difference between the well-known Bochner-Lebesgue spaces
	$$
	L_p\left(\mu ; X \right), \quad 1 \le p < \i
	$$
	and a general Orlicz space is the possibility that an element of $L_\varphi(\mu)$ need not vanish outside a $\sigma$-finite set. Many results about $L_p\left(\mu ; X \right)$ are easy to prove for $\sigma$-finite measures and may then be transferred to the case of an arbitrary measure by using this observation. Also, functions vanishing off a $\sigma$-finite set appear naturally when one characterizes the maximal linear subspace of $\dom I_\varphi$, cf. Theorem \ref{thm: C max subspace}. In order to capture this behaviour in our theory, we introduce
	\begin{definition}[$\sigma$-finite concentration] \label{def: sigma-finite concentration}
		A function $u \colon \Om \to X$ is \emph{$\sigma$-finitely concentrated} iff it vanishes outside a $\sigma$-finite set. For $L \subset L_\varphi(\mu)$, we denote by $L^\sigma$ the subset of $\sigma$-finitely concentrated elements in $L$.
	\end{definition}
	
	Our use of the word concentration in Definition \ref{def: sigma-finite concentration} has been criticized and the name support was suggested to us instead. However, in light of the topological implication of closedness implicit therein, we prefer the present terminology to differentiate more clearly in situations when $\Om$ is also a topological space.
	
	\begin{lemma} \label{lem: Lvarphisigma closed linear subspace}
		The space $L^\sigma_\varphi(\mu)$ is a closed linear subspace of $L_\varphi(\mu)$.
	\end{lemma}
	
	\begin{proof}
		By Theorem \ref{thm: L complete}.
	\end{proof}
	
	The property $L_\varphi(\mu) = L_\varphi^\sigma(\mu)$ can be characterized for the extensive class of separably measurable Orlicz integrands.
	
	\begin{theorem} \label{thm: Lvarphi = Lvarphisigma characterization}
		For $W \subset X$, we set $A_W = \left\{ \om \st \exists x \in W \setminus \left\{ 0 \right\} \colon \varphi \left( \om, x \right) = 0 \right\}$. If, for every $W_0 \in \SS \left( X \right)$, there exists $W \in \SS \left( X \right)$ with $W_0 \subset W$ such that $A_W \in \AA_\mu$ and $A_W$ is $\sigma$-finite, then $L_\varphi(\mu) = L_\varphi^\sigma(\mu)$. If $\varphi$ is separably measurable, then the converse is true as well.
	\end{theorem}
	
	The result applies in particular if the minimum of $\varphi$ at zero is strict for a.e. $\om \in \Om$ as happens for the Bochner-Lebesgue spaces with $1 \le p < \i$.
	
	\begin{proof}
		We want to show that any $u \in L_\varphi(\mu)$ is $\sigma$-finitely concentrated. We may for this assume $\| u \|_\varphi = 1$, hence $I_\varphi \left( u \right) \le 1$. We have $\left\{ \varphi \left[ \om, u \left( \om \right) \right] > 0 \right\} \in \AA_\sigma$ since this set admits a positive integrable function. As $u$ is almost separably valued, our assumption implies that
		$$
		\left\{ u \left( \om \right) \ne 0 \right\} \setminus \left\{ \varphi \left[ \om, u \left( \om \right) \right] > 0 \right\} \in \AA_\sigma
		$$
		whence $u \in L_\varphi^\sigma(\mu)$ follows as $\AA_\sigma$ is a $\sigma$-ring.
		
		Regarding the converse, suppose $W_0 \in \SS \left( X \right)$ with $A_W \not \in \AA_\sigma$ for every superspace $W \in \SS \left( X \right)$ of $W_0$. We shall show that $L_\varphi(\mu) \setminus L_\varphi^\sigma(\mu)$ is non-empty by constructing an element. The Orlicz integrand $\varphi$ being separably measurable, there exists a closed superspace $W \in \SS \left( X \right)$ such that the restriction of $\varphi$ is $\AA \otimes \BB(W)$-measurable. Hence, for any given $n \in \N$, the multimap
		$$
		\om \mapsto \Gamma_n \left( \om \right) :=
		\begin{cases}
			\left\{ x \in W \st \| x \| \ge \frac{1}{n}, \varphi \left( \om, x \right) = 0 \right\} & \text{ if this set is non-empty}, \\
			\left\{ 0 \right\} & \text{ else}
		\end{cases}
		$$
		has a $\AA_\mu \otimes \BB(W)$-measurable graph and thus permits an $\AA_\mu$-measurable selection $u_n \colon \Om \to W$ by the Aumann theorem in the form of \cite[Thm. 6.10]{FoLe}. The same theorem together with \cite[Thm. 6.5]{FoLe} shows that $\Gamma_n$ is Effros $\AA_\mu$-measurable. Note that the $\AA_\mu$-measurable sets $\Om \setminus \Gamma_n^- \left( \left\{ 0 \right\} \right)$ converge to $A_W \not \in \AA_\sigma$ as $n \to \i$ hence are non-$\sigma$-finite eventually. Therefore, $u_n$ eventually differs from zero on a set of non-$\sigma$-finite measure. In conclusion $u_n \in L_\varphi(\mu) \setminus L_\varphi^\sigma(\mu)$.
	\end{proof}
	
	\subsection{Embeddings and almost embeddings}
	
	We close the section by proving that on sets of finite measure, $L_\varphi(\mu)$ lies almost between $L_\i(\mu)$ and $L_1(\mu)$, which will be instrumental in deducing properties of $L_\varphi(\mu)$ from those of the better understood Bochner-Lebesgue spaces. We prepare this result with a simple embedding lemma providing continuous inclusions between Orlicz spaces in terms of their integrands.
	
	\begin{lemma} \label{lem: embedding}
		Let $\varphi$ and $\phi$ be Orlicz integrands such that
		$$
		\exists L > 0, f \in L_1(\mu) \colon \varphi \left( \om, x \right) \le \phi \left( \om, L x \right) + f \left( \om \right).
		$$
		Then
		$$
		\vvvert u \vvvert_\varphi \le L \left( 1 + \| f \|_{L_1} \right) \vvvert u \vvvert_\phi \quad \forall u \in L_\phi(\mu).
		$$
	\end{lemma}
	
	\begin{proof}
		For $u \in L_\phi(\mu)$, there holds
		\begin{align*}
			\vvvert u \vvvert_\varphi
			\le \inf_{\alpha > 0} \alpha^{-1} \left[ 1 + \| f \|_{L_1} + I_\phi \left( L \alpha u \right) \right]
			& \le \left( 1 + \| f \|_{L_1} \right) \vvvert L u \vvvert_\phi \\
			& = L \left( 1 + \| f \|_{L_1} \right) \vvvert u \vvvert_\phi. \qedhere
		\end{align*}
	\end{proof}
	
	\begin{lemma} \label{lem: a emb}
		Let $\mu$ be finite. Then there exists an isotonic family $\Om_\e \in \AA$ with $\lim_{\e \downarrow 0} \mu \left( \Om \setminus \Om_\e \right) = 0$ such that there hold the continuous embeddings
		$$
		L_\i \left( \Om_\e ; X \right) \to L_\varphi \left( \Om_\e \right) \to L_1 \left( \Om_\e ; X \right)
		$$
		via identical inclusion.
	\end{lemma}
	
	\begin{proof}
		Consider for $W \in \SS(X)$ the $\AA_\mu$-measurable sets
		$$
		E'_\e(W) = \left\{ \sup \varphi_\om \left( B_{\e, W} \right) \le 1 \right\}, \quad E''_\e(W) = \left\{ \inf \varphi_\om \left( W \setminus \overline{B}_{1 / \e} \right) \ge 1 \right\}.
		$$
		Measurability follows from separable measurability of $\varphi$. More precisely, the epigraphical multimap of $\varphi$ is $\AA_\mu$-measurable by Lemma \ref{lem: joint measurability}. Hence, the set $E'_\e(W)$ is $\AA_\mu$-measurable by the Hess theorem \cite[Thm. 6.5.14]{B} as the pre-image under the epigraphical multimap of the Wijsman-closed set
		$$
		\bigcap_{x \in B_{\e, W} \times \left[ 1, \i \right) } \left\{ F \in \CL(X \times \R) \st d_x(F) = 0 \right\}.
		$$
		For the set $E''_\e(W)$, this follows from the infimal measurability of normal integrands by Lemma \ref{lem: equivalence infimal measurability}. We may by \cite[Thm. 1.108]{FoLe} define the essential intersections
		$$
		\Om'_\e = \esscap_{W \in \SS(X) } E', \quad \Om''_\e = \esscap_{W \in \SS(X) } E''.
		$$
		By the same theorem and since $E'$ and $E''$ are non-increasing w.r.t. $W$, there exist $W'_\e, W''_\e \in \SS(X)$ with $\Om' = E' \left(W'_\e \right)$ and $\Om'' = E'' \left(W''_\e \right)$ a.e. so that, for any null sequence $\e_n$, we find $W' \in \SS(X)$ and $W'' \in \SS(X)$ independent of $n$ with $\Om'_{\e_n} = E' \left(W' \right)$ and $\Om''_{\e_n} = E'' \left(W'' \right)$ a.e. Hence
		$$
		\lim_{\e \to 0} \mu \left( \Om \setminus \Om'_\e \right) = \lim_{\e \to 0} \mu \left( \Om \setminus \Om''_\e \right) = 0
		$$
		because $\varphi$ is an Orlicz integrand thus vanishes continuously at the origin and has bounded sublevels. Setting $\Om_\e = \Om' \cap \Om''$, we have $\lim_{\e \to 0} \mu \left( \Om \setminus \Om_\e \right) = 0$. Denoting by $I_{B_X}$ the indicator in the sense of convex analysis of the unit ball $B_X$ we have $\varphi_\om(x) \le I_{B_X} \left( \frac{x}{\e} \right) + 1$ a.e. on $\Om'$ and $\e \| x \| \le \varphi_\om(x) + \frac{1}{\e}$ a.e. on $\Om''$ for all $x \in W$ for any $W \in \SS(X)$. Therefore,
		$$
		\vvvert u \vvvert_\varphi \le \frac{1}{\e} \left[ 1 + \mu \left( \Om \right) \right] \vvvert u \vvvert_\i, \quad
		\e \vvvert u \vvvert_1 \le \left( 1 + \frac{\mu(\Om)}{\e} \right) \vvvert u \vvvert_\varphi
		$$
		by Lemma \ref{lem: embedding} as any $u \in L_\varphi(\mu)$ is almost separably valued.
	\end{proof}
	
	\paragraph{Remark.} We found the idea for Lemma \ref{lem: a emb} in \cite[Thm. 3.2]{CaKa}, where it is credited to \cite{Gi1} in the case of a separable range space. Non-separability of $X$ is a major source of difficulty in conceiving the proof above, leading to the use of essential intersections instead of regular ones.
	\\
	
	In view of §\ref{sec: inf-int}, it becomes important to understand almost decomposability of $L_\varphi(\mu)$ and its subspaces. Obviously, $L_\varphi(\mu)$ and $L_\varphi^\sigma(\mu)$ are weakly decomposable. We also have
	
	\begin{corollary} \label{cor: L a decomp}
		$L_\varphi(\mu)$ and $L_\varphi^\sigma(\mu)$ are almost decomposable.
	\end{corollary}
	
	\begin{proof}
		Let $F \in \AA_f$ and $v \in L_\i \left( F ; X \right)$. Since $L_\varphi(\mu)$ and $L_\varphi^\sigma(\mu)$ are weakly decomposable and linear, it suffices to prove that, for $\e > 0$, there exists $F_\e \subset F$ with $\mu \left( F \setminus F_\e \right) < \e$ and $v \chi_{F_\e} \in L_\varphi(\mu)$, which follows from Lemma \ref{lem: a emb}.
	\end{proof}
	
	\section{The closure of simple functions $E_\varphi(\mu)$} \label{sec: E}
	
	We compile in this ancillary section basic facts about the space $E_\varphi(\mu)$ of the closure of simple functions in $L_\varphi(\mu)$. Even though the simple functions are in general not dense in $L_\varphi(\mu)$, their closure $E_\varphi(\mu)$ can still be used to approximate all of $L_\varphi(\mu)$ in a suitable sense, at least on $\sigma$-finite sets.
	
	\begin{definition}[convergence from below] \label{def: conv f bel}
		A sequence $u_n \colon \Om \to X$ of measurable functions is said to \emph{converge from below} to $u$ if there exists a sequence $\Om_n \in \AA$ with $\mu \left( \lim_n \Om \setminus \Om_n \right) = 0$ and $u_n = u \chi_{\Om_n}$. We write $u_n \uparrow u$ to signify that $u_n$ converges from below to $u$. We say that a convergence from below is monotonic if the sequence $\Om_n$ is isotonic, i.e., if $\Om_n \subset \Om_{n + 1}$ for all $n \ge 1$.
	\end{definition}
	
	We shall define the class of absolutely continuous functionals as those enjoying continuity from below and vanishing outside a $\sigma$-finite set. Such a functional is determined by its action on any almost decomposable subspace of $L_\varphi(\mu)$, for which $E_\varphi(\mu)$ is an example. This is the content of the next two lemmas and our primary use for $E_\varphi(\mu)$ in the duality theory. The space $E_\varphi(\mu)$ is obviously weakly decomposable. We also have
	
	\begin{lemma} \label{lem: E a decomp}
		$E_\varphi(\mu)$ and $E_\varphi^\sigma(\mu)$ are almost decomposable.
	\end{lemma}
	
	\begin{proof}
		Remembering the remark below Definition \ref{def: almost decomposable} on intersections of almost decomposable spaces, we need only consider $E_\varphi(\mu)$ since $E_\varphi^\sigma(\mu) = E_\varphi(\mu) \cap L_\varphi^\sigma(\mu)$ and these spaces are weakly decomposable in addition to $L_\varphi^\sigma(\mu)$ being almost decomposable by Corollary \ref{cor: L a decomp}.
		
		Let $F \in \AA_f$ and $v \in L_\i \left( F ; X \right)$. Since $E_\varphi(\mu)$ is weakly decomposable and linear, it suffices to prove that, for every $\e > 0$, there exists $F_\e \subset F$ with $\mu \left( F \setminus F_\e \right) < \e$ and $v \chi_{F_\e} \in E_\varphi(\mu)$. We find by the Egorov theorem a subset $E_\e \subset F$ with $\mu \left( F \setminus E_\e \right) < \frac{\e}{2}$ and $\lim_{\lambda \to 0} \varphi \left[ \om, \lambda v(\om) \right] = 0$ uniformly on $E_\e$, hence $v \chi_{E_\e} \in L_\varphi(\mu)$. Pick $v_n$ a sequence of simple functions with $v_n \to v$ a.e. By the Egorov theorem, we find a sequence $F_{k, \e} \subset E_\e$ with $\mu \left( E_\e \setminus F_{k, \e} \right) < 2^{- k - 1} \e$ and $\lim_n \varphi \left[ \om, k \left[ v(\om) - v_n(\om) \right] \right) = 0$ uniformly on $F_{k, \e}$ for fixed $k$. Consequently, the same holds on $F_\e = \bigcap_{k \ge 1} F_{k, \e}$ for all $k \ge 1$ so that $v_n \chi_{F_\e} \to v \chi_{F_\e}$ in $L_\varphi(\mu)$ by definition of the Luxemburg norm. In conclusion $v \chi_{F_\e} \in E_\varphi(\mu)$ and $\mu \left( F \setminus F_\e \right) = \mu \left( F \setminus E_\e \right) + \mu \left( E_\e \setminus F_\e \right) \le \frac{\e}{2} + \frac{\e}{2} = \e$.
	\end{proof}
	
	\begin{lemma} \label{lem: decomp ss abs cont dense}
		Given $u \in L^\sigma_\varphi(\mu)$ and an almost decomposable subspace $L \subset L_\varphi(\mu)$, there exists a sequence $u_n \in L$ with $u_n \uparrow u$ monotonically.
	\end{lemma}
	
	\begin{proof}
		Let $u$ vanish outside $\Sigma \in \AA_\sigma$ with $\Sigma = \bigcup_n F_n$ for an isotonic sequence $F_n \in \AA_f$. By the first remark below Definition \ref{def: almost decomposable}, it is immaterial that $u$ might be unbounded so that there exists an increasing sequence of sets $G_n \subset F_n$ with $\lim_n \mu \left( F_n \setminus G_n \right) = 0$ and $u \chi_{G_n} \in L$, hence $u \chi_{G_n} \uparrow u$.
	\end{proof}
	
	\section{The elements with absolutely continuous norm $C_\varphi(\mu)$} \label{sec: C}
	
	In this section, we study properties of the space $C_\varphi(\mu)$ of the elements in $L_\varphi(\mu)$ whose norm is absolutely continuous, i.e., for which $\lim_n \| u \chi_{E_n} \|_\varphi = 0$ whenever $E_n \in \AA$ is a sequence with $\mu\left( \lim_n E_n \right) = 0$. In the scalar theory $X = \R$, this space is important because $C_\varphi(\mu)^* = L_{\varphi^*}(\mu)$ if $\varphi$ is real-valued, inducing a weak* topology on $L_{\varphi^*}(\mu)$ that can serve to compensate if $L_{\varphi^*}(\mu)$ lacks reflexivity. The situation is similar, yet somewhat more complicated for the vector valued case. Nevertheless, our main interest in $C_\varphi(\mu)$ lies in its role of inducing a weak* topology on the function component of the dual space of $L_\varphi(\mu)$. Besides, the space $C_\varphi(\mu)$ is the key to understanding separability and reflexivity of its superspace $L_\varphi(\mu)$, as mentioned in the introduction. Indeed, we will recognize the linearity of $\dom I_\varphi$ as necessary for both these properties to occur. Since $C_\varphi(\mu)$ typically coincides with out to be the maximal linear subspace $M_\varphi(\mu)$ of $\dom I_\varphi$, so that the linearity of this domain is under mild conditions equivalent to $C_\varphi(\mu) = L_\varphi(\mu)$, settling these matters for $C_\varphi(\mu)$ solves the actual questions.
	
	\subsection{Basic properties}
	
	We start by proving the basic characterization of when $C_\varphi(\mu)$ coincides with the maximal linear Banach subspace of $\dom I_\varphi$.
	
	\begin{lemma} \label{lem: C closed subspace}
		$C_\varphi(\mu)$ and $C_\varphi^\sigma(\mu)$ are closed linear subspace of $L_\varphi(\mu)$ and $L_\varphi^\sigma(\mu)$.
	\end{lemma}
	
	\begin{proof}
		Linearity is clear. Closedness of $C_\varphi(\mu)$ follows by an obvious $\frac{\e}{2}$-argument. The case of $C_\varphi^\sigma(\mu)$ then obtains by Lemma \ref{lem: Lvarphisigma closed linear subspace}.
	\end{proof}
	
	\begin{theorem} \label{thm: C max subspace}
		For $A_\lambda = \left\{ u \in L_\varphi \st \lambda u \in \dom I_\varphi \right\}$, there holds
		\begin{equation} \label{eq: line space has abs cont norm}
			\bigcap_{\lambda > 0} A_\lambda = \bigcap_{n \in \N} A_n \subset C_\varphi^\sigma(\mu).
		\end{equation}
		If $\varphi$ is real-valued on atoms of finite measure, the inclusion in (\ref{eq: line space has abs cont norm}) is an equality. It is proper if $\varphi$ is not real-valued on an atom of finite measure.
	\end{theorem}
	
	\begin{proof}
		The first identity in (\ref{eq: line space has abs cont norm}) holds since $A_\lambda$ decreases as $\lambda$ increases. Inclusion: for $n \in \N$, $u \in \bigcap_{\lambda > 0} A_\lambda$ and an evanescent sequence $E_j \in \AA$ there holds
		$$
		\lim_j I_\varphi \left( n u \chi_{E_j} \right) = \lim_j \int_{E_j} \varphi\left[\om, n u(\om) \right] \, d \mu(\om) = 0
		$$
		by absolute continuity of the integral, hence $u \in C_\varphi(\mu)$. As each set in the union $\left\{ u \ne 0 \right\} = \bigcup_{n \in \N} \left\{ \varphi \left( n u \right) > 0 \right\}$ permits a positive integrable function hence is $\sigma$-finite, we conclude $u \in C_\varphi^\sigma(\mu)$.
		
		Addendum: fixing $u \in C^\sigma_\varphi(\mu)$, we may assume $\mu$ to be $\sigma$-finite. We claim that each set $R = R_n = \left\{ \varphi \left( n u \right) = \i \right\}$ and hence their union is null. Since $\varphi$ is real on atoms with finite measure, $R$ contains no atom so that $\mu$ is non-atomic on $R$. In conclusion, $\mu$ has the finite subset property on $R$, see \cite[Def. 1.16, Rem. 1.19]{FoLe}. If $\mu \left( R \right) > 0$, then there exists a sequence $Q_m \subset R$ with $\mu \left( Q_m \right) \searrow 0$ so that the contradiction $0 = \liminf_{m \to \i} \| u \chi_{Q_m} \|_\varphi \ge \frac{1}{n} > 0$ obtains and the claim follows. Let $f \colon \Om \to \R$ be an integrable positive function and consider the sets $A_{\lambda, n} = \left\{ \varphi \left( n u \right) \le \lambda f \right\}$ for $\lambda > 0$. By $\mu \left( \bigcup_{n \ge 1} R_n \right) = 0$, there holds $\mu\left( \lim_{\lambda \to \i} \Om \setminus A_{\lambda, n} \right) = 0$ so that, for $\lambda$ sufficiently large, we have $\| u \chi_{\Om \setminus A_{\lambda, n} } \|_\varphi < \frac{1}{n}$, hence
		\begin{align*}
			I_\varphi \left( n u \right)
			& = \int_{\Om \setminus A_{\lambda, n} } \varphi \left[ \om, n u \left( \om \right) \right] \, d \mu \left( \om \right)
			+ \int_{A_{\lambda, n} } \varphi \left[ \om, n u \left( \om \right) \right] \, d \mu \left( \om \right) \\
			& \le 1 + \lambda \int f \, d \mu < \i,
		\end{align*}
		whence $u \in \bigcap_{\lambda > 0} A_\lambda$. Regarding the necessity of $\varphi$ being real-valued on each atom $A \in \AA_f$, consider $x$ such that $\varphi \left( \om, x \right) = \i$ a.e. on $A$. Then
		\begin{equation*}
			x \chi_A \in C_\varphi \setminus \bigcap_{\lambda > 0} A_\lambda. \qedhere
		\end{equation*}
	\end{proof}
	
	\begin{corollary} \label{cor: C sigma fin}
		Let $\varphi$ be real-valued on atoms of finite measure and let $\mu$ have no atom of infinite measure. Then $C_\varphi(\mu) = C^\sigma_\varphi(\mu)$.
	\end{corollary}
	
	\begin{proof}
		Suppose there were $u \in C_\varphi \setminus C^\sigma_\varphi$. Then there exists $n \in \N$ with
		$$
		\int \varphi(n u) \, d \mu = +\i
		$$
		by Theorem \ref{thm: C max subspace}. Since we assume that no atom of infinite measure exists, Proposition \ref{pr: divergent subintegral} yields a set $\Sigma \in \AA_\sigma$ such that
		$$
		\int_\Sigma \varphi(nu) \, d \mu = +\i,
		$$
		which is impossible by Theorem \ref{thm: C max subspace} because $u \chi_\Sigma \in C^\sigma_\varphi(\mu)$.
	\end{proof}
	
	\begin{corollary} \label{cor: linear domain}
		If $\dom I_\varphi$ is linear, then $L_\varphi(\mu) = C^\sigma_\varphi(\mu)$. Conversely, if $L_\varphi(\mu) = C^\sigma_\varphi(\mu)$ and $\varphi$ is real-valued on atoms of finite measure, then $\dom I_\varphi$ is linear.
	\end{corollary}
	
	\begin{proof}
		By Theorem \ref{thm: C max subspace} since $\lin \dom I_\varphi = L_\varphi$.
	\end{proof}
	
	Theorem \ref{thm: C max subspace} allows a simple characterization of Orlicz integrands for which all elements of $L_\varphi(\mu)$ have absolutely continuous norms in terms of a growth condition often dubbed $\Delta_2$ or doubling condition. Similar conditions and their role in the theory of Orlicz spaces are well-known in the scalar and vector valued cases, cf. \cite{RaRe, Ko2}.
	
	\begin{definition}[$\Delta_2$-condition] \label{def: Delta2 cond}
		We say that the Orlicz integrand $\varphi$ satisfies the \emph{$\Delta_2$-condition} and write $\varphi \in \Delta_2$ if
		\begin{gather*}
			\forall A \in \AA_{a\sigma} \, \forall S \in \SS \left( X \right) \, \exists k \ge 1 \land f \in L_1(\mu) \colon \\
			\forall x \in S \quad \varphi \left( \om, 2 x \right) \le k \varphi (\om, x) + f(\om) \quad \mu \text{-a.e. on } A.
		\end{gather*}
	\end{definition}
	
	\begin{theorem} \label{thm: linear domain 2}
		There holds $L_\varphi(\mu) = C^\sigma_\varphi(\mu)$ if $\varphi \in \Delta_2$. If $\mu$ is non-atomic, then $\varphi \in \Delta_2$ is also necessary for $L_\varphi(\mu) = C^\sigma_\varphi(\mu)$ to hold.
	\end{theorem}
	
	\begin{proof}
		The first claim will follow by Theorem \ref{thm: C max subspace} once we prove that $\dom I_\varphi$ is linear if $\varphi \in \Delta_2$. As $\dom I_\varphi$ is an absolutely convex set, its linearity is equivalent to the implication
		$$
		I_\varphi(u) < \i \implies I_\varphi(2u) < \i.
		$$
		Arguing by contradiction, we assume $I_\varphi(2u) = \i$. Proposition \ref{pr: divergent subintegral} yields $A \in \AA_{a\sigma}$ with $I_\varphi\left( 2u \chi_A \right) = \i$. As $u$ is almost separably valued, the assumption $\varphi \in \Delta_2$ yields
		$$
		I_\varphi \left( 2 u \chi_A \right) \le k I_\varphi\left( u \chi_A \right) + \int_A f \, d \mu < \i,
		$$
		whence we have arrived at a contradiction. Regarding the necessity, let $\Sigma \in \AA_\sigma$ and $S \in \SS \left( X \right)$ as in Definition \ref{def: Delta2 cond}. Since $\varphi$ is integrally separably measurable, there exists a closed subspace $W \in \SS \left( X \right)$ with $S \subset W$ such that the restriction $\left. \varphi \right|_{\Sigma \times W}$ is $\AA(\Sigma) \otimes \BB(W)$-measurable. Hence, we may via restriction assume that $\mu$ is $\sigma$-finite and $\varphi$ is $\AA \otimes \BB$-measurable on a separable space. Let $\phi_\om (x) = \varphi_\om \left( 2 x \right)$ so that our assumption $L_\varphi(\mu) = C_\varphi(\mu)$ implies $\dom I_\varphi \subset \dom I_\phi$ by Theorem \ref{thm: C max subspace} as $\mu$ is non-atomic. Hence, $\varphi \in \Delta_2$ follows by \cite[Thm. 1.7]{Ko2}.
	\end{proof}
	
	If $\mu$ has an atom, then $L_\varphi(\mu) = C^\sigma_\varphi(\mu)$ may hold even if $\varphi \notin \Delta_2$. For example, regard $\R^n$ as an Orlicz space of real valued functions on the uniform measure space $\left\{ 1, \dots , n \right\}$ and take any real-valued map $\varphi \in \Gamma(\R)$ with $\varphi \notin \Delta_2$ as the Orlicz function.
	
	\subsection{Decomposability}
	
	As $C_\varphi(\mu)$ is the predual of the function component in $L_\varphi(\mu)^*$ if $\varphi$ is real-valued, it is interesting to understand convex duality also on $C_\varphi(\mu)$. For example, a convex functional on the function component of $L_{\varphi^*}(\mu)$ is weak* lower semicontinuous iff it arise as a convex conjugates w.r.t. this pairing.
	
	We saw the importance of decomposability for the representation of convex conjugates in Theorem \ref{thm: conjugate A}. This motivates to study this property for $C_\varphi(\mu)$. $C_\varphi(\mu)$ is weakly decomposable. We also have
	
	\begin{lemma} \label{lem: C a decomp}
		If $\varphi$ is real-valued, then $C_\varphi(\mu)$ and $C_\varphi^\sigma(\mu)$ are almost decomposable.
	\end{lemma}
	
	\begin{proof}
		Let $F \in \AA_f$ and $v \in L_\i \left( F ; X \right)$. Since $C_\varphi(\mu)$ and $C_\varphi^\sigma(\mu)$ are weakly decomposable linear spaces, it suffices by a remark below Definition \ref{def: almost decomposable} to prove that, for $\e > 0$, there exists a set $F_\e \subset F$ with $\mu \left( F \setminus F_\e \right) < \e$ and $v \chi_{F_\e} \in C_\varphi(\mu)$. Since $\varphi$ is real-valued, we find for $k \in \N$ a set $F_{k, \e} \subset F$ with $\mu \left( F \setminus F_{k, \e} \right) < 2^{-k} \e$ and $\sup_{\om \in F_{k, \e} } \varphi \left[ \om, k v \left( \om \right) \right] < \i$. Thus, for $F_\e = \bigcap_k F_{k, \e}$, there holds $ \mu \left( F \setminus F_\e \right) < \e$ and $v \chi_{F_\e} \in C_\varphi(\mu)$ by Theorem \ref{thm: C max subspace}.
	\end{proof}
	
	If $\varphi$ is not real-valued, then the maximal linear subspace of $\dom I_\varphi$ may be trivial, hence Lemma \ref{lem: C a decomp} ceases to hold. Consider the example $L_\i\left( \left[ 0, 1 \right] ; X \right)$ with the Orlicz function $I_{B_X}$.
	
	For every countable family in $L^\sigma_\varphi(\mu)$, there exists an evanescent sequence of sets outside which each element has absolutely continuous norm. This observation will yield insight into the dual spaces of $C_\varphi(\mu)$ and $L_\varphi(\mu)$.
	
	\begin{lemma} \label{lem: norm almost abs cont}
		If $\varphi$ is real-valued, then for any sequence $u_k \in L_\varphi^\sigma(\mu)$ there exists a decreasing sequence $E_\ell \in \AA$ with $\mu \left( \lim_\ell E_\ell \right) = 0$ and $u_k \chi_{\Om \setminus E_\ell} \in C_\varphi^\sigma(\mu)$.
	\end{lemma}
	
	\begin{proof}
		We may assume $\mu$ to be $\sigma$-finite. Hence, we may take $\mu$ to be finite by Proposition \ref{pr: equivalent finite measure}. Since $C_\varphi^\sigma(\mu)$ is almost decomposable by Lemma \ref{lem: C a decomp}, we find by Lemma \ref{lem: decomp ss abs cont dense} sequences of sets $D_{k, \ell} \in \AA$ decreasing in $\ell$ with
		$$
		\mu\left( D \right) \le 2^{- k - \ell};	\quad
		u_k \chi_{\Om \setminus D} \in C_\varphi^\sigma(\mu).
		$$
		Hence, for $E_\ell = \bigcup_{k \ge 1} D$, we have
		\begin{equation*}
			\mu\left( E_\ell \right) \le \sum_{k \ge 1} \mu\left( D \right) \le 2^{- \ell}; \quad
			u_k \chi_{\Om \setminus E_\ell} \in C_\varphi^\sigma(\mu) \quad \forall k \in \N. \qedhere
		\end{equation*}
	\end{proof}
	
	As our last fundamental fact on $C_\varphi(\mu)$ and a first step towards investigating separability, we prove the denseness of simple functions.
	
	\begin{lemma} \label{lem: C simple dense}
		There holds $C_\varphi^\sigma(\mu) \subset E_\varphi^\sigma(\mu)$. If $\varphi$ is real-valued, then integrable simple functions are dense in $C_\varphi^\sigma(\mu)$.
	\end{lemma}
	
	\begin{proof}
		The proof will be achieved by first providing a sequence of simple functions that is dense in $C_\varphi^\sigma(\mu)$ and second proving that each member of this sequence may be taken integrable if $\varphi$ is real-valued. It suffices to consider $\sigma$-finite measures $\mu$, hence we may equivalently consider a modified Orlicz integrand $\phi$ and a finite measure $\nu$ as in Proposition \ref{pr: equivalent finite measure}. Note however that we want to obtain a density set of $\mu$-integrable simple functions. For $u \in C_\varphi^\sigma(\mu)$, pick a sequence of simple functions with $u_n \to u$ a.e. so that, for $m \in \N$, there holds
		\begin{equation} \label{eq: convergence ae}
			\lim_n \phi \left( \om, m \left[ u(\om) - u_n(\om) \right] \right) = 0 \text{ for a.e. } \om \in \Om
		\end{equation}
		since $\phi$ vanishes continuously at the origin. For $k \in \N$, we find by the Egorov theorem a set $B = B_{k, m} \in \AA$ with $\nu \left( \Om \setminus B \right) < 2^{- m} k^{-1}$ and such that (\ref{eq: convergence ae}) uniformly on $B$. Hence, setting $C_k = \bigcap_{m \ge 1} B_{k, m}$, we have $\nu \left( \Om \setminus C_k \right) < k^{-1}$ and the convergence (\ref{eq: convergence ae}) holds uniformly on $C_k$. In particular, by the triangle inequality, the simple functions $u_n \chi_{C^c_k}$ eventually belong to $L_\varphi(\mu)$ hence to $E_\varphi(\mu)$. Lemma \ref{lem: norm almost abs cont} yields a sequence $D_j \in \AA$ with $\mu\left( \lim_j D_j \right) = 0$ such that, for all $n \in \N$ sufficiently large, there holds $u_n \chi_{C_k \cap D^c_j} \in C_\varphi^\sigma(\mu)$ for all $j \in \N$ if $\varphi$ is real-valued. Otherwise, we set $D_j = \emptyset$. In the former case, pick $F_j \in \AA_f$ with $\Om = \bigcup_j F_j$ and possibly replace $D_j$ by $D_j \cup \Om \setminus F_j$ so that $\mu \left( \Om \setminus D_j \right) < \i$, hence each $u_n \chi_{C_k \cap D^c_j}$ is a $\mu$-integrable simple function. Since $u$ has absolutely continuous norm, there exists for any given $\e > 0$ a $\delta > 0$ such that there holds $\| u \chi_{C^c_k \cup D_j} \|_\varphi < \e$ whenever $j, k > \delta^{-1}$. Therefore
		\begin{align*}
			\| u - u_n \chi_{C_k \cap D^c_j} \|_\varphi
			& \le \| u \chi_{C^c_k \cup D_j} \| + \| \left( u - u_n \right) \chi_{C_k \cap D^c_j} \|_\varphi \\
			& < \e + \| \left( u - u_n \right) \chi_{C_k \cap D^c_j} \|_\varphi \xrightarrow{n \to \i} \e.
		\end{align*}
		Since $\e > 0$ is arbitrary, the proof is complete.
	\end{proof}
	
	\subsection{Separability}
	
	To characterize separability of $C^\sigma_\varphi(\mu)$ hence of $L_\varphi(\mu)$ if the spaces agree, we introduce the following
	
	\begin{definition}[separable measure] \label{def: sep mes}
		The measure $\mu$ is called separable if $\left( \AA_f, d_\mu \right)$ is a separable space for the pseudometric $d_\mu \left( A, B \right) = \mu \left( A \Delta B \right)$.
	\end{definition}
	
	\paragraph{Remark.}
	
	\begin{enumerate}
		
		\item We differ from \cite[§3.5]{RaRe}, where $\AA$ instead of $\AA_f$ and $\arctan d_\mu$ instead of $d_\mu$ show up in the analogous definition. This renders $\left( \N, 2^\N \right)$ equipped with the counting measure non-separable, as $2^\N$ then becomes a discrete metric space of uncountable cardinality. However, the sequence space $\ell_1 \left( \N ; \R \right)$ with its Schauder basis of unit vectors is clearly a separable Orlicz space. Therefore, the old definition is not suited for characterizing the separability of Orlicz spaces, contrary to what \cite[§3.5, Thm. 1]{RaRe} claims.
		
		\item The notion of a separable measure relates to that of a separable measurable space $\left( \Om, \AA \right)$: If $\AA = \sigma \left( A_n \colon n \ge 1 \right)$ is the $\sigma$-algebra generated by the sequence $A_n$ and the measure $\mu$ is $\sigma$-finite, then $\mu$ is separable and the countable algebra $\alpha \left( A_n \colon n \ge 1 \right)$ generated by the sequence $A_n$ is dense in $\left( \AA_f, d_\mu \right)$. This is implicit in the proof of \cite[Thm. 2.16]{FoLe}. A measure is separable iff its completion is.
		
	\end{enumerate}
	
	\begin{theorem} \label{thm: C sep}
		Let $\varphi$ be real-valued. If $\mu$ and $X$ are separable, then there exists a dense sequence of integrable simple functions in $C_\varphi^\sigma(\mu)$, hence the space is separable. Conversely, if $C_\varphi^\sigma(\mu)$ is separable, then $X$ is separable. If in addition $\mu$ has no atom of infinite measure, then $\mu$ also is separable.
	\end{theorem}
	
	\begin{proof}
		$\implies$: We shall pass to several subsequences in the proof none of which we relabel. Lemma \ref{lem: C simple dense} reduces our task to constructing a sequence whose closure includes each function $x \chi_A \in C_\varphi(\mu)$ with $x \in X$ and $A \in \AA_f$. Let $x_k \in X$ and $E_\ell \in \AA$ be dense sequences. Each $A \in \AA_f$ is contained in the $\sigma$-finite set $\bigcup_{\ell \ge 1} A_\ell$ except for a null set. Therefore, we may assume $\mu$ to be $\sigma$-finite hence finite by Proposition \ref{pr: equivalent finite measure}. Lemma \ref{lem: C a decomp} yields sets $B = B(k, m) \in \AA$ such that $\mu \left( \Om \setminus B \right) < \frac{1}{m}$ and $x_k \chi_B \in C_\varphi(\mu)$. We claim the countable family $x_k \chi_{B(k, m)} \chi_{A_\ell}$ for $k, \ell, m \in \N$ to yield the required sequence. Indeed, each function of the form $x_k \chi_B \chi_A$ belongs to its closure since there exists a subsequence with $d_\mu \left( A, A_\ell \right) \to 0$ as $\ell \to \i$. Pick a subsequence with $x_k \to x$ so that the Egorov theorem yields for $\delta > 0$ a set $C_\delta \in \AA$ such that $\mu \left( C_\delta \right) < \delta$ and for any given $n \in \N$ there holds $\varphi \left[ n \left( x - x_k \right) \right] \to 0$ uniformly on $\Om \setminus C_\delta$ as $k \to \i$. Because $x \chi_A$ has absolutely continuous norm, we find for any given $\e > 0$ a $\delta_1 > 0$ such that $\| x \chi_A \chi_{B^c \cup C_\delta} \| < \e$ whenever $\delta, m^{-1} < \delta_1$. Choosing $\delta, m^{-1}$ sufficiently small and combining the last two statements yields
		\begin{align*}
			\| x \chi_A - x_k \chi_A \chi_B \chi_{C^c_\delta} \|_\varphi
			& \le \| x \chi_A \chi_{B^c \cup C_\delta} \|_\varphi
			+ \| \left( x - x_k \right) \chi_{A \cap B \cap C^c_\delta} \|_\varphi \\
			& < \e + \| \left( x - x_k \right) \chi_{A \cap B \cap C^c_\delta} \|_\varphi
			\xrightarrow{k \to \i} \e
		\end{align*}
		Sending $\e \to 0^+$ completes the first part of the proof.
		
		$\impliedby$: Let $u_n \in C_\varphi^\sigma(\mu)$ be a dense sequence. The set $\bigcup_{n \ge 1} u_n \left( \Om \right)$ is almost separably valued, hence there exists a null set $N$ such that
		$$
		\overline{\bigcup_{n \ge 1} u_n \left( \Om \setminus N \right) } =: S \in \SS(X).
		$$
		But then every element of $C_\varphi^\sigma(\mu)$ is $S$-valued a.e. since convergence in $C_\varphi^\sigma(\mu)$ implies convergence a.e. up to subsequences. In conclusion $S = X \in \SS(X)$ because $C_\varphi^\sigma(\mu)$ is almost decomposable by Lemma \ref{lem: C a decomp}. Now, consider the separable $\sigma$-algebra $\AA' = \sigma \left( u_n \colon n \ge 1 \right)$. To see that $\AA'$ is separable, note that it is generated by the sets $u^{-1}_n \left( G_m \right)$ for $G_m$ a sequence generating the Borel $\sigma$-algebra $\BB(X)$. As $u_n$ is dense, each element $u \in C_\varphi^\sigma(\mu)$ is $\AA'$-measurable so that since $C_\varphi^\sigma(\mu)$ is almost decomposable, we deduce $\AA_f \subset \AA'_\mu$, hence $\AA_\sigma \subset \AA'_\mu$. Therefore, our proof will be finished if we prove that $\mu$ is $\sigma$-finite. As each member of the dense sequence $u_n$ is $\sigma$-finitely concentrated, all elements of $C_\varphi^\sigma(\mu)$ vanish outside some $A \in \AA_\sigma$ that is independent of the element under consideration. Suppose $\mu \left( \Om \setminus A \right) > 0$. Since $\mu$ has no atom of infinite measure, we find $B \in \AA_f$ with $B \subset \Om \setminus A$ and $\mu \left( B \right) > 0$. As $C_\varphi^\sigma(\mu)$ is almost decomposable, there exists a non-trivial element $v \in C_\varphi^\sigma(\mu)$ vanishing outside of $B$, hence $v$ does not belong to the closure of $u_n$, which contradicts density of this sequence.
	\end{proof}
	
	Theorem \ref{thm: C sep} settles the separability of $L_\varphi(\mu)$ if it happens to coincide with its subspace $C_\varphi^\sigma(\mu)$.This coincidence is also necessary for $L_\varphi(\mu)$ to be separable:
	
	\begin{theorem} \label{thm: sep implies L = C}
		If $L_\varphi(\mu)$ is separable and $\mu$ has no atom of infinite measure, then $L_\varphi(\mu) = C_\varphi(\mu)$. The same is true for $L_\varphi^\sigma(\mu)$ and $C_\varphi^\sigma(\mu)$ without restriction on the measure.
	\end{theorem}
	
	\begin{proof}
		Let $u \in L_\varphi(\mu) \setminus C_\varphi(\mu)$ so that we may pick $E_n \in \AA$ with
		$$
		\mu\left( \lim_n E_n \right) = 0,
		\quad \| u \chi_{E_n} \|_\varphi \ge \delta > 0.
		$$
		By Lemma \ref{lem: equi norm}, we find $v_n \in V_{\varphi^*}(\mu)$ with
		$$
		\vvvert v_n \vvvert_\varphi^* \le 1,
		\quad \lim_n \int_{E_n} \langle v_n, u \rangle \, d \mu \ge \frac{\delta}{2}.
		$$
		These integrals being finite, it is not restrictive to assume that $\mu$ is $\sigma$-finite by restricting it to a set outside which the sequence of integrands $\langle v_n, u \rangle$ vanishes. Separability of $L_\varphi(\mu)$ yields a weak* convergent subsequence (not relabeled) of $v_n$ that is weak* equi-integrable on $L_\varphi(\mu)$ since by Lemma \ref{lem: closed subspaces} and Theorem \ref{thm: A = V} the space $V_{\varphi^*}(\mu)$ is sequentially weak* closed and by Theorem \ref{thm: V compact} the sequence $v_n$ is weak* equi-integrable in the sense of Definition \ref{def: weak eq-int}, hence we arrive at the contradiction
		\begin{equation*}
			0 < \frac{\delta}{2} \le \lim_n \int_{E_n} \langle v_n, u \rangle \, d \mu = 0.
		\end{equation*}
		The addendum follows by the first part and restriction of the measure.
	\end{proof}
	
	\paragraph{Remark.} Theorem \ref{thm: C sep} and Theorem \ref{thm: sep implies L = C} resemble the results in \cite[§1]{BeSh} on the separability of scalar-valued Banach function spaces.
	
	To conclude our section on separability, we characterize under what conditions $C_\varphi(\mu)$ has the Asplund property. A Banach space is called an Asplund space if each of its separable subspaces has a separable dual. This is equivalent to $X^*$ having the Radon-Nikodym property, which is relevant in the duality theory §\ref{sec: duality}.
	
	\begin{theorem} \label{thm: C Asplund}
		Let $\varphi$ be real-valued. If $X$ is an Asplund space and $I^*_\varphi$ is real-valued, then $C^\sigma_\varphi(\mu)$ is an Asplund space. Conversely, if $C^\sigma_\varphi(\mu)$ is Asplund and $\mu$ has no atom of infinite measure, then $X$ is Asplund and $I^*_\varphi$ is real-valued.
	\end{theorem}
	
	\begin{proof}
		Regarding the first claim, it suffices to prove that any separable subspace of $C^\sigma_\varphi(\mu)$ has a superspace with separable dual as then the subspace itself will have a separable dual via restriction of the density set. Here, the Hahn-Banach extension theorem enters. Therefore, it suffices if any sequence $u_n \in C^\sigma_\varphi(\mu)$ is contained in a superspace with separable dual. We may assume $\mu$ to be complete. Pick $W \in \SS(X)$ such that the sequence $u_n$ is almost $W$-valued and take $\Sigma \in \AA_\sigma$ a set off which all $u_n$ vanish. We denote by $\phi$ the restriction of $\varphi$ to $\Sigma \times W$. As $\phi$ is a normal integrand hence strongly measurable in the Wijsman topology by the Hess theorem, the $\sigma$-algebra $\AA'$ generated by $u_n$ and $\phi$ on $\Sigma$ is separable. Consequently, the restriction $\nu$ of $\mu$ to $\AA'$ is separable if we arrange $\nu$ to be $\sigma$-finite, which is possible without restriction by enlarging $\AA'$ by a sequence $F_n \in \AA$ such that $\Sigma = \bigcup F_n$. There holds $C_\phi(\nu)^* = L_{\phi^*}(\nu)$ by Corollary \ref{cor: C*} since $X$ is an Asplund space iff $X^*$ has the Radon-Nikodym property. Therefore, we are finished if we prove that $L_{\varphi^*}(\nu)$ is separable. Denoting by $E \colon C_\phi(\nu) \to C_\varphi(\mu)$ the identical embedding, we have $I^*_\phi = E^* I^*_\varphi$ by \cite[§3.4, Thm. 3]{IT}, hence $I^*_\phi = I_{\phi^*}$ is real-valued so that the dual space $L_{\phi^*}(\nu)$ equals $C^\sigma_{\phi^*}(\nu)$ by Theorem \ref{thm: C max subspace}. In conclusion, this space is separable by Theorem \ref{thm: C sep} as $W^*$ and $\nu$ are.
		
		For the converse, since, for any sequence $u_n \in C^\sigma_\varphi(\mu)$, the measure $\nu$ defined as above is separable, the space $C_\phi(\nu)$ is separable by Theorem \ref{thm: C sep}. Here, it enters that $\mu$ has no atom of infinite measure. Thus, $C_\phi(\mu)^* = L_{\phi^*}(\nu)$ is separable by the Asplund property. Consequently, Theorem \ref{thm: C sep} yields that $W^*$ is separable, hence $X$ is an Asplund space.
		
		Suppose now that, for $v \in C_\varphi(\mu)^* = L_{\varphi^*}(\mu)$, the sequence $u_n$ is such that
		$$
		I^*_\varphi(v) = \lim_n \langle v, u_n \rangle - I_\varphi(u_n).
		$$
		Since the separable dual space $L_{\phi^*}(\nu)$ coincides with $C_{\phi^*}(\nu)$ by Theorem \ref{thm: sep implies L = C}, there holds $I^*_\varphi(v) = I^*_\phi(v) < \i$ by Theorem \ref{thm: C max subspace}, hence $I^*_\varphi$ is real-valued.
	\end{proof}
	
	\section{Duality theory} \label{sec: duality}
	
	In this section, we obtain an abstract direct sum decomposition of $L_\varphi(\mu)^*$ into three fundamentally different types of functionals: Absolutely continuous, diffuse and purely finitely additive. We represent the absolutely continuous component, which turns out to agree with both $C_\varphi(\mu)^*$ and the function component of $L_\varphi(\mu)^*$. We then characterize the reflexivity of $L_\varphi(\mu)$ and represent the convex conjugate and the subdifferential of a general integral functional on $L_\varphi(\mu)$.
	
	\subsection{Types of functionals} \label{ssec: functionals}
	
	We denote by $\ba\left( \AA \right)$ the linear space of bounded, finitely additive, real set functions on $\AA$ equipped with the total variation norm
	$$
	\| \nu \| = \left| \nu \right|(\Om) = \sup \left\{ \sum_{i = 1}^n \left| \nu(A_i) \right| \st A_i \in \AA \text{ measurable parition of } \Om \right\}.
	$$
	For $\nu \in \ba\left( \AA \right)$, consider the positive part
	$$
	\nu^+ \colon \Sigma \to \R \colon A \mapsto \sup_{B \subset A} \nu(B).
	$$
	The negative part is defined as $\nu^- = \left( - \nu \right)^+$. Then $\nu^+$ and $\nu^-$ belong to $\ba \left( \AA \right)$ and there hold the relations
	\begin{equation} \label{eq: Jordan decomp}
		\nu = \nu^+ - \nu^-; \quad \left| \nu \right| = \nu^+ + \nu^-; \quad \| \nu \| = \| \nu^+ \| + \| \nu^- \|.
	\end{equation}
	Cf. \cite[Thm.III.1.8]{DSch}. One has the following refinement of the classical Hewitt-Yosida theorem due to Giner:
	
	\begin{theorem} \label{thm: giner 1.3.5}
		The space $\ba\left( \AA \right)$ is a direct topological sum of its linear subspaces $\Sigma(\AA)$ and $F(\AA)$ consisting of the $\sigma$-additive and the purely finitely additive elements, respectively. The projectors onto $\Sigma(\AA)$ and $F(\AA)$ are monotone, i.e.
		\begin{equation} \label{eq: monotonicity}
			\nu = \nu_\sigma + \nu_f \ge 0 \implies \nu_\sigma \ge 0; \quad \nu_f \ge 0.
		\end{equation}
		Furthermore, there holds
		\begin{equation} \label{eq: norm decomp}
			\| \nu \| = \| \nu_\sigma \| + \| \nu_f \| \quad \forall \nu \in \ba\left( \AA \right).
		\end{equation}
		Finally, setting $\nu_B = \nu \left( \cdot \cap B \right)$ for $\nu \in \ba\left( \AA \right)$ and $B \in \AA$, there holds $\nu_B \in \ba\left( \AA \right)$ and
		\begin{equation} \label{eq: comm}
			\left( \nu_a \right)_B = \left( \nu_B \right)_a; \quad \left( \nu_f \right)_B = \left( \nu_B \right)_f.
		\end{equation}
	\end{theorem}
	
	\begin{proof}
		The first part of the theorem up to (\ref{eq: monotonicity}) is classical, cf. \cite[Thm. 1.24]{YoHe}. The rest is due to \cite[Cor. A1.4]{Gi1}. We repeat his argument for the sake of completeness since the source is hard to obtain. (\ref{eq: norm decomp}): We start by showing $\nu^+_\sigma \le \left( \nu^+ \right)_\sigma$. There holds $\nu_\sigma = \left( \nu^+ \right)_\sigma - \left( \nu^- \right)_\sigma$ since the projector onto $\Sigma(\AA)$ is linear.
		$$
		\nu^+(A) = \sup_{B \subset A} \nu_\sigma(B) \le \sup_{B \subset A} \left( \nu^+ \right)_\sigma(B) = \left( \nu^+ \right)_\sigma(A)
		$$
		by (\ref{eq: Jordan decomp}). Next, we check that $\nu^+_\sigma \le \left( \nu^+ \right)_\sigma$.
		$$
		\nu_\sigma = \left( \nu^+ \right)_\sigma - \left( \nu^- \right)_\sigma = \nu^+_\sigma - \nu^-_\sigma \implies \left( \nu^+ \right)_\sigma - \nu^+_\sigma = \left( \nu^- \right)_\sigma - \nu^-_\sigma \ge 0.
		$$
		In the same way we obtain $\nu^+_f \le \left( \nu^+ \right)_f$ and $\nu^-_f \le \left( \nu^- \right)_f$. Now, we prove the announced identity of norms.
		\begin{align*}
			\| \nu \| \le
			\| \nu_\sigma \| + \| \nu_f \|
			& = \| \nu^+_\sigma \| + \| \nu^-_\sigma \| + \| \nu^+_f \| + \| \nu^-_f \| \\
			& \le \| \left( \nu^+ \right)_\sigma \| + \| \left( \nu^- \right)_\sigma \| + \| \left( \nu^+ \right)_f \| + \| \left( \nu^- \right)_f \| \\
			& = \left( \nu^+ \right)_\sigma \left( \Om \right) + \left( \nu^+ \right)_f \left( \Om \right) + \left( \nu^- \right)_\sigma \left( \Om \right) + \left( \nu^- \right)_f \left( \Om \right) \\
			& = \nu^+ \left( \Om \right) + \nu^- \left( \Om \right) = \| \nu^+ \| + \| \nu^- \| = \| \nu \|.
		\end{align*}
		(\ref{eq: comm}): Since $\left( \nu_A \right)^+ = \left( \nu^+ \right)_A$ and $\left( \nu_A \right)^- = \left( \nu^- \right)_A$ for $A \in \AA$, we may assume $\nu \ge 0$. Hence,
		$$
		0 \le \left( \nu_\sigma \right)_A = \nu_\sigma \left( \cdot \cap A \right) \le \nu_\sigma \in \Sigma(\AA)
		$$
		so that $\left( \nu_\sigma \right)_A \in \Sigma(\AA)$ and likewise we obtain $\left( \nu_f \right)_A \in F(\AA)$ by definition of a purely finitely additive measure \cite[Def. 1.13]{YoHe}. Since $\nu_A = \left( \nu_\sigma \right)_A + \left( \nu_f \right)_A$, one deduces from the uniqueness of the decomposition the claimed result.
	\end{proof}
	
	Let $\nu \colon \AA \to \left[ 0, \i \right]$ be a measure. By a result attributed to E. De Giorgi \cite[Thm. 1.114]{FoLe}, we can decompose $\nu$ into the sum of three measures
	\begin{equation} \label{eq: decom De Giorgi}
		\nu = \nu_a + \nu_d + \nu_s
	\end{equation}
	with $\nu_a \ll \mu$ and $\nu_d$ diffuse with respect to $\mu$. Moreover, if $\nu$ is $\sigma$-finite, then these three measures are mutually singular and $\nu_s \perp \mu$. Cf. \cite{FoLe} for the terminology. The decomposition (\ref{eq: decom De Giorgi}) is constructed explicitly in \cite{FoLe}:
	\begin{equation} \label{eq: ac def}
		\nu_a(A) = \sup \left\{ \int_A u \, d \mu \st u \colon \Om \to \left[ 0, \i \right] \text{ with } \int_E u \, d \mu \le \nu(E) \text{ if } E \subset A \right\},
	\end{equation}
	\begin{equation} \label{eq: d def}
		\nu_d(A) = \sup \left\{ \nu(E) \st E \subset A \colon E' \subset E \land \nu(E') > 0 \implies \mu(E') = \i \right\},
	\end{equation}
	and
	\begin{equation} \label{eq: s def}
		\nu_s(A) = \sup \left\{ \nu(E) \st E \subset A \text{ with } \mu(E) = 0 \right\}.
	\end{equation}
	All functions and sets in these definitions are assumed measurable. If $\nu$ is a signed measure, then $\nu^+$ and $\nu^-$ are mutually singular by \cite[Thm. 1.178]{FoLe} and we can decompose $\nu^+$ and $\nu^-$ according to (\ref{eq: decom De Giorgi}). We then define $\nu_a = \left( \nu^+ \right)_a - \left( \nu^- \right)_a$ etc. While every diffuse measure is absolutely continuous, $\nu_a$ given by (\ref{eq: ac def}) is distinguished against the diffuse part under additional assumptions:
	
	\begin{proposition} \label{pr: sgm spprt}
		Let $\nu$ be finite. Then there exists a set $\Sigma \in \AA$ that is $\sigma$-finite for $\mu$ with
		\begin{equation} \label{eq: abs cnt cnc}
			\nu_a(A) = \nu_a \left( A \cap \Sigma \right) \quad \forall A \in \AA.
		\end{equation}
	\end{proposition}
	
	\begin{proof}
		The claim follows if we prove that
		$$
		\nu_a(\Om) = \sup \left\{ \nu_a(\Om \cap F) \st F \in \AA_f \right\}.
		$$
		By \cite[Lem. 1.102]{FoLe}, there exists $u$ attaining the supremum in (\ref{eq: ac def}) for $A = \Om$. As $\nu$ is finite, $u$ is integrable hence vanishes outside a $\sigma$-finite set. Therefore, we find a sequence of sets $F_n$ with $\mu \left( F_n \right) < \i$ increasing towards $\left\{ u \ne 0 \right\}$. Consequently
		$$
		\nu_a(\Om) = \lim_n \int_{F_n} u \, d \mu \le \lim_n \nu_a \left( \Om \cap F_n \right) \le \nu_a(\Om)
		$$
		yields the claim.
	\end{proof}
	
	Whenever we say that a finite measure $\nu$ is absolutely continuous with respect to $\mu$ in the following, we mean this to include the property (\ref{eq: abs cnt cnc}). In analogy to Theorem \ref{thm: giner 1.3.5} we can decompose $\Sigma(\AA)$ with respect to $\mu$ into a direct topological sum.
	
	\begin{theorem} \label{thm: tp decom De Giorgi}
		The space $\Sigma \left( \AA \right)$ is a direct topological sum of its subspaces $A(\mu)$, $D(\mu)$ and $S(\mu)$ consisting of the absolutely continuous, the diffuse and the singular elements with respect to $\mu$, respectively. The projectors onto the subspaces are monotone, i.e.
		\begin{equation} \label{eq: mon}
			\nu = \nu_a + \nu_d + \nu_s \ge 0 \implies \nu_a \ge 0; \quad \nu_d \ge 0; \quad \nu_s \ge 0.
		\end{equation}
		Furthermore, there holds
		\begin{equation} \label{eq: nrm dcmp}
			\| \nu \| = \| \nu_a \| + \| \nu_d \| + \| \nu_s \| \quad \forall \nu \in \Sigma\left( \AA \right).
		\end{equation}
		Finally, setting $\nu_A = \nu \left( \cdot \cap A \right)$ for $\nu \in \Sigma\left( \AA \right)$ and $A \in \AA$, there holds $\nu_A \in \Sigma\left( \AA \right)$ and
		\begin{equation} \label{eq: cmmt}
			\left( \nu_a \right)_A = \left( \nu_A \right)_a; \quad \left( \nu_d \right)_A = \left( \nu_A \right)_d \quad \left( \nu_s \right)_A = \left( \nu_A \right)_s.
		\end{equation}
	\end{theorem}
	
	\begin{proof}
		Since it is trivial to check that $A(\mu)$, $D(\mu)$ and $S(\mu)$ are linear subspaces, we start by proving uniqueness of the decomposition. Suppose
		$$
		\nu = \nu^i_a + \nu^i_d + \nu^i_s, \quad i \in \left\{ 1, 2 \right\}
		$$
		with $\nu^i_a \ll \mu$ and satisfying (\ref{eq: abs cnt cnc}) but not necessarily given by (\ref{eq: ac def}). Ditto for $\nu^i_d$ and $\nu^i_s$. Suppose first $\nu^i_s = 0$ for all $i$. By our definition of absolute continuity, there exists $\Sigma \in \AA_\sigma$ with $\nu^i_a(A) = \nu^i_a(A \cap \Sigma)$ for all $A \in \AA$ and $i \in \left\{ 1, 2 \right\}$. We have $\nu^i_d(\Sigma) = 0$ due to diffusivity. Hence, the finite measure $\nu^1_a - \nu^2_a = \nu^2_d - \nu^1_d$ vanishes on $\Sigma$ and $\Om \setminus \Sigma$, hence $\nu^1_a = \nu^2_a$ and $\nu^1_d = \nu^2_d$.
		
		We come to the general case. As $\nu$ is finite, we have $\nu^i_s \perp \mu$ by \cite[Thm. 1.114]{FoLe} so that there exist $S_i \in \AA$ with $\mu \left( S_i \right) = 0$ and $\nu^i_s \left( A \cap S_i \right) = \nu^i_s(A)$ for all $A \in \AA$. Setting $S = S_1 \cup S_2$ we have $\mu(S) = 0$ thus $\nu^i_a$ and $\nu^i_d$ vanish on $S$. Consequently, the restriction of $\nu$ to $S$ agrees with the restrictions of both $\nu^i_s$ so that $\nu^1_s = \nu^2_s$. Now, the uniqueness of $\nu^i_a$ and $\nu^i_d$ follows by the first case. The mutual singularity of the components of (the mutually singular positive and negative parts of) $\nu$ yields (\ref{eq: mon}) and (\ref{eq: nrm dcmp}).
		
		$\nu_A \in \Sigma(\AA)$ is immediate. For the rest, it suffices to consider $\nu \ge 0$. We claim that
		$$
		\left( \nu_A \right)_a \le \left( \nu_a \right)_A; \quad \left( \nu_A \right)_d \le \left( \nu_d \right)_A; \quad \left( \nu_A \right)_s \le \left( \nu_s \right)_A.
		$$
		Directly from (\ref{eq: ac def}) we have that $\nu_a$ is monotone, i.e., $\nu^1 \le \nu^2 \implies \nu^1_a \le \nu^2_a$. The same is true for $\nu_d$ and $\nu_s$ by (\ref{eq: d def}) and (\ref{eq: s def}). Also, restriction of a measure is monotone. Therefore
		$$
		\left( \nu_A \right)_a \le \nu_a \implies \left( \nu_A \right)_a = \left( \left( \nu_A \right)_a \right)_A \le \left( \nu_a \right)_A.
		$$
		In the same way, we obtain $\left( \nu_A \right)_d \le \left( \nu_d \right)_A$ and $\left( \nu_A \right)_s \le \left( \nu_s \right)_A$. Finally
		$$
		\nu_A = \left( \nu_A \right)_a + \left( \nu_A \right)_d + \left( \nu_A \right)_s \le \left( \nu_a \right)_A + \left( \nu_d \right)_A + \left( \nu_s \right)_A = \nu_A
		$$
		obtains the claim.
	\end{proof}
	
	The following was first observed in \cite{Gi1}.
	
	\begin{proposition} \label{pr: giner 1.3.2}
		For all $\ell \in L_\varphi(\mu)^*$ and $u \in L_\varphi(\mu)$, the mapping
		$$
		\nu_{\ell, u} \colon \AA \to \R \colon A \mapsto \ell \left( u \chi_A \right)
		$$
		is an additive set function of bounded variation. Moreover
		$$
		\forall A \in \Sigma \quad \nu_{\ell, u \chi_A} = \nu_{\ell, u} \left( \cdot \cap A \right) = \left( \nu_{\ell, u} \right)_A.
		$$
		The mapping
		$$
		\nu \colon L_\varphi(\mu)^* \times L_\varphi(\mu) \to \ba\left( \Sigma \right) \colon \left( \ell, u \right) \to \nu_{\ell, u}
		$$
		is bilinear and continuous.
	\end{proposition}
	
	\begin{proof}
		It is obvious that $\nu_{\ell, u}$ is an additive set function.	The variation of $\nu_{\ell, u}$ is bounded since
		\begin{equation} \label{eq: bln cnt}
			\| \nu_{\ell, u} \|_\i = \sup_{A \in \Sigma} \left| \langle \ell, u \chi_A \rangle \right| \le \| \ell \|^*_\varphi \| u \|_\varphi.
		\end{equation}
		We used the equivalence of norms $\| \cdot \|_\i \le \left| \, \cdot \, \right| \left( \Om \right) \le 2 \| \cdot \|_\i$ for the total variation norm, cf. \cite{DSch}. The bilinearity holds because $\langle \cdot, \cdot \rangle$ is bilinear. The continuity follows by (\ref{eq: bln cnt}).
	\end{proof}
	
	We now generalize the abstract dual space decomposition \cite[Thm. 1.3.7]{Gi1} to an arbitrary measure, obtaining an additional diffuse component that drops out if $\mu$ is $\sigma$-finite. Like \cite[Thm. 1.3.7]{Gi1} our result would easily extends to a broader class of function spaces.
	
	\begin{definition}
		A continuous linear functional $\ell \in L_\varphi(\mu)^*$ belongs to the $\sigma$-finite functionals $\Sigma_{\varphi^*}(\mu)$ if $\nu_{\ell, u}$ belongs to $\Sigma(\AA)$ for every $u \in L_\varphi(\mu)$. The absolutely continuous functionals $A_{\varphi^*}(\mu)$, the diffuse ones $D_{\varphi^*}(\mu)$ and the purely finitely additive ones $F_{\varphi^*}(\mu)$ are defined analogously with $A(\mu)$, $D(\mu)$ and $F(\AA)$ taking the role of $\Sigma(\AA)$.
	\end{definition}
	
	\begin{theorem} \label{thm: xtnsn Giner}
		There holds
		\begin{equation} \label{eq: dual decom}
			L_\varphi(\mu)^* = A_{\varphi^*}(\mu) \oplus D_{\varphi^*}(\mu) \oplus F_{\varphi^*}(\mu).
		\end{equation}
		More explicitly, every $\ell \in L_\varphi(\mu)^*$ has a unique sum decomposition
		$$
		\ell = \ell_a + \ell_d + \ell_f
		$$
		with $\ell_a \in A_{\varphi^*}(\mu)$, $\ell_d \in D_{\varphi^*}(\mu)$ and $\ell_f \in F_{\varphi^*}(\mu)$. There holds
		\begin{equation} \label{eq: norm decom}
			\| \, \ell \, \|^*_\varphi = \| \, \ell_a \, \|^*_\varphi + \| \, \ell_d \, \|^*_\varphi +	\| \, \ell_f \, \|^*_\varphi.
		\end{equation}
	\end{theorem}
	
	\begin{proof}
		For uniqueness and existence, we adapt the argument in \cite[Thm. 1.3.7]{Gi1} to our setting. Uniqueness follows since $\ell$ has at most one sum decomposition $\ell = \ell_\sigma + \ell_f$ with $\ell_\sigma \in \Sigma_{\varphi^*}(\mu)$ and $\ell_f \in F_{\varphi^*}(\mu)$ by Theorem \ref{thm: giner 1.3.5}, while $\ell_\sigma$ has at most one sum decomposition $\ell_\sigma = \ell_a + \ell_d$ with $\ell_a \in A_{\varphi^*}(\mu)$ and $\ell_d \in D_{\varphi^*}(\mu)$.
		
		Regarding existence, we first prove that
		$$
		L_\varphi(\mu)^* = \Sigma_{\varphi^*}(\mu) \oplus F_{\varphi^*}(\mu).
		$$
		Employing the notation of Theorem \ref{thm: giner 1.3.5} and Proposition \ref{pr: giner 1.3.2}, we set for $u \in L_\varphi(\mu)$
		\begin{equation} \label{eq: prts}
			\ell_\sigma(u) = \left( \nu_{\ell, u } \right)_\sigma(\Om); \quad \ell_f(u) = \left( \nu_{\ell, u } \right)_f(\Om).
		\end{equation}
		These functions belong to $L_\varphi(\mu)^*$ since the mappings
		$$
		\nu_{\ell, \cdot}; \quad \left( \nu_{\ell, \cdot} \right)_\sigma; \quad \left( \nu_{\ell, \cdot} \right)_\sigma(\Om); \quad \left( \nu_{\ell, \cdot} \right)_f; \quad \left( \nu_{\ell, \cdot} \right)_f(\Om)
		$$
		are linear and continuous from $L_\varphi(\mu)$ to their respective image spaces by Proposition \ref{pr: giner 1.3.2} and Theorem \ref{thm: giner 1.3.5}. Continuity of the projectors enters. Clearly $\ell = \ell_\sigma + \ell_f$. We have $\ell_\sigma \in \Sigma_{\varphi^*}(\mu)$ since for every $A \in \AA$ there holds
		\begin{equation} \label{eq: msr}
			\begin{aligned}
				\ell_\sigma \left( u \chi_A \right) = \left( \nu_{\ell, u \chi_A} \right)_\sigma (\Om) = \left( \left( \nu_{\ell, u} \right)_A \right)_\sigma (\Om)
				& = \left( \left( \nu_{\ell, u} \right)_\sigma \right)_A (\Om)
				\\
				& = \left( \nu_{\ell, u} \right)_\sigma (A)
			\end{aligned}
		\end{equation}
		by Proposition \ref{pr: giner 1.3.2} and Theorem \ref{thm: giner 1.3.5}. In the same way, one checks $\ell_f \in F_{\varphi^*}(\mu)$.
		
		Let us decompose $\Sigma_{\varphi^*}(\mu) = A_{\varphi^*}(\mu) \oplus D_{\varphi^*}(\mu)$ to finish existence. We set
		$$
		\ell_a(u) = \left( \nu_{\ell_\sigma, u } \right)_a (\Om); \quad \ell_d(u) = \left( \nu_{\ell_\sigma, u } \right)_d (\Om).
		$$
		Similar to (\ref{eq: prts}) these functions belong to $L_\varphi(\mu)^*$ by Proposition \ref{pr: giner 1.3.2} and Theorem \ref{thm: tp decom De Giorgi}. Note $\ell_\sigma = \ell_a + \ell_d$ since $\ell_\sigma(0) = 0$ renders the singular part of $\nu_{\ell_\sigma, u}$ trivial. As in (\ref{eq: msr}) one checks $\ell_a \in A_{\varphi^*}(\mu)$ and $\ell_d \in D_{\varphi^*}(\mu)$ by Proposition \ref{pr: giner 1.3.2} and Theorem \ref{thm: tp decom De Giorgi}.
		
		(\ref{eq: norm decom}): We start by proving that $\| \, \ell \, \|^*_\varphi = \| \, \ell_\sigma \, \|^*_\varphi + \| \, \ell_f \, \|^*_\varphi$. Let $u_n$ and $v_n$ be sequences in $L_\varphi(\mu)$ with $\| u_n \|_\varphi < 1$ and $\| v_n \| < 1$ such that $\lim_n \ell_\sigma (u_n) = \| \, \ell_\sigma \, \|^*_\varphi$ and $\lim_n \ell_f (v_n) = \| \, \ell_f \, \|^*_\varphi$. Let $\Sigma \in \AA_\sigma$ with $I_\varphi \left( u_n \chi_\Sigma \right) = I_\varphi \left( u_n \right)$ and $I_\varphi \left( v_n \chi_\Sigma \right) = I_\varphi \left( v_n \right)$ for all $n \in \N$. We find $\nu$ a finite measure that is equivalent to $\mu_\Sigma$ with $\frac{d \nu}{d \mu} = f$ a positive integrable function on $\Sigma$ vanishing on $\Om \setminus \Sigma$. By absolute continuity of the integral, there exists for $n \in \N$ a $\delta > 0$ such that
		$$
		\nu(A) < \delta \implies I_\varphi \left( v_n \chi_A \right) < 1 - I_\varphi(u_n). 
		$$
		Applying \cite[Thm. 1.19]{YoHe} to the countably additive measure $\nu + \left| \nu_{\ell_\sigma, u_n} \right| + \left| \nu_{\ell_\sigma, v_n} \right|$ and the purely finitely additive set function $\left| \nu_{\ell_f, u_n} \right| + \left| \nu_{\ell_f, v_n} \right|$,	we find $A_n \in \AA$ with $\nu \left( A_n \right) < \delta$ while
		$$
		\left| \ell_\sigma \left( u_n \chi_A \right) \right| < \frac{1}{n}; \quad \left| \ell_\sigma \left( v_n \chi_A \right) \right| < \frac{1}{n}
		$$
		and
		$$
		\ell_f \left( u_n \chi_{A_n} \right) = \ell_f(u_n); \quad \ell_f \left( v_n \chi_{A_n} \right) = \ell_f(v_n).
		$$
		In particular, we have
		\begin{equation} \label{eq: rf}
			\lim_n \ell_f \left( u_n \chi_{\Om \setminus A_n} \right) = \lim_n \ell_\sigma \left( v_n \chi_{A_n} \right) = 0.
		\end{equation}
		Hence
		$$
		I_\varphi \left( u_n \chi_{\Om \setminus A_n} + v_n \chi_{A_n} \right) = I_\varphi \left( u_n \chi_{\Om \setminus A_n} \right) + I_\varphi \left( v_n \chi_{A_n} \right) < 1
		$$
		so that $\| u_n \chi_{\Om \setminus A_n} + v_n \chi_{A_n} \|_\varphi \le 1$. This together with (\ref{eq: rf}) yields
		\begin{align*}
			\| \, \ell \, \|^*_\varphi
			& \le \| \, \ell_\sigma \, \|^*_\varphi + \| \, \ell_f \, \|^*_\varphi
			= \lim_n \ell_\sigma(u_n) + \lim_n \ell_f(v_n) \\
			& = \lim_n \ell_\sigma \left( u_n \chi_{\Om \setminus A_n} \right) + \lim_n \ell_f \left( v_n \chi_{A_n} \right) \\
			& = \lim_n \ell \left( u_n \chi_{\Om \setminus A_n} + v_n \chi_{A_n} \right) \le \| \, \ell \, \|^*_\varphi.
		\end{align*}
		
		We finish the proof of (\ref{eq: norm decom}) by showing that $\| \, \ell_\sigma \, \|^*_\varphi = \| \, \ell_a \, \|^*_\varphi + \| \, \ell_d \, \|^*_\varphi$. Let $u_n$ and $v_n$ be as above but now with $\lim_n \ell_a (u_n) = \| \, \ell_a \, \|^*_\varphi$ and $\lim_n \ell_d (v_n) = \| \, \ell_d \, \|^*_\varphi$ instead of the corresponding conditions above.
		
		As $\ell_a \in A_{\varphi^*}(\mu)$, we may arrange that all members of the sequences $\nu_{\ell_a, u_n}$ and $\nu_{\ell_a, v_n}$ vanish off $\Sigma$ by possibly enlarging the set while keeping it $\sigma$-finite by Proposition \ref{pr: sgm spprt}, i.e., $\ell_a \left( u_n \right) = \ell_a \left( u_n \chi_\Sigma \right)$ and $\ell_a \left( v_n \right) = \ell_a \left( v_n \chi_\Sigma \right)$ for all $n \in \N$. Remember that $\nu_{\ell_d, u}$ for every $u \in L_\varphi(\mu)$ vanishes on any $\sigma$-finite set by diffusivity. Consequently,
		\begin{align*}
			\| \, \ell_\sigma \, \|^*_\varphi
			& \ge \lim_n \ell_\sigma \left( u_n \chi_\Sigma + v_n \chi_{\Om \setminus \Sigma} \right)
			= \lim_n \ell_a \left( u_n \chi_\Sigma \right) + \lim_n \ell_d \left( v_n \chi_{\Om \setminus \Sigma} \right) \\
			& = \lim_n \ell_a ( u_n ) + \lim_n \ell_d ( v_n )
			= \| \, \ell_a \, \|^*_\varphi + \| \, \ell_d \, \|^*_\varphi \ge \| \, \ell_\sigma \, \|^*_\varphi. \qedhere
		\end{align*}
	\end{proof}
	
	We close the section showing that absolutely continuous functionals are rather stable classes.
	
	\begin{lemma} \label{lem: closed subspaces}
		If $L_\varphi(\mu) = L_\varphi^\sigma(\mu)$, then $A_{\varphi^*}(\mu)$ is sequentially weak* closed.
	\end{lemma}
	
	\begin{proof}
		For $\ell_n \in A_{\varphi^*}(\mu)$ with $\ell_n \weakast \ell$ and $u \in L_\varphi(\mu)$, we define a signed finite measure $\lambda_n(E) = \ell_n \left( u \chi_E \right)$ that is absolutely continuous w.r.t. $\mu$. Set $\lambda(E) = \ell \left( u \chi_E \right)$. Restricting to a relevant set $\Sigma \in \AA_\sigma$ outside of which $u$ vanishes, we find a finite measure $\nu$ as in Proposition \ref{pr: equivalent finite measure} w.r.t. which each $\lambda_n$ is absolutely continuous. As $\lambda(E) = \lim_n \lambda_n(E)$ for all $E \in \AA$, we conclude by the Vitali-Hahn-Saks theorem \cite[Thm. 2.53]{FoLe} that the sequence $\lambda_n$ hence its limit $\lambda$ is uniformly absolutely continuous w.r.t. $\nu$ and $\mu$ so that $\ell$ is absolutely continuous as a functional.
	\end{proof}
	
	\subsection{Representation results}
	
	Throughout this subsection, we assume that $\mu$ has no atom of infinite measure. We denote by $\mathcal{W} \left( \Om ; X^* \right)$ the space of weak* measurable functions $w \colon \Om \to X^*$. Let $w_1, w_2 \in \mathcal{W} \left( \Om ; X^* \right)$. We say that $w_1 = w_2$ weak* a.e. if $\langle w_1, x \rangle = \langle w_2, x \rangle$ a.e. for every $x \in X$ with the exceptional null set possibly depending on $x$. We call a mapping $v \colon \AA_{a\sigma} \to \mathcal{W} \left( \Om ; X^* \right)$ with $v_A = v_B$ weak* a.e. on $A \cap B$ for $A, B \in \AA_\sigma$ a linear weak* integrand. Since any strongly measurable function $u \colon \Om \to X$ is the pointwise limit of a sequence of simple functions, the assignment $\om \mapsto \langle v_A(\om), u(\om) \rangle$ defines a family of functions indexed by $A \in \AA_\sigma$ for which we can attempt an exhausting integration, cf. the explanation before Theorem \ref{thm: inf int}. In this sense, a linear weak* integrand induces an integral functional. We introduce the space
	$$
	\mathcal{V}_{\varphi^*}(\mu) = \left\{ v \st v \colon \Om \to X^* \text{ a linear weak* integrand and } \vvvert v \vvvert^*_\varphi < \i \right\}
	$$
	with the operator seminorm
	$$
	\vvvert v \vvvert^*_\varphi = \sup_{\| u \|_\varphi \le 1} \int \langle v(\om), u(\om) \rangle \, d \mu(\om).
	$$
	Applying the standard procedure of identifying elements whose difference belongs to the kernel of $\vvvert \cdot \vvvert^*_\varphi$, we obtain a normed space $V_{\varphi^*}$ of continuous linear functionals on $L_\varphi$. It is insightful to describe this kernel more explicitly. There holds $\vvvert v \vvvert^*_\varphi = 0$ iff $\sup_{u \in L_\varphi} \int_\Sigma \langle v_\Sigma, u \rangle \, d \mu = 0$ for all $\Sigma \in \AA$. To interchange this supremum with the integral, we want to apply Theorem \ref{thm: inf int} and need to check its assumptions. First, the space $L_\varphi$ is almost decomposable. Second, the integrand is Carathéodory hence separably measurable. Third, $\mu$ is restricted to the $\sigma$-finite set $\Sigma$ where it has no atom of infinite measure. Fourth, the value of the supremum is not $+ \i$ by assumption. Therefore,
	$$
	\int_\Sigma \esssup_{W \in \SS(X)} \sup_{x \in W} \langle v_\Sigma(\om), x \rangle \, d \mu(\om) = \sup_{u \in L_\varphi} \int \langle v_\Sigma, u \rangle \, d \mu = 0,
	$$
	hence $\esssup_{W \in \SS(X)} \sup_{x \in W} \langle v_\Sigma(\om), x \rangle = 0$ a.e. This is equivalent to (the equivalence class of) $v$ vanishing i.a.e. if $X$ is separable. In general, it suggests one should think of equivalence classes in $V_{\varphi^*}$ as linear integrands that accept strongly measurable functions as sensible arguments. We call $V_{\varphi^*}$ the space of \emph{linear integrands} on $L_\varphi$. We shall prove that the absolutely continuous functionals $A_{\varphi^*}$ agree with $V_{\varphi^*}$. One half of this inclusion is easy to obtain:
	\begin{proposition} \label{pr: iso A to V}
		Identifying $v \in V_{\varphi^*}(\mu)$ with the continuous linear functional
		$$
		L_\varphi(\mu) \to \R \colon u \mapsto \int \langle v \left( \om \right), u \left( \om \right) \rangle \, d \mu \left( \om \right)
		$$
		induces an isometric embedding
		$$
		V_{\varphi^*}(\mu) \to A_{\varphi^*}(\mu).
		$$
	\end{proposition}
	
	\begin{proof}
		The induced functional is absolutely continuous by dominated convergence and since an integrable function vanishes off a $\sigma$-finite set. The embedding is obviously isometric if $V_{\varphi^*}(\mu)$ and $A_{\varphi^*}(\mu)$ carry the operator norm.
	\end{proof}
	
	In preparation of proving the converse inclusion, we need to study Hölder and reverse Hölder inequalities to determine if a given measurable function belongs to $V_{\varphi^*}$ or $L_{\varphi^*}$.\footnote{Whenever the space $L_{\varphi^*}(\mu)$ appears, we tacitly assume $\varphi^*$ to be an Orlicz integrand.} We start by observing that each element of $L_\varphi$ induces at least one continuous linear functional on $L_{\varphi^*}$ and vice versa.
	
	\begin{lemma} \label{lem: Hölder inequality}
		For $u \in L_\varphi(\mu)$ and $v \in L_{\varphi^*}(\mu)$, there holds
		\begin{equation} \label{eq: Hölder inequality}
			\int \left| \langle v \left( \om \right), u \left( \om \right) \rangle \right| \, d \mu \left( \om \right) \le 2 \| v \|_{\varphi^*} \| u \|_\varphi.
		\end{equation}
	\end{lemma}
	
	\begin{proof}
		Fenchel-Young inequality and Lemma \ref{lem: modular-norm}.
	\end{proof}
	
	\begin{lemma} \label{lem: rev Hölder}
		Let $L \subset L_\varphi(\mu)$ be either linear and almost decomposable or such that $L \chi_A \subset L$ for all $A \in \AA$ while the closure $\cl L$ is linear and almost decomposable. Let a weak* measurable function $v \colon \Om \to X^*$ satisfy
		\begin{equation} \label{eq: rev Höld}
			\int \langle v(\om), u(\om) \rangle \, d \mu(\om) \le \| u \|_\varphi \quad \forall u \in L.
		\end{equation}
		Then $v \in V_{\varphi^*}(\mu)$ with $\vvvert v \vvvert_\varphi^* \le 1$. In particular, there then holds $\vvvert v \vvvert_{\varphi^*} \le 1$ if $I^*_\varphi = I_{\varphi^*}$ on $\lin(v)$. If moreover $v$ is strongly measurable, then $v \in L_{\varphi^*}(\mu)$.
	\end{lemma}
	
	The lemma applies if $L$ are the simple functions in $L_\varphi^\sigma(\mu)$ or if $\varphi$ is real-valued and $L$ are the simple functions in $C_\varphi^\sigma(\mu)$. These assertions follow from Lemmas \ref{lem: E a decomp}, \ref{lem: C a decomp}, an \ref{lem: C simple dense}.
	
	\begin{proof}
		We start by proving that (\ref{eq: rev Höld}) holds for $u \in \cl L$ if $L \chi_A = L$ for $A \in \AA$, thereby reducing to the case when $L$ is linear and almost decomposable. Pick by Theorem \ref{thm: L complete} a sequence $u_n \in L$ with $u_n \to u$ in $L_\varphi(\mu)$ and a.e. Let $\Om^+ = \left\{ \langle v, u \rangle > 0 \right\}$ and $\Om^+_n = \left\{ \langle v, u_n \rangle > 0 \right\}$. We have $u_n \chi_{\Om^+_n} \in L$ by assumption. The Fatou lemma yields
		$$
		\int_{\Om^+} \langle v, u \rangle \, d \mu \le \liminf_n \int_{\Om^+_n} \langle v, u_n \rangle \, d \mu \le \lim_n \| u_n \|_\varphi = \| u \|_\varphi.
		$$
		Since $\cl L$ is linear, we may argue analogously for the negative part $\langle v, u \rangle^-$ so that $\langle v, u \rangle$ is integrable and (\ref{eq: rev Höld}) holds for $u \in \cl L$. Replacing if necessary $L$ by its closure $\cl L$, we may from now on assume $L$ linear and almost decomposable.
		
		We claim that (\ref{eq: rev Höld}) holds for all $u \in L_\varphi(\mu)$. Otherwise there were $u \in L_\varphi(\mu)$ with $\langle v, u \rangle \ge 0$ a.e. and $\int \langle v, u \rangle \, d \mu = +\i$. According to Proposition \ref{pr: divergent subintegral} there is $\Sigma \in \AA_\sigma$ with $\int_\Sigma \langle v, u \rangle \, d \mu = +\i$ since we ruled out atoms with infinite measure. Choose $u_n \in L$ with $u_n \uparrow u$ on $\Sigma$ by Lemma \ref{lem: decomp ss abs cont dense} and recognize the contradiction
		$$
		+\i = \lim_n \int_\Sigma \langle v, u_n \rangle \, d \mu \le \lim_n \| u_n \chi_\Sigma \|_\varphi = \| u \chi_\Sigma \|_\varphi \le \| u \|_\varphi < +\i.
		$$
		Consequently, the integral $\int \langle v, u \rangle \, d \mu$ exists for every $u \in L_\varphi(\mu)$. In particular, there exists $\Sigma \in \AA_\sigma$ outside of which $\langle v, u \rangle$ vanishes, so that we find
		$$
		\int \langle v, u \rangle \, d \mu = \lim_n \int_\Sigma \langle v, u_n \rangle \, d \mu \le \lim_n \| u_n \chi_\Sigma \|_\varphi \le \| u \|_\varphi.
		$$	
		Therefore, $v \in V_{\varphi^*}(\mu)$ with $\vvvert v \vvvert_\varphi^* \le 1$. The addenda are obvious by definition of the dual Amemiya norm.
	\end{proof}
	
	\begin{corollary} \label{cor: Lux-Orl equi}
		Let $L \subset L_\varphi(\mu)$ be an almost decomposable linear subspace. If $\varphi$ is dualizable i.a.e. for every element of $L_{\varphi^*}(\mu)$, then
		\begin{equation} \label{eq: Lux-Orl equi}
			\vvvert v \vvvert_{\varphi^*} = \sup_{u \in B_L} \langle v, u \rangle \quad \forall v \in L_{\varphi^*}(\mu).
		\end{equation}
	\end{corollary}
	
	\begin{proof}
		By Theorem \ref{thm: conjugate A}, there holds $I^*_\varphi = I_{\varphi^*}$ on $L_{\varphi^*}(\mu)$ since we ruled out atoms of infinite measure. Therefore, the Amemiya norm $\vvvert \cdot \vvvert_{\varphi^*}$ and the dual Amemiya norm $\vvvert \cdot \vvvert^*_\varphi$ coincide there. Because $L$ is almost decomposable and the functional induced by $v$ is absolutely continuous, the right-hand side in (\ref{eq: Lux-Orl equi}) coincides with the operator norm, that is, with the dual Amemiya norm $\vvvert v \vvvert_\varphi^*$ according to Lemma \ref{lem: decomp ss abs cont dense}.
	\end{proof}
	
	Our definition of $F_{\varphi^*}(\mu)$ agrees with the so-called singular functionals of Kozek \cite{Ko1, Ko2} and Castaing/Valadier \cite[Ch. VIII, §1]{CaVa} if $\mu$ is $\sigma$-finite. We prove this to make the results in \cite{CaVa} available to us.
	
	\begin{lemma} \label{lem: F if mu sgm-f}
		Let $\ell \in L_\varphi(\mu)^*$ and $C_n \in \AA$ be a sequence with $\mu \left( \lim_n C_n \right) = 0$ such that $\ell \left( u \chi_{\Om \setminus C_n} \right) = 0$ for all $u \in L_\varphi(\mu)$. Then $\ell \in F_{\varphi^*}(\mu)$. If $\mu$ is $\sigma$-finite, the converse is true as well.
	\end{lemma}
	
	\begin{proof}
		$\implies$: Given any $u \in L_\varphi(\mu)$ we need to show that $\nu_{\ell, u} \in F(\AA)$. Let $\nu \in \Sigma(\AA)$ with $0 \le \nu \le \left| \nu_{\ell, u} \right|$. We have
		$$
		\nu^+_{\ell, u} \left( \Om \setminus C_n \right) = \sup_{B \subset \Om \setminus C_n} \nu_{\ell, u}(B) = 0.
		$$
		The same is true for the negative part $\nu^-_{\ell, u}$. Consequently,
		$$
		0 \le \nu \left( \Om \right) = \lim_n \nu \left( \Om \setminus C_n \right) = \lim_n \left| \nu_{\ell, u} \right| \left( \Om \setminus C_n \right) = 0,
		$$
		hence $\nu$ vanishes identically. We conclude $\nu_{\ell, u} \in F(\AA)$ by the arbitrariness of $\nu$ and the definition of $F(\AA)$.
		
		$\impliedby$: Let $\ell \in F_{\varphi^*}(\mu)$ and
		$$
		u_m \in L_\varphi(\mu); \quad \| u_m \|_\varphi < 1; \quad \ell \left( u_m \right) > \vvvert \ell \vvvert^*_\varphi - \frac{1}{m}.
		$$
		Let $\nu$ be a finite measure with $d \nu = f d \mu$ for a positive integrable function $f$. For $m, n \in \N$, there exists by \cite[Thm. 1.22]{YoHe} a set $E = E(m, n) \in \AA$ with
		$$
		\nu \left( E \right) <2^{- n - m}; \quad \ell \left( u_m \chi_E \right) = \ell \left( u_m \right); \quad I_\varphi \left( u_m \chi_E \right) < \frac{1}{2}.
		$$
		Setting $C = C_n = \bigcup_{m \ge 1} E(m, n)$, we argue by contradiction that $\ell \left( u \chi_{\Om \setminus C} \right) = 0$ for every $u \in L_\varphi(\mu)$. Suppose
		$$
		v \in L_\varphi(\mu); \quad \| v \|_\varphi < \frac{1}{2}; \quad \ell \left( v \chi_{\Om \setminus C} \right) = a > 0.
		$$
		Pick $m$ sufficiently large that $\frac{1}{m} < a$. Then
		$$
		I_\varphi \left( u_m \chi_{E(m, n) } + v \chi_{\Om \setminus C_n} \right) = I_\varphi \left( u_m \chi_E \right) + I_\varphi \left( v \chi_{\Om \setminus C} \right) < \frac{1}{2} + \frac{1}{2} = 1
		$$
		hence $\| u_m \chi_E + v \chi_{\Om \setminus C} \|_\varphi \le 1$. Consequently,
		$$
		\ell \left( u_m \chi_E + v \chi_{\Om \setminus C} \right) = \ell \left( u_m \right) + \ell \left( v \chi_{\Om \setminus C} \right) > \vvvert \ell \vvvert^*_\varphi - \frac{1}{m} + a > \vvvert \ell \vvvert^*_\varphi
		$$
		yields a contradiction.
	\end{proof}
	
	We are now ready to recast the characterization in \cite{CaVa} of the absolutely continuous functionals in the dual space $L_\i \left( \mu ; X \right)^*$ to match our setting. This will be the foundation on which we build the general case by means of the almost embedding Lemma \ref{lem: a emb}.
	
	\begin{proposition} \label{pr: CaVa}
		Let $\mu$ be $\sigma$-finite and $\ell \in A_1 \left( \mu ; X^* \right)$ an absolutely continuous element of $L_\i \left( \mu ; X \right)^*$. Then there exists a weak* measurable function $v \colon \Om \to X^*$ such that
		\begin{equation} \label{eq: id A1 and V1}
			\ell(u) = \int \langle v(\om), u(\om) \rangle \, d \mu(\om) \quad \forall u \in L_\i \left( \mu; X \right).
		\end{equation}
		In particular, we have an isometric isomorphism $A_1 \left( \mu ; X^* \right) = V_1 \left( \mu ; X^* \right)$ via this identification, where $V_1 \left( \mu ; X^* \right) = V_{\varphi^*} \left( \mu \right)$ for $\varphi = I_{B_X}$, namely $\varphi^* = \| \cdot \|_{X^*}$ is the dual norm.
	\end{proposition}
	
	\begin{proof}
		Observe that, for $v \in V_1\left( \mu ; X^* \right)$, there holds
		$$
		\| v \|^*_\varphi = \sup_{\| u \|_\i \le 1} \int \langle v(\om), u(\om) \rangle \, d \mu(\om) = \int \esssup_{W \in \SS(X)} \sup_{x \in B_W} \langle v(\om), x \rangle \, d \mu(\om)
		$$
		according to Theorem \ref{thm: inf int}, as can be seen by absorbing the pointwise a.e. restriction $\| u \|_\i \le 1$ into the integrand as an indicator of the ball $B_X$. Therefore, the space $V_1\left(\mu; X^* \right)$ is isometrically isomorphic to the space $L^1_{X^*} \left[ X \right]$ defined in \cite[VIII]{CaVa} through the identification remarked below \cite[Lem. VIII.3]{CaVa}. Moreover, the definition \cite[VIII, Def. 5]{CaVa} of singular functionals agrees with our definition of $F_{\varphi^*}(\mu)$ in the current situation by Lemma \ref{lem: F if mu sgm-f} since for any family of measurable sets on a $\sigma$-finite measure space there exists a countable subfamily whose intersection returns the essential intersection of the entire family by \cite[Thm. 1.108]{FoLe}. It is then obvious that $F_{\varphi^*}(\mu)$ and the so-called singular functionals are isometrically isomorphic if both carry their operator norm. In total
		$$
		L_\i\left( \mu ; X \right)^* = V_1\left( \mu; X^* \right) \oplus F_1\left( \mu; X^* \right).
		$$
		Since
		$$
		L_\i\left( \mu ; X \right)^* = A_1\left( \mu; X^* \right) \oplus F_1\left( \mu; X^* \right)
		$$
		by Theorem \ref{thm: xtnsn Giner} and because (\ref{eq: id A1 and V1}) defines an isometric embedding of $V_1\left( \mu; X^* \right)$ into $A_1\left( \mu; X^* \right)$ by Proposition \ref{pr: iso A to V}, we find the embedding induced by (\ref{eq: id A1 and V1}) surjective.
	\end{proof}
	
	In the following, we consider functions defined on a set $A \in \AA$ as trivially extended to all of $\Om$.
	
	\begin{proposition} \label{pr: abs cont local}
		Let $\ell \in A_{\varphi^*}(\mu)$ such that, for any $F \in \AA_f$, there exists a sequence of sets $F_n \in \AA$ with $F = \bigcup_{n \ge 1} F_n$ and elements $v_n \in V_{\varphi^*}\left( F_n \right)$ such that
		\begin{equation} \label{eq: proj cond}
			\ell \left( u \chi_{F_n} \right) = \int \langle v_n, u \rangle \, d \mu \quad \forall u \in L_\varphi(\mu).
		\end{equation}
		Then there exists a unique $v \in V_{\varphi^*}(\mu)$ with
		\begin{equation} \label{eq: density repr}
			\ell \left( u \right) = \int \langle v, u \rangle \, d \mu \quad \forall u \in L_\varphi(\mu).
		\end{equation}
		Moreover, if each such $v_n$ is strongly measurable, then is $v$ is integrally strongly measurable. If in addition $\varphi$ is dualizable i.a.e. for every element of $L_{\varphi^*}(\mu)$ and the minimum of $\varphi^*$ at the origin is strict outside a $\sigma$-finite set, then $v$ is uniquely determined as an element of $L_{\varphi^*}^\sigma(\mu)$.
	\end{proposition}
	
	\begin{proof}
		Uniqueness: Such a representation is unique by Proposition \ref{pr: iso A to V}. Existence: We may assume $\mu$ to be $\sigma$-finite since if to every $\Sigma \in \AA_\sigma$ there corresponds $v_\Sigma \in V_{\varphi}(\Sigma)$ with
		$$
		\nu_{\ell, u} \left( \Sigma \right) = \int \langle v_\Sigma(\om), u(\om) \rangle \, d \mu(\om) \quad \forall u \in L_\varphi(\mu),
		$$
		then this defines a linear weak* integrand hence an equivalence class $v \in V_{\varphi^*}(\mu)$ representing $\ell$ by Proposition \ref{pr: sgm spprt}. We may even assume $\mu$ to be finite since every element of $A_{\varphi^*}(\mu)$ is $\sigma$-additive and every $\sigma$-finite set can be written as a disjoint union of sets having finite measure.
		
		We have $v_n = v_m$ weak* a.e. on $F_n \cap F_m$ by the considerations on the kernel of the operator norm on $V_{\varphi^*}(\mu)$. Therefore,
		$$
		v(\om) = v_n(\om) \text{ if } \om \in F_n \cap F_m
		$$
		defines an a.e. equivalence class of weak* measurable functions. We have
		$$
		G_n := \bigcup_{m = 1}^{n -1} F_m; \quad \ell(u) = \sum_{n = 1}^\i \ell \left( u \chi_{F_n \setminus G_n} \right) = \sum_{n = 1}^\i \int_{F_n \setminus G_n} \langle v_n, u \rangle \, d \mu = \int \langle v, u \rangle \, d \mu.
		$$
		The series converges by $\sigma$-additivity of $\ell$. Measurability: This is obvious by our construction of $v$. Addendum: For $\Sigma \in \AA_\sigma$, there exists a unique $v_\Sigma \in L_{\varphi^*}\left(\Sigma\right)$ representing $\ell$ on $\Sigma$. We have
		$$
		\sup_{\Sigma \in \AA_\sigma} \| v_\Sigma \|_{\varphi^*} \le \sup_{\Sigma \in \AA_\sigma} \vvvert v_\Sigma \vvvert^*_\varphi \le \vvvert \, \ell \, \vvvert^*_\varphi \le 1
		$$
		so there exists $\Sigma_0 \in \AA_\sigma$ with
		\begin{equation} \label{eq: ass sup}
			\int_{\Sigma_0} \varphi^* \left[\om, v_{\Sigma_0}(\om) \right] \, d \mu(\om) = \sup_{\Sigma \in \AA_\sigma} \int_\Sigma \varphi^* \left[ \om, v_\Sigma \left( \om \right) \right] \, d \mu \left( \om \right) \le 1.
		\end{equation}
		We may assume $\Sigma_0$ to contain the $\sigma$-finite set off which the minimizer of $\varphi^*$ at the origin is isolated. If $v_{\Sigma_0}$ could be extended outside of $\Sigma_0$ in a non-trivial way to still represent $\ell$, then there were $u \in L_\varphi^\sigma(\mu)$ concentrated on $\Om \setminus \Sigma_0$ with $\ell \left( u \right) = b > 0$. Since $u$ is concentrated on a $\sigma$-finite set, we can extend $v_{\Sigma_0}$ to this set, which contradicts the definition of $\Sigma_0$ as the extension would surpass the supremum in (\ref{eq: ass sup}) if the minimizer of $\varphi^*$ at the origin is isolated.
	\end{proof}
	
	\begin{theorem} \label{thm: A = V}
		Identifying $v \in V_{\varphi^*}(\mu)$ with the continuous linear functional
		\begin{equation} \label{eq: v induces functional}
			L_\varphi(\mu) \to \R \colon u \mapsto \int \langle v \left( \om \right), u \left( \om \right) \rangle \, d \mu \left( \om \right),
		\end{equation}
		induces an isometric isomorphism
		\begin{equation} \label{eq: A = V}
			A_{\varphi^*}(\mu) = V_{\varphi^*}(\mu).
		\end{equation}
		If moreover $X^*$ has the Radon-Nikodym property with respect to the restriction of $\mu$ to sets of finite measure, then elements of $V_{\varphi^*}(\mu)$ are integrally strongly measurable. If in addition $\varphi$ is dualizable i.a.e. for every $v \in L_{\varphi^*}(\mu)$, the minimum of $\varphi^*$ at the origin is strict outside a $\sigma$-finite set and $v$ is identified with the continuous linear function (\ref{eq: v induces functional}), then (\ref{eq: A = V}) induces an isomorphism
		\begin{equation} \label{eq: A = L}
			A_{\varphi^*}(\mu) = L_{\varphi^*}(\mu).
		\end{equation}
	\end{theorem}
	
	\paragraph{Remark.} The Radon-Nikodym property always holds if $\mu$ is purely atomic, cf. \cite[p. 62]{DU}.
	
	\begin{proof}
		By Proposition \ref{pr: iso A to V}, it remains to represent a given $\ell \in A_{\varphi^*}(\mu)$ by some $v \in V_{\varphi^*}(\mu)$ this way. Let $F \in \AA_f$ and consider the functional $\ell_F(u) = \ell \left( u \chi_F \right)$. If each $\ell_F$ permits a representation via $v_F \in V_{\varphi^*}\left(F\right)$ by (\ref{eq: v induces functional}), then (\ref{eq: A = V}) follows by Proposition \ref{pr: abs cont local}.
		
		Let $F_\e$ be an isotonic family with $F_\e \uparrow F$ as in Lemma \ref{lem: a emb}. Consider the mapping
		$$
		\nu_\e \colon L_\i \left( \mu ; X \right) \to \R \colon u \mapsto \nu_{\ell, u} \left( F_\e \right) = \ell \left( u \chi_{F_\e} \right).
		$$
		It is linear continuous since $\| v \chi_{F_\e} \|_\varphi \le C_\e \| v \|_\i$. Proposition \ref{pr: CaVa} yields $v_\e \in V_1\left(\mu; X^* \right)$ with
		\begin{equation} \label{eq: nu int repr}
			\nu_\e (u) = \int \langle v_\e, u \rangle \, d \mu \quad \forall u \in L_\i \left( \mu ; X \right).
		\end{equation}
		As $L_\i\left( \mu; X \right)$ is linear and almost decomposable, we may invoke Lemma \ref{lem: rev Hölder} to find that (\ref{eq: nu int repr}) defines an element $v_\e \in V_{\varphi^*}\left( F_\e \right)$. Lemma \ref{lem: decomp ss abs cont dense} together with the absolute continuity of $\ell$ then implies that the functional induced by $v_\e$ through (\ref{eq: nu int repr}) agrees with $\nu_\e$ on all of $L_\varphi(\mu)$, hence we conclude existence of a representing function $v_F \in V_{\varphi^*}\left( F \right)$ as required by Proposition \ref{pr: abs cont local}. The first claim has been proved.
		
		First addendum and (\ref{eq: A = L}): Arguing as in the first step, we may reduce the problem to the set $F \in \AA_f$ by the addendum in Proposition \ref{pr: abs cont local}. Consider the restriction of the mapping $\nu_\e$ (not relabeled)
		$$
		\nu_\e \colon \AA \times X \to \R \colon \left( A, x \right) \mapsto \ell \left( x \chi_A \chi_{F_\e} \right).
		$$
		We may regard $A \mapsto \nu_\e \left( A \right)$ as an $X^*$-valued vector measure because $\nu_\e \left( A, \cdot \right) \in X^*$ by $\| v \chi_{F_\e} \|_\varphi \le C_\e \| v \|_\i$. Since $\ell$ is absolutely continuous, the vector measure $\nu_\e$ is weak* $\sigma$-additive. Let $\Om_j \in \AA$ be a countable measurable partition of $\Om$ and pick for $\nu_\e \left( \Om_j \right)$ an $x_j \in B_X$ with $\| \nu_\e\left( \Om_j \right) \|_{X^*} < 2^{-j} + \nu_\e \left( \Om_j, x_j \right)$ so that absolute continuity of $\ell$ yields the estimate
		$$
		\sum_{j \ge 1} \| \nu_\e\left( \Om_j \right) \|_{X^*} \le 1 + \ell \left( \sum_{j \ge 1} x_j \chi_{\Om_j} \chi_{F_\e} \right)
		\le 1 + C_\e \vvvert \ell \vvvert^*_\varphi < \i.
		$$
		Consequently, $\nu_\e$ is $\sigma$-additive in norm convergence and its total variation
		$$
		\| \nu_\e \| \left( F \right) := \sup \left\{ \sum_{i = 1}^n \| \nu_\e \left( A_i \right) \|_{X^*} \st A_i \in \AA \text{ a finite partition of } F \right\}
		$$
		is finite. Applying the Radon-Nikodym theorem, we deduce existence of a density $v_\e \in L_1 \left( \mu ; X^* \right)$ with $\nu_\e \left( A \right) = \int_A v_\e \, d \mu$ for all $A \in \AA$. We claim that
		$$
		\ell \left( u \chi_{F_\e} \right) = \int_{F_\e} \langle v_\e, u \rangle \, d \mu \quad \forall u \in L_\varphi(\mu).
		$$
		Since this identity holds if $u$ is simple, Lemma \ref{lem: rev Hölder} and the remark below it imply $\vvvert v_\e \vvvert^*_\varphi \le 1$, hence, if $\varphi$ is dualizable, $v_\e \in L_{\varphi^*}(\mu)$ with $\vvvert v_\e \vvvert_{\varphi^*} \le 1$. Lemma \ref{lem: decomp ss abs cont dense} and the absolute continuity of $\ell$ then imply that $v_\e$ induces an integral representation on $F_\e$ for all $u \in L_\varphi(\mu)$. The claim follows by Proposition \ref{pr: abs cont local} since $F_\e \uparrow F$.
	\end{proof}
	
	\begin{corollary} \label{cor: C*}
		Let $\varphi$ be real-valued. Then
		\begin{equation} \label{eq: C*}
			C_\varphi(\mu)^* = V_{\varphi^*}(\mu)
		\end{equation}
		via the isometric isomorphism identifying $v \in V_{\varphi^*}(\mu)$ with the functional
		\begin{equation} \label{eq: C* 2}
			\ell \left( u \right) = \int \langle v \left( \om \right), u \left( \om \right) \rangle \, d \mu \left( \om \right), \quad u \in C_\varphi(\mu).
		\end{equation}
	\end{corollary}
	
	\begin{proof}
		We claim that
		\begin{equation} \label{eq: annihilator}
			C_\varphi(\mu)^\bot = D_{\varphi^*}(\mu) \oplus F_{\varphi^*}(\mu).
		\end{equation}
		As $\mu$ has no atom of infinite measure, there holds $C_\varphi(\mu) = C^\sigma_\varphi(\mu)$ by Corollary \ref{cor: C sigma fin} so that $D_{\varphi^*}(\mu)$ is contained in the annihilator. Fix $u \in C^\sigma_\varphi(\mu)$ and pick $\Sigma \in \AA_\sigma$ off which $u$ vanishes. Let $\nu$ be a finite measure defined by $d \nu = f d \mu$ for a positive integrable function $f$ on $\Sigma$ and $\nu \left( \Om \setminus \Sigma \right) = 0$. Fix $\ell_f \in F_{\varphi^*}(\mu)$. Then there exists a sequence $A_n \in \AA$ with $\nu \left( A_n \right) < \frac{1}{n}$ and $\nu_{\ell_f, u} \left( \Om \setminus A_n \right) = 0$ by \cite[Thm. 1.19]{YoHe}. As $u$ has absolutely continuous norm, there holds $\lim_n \| u - u \chi_{\Om \setminus A_n} \|_\varphi = 0$, hence $\ell_f(u) = \lim_n \ell_f \left( u \chi_{\Om \setminus A_n} \right) = 0$, whence $\ell_f$ is contained in the annihilator.
		
		By Theorem \ref{thm: xtnsn Giner}, it remains to prove that no non-trivial element of $A_{\varphi^*}(\mu)$ vanishes on all of $C_\varphi(\mu)$. Let $\ell_a \in A_{\varphi^*}(\mu)$ and $u \in L_\varphi(\mu)$ with $\ell_a(u) > 0$. We may assume $u \in L^\sigma_\varphi(\mu)$ by Proposition \ref{pr: sgm spprt}. Invoking Lemma \ref{lem: decomp ss abs cont dense} together with the almost decomposability of $C_\varphi(\mu)$ by Lemma \ref{lem: C a decomp}m we find a sequence $u_n \in C_\varphi(\mu)$ converging to $u$ from below, hence $\ell_a \left( u_n \right) > 0$ eventually by the $\sigma$-additivity of $\ell_a$.
		
		Having computed (\ref{eq: annihilator}), we now use that, for a Banach space $X$ and a closed subspace $U$, there holds $U^* = X^* / U^\bot$ through the isometric isomorphism induced by $x' + U^\bot \mapsto \left. x' \right|_U$. In the situation at hand, this implies by Theorems \ref{thm: xtnsn Giner} and \ref{thm: A = V} that
		\begin{align*}
			C_\varphi(\mu)^*
			& = \left[ V_{\varphi^*}(\mu) \oplus D_{\varphi^*}(\mu) \oplus F_{\varphi^*}(\mu) \right] / C_\varphi(\mu)^\bot \\
			& = \left[ V_{\varphi^*}(\mu) \oplus D_{\varphi^*}(\mu) \oplus F_{\varphi^*}(\mu) \right] / \left[ D_{\varphi^*}(\mu) \oplus F_{\varphi^*}(\mu) \right] 
			= V_{\varphi^*}(\mu)
		\end{align*}
		where the action of a functional is described by (\ref{eq: C* 2}). The norm of this quotient space is the operator norm by the decomposition (\ref{eq: norm decom}) so that the isomorphism induced by (\ref{eq: C* 2}) indeed is isometric.
	\end{proof}
	
	\begin{corollary} \label{cor: RNP necessary}
		If all elements of $V_{\varphi^*}(\mu)$ are integrally strongly measurable, then $X^*$ has the Radon-Nikodym property w.r.t. the restriction of $\mu$ to any set of finite measure. In particular, this is necessary for $V_{\varphi^*}(\mu) = L_{\varphi^*}(\mu)$ to hold.
	\end{corollary}
	
	\begin{proof}
		Arguing by contradiction, we suppose there were $F \in \AA_f$ on which the restriction of $\mu$ fails the Radon-Nikodym property. Theorems \ref{thm: xtnsn Giner} and \ref{thm: A = V} yield $L_1\left( F; X \right)^* = V_\i\left( F; X^* \right)$ and by \cite[§4.1, Thm. 1]{DU} we know then that there exists
		$$
		v \in V_\i\left( F; X^* \right) \setminus L_\i\left( F; X^* \right).
		$$
		Lemma \ref{lem: a emb} yields an isotonic exhausting sequence $F_n$ with $\lim_n \mu\left( F \setminus F_n \right) = 0$ such that
		$$
		L_\i\left(F_n ; X \right) \to L_\varphi\left( F_n \right) \to L_1\left( F_n ; X \right).
		$$
		As the second embedding in this chain is dense, its adjoint operator is an embedding, too, so that $v \chi_{F_n} \in V_{\varphi^*}(\mu)$. If each $v \chi_{F_n}$ were strongly measurable, then its limit from below $v$ were likewise, which would yield $v \in L_\i\left( F; X^* \right)$ because of Lemma \ref{lem: rev Hölder} since Orlicz functions are dualizable by Lemma \ref{lem: dualizable suff cond}. We have arrived at a contradiction.
	\end{proof}
	
	As a consequence of our duality theory, we obtain a characterization of reflexivity for the Orlicz space $L_\varphi(\mu)$.
	
	\begin{theorem} \label{thm: rflxv}
		Let the range space $X$ be reflexive and let the conjugate Orlicz integrands $\varphi$ and $\varphi^*$ be real-valued, dualizable i.a.e. for every element of $L_{\varphi^*}(\mu)$ and $L_\varphi(\mu)$, respectively.	Then $L_\varphi(\mu)$ is reflexive if and only if
		\begin{equation} \label{eq: reflexivity}
			L_\varphi(\mu) = C_\varphi^\sigma(\mu) \text{ and }
			L_{\varphi^*}(\mu) = C_{\varphi^*}^\sigma(\mu).
		\end{equation}
		Hence, if $\varphi \in \Delta_2$ and $\varphi^* \in \Delta_2$, then $L_\varphi(\mu)$ is reflexive. If $\mu$ is non-atomic, then the $\Delta_2$-conditions are also necessary for $L_\varphi(\mu)$ to be reflexive.
	\end{theorem}
	
	\paragraph{Remark.} By the almost embedding result Lemma \ref{lem: a emb}, we know that $L_\varphi(\mu)$ contains a copy of the range space $X$, whence reflexivity of $X$ is clearly a necessary assumption unless in the trivial case when no set of positive measure exists, which we ruled out in our remark on notation. This contrasts with \cite{Tu}, where a reflexive Orlicz space with an Orlicz function on a non-reflexive range space is presented. The catch is that the author of \cite{Tu} does not require continuity of the Orlicz function at the origin, which makes such a pathology possible.
	
	\begin{proof}
		Remember $C_\varphi(\mu) = C_\varphi^\sigma(\mu)$ and $C_{\varphi^*}(\mu) = C_{\varphi^*}^\sigma(\mu)$ by Corollary \ref{cor: C sigma fin} since we assume $\mu$ to have no atom of infinite measure. $\implies$: If $L_\varphi(\mu)$ is reflexive, then $C_\varphi(\mu)$ is as a closed subspace by Lemma \ref{lem: C closed subspace}. Therefore,
		\begin{equation} \label{eq: est spaces}
			L_\varphi(\mu) \supset C_\varphi(\mu) = C_\varphi(\mu)^{**} = V_{\varphi^*}(\mu)^* \supset L_\varphi(\mu) \implies L_\varphi(\mu) = C_\varphi^\sigma(\mu).
		\end{equation}
		More precisely, the dual space of $C_\varphi(\mu)$ is $V_{\varphi^*}(\mu)$ by means of the standard integral pairing, while the dual of $V_{\varphi^*}(\mu)$ contains $L_\varphi(\mu)$ with the functionals $L_\varphi(\mu)$ acting again through the integral pairing. Since the canonical embedding of $C_\varphi(\mu)$ into the bidual $V_{\varphi^*}(\mu)$ via the integral pairing is surjective by reflexivity, we deduce the last inclusion in (\ref{eq: est spaces}) and consequently the final claim. By our assumptions, the space $L_{\varphi^*}(\mu)$ is a closed subspace of $C_{\varphi^*}(\mu)^*$ due to Corollary \ref{cor: Lux-Orl equi}, hence it is reflexive if $L_\varphi(\mu)$ is. Consequently, the argument for $L_{\varphi^*}(\mu) = C_{\varphi^*}^\sigma(\mu)$ is the same as for the first identity.
		
		$\impliedby$: A Banach space is reflexive iff its unit ball is (sequentially) weakly compact. As our assumption $L_\varphi(\mu) = C_\varphi^\sigma(\mu)$ implies that any given sequence in $L_\varphi(\mu)$ vanishes off a $\sigma$-finite set, we may assume $\mu$ to be $\sigma$-finite. Now, the space $L_\varphi(\mu)$ is reflexive as
		$$
		L_\varphi(\mu)^{**} = C_\varphi(\mu)^{**} = L_{\varphi^*}(\mu)^* = C_{\varphi^*}(\mu)^* = L_\varphi(\mu)
		$$
		by Theorem \ref{thm: A = V}. More precisely, the space $C_\varphi(\mu)$ has $L_{\varphi^*}(\mu) = C_{\varphi^*}(\mu)$ as a dual space via the integral pairing, while this dual space has $L_\varphi(\mu) = C_\varphi(\mu)$ as a bidual via the same pairing so that the canonical embedding of $C_\varphi(\mu)$ into its bidual is surjective, i.e., reflexivity. Theorem \ref{thm: linear domain 2} settles the addendum on the $\Delta_2$-conditions.
	\end{proof}
	
	\begin{corollary}
		Let $X$ be reflexive while $\varphi$ and $\varphi^*$ are real-valued. Then $L_\varphi(\mu)$ is reflexive if and only if $\dom I_\varphi = L_\varphi(\mu)$ and $\dom I_{\varphi^*} = L_{\varphi^*}(\mu)$.
	\end{corollary}
	
	\begin{proof}
		$\implies$: If $X$ is separable, then $\varphi$ and $\varphi^*$ are dualizable so this follows by Theorem \ref{thm: rflxv} and Corollary \ref{cor: linear domain}. Hence, we also have $\dom I_\varphi = L_\varphi(\mu)$ if $X$ is not separable since every given $u \in L_\varphi(\mu)$ has almost separable range. Let $\ell \in L_\varphi(\mu)^*$ and pick $u_n \in L_\varphi(\mu)$ with $I^*_\varphi(\ell) = \lim_n \langle \ell, u_n \rangle - I_\varphi(u_n)$. Let $W \in \SS(X)$ such that each member of the sequence $u_n$ is almost $W$-valued and set $\phi = \varphi_W$. We denote the restriction of $\ell$ to $L_\phi(\mu)$ again by $\ell$. Then $I^*_\varphi(\ell) = I^*_\phi(\ell) < \i$ by Theorem \ref{thm: rflxv}. $\impliedby$: By Corollary \ref{cor: linear domain}.
	\end{proof}
	
	Our final application of the duality theory obtained so far is a result for the convex conjugate of integral functionals on a vector valued Orlicz space. Remember that the exhausting integral of an essential infimum function always exists, even if the infimum function does not exist globally.
	
	\begin{theorem} \label{thm: conjugate B}
		Let $f \colon \Om \times X \to \left( - \i, \i \right]$ be an integrally separably measurable integrand.
		Then, if
		$$
		I_f \not \equiv \i \text{ on } L_\varphi \text{, where } I_f \left( u \right) = \int f \left[ \om, u \left( \om \right) \right] \, d \mu \left( \om \right) \text{ and } D = \dom I_f,
		$$
		the convex conjugate $I^*_f$ of $I_f$ with respect to the norm topology is given by
		\begin{equation} \label{eq: convex conjugate}
			\begin{gathered}
				I^*_f(\ell)
				= I^*_f \left( \ell_a \right) + s_D \left( \ell_d \right) + s_D \left( \ell_f \right)
				\\
				= \int \esssup_{W \in \SS \left( X \right) } \sup_{x \in W} \langle \ell_a(\om), x \rangle - f_\om(x) \, d \mu \left( \om \right)
				+ \sup_{u \in D} \ell_d \left( u \right) + \sup_{u \in D} \ell_f \left( u \right)
			\end{gathered}
		\end{equation}
		wherever $I^*_f$ is finite.	Let $\SS_u(X)$ be the separable subspaces of $X$ almost containing the range of $u$. The Fenchel-Moreau subdifferential	of $I_f$ on $D$ is given by
		\begin{equation} \label{eq: subdifferential representation}
			\begin{gathered}
				\p I_f \left( u \right) = \p_a I_f \left( u \right) + \p_d I_f \left( u \right) + \p_f I_f \left( u \right) \\
				= \bigcap_{W \in \SS_u(X) } \left\{ u^* \in V_{\varphi^*} \st u^*_W \left( \om \right) \in \p f_W \left[ \om, u \left( \om \right) \right] \text{ i.a.e.} \right\} \\
				+ \left\{ \ell_d \in D_{\varphi^*} \st \langle \ell_d , v - u \rangle \le 0 \quad \forall v \in D \right\}. \\
				+ \left\{ \ell_f \in F_{\varphi^*} \st \langle \ell_f , v - u \rangle \le 0 \quad \forall v \in D \right\}.
			\end{gathered}
		\end{equation}
		Moreover, if $f$ is a convex integrand, denoting by $p \left( \om, v \right) = f' \left( \om, u \left( \om \right) ; v \right)$ its radial derivative, the closure of the radial derivative $v \mapsto I' \left( u ; v \right)$ at a point $u \in \dom \left( \p I_f \right)$ is given by
		\begin{equation} \label{eq: derivative closure}
			\textnormal{cl}_v I'_f \left( u ; v \right) = \int \essinf_{W \in \SS_v(X) } \textnormal{cl}_v p_W \left( \om, v \left( \om \right) \right) \, d \mu \left( \om \right) + s_{\p_d I \left( u \right)}(v) + s_{\p_f I \left( u \right)}(v).
		\end{equation}
		Finally, if $f$ is dualizable i.a.e. for $\ell_a \in V_{\varphi^*}(\mu)$, then $I^*_f\left(\ell_a \right) = I_{f^*}\left(\ell_a \right)$ with $I_{f^*}\left(\ell_a\right) = \int f^*\left[\om, \ell_a(\om) \right] \, d \bar{\mu}(\om)$ and the intersection in (\ref{eq: subdifferential representation}) over $W \in \SS_u(X)$ is to be replaced by $W = X$.
	\end{theorem}
	
	\begin{proof}
		(\ref{eq: convex conjugate}): Clearly $I^*_f \left( \ell \right) \le I^*_f \left( \ell_a \right) + s_D \left( \ell_d \right) + s_D \left( \ell_f \right)$ so that it remains to prove the converse inequality. Let $u_i \in D$ for $1 \le i \le 3$. By Proposition \ref{pr: sgm spprt}, we find $\Sigma \in \AA_\sigma$ with
		$$
		\left| \nu_{\ell_a, u_i} \right| \left( \Om \setminus \Sigma \right) = 0; \quad f \left( \om, u_i(\om) \right) = 0 \text{ on } \Om \setminus \Sigma \quad \forall i.
		$$
		Let $\nu$ be a finite measure defined by $d \nu = f d \mu$ for $f$ a positive integrable function on $\Sigma$ and $\nu \left( \Om \setminus \Sigma \right) = 0$. There exists for $n \in \N$ a set $A_n \in \AA$ such that
		$$
		\nu \left( A_n \right) < \frac{1}{n}; \quad \left| \nu_{\ell_\sigma, u_i} \right| \left( A_n \right) < \frac{1}{n}; \quad \left| \nu_{\ell_f, u_i} \right| \left( \Om \setminus A_n \right) = 0 \quad \forall i
		$$
		by \cite[Thm. 1.19]{YoHe}. Setting $\bar{u}_n = u_1 \chi_\Sigma \chi_{\Om \setminus A_n} + u_2 \chi_{\Om \setminus \Sigma} \chi_{\Om \setminus A_n} + u_3 \chi_{A_n}$, we have $\bar{u}_n \in D$ so that
		\begin{align*}
			I^*_f \left( \ell \right)
			& \ge \ell \left( \bar{u}_n \right) - \int f \left[ \om, \bar{u}_n \left( \om \right) \right] \, d \mu \left( \om \right) \\
			& = \ell_a \left( u_1 \chi_{\Om \setminus A_n} \right) + \ell_d \left( u_2 \chi_{\Om \setminus A_n} \right) + \ell_\sigma \left( u_3 \chi_{A_n} \right) + \ell_f \left( u_3 \chi_{A_n} \right) \\
			& - \int_{\Om \setminus A_n} f \left[ \om, u_1 \left( \om \right) \right] \, d \mu \left( \om \right) - \int_{A_n} f \left[ \om, u_3 \left( \om \right) \right] \, d \mu \left( \om \right) \\
			& \ge \ell_a \left( u_1 \chi_{\Om \setminus A_n} \right) + \ell_d \left( u_2 \chi_{\Om \setminus A_n} \right) - \frac{1}{n} + \ell_f \left( u_3 \right) \\
			& - \int_{\Om \setminus A_n} f \left[ \om, u_1 \left( \om \right) \right] \, d \mu \left( \om \right) - \int_{A_n} f \left[ \om, u_3 \left( \om \right) \right] \, d \mu \left( \om \right) \\
			& \to \ell_a \left( u_1 \right) + \ell_d \left( u_2 \right) + \ell_f \left( u_3 \right) - \int_\Om f \left[ \om, u_1 \left( \om \right) \right] \, d \mu \left( \om \right) \text{ as } n \to \i.
		\end{align*}
		Taking the supremum over all $u_i$ concludes the proof by Theorem \ref{thm: conjugate A}. Observe that the finiteness of $I^*_f(\ell)$ enters so that the integrand $\langle \ell_a(\om), x \rangle - f \left[ \om, x \right]$ may be restricted to a suitable $\sigma$-finite set where $\ell_a$ is well-defined a.e. as a function.
		
		(\ref{eq: subdifferential representation}): By the Fenchel-Young equality and (\ref{eq: convex conjugate}), there holds
		\begin{align*}
			& \ell \in \p I_f \left( u \right)
			\iff I_f \left( u \right) + I^*_f \left( \ell_a \right) + s_D \left( \ell_d \right) + s_D \left( \ell_f \right) = \langle \ell, u \rangle \\
			& \iff I_f \left( u \right) + I^*_f \left( \ell_a \right) = \langle \ell_a, u \rangle \land s_D \left( \ell_d \right) = \langle \ell_d, u \rangle \land s_D \left( \ell_f \right) = \langle \ell_f, u \rangle,
		\end{align*}
		which is equivalent to $\ell$ fulfilling $\ell_{a, W} \left( \om \right) \in \p f_W \left[ \om, u \left( \om \right) \right]$ a.e. while $\langle \ell_d, v - u \rangle \le 0$ and $\langle \ell_f, v - u \rangle \le 0$ for all $v \in D$. The first assessment follows by Theorem \ref{thm: conjugate A}.
		
		(\ref{eq: derivative closure}): Since $\p_d I_f(u)$ and $\p_f I_f (u)$ are non-empty cones, the functions
		$$
		v \mapsto s_{\p_d I \left( u \right)}(v) = \sup_{\ell_d \in \p_d I_f (u) } \langle \ell_f, v \rangle; \quad v \mapsto s_{\p_f I \left( u \right)}(v)
		$$
		take values in $\left\{ 0, +\i \right\}$. Let $v \in D$. Remember that the subdifferential of a convex function at a point $u$ consists of those continuous linear functionals that are dominated by the radial derivative of the function at $u$. In particular, the closure of the sublinear derivative functional is the supremum of the subgradients. One-sided difference quotients of convex functions being monotone decreasing, we have
		\begin{align*}
			I' \left( u ; v \right)
			& \ge \int p \left( \om, v \left( \om \right) \right) \, d \mu \left( \om \right) + s_{\p_d I \left( u \right)}(v) + s_{\p_f I \left( u \right)}(v) \\
			& \ge \int \essinf_{W \in S (X) } \textnormal{cl}_v p_W \left( \om, v \left( \om \right) \right) \, d \mu \left( \om \right) + s_{\p_d I \left( u \right)}(v) + s_{\p_f I \left( u \right)}(v) \\
			& \ge \int \langle \ell_a \left( \om \right), v \left( \om \right) \rangle \, d \mu \left( \om \right) + s_{\p_d I \left( u \right)}(v) + s_{\p_f I \left( u \right)}(v)
		\end{align*}
		for any $\ell_a \in \p_a I_f \left( u \right)$ by (\ref{eq: subdifferential representation}). Taking the supremum over all such $\ell_a$ yields
		\begin{align*}
			I' \left( u ; v \right)
			& \ge \int \essinf_{W \in S (X) } \textnormal{cl}_v p_W \left( \om, v \left( \om \right) \right) \, d \mu \left( \om \right) + s_{\p_d I \left( u \right)}(v) + s_{\p_f I \left( u \right)}(v) \\
			& \ge \textnormal{cl}_v I' \left( u ; v \right).
		\end{align*}
		Therefore, the claim will obtain if we prove that the function (\ref{eq: derivative closure}) is lower semicontinuous. For this, let $v_n \to v \in L_\varphi$. It suffices to extract a subsequence $n_k$ such that
		\begin{equation} \label{eq: lsc along subsequence}
			\begin{aligned}
				\liminf_k \int \essinf_{W \in S (X) } \textnormal{cl}_v p_W \left( \om, v_{n_k} \left( \om \right) \right) \, d \mu \left( \om \right) \\
				\ge \int \essinf_{W \in \SS_v(X) } \textnormal{cl}_v p_W \left( \om, v \left( \om \right) \right) \, d \mu \left( \om \right).
			\end{aligned}
		\end{equation}
		By extracting a subsequences, we may assume the left-hand integrals in (\ref{eq: lsc along subsequence}) to be finite. Hence, we find a set $\Sigma \in \AA_\sigma$ outside of which the pertaining integrands vanish. As the right-hand integral is exhausting and the integrand has an integrable minorant $\langle \ell_a(\om), v(\om) \rangle$, we find $A \in \AA_\sigma$ with $\Sigma \subset A$ over which it attains its value.	By restricting $f$ on $A$ to a suitable separable subspace $W_0$, we may replace the essential infimum functions in both sides of (\ref{eq: lsc along subsequence}) by $\textnormal{cl}_v p_{W_0}$ attaining the essential infimum function as explained below Proposition \ref{pr: divergent subintegral}. For $\ell \in \p I_f \left( u \right)$, choose a subsequence $n_k$ with $v_{n_k} \to v$ a.e. and such that the $L_1(\mu)$-convergent sequence $\langle \ell_a, v_{n_k} \rangle$ has an integrable minorant $m$.	This implies
		$$
		\textnormal{cl}_v p_{W_0} \left( \om, v_{n_k} \left( \om \right) \right) \ge \langle \ell_a \left( \om \right), v_{n_k} \left( \om \right) \rangle \ge m \left( \om \right) \text{ a.e.}
		$$
		so that the Fatou lemma yields (\ref{eq: lsc along subsequence}). The addendum on (\ref{eq: subdifferential representation}) follows by the corresponding addendum in Theorem \ref{thm: conjugate A}.
	\end{proof}
	
	\section{Weak topologies and compactness} \label{sec: weak compactness}
	
	We obtain in this section our characterizing Theorems \ref{thm: L compact} and \ref{thm: V compact} on sequential compactness for various weak topologies that $L_\varphi(\mu)$ and $V_{\varphi^*}(\mu)$ or subsets of them induce on each other through the standard integral pairing
	$$
	(u, v) \mapsto \int_\Om \langle v(\om), u(\om) \rangle \, d \mu(\om)
	$$
	For a $\sigma$-finite measure, both theorems provide sufficient conditions that are also necessary if the inducing subspace is rich enough in terms of decomposability properties, as happens in particular if all of $V_{\varphi^*}(\mu)$ or $L_\varphi(\mu)$ induce the topology. Our approach to the matter is the same in both theorems: Starting with the prototypical cases of the almost superspaces $L_1\left( \mu; X \right)$ and $V_1\left( \mu ; X^* \right)$, into which $L_\varphi(\mu)$ and $V_{\varphi^*}(\mu)$ almost embed by Lemma \ref{lem: a emb} if $\mu$ is $\sigma$-finite, we seek additional conditions under which a sequence whose images under one of the sequences of embeddings
	$$
	L_\varphi(\mu) \to L_1\left(\Om_j; X\right),
	\quad V_{\varphi^*}(\mu) \to V_1\left(\Om_j; X\right) \quad j \in \N
	$$
	is (weakly) compact has a compact pre-image. Thereby, we reduce the general matter to these prototypes.	While characterizations of weak compactness in $L_1\left( \mu ; X \right)$ are abundant and will be discussed to some extent below, we establish a corresponding result for $\sigma \left( V_1\left(\mu ; X^* \right) ; L_\i \left( \mu ; X \right) \right)$ in Lemma \ref{lem: V1 compact}. Throughout this section, we assume that $\mu$ has no atom of infinite measure.
	
	\subsection{Weak topologies}
	
	Before we study weak compactness, we need to obtain a pair of auxiliary results that will serve to prove that $V_{\varphi*}(\mu)$ is a norming subspace of the dual space, which we require of the subspaces that induce topologies. The first of these results is of interest for its own sake, as it implies in particular that the convex functional $I_\varphi$ is lower semicontinuous in the weak topology induced on $L_\varphi(\mu)$ by $V_{\varphi^*}(\mu)$ if $\mu$ is $\sigma$-finite, as the proof of Lemma \ref{lem: equi norm} will show.
	
	\begin{lemma} \label{lem: Mackey}
		The Mackey topology $\mm \left( L_\varphi(\mu); V_{\varphi^*}(\mu) \right)$ implies local convergence in $\mu$.
	\end{lemma}
	
	\begin{proof}
		We may assume $\mu$ finite by definition of local convergence in measure. Clearly, it suffices to obtain an isotonic sequence $\Om_\e \in \AA$ with $\lim_{\e \downarrow 0} \mu \left( \Om \setminus \Om_\e \right) = 0$ such that, on every member of $\Om_\e$, convergence in $\mu$ obtains. Lemma \ref{lem: a emb} yields an isotonic sequence $\Om_\e$ with
		$$
		L_\i\left( \Om_\e; X \right) \to L_\varphi \left( \Om_\e \right) \to L_1 \left( \Om_\e ; X \right).
		$$
		As the last embedding in this chain is dense, its adjoint operator induces an embedding of $L_1 \left( \Om_\e ; X \right)^* = V_\i \left( \Om_\e ; X^* \right)$ into $V_{\varphi^*}\left( \Om_\e \right)$. Since the ball of $V_\i \left( \Om_\e ; X^* \right)$ is weak* compact by the Alaoglu theorem, it is weak* compact in $V_{\varphi^*}\left( \Om_\e \right)$. Therefore, the $L_1 \left( \Om_\e ; X \right)$-norm is Mackey continuous as a supremum over a convex set that is compact in $\ss \left( V_{\varphi^*}(\mu); L_\varphi(\mu) \right)$. Hence, since norm convergence in $L_1$ implies convergence in measure,	Mackey convergence implies convergence in measure on any $\Om_\e$.
	\end{proof}
	
	\begin{lemma} \label{lem: equi norm}
		The space $V_{\varphi^*}(\mu)$ is norming, i.e., the support functional of its ball
		$$
		S(u) = \sup_{\vvvert v \vvvert^*_\varphi \le 1} \int \langle v(\om), u(\om) \rangle \, d \mu(\om)
		$$
		defines an equivalent norm on $L_\varphi(\mu)$. There holds
		$$
		\| u \|_\varphi \le S(u) \le 2 \| u \|_\varphi \quad \forall u \in L_\varphi(\mu).
		$$
	\end{lemma}
	
	\begin{proof}
		It suffices to show that the support functional controls the Luxemburg norm. We first assume $\mu$ to be $\sigma$-finite. The functional $I_\varphi$ is closed w.r.t. local convergence in $\mu$ by the Fatou lemma. More precisely, since the Fatou lemma holds for sequences but not for nets, we may argue as follows: Given $u_i \in L_\varphi(\mu)$ a net with $\lim_i u_i = u$ in the Mackey topology, there holds $\lim_i u_i = u$ locally in measure by Lemma \ref{lem: Mackey} so we may conclude $I_\varphi(u) \le \liminf_i I_\varphi(u_i)$ once we prove lower semicontinuity of $I_\varphi$ for the local convergence in $\mu$. The latter is metrizable on $\LL_0 \left(\Om; X \right)$ if $X$ is separable. Combining this with the separable valuedness of strongly measurable functions and the $\sigma$-finiteness of $\mu$, we may dispose of the separability assumption on $X$. Hence, we have reduced to checking lower semicontinuity for sequences $w_n \in \LL_0\left(\Om; X \right)$ with $\lim_n w_n = w$ locally in $\mu$. Applying the Fatou lemma in the version discussed below Proposition \ref{pr: divergent subintegral}, we conclude $I_\varphi(w) \le \liminf_n I_\varphi(w_n)$. Hence, $I_\varphi$ is lower semicontinuous for the Mackey topology. Consequently, $I_\varphi = I_\varphi^{**}$ w.r.t. this pairing by \cite[Thm. 4.92]{FoLe} so that if $S(u) \le \frac{1}{2}$, then Lemma \ref{lem: modular-norm} yields the estimate
		$$
		I_\varphi(u) = \sup_{v \in V_{\varphi^*} } \langle v, u \rangle - I_\varphi^*(v) \le \sup_{v \in V_{\varphi^*} } \frac{1}{2} \vvvert v \vvvert_\varphi^* - I_\varphi^*(v) \le \sup_{v \in V_{\varphi^*} } \| v \|_\varphi^* - I_\varphi^*(v) \le 1
		$$
		hence $\| u \|_\varphi \le 1$, whence $\| \cdot \|_\varphi \le 2 S$ follows. It remains to remove the restriction of $\sigma$-finiteness. Let $\| u \|_\varphi > \alpha$ so that $I_\varphi\left( \alpha^{-1} u \right) > 1$. Because we assume $\mu$ to have no atom of infinite measure, we may invoke Proposition \ref{pr: divergent subintegral} to find a set $\Sigma \in \AA_\sigma$ such that $I_\varphi\left( \alpha^{-1} u \chi_\Sigma \right) > 1$, hence $\| u \chi_\Sigma \|_\varphi > \alpha$. Consequently, we find a sequence $\Sigma_n \in \AA_\sigma$ with
		$$
		\| u \|_\varphi = \lim_n \| u \chi_{\Sigma_n} \|_\varphi \le \limsup_n 2 S\left( u \chi_{\Sigma_n} \right) \le 2 S(u),
		$$
		whence the restriction has been lifted.
	\end{proof}
	
	\subsection{Weak compactness}
	
	Before addressing weak compactness, let us briefly settle the matter of characterizing strong compactness in $L_\varphi(\mu)$, thereby explaining why we feature no section on strong compactness. We know from Theorem \ref{thm: L complete} that up to subsequences convergence a.e. is necessary for strong convergence in $L_\varphi(\mu)$. By definition of the Luxemburg norm, the Vitali convergence theorem implies that a sequence $u_n$ converging to a limit $u$ a.e. up to subsequences will converge strongly in $L_\varphi(\mu)$ if and only if for every $\lambda > 0$ the extended real-valued function $\varphi \left( \om, \lambda \left[ u\left( \om \right) - u_n \left( \om \right) \right] \right)$ is equi-integrable and has non-escaping mass in $L_1(\mu)$. This is equivalent to the existence of integrable majorants up to subsequences by the Lebesgue dominated convergence theorem. These conditions become considerably easier to establish if $\varphi$ satisfies a $\Delta_2$-type condition on all of $\Om$ since then equi-integrability and non-escaping mass hold for all $\lambda > 0$ if they hold for some $\lambda > 0$. Notably, these considerations are completely analogous to the scalar case without requiring any noteworthy adaption. Nevertheless, more apt characterizations of strongly compact sets cannot be given in situations with additional structure on the measure space $\left( \Om, \AA, \mu \right)$, cf., e.g., \cite{KR} for a result in this direction.
	
	\subsubsection{Compactness in $L_\varphi(\mu)$}
	
	The following notions mimic the definition of equi-integrability as known in the theory of Lebesgue spaces in a way adapted to weak topologies. Their role in the theory of weakly compact subsets in scalar valued Orlicz spaces is well-established, cf., e.g., \cite[§4.5, Thm. 1]{RaRe}.
	
	\begin{definition}[weak and weak* equi-integrability] \label{def: weak eq-int}
		A subset $\FF \subset L_\varphi(\mu)$ is \emph{weakly equi-integrable on $\GG \subset V_{\varphi^*}(\mu)$} if for each evanescent sequence of sets $E_n \in \AA$, i.e., $\mu\left( \lim_n E_n\right) = 0$, and every $v \in \GG$ there holds
		\begin{equation} \label{eq: weak equi-integrability}
			\lim_n \sup_{u \in \FF} \left| \int_{E_n} \langle v, u \rangle \, d \mu \right| = 0.
		\end{equation}
		Similarly, the subset $\GG$ is called \emph{weak* equi-integrable on $\FF$} if for $E_n$ and $u \in \FF$ there holds
		\begin{equation} \label{eq: weak* equi-integrability}
			\lim_n \sup_{v \in \GG} \left| \int_{E_n} \langle v, u \rangle \, d \mu \right| = 0.
		\end{equation}
	\end{definition}
	
	All of $L_\varphi(\mu)$ is weakly equi-integrable on $C_{\varphi^*}(\mu)$ and the entire space $V_{\varphi^*}(\mu)$ is weak* equi-integrable on $C_\varphi(\mu)$. The sets $\FF$ and $\GG$ may always be taken linear by passing to their linear hull. Weak equi-integrability agrees with the formally stronger notion of absolute weak equi-integrability, i.e., the absolute value in (\ref{eq: weak equi-integrability}) may be equivalently placed inside the integral. Indeed, suppose $\FF$ is weakly equi-integrable but fails to be absolutely weakly equi-integrable. Then we find a sequence $E_n$ as above and $u_n \in \FF$ for which there obtains the contradiction
	$$
	\lim_n \int_{E_n} \left| \langle v, u_n \rangle \right| \, d \mu = \lim_n \left| \int_{E_n} \langle v, u_n \rangle \, d \mu \right| \ge \e > 0.
	$$
	The same is true for weak* equi-integrability. To localize weak sequential compactness, we need several preparatory results.
	
	\begin{proposition} \label{pr: a dense convergence principle}
		Let $u_n \in L_\varphi(\mu)$ be weakly equi-integrable at $v \in V_{\varphi^*}(\mu)$ and $v_m \uparrow v$. Then
		$$
		\lim_m \lim_n \langle u_n, v_m \rangle = \lim_n \lim_m \langle u_n, v_m \rangle
		$$
		whenever one of these iterated limits exists. Mutatis mutandis, the same is true for weak* equi-integrability.
	\end{proposition}
	
	\begin{proof}
		This follows easily from the definition of weak equi-integrability and a straightforward $\frac{\e}{2}$-argument.
	\end{proof}
	
	The next result reduces convergence considerations to a finite measure space.
	
	\begin{proposition} \label{pr: localize L}
		Let $\FF \subset L_\varphi(\mu)$ be a norm bounded set that is weakly equi-integrable on a subset $\GG \subset V_{\varphi^*}(\mu)$ inducing on $\FF$ an equivalent norm via
		\begin{equation} \label{eq: equi norm 1}
			\| u \|_\GG = \sup_{v \in B_\GG} \left| \langle v, u \rangle \right|.
		\end{equation}
		Then the sequential relative compactness in $\ss \left( L_\varphi, \GG \right)$ of the following sets is equivalent: \emph{(i)} $\FF$ \emph{(ii)} $\FF \chi_A$ for each $A \in \AA$ \emph{(iii)} $\FF \chi_{A_k}$ for each member $A_k$ of an isotonic sequence with $\lim_k A_k = \Om$.
	\end{proposition}
	
	We call $\GG$ a norming subset for $\FF$ if (\ref{eq: equi norm 1}) defines an equivalent norm.
	
	\begin{proof}
		It suffices to deduce (i) from (iii). Let $B_k$ be the partition of $\Om$ defined by
		$$
		B_1 = A_1, \quad B_k = A_k \setminus A_{k - 1} \quad k \ge 2.
		$$
		Pick a sequence $u_n \in \FF$ and extract a diagonal sequence (not relabeled) such that $u_n \chi_{B_k}$ converges for every $k$. Define $u$ to agree with the limit $\lim_n u_n \chi_{B_k}$ on $B_k$. Then $u \in L_\varphi(\mu)$ since
		\begin{align*}
			c \| u \|_\varphi = c \lim_k \| u \chi_{A_k} \|_\varphi \le \limsup_k \| u \chi_{A_k} \|_\GG
			& \le \limsup_k \liminf_n \| u_n \chi_{A_k} \|_\GG \\
			& \le C_\FF < \i.
		\end{align*}
		We used that $\| \cdot \|_\GG$ is lower semicontinuous as a supremum of continuous functions. By Proposition \ref{pr: a dense convergence principle},
		$$
		\lim_n \langle u_n, v \rangle = \lim_n \lim_k \langle u_n, v \chi_{B_k} \rangle = \lim_k \lim_n \langle u_n, v \chi_{B_k} \rangle = \lim_k \langle u, v \chi_{B_k} \rangle = \langle u, v \rangle
		$$
		for any $v \in \lin \GG$ so that a convergent subsequence has been extracted.
	\end{proof}
	
	An analogous statement and proof holds for subsets of $V_{\varphi^*}(\mu)$:
	
	\begin{proposition} \label{pr: localize V}
		Let $\GG \subset V_{\varphi^*}(\mu)$ be a norm bounded set that is weak* equi-integrable on $\FF \subset V_{\varphi^*}(\mu)$ inducing on $\GG$ an equivalent norm via
		\begin{equation} \label{eq: equi norm 2}
			\| v \|_\FF = \sup_{u \in B_\FF} \left| \langle v, u \rangle \right|.
		\end{equation}
		Then the sequential relative compactness in $\ss \left( V_{\varphi^*}, \FF \right)$ of the following sets is equivalent: \emph{(i)} $\GG$ \emph{(ii)} $\GG \chi_A$ for each $A \in \AA$ \emph{(iii)} $\GG \chi_{A_k}$ for each member $A_k$ of an isotonic sequence with $\lim_k A_k = \Om$.
	\end{proposition}
	
	We call $\FF$ a norming subset for $\GG$ if (\ref{eq: equi norm 2}) defines an equivalent norm. Further terminology:
	
	\begin{definition}[weak tightness] \label{def. loc w-tight}
		A family $\FF$ of strongly measurable functions $u \colon \Om \to X$ is \emph{weakly tight} if there exists a separably measurable integrand $h \colon \Om \times X \to \left[ 0, \i \right]$ with weakly compact sublevels $\left\{ x \in X \st h(\om, x) \le \alpha \right\}$, $\alpha \ge 0$, such that
		$$
		\sup_{u \in \FF} \int h \left[ \om, u(\om) \right] \, d \mu(\om) < +\i.
		$$
		If $\FF$ is weakly tight on every set of finite measure, then $\FF$ is called \emph{locally weakly tight}.
	\end{definition}
	
	\begin{definition}[weak biting convergence] \label{def. biting}
		Given a sequence of strongly measurable functions $u_n \colon \Om \to X$, we say that $u_n$ converges to $u$ in the weak or $\ss \left( X, X^* \right)$-biting sense if there exists an exhausting sequence $\Om_j \in \AA$, $\Om = \bigcup_j \Om_j$, such that $u_n \weak u$ in $L_1 \left( \Om_j ; X\right)$ for every $j \in \N$.
	\end{definition}
	
	Remember that $L_1 \left( \Om_j ; X\right)^* = V_\i \left( \Om_j ; X^* \right)$ by Corollary \ref{cor: C*}. In particular, the limit function is almost separably valued and $\AA_\mu$-measurable hence agrees with a strongly measurable function a.e. by Lemma \ref{lem: strong mb completion} so that it is unique up to null sets and strongly measurable up to modification on a null set. In view of the role that biting convergence plays in the upcoming compactness theorem, we are interested in approximating elements of $V_{\varphi^*}(\mu)$ by sequences in $V_\i\left(\mu; X^* \right)$ in convergence from below.
	
	\begin{proposition} \label{pr: Vinf mon dns}
		Let $\mu$ be $\sigma$-finite. For every weak* measurable function $v \colon \Om \to X^*$ with $v \in V_{\varphi^*}(\mu)$, there exists a sequence $v_k \in V_\i\left(\mu; X^* \right)$ such that $v_m \uparrow v$.
	\end{proposition}
	
	\begin{proof}
		By Lemma \ref{lem: a emb}, we find an isotonic sequence $F_n \in \AA_f$ with $\Om = \bigcup_n F_n$ and $L_\i \left( F_n ; X \right) \to L_\varphi\left( F_n \right)$. We claim that the non-negative function
		$$
		M(\om) = \esssup_{W \in \SS(X)} \sup_{x \in B_W} \langle v(\om), x \rangle
		$$
		is integrable on $F_n$. To see this, note that
		$$
		\int_{F_n} M \, d \mu = \sup_{u \in L_\i} \int_{F_n} \langle v, u \rangle - I_{B_X}(u) \, d \mu = \sup_{\| u \|_\i \le 1} \int_{F_n} \langle v, u \rangle \, d \mu < +\i
		$$
		by Theorem \ref{thm: inf int}. In particular, $M$ is finite a.e. on all of $F_n$ hence on $\Om$. Therefore, $v_k = \chi_{\left\{ M \le k \right\}} v \uparrow v$ on $\Om$ as $k \uparrow \i$, where $v_k \in V_\i\left( \Om; X^* \right)$.
	\end{proof}
	
	\begin{proposition} \label{pr: a unif integrab}
		Let $u_n$ be a sequence bounded in $L_1\left(\mu; X \right)$. Given $\e > 0$ there exists $\Om_\e \in \AA$ with $\mu \left( \Om \setminus \Om_\e \right) < \e$ and a subsequence $u_{n_k}$ that is equi-integrable in $L_1\left( \Om_\e ; X \right)$.
	\end{proposition}
	
	\begin{proof}
		It suffices to consider scalar functions. \cite[Lem. 2.31]{FoLe} yields a subsequence of $u_n$ (not relabeled) and a sequence of sets $\Om_n$ with $\lim_n \mu \left( \Om \setminus \Om_n \right) = 0$ such that $u_n \chi_{\Om_n}$ is equi-integrable. Passing to a subsequence $u_{n_k}$ with
		$$
		\mu \left( \Om \setminus \Om_{n_k} \right) < 2^{-k} \e,
		$$
		we find that $\Om_\e = \bigcup_{k \ge 1} \Om_{n_k}$ and $u_{n_k}$ have the claimed properties.
	\end{proof}
	
	\begin{theorem} \label{thm: L compact}
		Let $\mu$ be $\sigma$-finite, $\GG \subset V_{\varphi^*}(\mu)$ a norming subset and $A \subset L_\varphi(\mu)$. Then $A$ is relatively sequentially compact if \emph{(i)} $A$ is norm bounded \emph{(ii)} $A$ is weakly equi-integrable on $\GG$ \emph{(iii)} $A$ is relatively sequentially compact in the local convergence of the $\ss\left( X, X^* \right)$-biting sense. Conversely, if $A$ is relatively sequentially compact, then \emph{(i)} holds if $\lin \GG$ is closed, \emph{(ii)} holds if $\lin \GG$ is closed under multiplication with indicators of measurable sets
		\emph{(iii)} holds if in addition to the latter closedness, the sequential closure of $\lin \GG$ w.r.t. convergence from below contains $V_\i\left( \mu ; X^*\right)$. Moreover, condition \emph{(iii)} may be equivalently replaced by any of the following conditions: \emph{(iiia)} Given any sequence in $A$ there exists a subsequence $v_k \in \co \left\{ u_n \st n \ge k \right\}$ such that $v_k(\om)$ is norm convergent for a.e. $\om \in \Om$ \emph{(iiib)} $v_k(\om)$ is weakly convergent for a.e. $\om \in \Om$ \emph{(iiic)} $v_k$ is locally weakly tight.	
	\end{theorem}
	
	\paragraph{Remark.} The condition involving $V_\i\left(\mu; X^*\right)$ holds if $\lin \GG$ agrees with all of $V_{\varphi^*}(\mu)$, as can be seen by exhausting the $\sigma$-finite measure space $\Om$ with an isotonic sequence $F_j \in \AA_f$ such that $\Om = \bigcup_j F_j$ and $L_\i\left( F_j ; X \right) \to L_\varphi\left( F_j \right) \to L_1\left( F_j ; X \right)$ by Lemma \ref{lem: a emb} so that $V_\i \left( F_j ; X^* \right) \to V_{\varphi^*}\left( F_j \right)$ for all $j$.
	
	\begin{proof}
		Sufficiency. We may assume $\mu$ finite by Proposition \ref{pr: localize L} and the $\sigma$-finiteness of $\mu$. Let $u_n \in A$ be a sequence. We extract a subsequence (not relabeled) such that, for $\Om_j$ an isotonic sequence with $\Om = \bigcup_j \Om_j$, there holds convergence to a function $u$ weakly in $L_1\left( \Om_j ; X \right)$ for all $j \in \N$. We claim that $u \in L_\varphi(\mu)$. Indeed, if $B$ denotes the ball of $V_{\varphi^*}(\mu)$, then the support functional $S(u)$ of the ball $B$ in $V_{\varphi^*}(\mu)$ induces an equivalent norm on $L_\varphi(\mu)$ by Corollary \ref{cor: Lux-Orl equi}. We may restrict to taking this supremum over $B \cap \bigcup_j V_\i \left( \Om_j ; X^* \right)$ by Proposition \ref{pr: Vinf mon dns}, hence $S$ is lower semicontinuous w.r.t. the $\ss \left( X, X^* \right)$-biting convergence of $u_n$ so that $\| u \|_\varphi \le S(u) \le \liminf_n S(u_n) < \i$. As $u$ is strongly measurable by the remark below Definition \ref{def. biting}, we conclude $u \in L_\varphi(\mu)$. Consequently, by Proposition \ref{pr: localize L}, it suffices if for any given $\Om_j$ the sequence $u_n \chi_{\Om_j}$ converges to $u \chi_{\Om_j}$ in $\ss \left( L_\varphi, \lin \GG \right)$. Since we can obtain any $v \in \GG$ as a limit from below of elements belonging to $V_\i \left( \mu; X^* \right)$ by Proposition \ref{pr: Vinf mon dns}, the sufficiency has been proved due to Proposition \ref{pr: a dense convergence principle}. More precisely, for $v \in \GG$ pick $v_m \in V_\i\left(\mu ; X^* \right)$ with $v_m \uparrow v$. Then $\langle u, v \rangle = \lim_m \langle u, v_m \rangle = \lim_n \lim_m \langle u_n, v_m \rangle = \lim_m \lim_n \langle u_n, v_m \rangle = \lim_m \langle u_n, v \rangle$.
		
		Necessity. (i): it suffices to prove that any convergent sequence $u_n$ is bounded. The Banach Steinhaus theorem yields a bound for $u_n$ in the dual of $\lin \GG$, as this is a Banach space by assumption. Since $\GG$ is norming, we have obtained a bound in $L_\varphi(\mu)$. (ii): Again, it suffices to prove that $u_n$ is weakly equi-integrable on $\GG$. Otherwise, there were $v \in \GG$ and $E_n \in \AA$ with $\mu \left( \lim_n E_n \right) = 0$ such that
		\begin{equation} \label{eq: contrapos equi-int}
			\inf_{n \in \N} \left| \int_{E_n} \langle v, u_n \rangle \, d \mu \right| \ge \e > 0.
		\end{equation}
		Let $\lambda_n \left( E \right) = \int_E \langle u_n, v \rangle \, d \mu$, a signed finite measure for which the limit $\lambda \left( E \right) = \lim_n \lambda_n \left( E \right)$ exists for every $E \in \AA$ since $v \chi_E \in \lin \GG$ by assumption. We shall invoke the Vitali-Hahn-Saks theorem \cite[Thm. 2.53]{FoLe} to conclude that $\lim_n \lambda_n \left( E_n \right) = 0$, thus contradicting (\ref{eq: contrapos equi-int}). To justify this, let $\Sigma \in \AA_\sigma$ such that every $\langle u_n, v \rangle$ vanishes outside $\Sigma$. Then $\mu$ is $\sigma$-finite on $\Sigma$ and hence equivalent to some finite measure $\nu$ on $\Sigma$. Therefore, each $\lambda_n$ is absolutely continuous w.r.t. the finite measure $\nu$ and Vitali-Hahn-Saks has been justified.
		(iii): By Lemma \ref{lem: a emb}, there exists $F_n \in \AA_f$ an isotonic sequence with $\Om = \bigcup_n F_n$ such that, for $\Om_n = \Om \setminus F_n$, there holds $L_\i \left(\Om_n ; X \right) \to L_\varphi\left( \Om_n \right) \to L_1\left( \Om_n ; X \right)$, hence $V_\i \left( \Om_n ; X^* \right) \to V_{\varphi^*} \left( \Om_n \right)$. For $v \in V_\i\left( \Om_n ; X^* \right)$, pick $v_m \in \GG$ a sequence with $v_m \uparrow v$. By Proposition \ref{pr: a dense convergence principle}, we have $\langle u, v \rangle = \lim_m \langle u, v_m \rangle = \lim_m \lim_n \langle u_n, v_m \rangle = \lim_n \lim_m \langle u_n, v_m \rangle = \lim_n \langle u_n, v \rangle$, hence necessity obtains.
		
		Addendum: Assuming (i) and (ii), we shall prove that (iii) may be replaced by the other (iii)'s in a circular fashion. We start with (iii) to obtain (iiia). If $u_n \in A$ is a sequence converging locally in the $\ss\left( X^*, X \right)$-biting sense, then there exists an exhausting sequence $\Om_j \in \AA_f$, $\Om = \bigcup_j \Om_j$, such that $u_n \weak u$ in $L_1\left( \Om_j ; X \right)$. By \cite[Thm. 2.1]{DRS}, this yields $v_{k, j} \in \co \left\{ u_n \st n \ge k \right\}$ such that $v_{k, j} \chi_{\Om_j} (\om)$ is norm convergent for a.e. $\om \in \Om$. Inductively, we find $v_{k, j + 1} \in \co \left\{ v_{n, j} \st n \ge k \right\}$ such that $v_{k, j + 1} \chi_{\Om_{j + 1} } (\om)$ is norm convergent for a.e. $\om \in \Om$. Passing to the diagonal sequence, we find $v_{k, k} \in \co \left\{ u_n \st n \ge k \right\}$ such that $v_{k, k}(\om)$ is norm convergent for a.e. $\om \in \Om_j$ hence a.e. on $\Om$.
		
		Using (iiia) to obtain (iiib) is trivial.
		
		From (iiib) to (iiic): Let $u_n \in A$. There exists an isotonic sequence $F_j \in \AA_f$ such that $\Om = \bigcup_j F_j$ and $L_\i\left( F_j ; X \right) \to L_\varphi\left( F_j \right) \to L_1\left( F_j ; X \right)$ by Lemma \ref{lem: a emb}. Invoking Proposition \ref{pr: a unif integrab}, after possibly decreasing each $F_j$ and extracting a diagonal subsequence from $u_n$ (not relabeled), we may assume that $u_n$ is equi-integrable in each $L_1\left( F_j ; X \right)$. The sequence $u_n$ is weakly relatively compact in $L_1\left( F_j ; X \right)$ by \cite[Thm. 2.1]{DRS} so that we may pass to a diagonal sequence (not relabeled) that converges weakly in $L_1\left( F_j ; X \right)$ for all $j$. From this, we deduce by \cite[Thm. 8]{Sa} existence of $v_{k, j} \in \co \left\{ u_n \st n \ge k \right\}$ such that $v_{k, j}$ is weakly tight on $F_j$. Inductively, we find $v_{k, j + 1} \in \co \left\{ v_{n, j} \st n \ge k \right\}$ that is weakly tight on $F_{j + 1}$. Passing to the diagonal sequence, we have found $v_{k, k} \in \co \left\{ u_n \ge n \ge k \right\}$ that is locally weakly tight, i.e., (iiic) holds.
		
		From (iiic) to (iii): It suffices if for any sequence $u_n \in A$ we furnish a sequence of sets $\Om_j \in \AA$, $\Om = \bigcup_j \Om_j$, and a subsequence of $u_n$ that is weakly relatively compact in $L_1\left( \Om_j ; X \right)$ for $j \in \N$ since then $u_n$ will have a diagonal subsubsequence that converges locally in the $\ss\left( X^*, X \right)$-biting sense.
		Again by Lemma \ref{lem: a emb} and the $\sigma$-finiteness of $\mu$ we find $F_j \in \AA_f$ with $\Om = \bigcup_j F_j$ such that $u_n$ is bounded in $L_1\left( F_j ; X \right)$. Applying Proposition \ref{pr: a unif integrab}, we find a measurable set $\Om_j \subset F_j$ with $\lim_j \mu\left( F_j \setminus \Om_j \right) = 0$ and a subsequence of $u_n$ (not relabeled) that is equi-integrable in a fixed $L_1\left( \Om_j ; X \right)$. A standard diagonal argument allows then to extract a subsequence $u_n$ that is equi-integrable in $L_1\left( \Om_j ; X \right)$ for every $j \in \N$.
		
		Now, in order to conclude that $u_n$ is weakly relatively compact in $L_1\left( \Om_j ; X \right)$, it suffices by \cite[§24.3(8)]{Koe} if $u_n$ is weakly relatively convex compact, i.e., if we may extract from the decreasing convex hull $\co \left\{ u_n \st n \ge k \right\}$ of any subsequence of $u_n$ (not relabeled) a weakly convergent subsequence $v_k$. By (iiic), we find $v_k \in \co \left\{ u_n \st n \ge k \right\}$ that is locally weakly tight. Therefore, our proof is finished if we show that a bounded, equi-integrable and locally weakly tight sequence has a weakly convergent subsequence in $L_1 \left( \Om_j ; X \right)$. Pick
		$$
		\Om'_j \subset \Om_j, \quad \lim_j \mu \left( \Om_j \setminus \Om'_j \right) = 0
		$$
		such that each $v_k \chi_{\Om'_j}$ is weakly tight hence has a weakly convergent subsequence by \cite[Thm. 8]{Sa}. Extract a diagonal sequence (not relabeled) such that $v_k \chi_{\Om_j}$ converges weakly for all $j \in \N$. Then, as $\lim_j \sup_k \| v_k \chi_{\Om_j \setminus \Om'_j} \| = 0$ by equi-integrability, the sequence $v_k$ converges weakly, too. The proof is finished.
	\end{proof}
	
	\begin{corollary}
		Theorem \ref{thm: L compact} fully applies if either
		\begin{enumerate}
			
			\item $\GG$ is the unit ball of $V_{\varphi^*}(\mu)$ or; \label{it: f app 1}
			
			\item $X^*$ has the Radon-Nikodym property, $\GG$ is the unit ball of $L_{\varphi^*}(\mu)$, and $\varphi$ is dualizable i.a.e. for every element of $L_{\varphi^*}(\mu)$. \label{it: f app 2}
			
		\end{enumerate}
	\end{corollary}
	
	\begin{proof}
		\ref{it: f app 1}: The set $\GG$ is norming by Lemma \ref{lem: equi norm}. The closedness under multiplication with indicators is obvious from the definition of $V_{\varphi^*}(\mu)$. To see that any $v \in V_\i\left( \mu ; X^* \right)$ admits a sequence $v_n \in V_{\varphi^*}(\mu)$ with $v_n \uparrow v$, use Lemma \ref{lem: a emb} to find an exhausting sequence $F_n \in \AA_f$ with $\Om = \bigcup_n F_n$ such that $v_n = v \chi_{F_n} \in V_{\varphi^*}(\mu)$.
		
		\ref{it: f app 2}: This is a particular case of the former by Theorem \ref{thm: A = V}.
	\end{proof}
	
	\subsubsection{Compactness in $V_{\varphi^*}(\mu)$}
	
	Our second main theorem on compactness primarily applies to the weak* topology of $V_{\varphi^*}(\mu)$ hence of $L_{\varphi^*}(\mu)$, should these spaces coincide. As before, we state our result in a more general abstract setting. We start by analysing the particular case of the space $V_1 \left( \mu ; X^* \right)$ to obtain a preparatory result for the general case, characterizing relatively sequentially compact sets if $\mu$ is a finite separable measure.
	
	\begin{lemma} \label{lem: V1 compact}
		Let $\mu$ be a finite measure. A subset $\FF \subset V_1 \left( \mu ; X^* \right)$ is relatively sequentially compact in $\ss \left( V_1 \left( \mu ; X^* \right) ; L_\i \left( \mu ; X \right) \right)$ if \emph{(i)} $\FF$ is norm bounded \emph{(ii)} $\FF$ is weak* equi-integrable on $L_\i \left( \mu ; X \right)$ \emph{(iii)} For any $E \in \AA$, the set $\left( \int_E u \, d \mu \right)_{u \in \FF}$ is relatively sequentially compact in $\ss \left( X^*, X \right)$ \emph{(vi)} Given a sequence in $\FF$, the initial $\sigma$-algebra $\AA'$ generated by the sequence via its weak* measurability such that the restriction of $\mu$ to $\AA'$ is separable. Conversely, the conditions \emph{(i)}, \emph{(ii)} and \emph{(iii)} are necessary.
	\end{lemma}
	
	\paragraph{Remark.}
	
	\begin{enumerate}
		
		\item The integral of $u$ is to be understood in the sense of Pettis for the duality of $\ss\left( X^*, X \right)$, i.e., $\int_E u \, d \mu$ is defined as the unique element $x' \in X^*$ for which
		$$
		\langle x', x \rangle_{X^*, X} = \int_E \langle u(\om), x \rangle_{X^*, X} \, d \mu(\om) \quad \forall x \in X.
		$$
		
		\item Remember that a $\sigma$-finite measure on a separable $\sigma$-algebra is itself separable. The $\sigma$-algebra $\AA'$ generated by $u_n \in \FF$ via its weak* measurability is the smallest one for which any $u_n$ is weak* measurable, i.e., it is generated by the family of scalar functions $\om \mapsto \langle u_n(\om), x \rangle$ for $x \in X$. Equivalently, we may restrict to a dense subset of $X$, which shows in particular that $\AA'$ is separable if $X$ is. Another sufficient condition for separability of $\AA'$ is if $\FF$ consists of strongly measurable functions, as is implicit in the proof of Lemma \ref{lem: strong mb completion}. Finally, another condition is separability of the superalgebra $\AA$. Indeed, since the pseudometric in Definition \ref{def: sep mes} becomes a metric upon passing to a.e. equivalence classes of sets, this follows from the fact that subspaces of metric spaces retain separability.
		
	\end{enumerate}
	
	\begin{proof}
		Sufficiency: Let $u_n \in \FF$ with $\AA'$ the separable $\sigma$-algebra generated by the sequence. Conditional expectations are nonexpansive on $L_p\left( \mu ; X \right)$ for $1 \le p < \i$ by \cite[Ch. 5.1, Thm. 4]{DU}. We may use $\lim_{p \to \i} \| u \|_p = \| u \|_\i$ for $u \in L_\i\left( \mu ; X \right)$ to find that the conditional expectation operators $E\left( \cdot \st \AA' \right)$ are also nonexpansive on $L_\i\left( \mu ; X \right)$. Hence, if $v_n \to v$ a.e. and $\sup_n \| v_n \|_\i < \i$, then $E\left( v_n \st \AA' \right) \to E\left( v \st \AA' \right)$ a.e. and bounded in $L_\i \left( \mu ; X \right)$. Consequently, elementary properties of conditional expectations and approximation by simple functions yield the identity
		$$
		\int \langle u_n(\om), v(\om) \rangle \, d \mu(\om) = \int \langle u_n(\om), E\left( v \st \AA' \right)(\om) \rangle \, d \mu(\om) \quad \forall v \in L_\i\left( \mu ; X \right)
		$$
		so that we have reduced to considering pairings of $u_n$ with $\AA'$-measurable elements of $L_\i \left( \mu ; X \right)$. Let $\Om_j$ be a dense sequence for $\mu$. By a diagonal argument and the boundedness of $\FF$, we may pass to a subsequence (not relabeled) such that
		$$
		\exists \ss \left( X^*, X \right) \text{-} \lim_n \int_{\Om_j} u_n \, d \mu \quad \forall j \in \N.
		$$
		Consequently, this limit exists for any $\Om' \in \AA'$ by weak* equi-integrability of $\FF$ and since for $\Om'$ we find $\Om_{j_k}$ with $\lim_k \mu \left( \Om' \Delta \Om_{j_k} \right) = 0$. Hence, the pairing $\langle u_n, v \rangle$ with any simple function $v$ converges. By strong measurability, we may approximate a general $v \in L_\i \left( \mu ; X \right)$ with a sequence $v_i = \sum_{j = 1}^\i x_j \chi_{A_j}$ of countably valued functions converging uniformly to $v$, where $v_i$ in turn is approximated by its simple partial sums. Using the weak* equi-integrability of $u_n$ to invoke Proposition \ref{pr: a dense convergence principle}, we deduce that $\langle u_n, v \rangle$ converges for any $v \in L_\i \left( \mu ; X \right)$. In particular, there exists $u$ belonging to the dual space of $L_\i \left( \mu ; X \right)$ such that $w^*\text{-}\lim_n u_n = u$. Lemma \ref{lem: closed subspaces} then yields $u \in V_1\left(\mu ; X^* \right)$.
		
		Necessity: $\FF$ is bounded since $L_\i \left( \mu ; X \right)^* = V_1 \left( \mu ; X^* \right) \oplus S_1 \left( \mu ; X^* \right)$ by Theorem \ref{thm: xtnsn Giner}. As, for any $v \in L_\i \left( \mu; X \right)$, the set $\left( \langle u, v \rangle \right)_{u \in \FF}$ is weakly relatively compact in $L_1(\mu)$, the set $\FF$ is weak* equi-integrable on $L_\i \left( \mu ; X \right)$. The necessity of (iii) follows by considering functions $x \chi_E \in L_\i \left( \mu ; X \right)$ for $x \in X$ and $E \in \AA$.
	\end{proof}
	
	\begin{definition}[weak* biting convergence] \label{def. weak* biting}
		Given integrally weak* measurable functions $v_n \colon \Om \to X^*$, we say that $v_n$ converges locally to $v$ in the weak* or $\ss \left( X^*, X \right)$-biting sense if there exists an exhausting increasing sequence $\Om_j \in \AA$, $\Om = \bigcup_j \Om_j$, such that $v_n \to v$ in $\ss \left( V_1 \left( \Om_j ; X^* \right) ; L_\i \left( \Om_j ; X \right) \right)$ for every $j \in \N$.
	\end{definition}
	
	\paragraph{Remark.} Weak* biting convergence implies the limit to be integrally measurable in the weak* sense. Let $v_n \in V_1\left( \Om ; X^* \right)$ be a bounded sequence. We may extract from $v_n$ a subsequence (not relabeled) such that, for the integrable function $M_n$ defined in the proof of Proposition \ref{pr: Vinf mon dns}, there exists a sequence $\Om_n$ with $\lim_n \mu \left( \Om \setminus \Om_n \right) = 0$ such that $M_n \chi_{\Om \setminus \Om_n}$ is equi-integrable by Proposition \ref{pr: a unif integrab}, hence $\chi_{\Om \setminus \Om_n} v_n$ is weak* equi-integrable on $L_\i\left( \mu ; X \right)$. The measure in Proposition \ref{pr: Vinf mon dns} was $\sigma$-finite, however, we only used that $\mu$ has no atom of infinite measure to obtain the function $M_n$ there. Thus, by Lemma \ref{lem: V1 compact}, the question of convergence in the weak* biting sense reduces to obtaining (iii) and (iv) in Lemma \ref{lem: V1 compact}.The same is true for sequences bounded in $V_{\varphi*}(\mu)$ because this space embeds into $V_1\left(\mu; X^* \right)$ on an exhausting sequence by Lemma \ref{lem: a emb}.
	
	\begin{theorem} \label{thm: V compact}
		Let $\mu$ be $\sigma$-finite, $\FF \subset L_\varphi(\mu)$ a norming subset for $V_{\varphi^*}(\mu)$. Then $A$ is relatively sequentially compact in $\ss \left( V_{\varphi^*}, \lin \FF \right)$ if \emph{(i)} $A$ is norm bounded \emph{(ii)} $A$ is weak* equi-integrable on $\FF$ \emph{(iii)} $A$ is relatively sequentially compact in the $\ss \left( X^*, X \right)$-biting sense. Conversely, if $A$ relatively sequentially compact, then \emph{(i)} holds if $\lin \FF$ is closed, \emph{(ii)} holds if $\lin \FF$ is closed under multiplication with indicators of measurable sets \emph{(iii)} holds if in addition to the latter closedness, the sequential closure of $\lin \FF$ w.r.t. the convergence from below contains $L_\i\left( \mu ; X \right)$.
	\end{theorem}
	
	\paragraph{Remark.} The condition involving $L_\i\left(\mu; X\right)$ is satisfied if $\mu$ is $\sigma$-finite and $\lin \FF$ almost decomposable.
	
	\begin{proof}
		Sufficiency: Let $v_n \in A$ be a sequence. We may assume $\mu$ to be finite since then the $\sigma$-finite case will follow by Proposition \ref{pr: localize V}. Let $\Om_j$ be an exhausting sequence for which $v_n$ is sequentially relatively compact in $\ss \left( V_1 \left( \Om_j ; X^* \right) ; L_\i \left( \Om_j ; X \right) \right)$ for $j \in \N$. Invoking again Proposition \ref{pr: localize V}, we may reduce to extracting from $v_n$ a subsequence converging in $V_{\varphi^*}\left( \Om_j \right)$ for any given $j$. We shall extract several subsequences, none of which we relabel. Select $v_n$ with
		$$
		\lim_n v_n = v \text{ in } \ss \left( V_1 \left( \Om_j ; X^* \right) ; L_\i \left( \Om_j ; X \right) \right).
		$$
		Then for $u \in E_\varphi\left( \Om_j \right)$ and simple functions $u_m$ with $\lim_m \| u - u_m \|_\varphi = 0$ there holds
		\begin{align*}
			\left| \langle v, u \rangle - \langle v_n, u \rangle \right|
			& \le \left| \langle v, u - u_m \rangle \right| + \left| \langle v - v_n, u_m \rangle \right| + \left| \langle v_n, u - u_m \rangle \right| \\
			& \le C \| u - u_m \|_\varphi + \left| \langle v - v_n, u_m \rangle \right|
		\end{align*}
		by the boundedness of $A$ so that sending $n \to \i$ and then $m \to \i$ yields
		$$
		\lim_n \langle v_n, u \rangle = \langle v, u \rangle \quad \forall u \in E_\varphi\left(\Om_j\right).
		$$
		Proposition \ref{pr: a dense convergence principle} together with Lemma \ref{lem: decomp ss abs cont dense} then yields the same convergence for $u \in \FF$ since $A$ is weak* equi-integrable on $\FF$.
		
		Necessity: (i) Given a sequence $v_n \in A$, the Banach-Steinhaus theorem yields a bound on $v_n$ in the dual space of the Banach space $\lin \FF$ so that we obtain a bound in $V_{\varphi^*}(\mu)$ because $\FF$ is norming.
		(ii) If $A$ were not weak* equi-integrable on $u \in \FF$, then we would find sequences $v_n \in A$ and $E_n \in \AA$ with $\mu \left( \lim_n E_n \right) = 0$ and
		\begin{equation} \label{eq: contrapos equi-int 2}
			\inf_{n \in \N} \left| \int_{E_n} \langle v_n, u \rangle \, d \mu \right| \ge \e > 0.
		\end{equation}
		Let $\lambda_n \left( E \right) = \int_E \langle u_n, v \rangle \, d \mu$, a signed finite measure for which the limit $\nu \left( E \right) = \lim_n \nu_n \left( E \right)$ exists for every $E \in \AA$ since $v \chi_E \in \lin \FF$ by the closedness under multiplication with indicators. We shall invoke the Vitali-Hahn-Saks theorem \cite[Thm. 2.53]{FoLe} in order to conclude that $\lim_n \nu_n \left( E_n \right) = 0$ thus contradicting (\ref{eq: contrapos equi-int 2}). To justify this, note that $\mu$ is $\sigma$-finite hence equivalent to some finite measure $\nu$ on $\Sigma$. Therefore, each $\lambda_n$ is absolutely continuous w.r.t. the finite measure $\nu$ and Vitali-Hahn-Saks has been justified.
		(iii) By Lemma \ref{lem: a emb}, there exists $F_n \in \AA_f$ an isotonic sequence with $\Om = \bigcup_n F_n$ such that $L_\i\left( F_n ; X \right) \to L_\varphi\left( F_n \right)$ hence the adjoint mapping of this embedding is itself an embedding $V_{\varphi^*} \left( F_n \right) \to V_1 \left( F_n ; X^* \right)$ as $L_\i\left( F_n ; X \right)$ is dense from below in $L_\varphi\left( F_n \right)$ by Lemma \ref{lem: decomp ss abs cont dense}. For $v \in L_\i\left( F_n ; X \right)$, pick $v_m \in \FF$ a sequence with $v_m \uparrow v$. By Proposition \ref{pr: a dense convergence principle}, we have
		\begin{equation*}
			\langle u, v \rangle = \lim_m \langle u, v_m \rangle = \lim_m \lim_n \langle u_n, v_m \rangle = \lim_n \lim_m \langle u_n, v_m \rangle = \lim_n \langle u_n, v \rangle. \qedhere
		\end{equation*}
	\end{proof}
	
	\begin{corollary}
		Theorem \ref{thm: V compact} fully applies if $\FF$ is the unit ball of $L_\varphi(\mu)$.
	\end{corollary}
	
	\begin{proof}
		The space $L_\varphi(\mu)$ being predual to $V_{\varphi^*}(\mu) \oplus S_{\varphi^*}(\mu)$, it is norming. It is closed under under multiplication with indicators. As $L_\varphi(\mu)$ is almost decomposable by Corollary \ref{cor: L a decomp}, its sequential closure from below contains $L_\i\left( \mu ; X \right)$ if $\mu$ is $\sigma$-finite.
	\end{proof}
	
	\begin{corollary}
		If $\varphi$ is real-valued so that $C_\varphi(\mu)^* = V_{\varphi^*}(\mu)$, then Theorem \ref{thm: V compact} characterizes weak* convergent sequences in $V_{\varphi^*}(\mu)$. The weak* equi-integrability is always satisfied in this case. If moreover $X^*$ has the Radon-Nikodym property, then the ball of $V_{\varphi^*}(\mu)$ is relatively sequentially compact.
	\end{corollary}
	
	\begin{proof}
		We have $C_\varphi(\mu)^* = V_{\varphi^*}(\mu)$ by Corollary \ref{cor: C*}.	Theorem \ref{thm: V compact} fully applies since the sequential closure of $C_\varphi(\mu)$ from below contains $L_\i\left( \mu ; X \right)$ as $C_\varphi(\mu)$ is almost decomposable by Lemma \ref{lem: C a decomp} if $\mu$ is $\sigma$-finite. The closedness under multiplication with measurable indicators is obvious for $C_\varphi(\mu)$. Clearly, weak* equi-integrability w.r.t. $C_\varphi(\mu)$ is a vacuous assumption.
		
		If $X^*$ has the Radon-Nikodym property, then elements of $V_{\varphi^*}(\mu)$ are strongly measurable by Theorem \ref{thm: A = V} if $\mu$ is $\sigma$-finite as in Theorem \ref{thm: V compact}. Let $F_j \in \AA_f$ be a sequence with $\Om = \bigcup_j F_j$ such that
		$$
		V_{\varphi^*}\left( F_j \right) \to V_1\left( F_j ; X^* \right)
		$$
		according to Lemma \ref{lem: a emb}. By strong measurability, any sequence in $V_{\varphi^*}\left( F_j \right)$ generates a countable $\sigma$-algebra $\AA'$, hence $\mu$ is separable on $\AA'$. Now, by the remark below Lemma \ref{lem: V1 compact}, we have reduced the matter to checking (iii) in Lemma \ref{lem: V1 compact} for $E \in \AA'$. This obtains since the ball $B_{X^*}$ is sequentially weak* compact by \cite[XIII, Thm. 6]{Die} upon remembering that $X^*$ has the Radon-Nikodym property iff separable subspaces of $X$ have separable duals by \cite[VII.2, Cor. 8]{DU}.
	\end{proof}
	
	\newpage
	
	\appendix
	
	\chapter{Appendix} \label{prt: Appendix}
	
	\section{Multimaps}
	
	We compile here auxiliary results about (Effros) measurable multimaps. Throughout this section, the metric space $M$ is separable unless stated otherwise and $\Gamma \colon \Om \to \mathcal{P} \left( M \right)$ is a multimap.
	
	\begin{lemma} \label{lem: retain measurability}
		Let the multimap $\Gamma$ be closed and measurable.
		\begin{enumerate}
			\item The graph $\graph \Gamma$ is $\AA \otimes \BB \left( M \right)$-measurable. \label{en. it. graph measurable}
			\item The multimap $\p \Gamma \colon \Om \to \mathcal{P} \left( M \right)$ is measurable. \label{en. it. boundary measurable}
		\end{enumerate} 
	\end{lemma}
	
	\begin{proof}
		\ref{en. it. graph measurable}: The proof for $M = \R^d$ is contained in \cite[Thm. 14.8]{RocW} and may be adapted by replacing $\Q^d$ with a dense sequence in $M$. \ref{en. it. boundary measurable}: Let $O \subset M$ be open. The set
		$$
		V \left( O \right)
		= \left\{ A \in \CL \left( M \right) \st O \subset A \right\}
		= \bigcap_{x \in O} \left\{ A \in \CL \left( M \right) \st d_x (A) = 0 \right\}
		$$
		is closed in the Wijsman topology $\tau_W$. By the Hess theorem \cite[Thm. 6.5.14]{B}, the multimap $\Gamma$ is $\tau_W$-measurable as a single-valued mapping to $\CL \left( M \right)$. Hence
		\begin{align*}
			\Gamma^- \left( O \right)
			& = \left\{ \om \st O \subset \Gamma \left( \om \right) \right\} \dot{\cup} \left\{ \om \st \p \Gamma \left( \om \right) \cap O \ne \emptyset \right\} \\
			& = \Gamma^{-1} \left[ V \left( O \right) \right] \dot{\cup} \left\{ \om \st \p \Gamma \left( \om \right) \cap O \ne \emptyset \right\} \\
			& = \Gamma^{-1} \left[ V \left( O \right) \right] \dot{\cup} \left( \p \Gamma \right)^{-} \left( O \right)
		\end{align*}
		so that $\left( \p \Gamma \right)^{-} \left( O \right)$ is measurable as a difference of measurable sets.
	\end{proof}
	
	\begin{corollary} \label{cor: open graph measurable}
		If $\Gamma$ is open and measurable, then $\graph \Gamma \in \AA \otimes \BB \left( M \right)$.
	\end{corollary}
	
	\begin{proof}
		The multimaps $\cl \Gamma$ and $\p \cl \Gamma$ are measurable with measurable graphs by Lemma \ref{lem: retain measurability}. As $\Gamma$ is open, we have $\Gamma = \cl \Gamma \setminus \p \cl \Gamma$ so that $\graph \Gamma = \graph \cl \Gamma \setminus \graph \p \cl \Gamma$ belongs to $\AA \otimes \BB \left( M \right)$.
	\end{proof}
	
	\begin{lemma} \label{lem: solid multimaps simpler measurability}
		Let $\cl \Gamma = \cl \interior \Gamma$. Then $\Gamma$ is measurable iff $\Gamma^- \left( \left\{ x \right\} \right)$ is measurable for every $x \in M$.
	\end{lemma}
	
	\begin{proof}
		$\implies$: Pre-images of compact sets under measurable multimaps are measurable. $\impliedby$: Let $O \subset M$ be open and $\left\{ x_n \right\} \subset O$ be a dense sequence. From $\cl \Gamma = \cl \interior \Gamma$ follows
		\begin{align*}
			\Gamma^- \left( O \right)
			= \left\{ \om \st \Gamma \left( \om \right) \cap O \ne \emptyset \right\}
			& = \left\{ \om \st \interior \Gamma \left( \om \right) \cap O \ne \emptyset \right\} \\
			& = \bigcup_{n \ge 1} \left\{ \om \st \interior \Gamma \left( \om \right) \cap \left\{ x_n \right\} \ne \emptyset \right\}.
		\end{align*}
		As measurability of the multimaps $\interior \Gamma$ and $\cl \interior \Gamma$ is equivalent, the last set is measurable and our claim obtains.
	\end{proof}
	
	\begin{lemma} \label{lem: strong mb completion}
		Let $M$ be an arbitrary metric space and $\AA_\mu$ the completion of $\AA$ w.r.t. $\mu$. A function $u \colon \Om \to M$ is $\AA_\mu$-$\BB \left( M \right)$-measurable and almost separably valued iff there exists a strongly $\AA$-$\BB \left( M \right)$-measurable function $v \colon \Om \to M$ with $u = v$ a.e.
	\end{lemma}
	
	\begin{proof}
		$\implies$: Modify $u$ on a null set to obtain a separably valued $\AA_\mu$-measurable function $w$ and take a sequence $B_n$ of balls generating the topology of $w \left(\Om\right)$. Express $w^{-1} \left( B_n \right)$ as a disjoint union of two sets $\Om_n \in \AA$ and $M_n$ such that $M_n \subset N_n$ for a null set $N_n \in \AA$. For $N = \bigcup_{n \ge 1} N_n$, define $w$ to agree with $w$ on $\Om \setminus N$ and assign any constant value on $N$. Any $v^{-1} \left( B_n \right)$ is $\AA$-measurable, whence $v$ is strongly $\AA$-measurable and $u = v$ a.e.
		
		$\impliedby$: Let $N \in \AA$ be negligible with $u = v$ on $\Om \setminus N$. If $B \in \BB \left( M \right)$, then
		$$
		u^{-1} \left( B \right) = \LaTeXunderbrace{ \left( v^{-1} \left( B \right) \cap N^c \right) }_{\in \AA} \cup \LaTeXunderbrace{ \left( u^{-1} \left( B \right) \cap N \right) }_{\subset N \text{ with } \mu \left( N \right) = 0} \in \AA_\mu.
		$$
		Hence, $u$ is $\AA_\mu$-$\BB \left( M \right)$-measurable and $u \left( \Om \setminus N \right) = v \left( \Om \setminus N \right)$ is separable.
	\end{proof}
	
	\section{Integrands}
	
	We compile here auxiliary results about measurability of integrands. Since none of the standard references \cite{B, CaVa, He, Roc} contain these statements directly in the required form, we give proofs.
	
	\begin{definition}[Infimal measurability] \label{def: infimal measurability}
		An integrand $f \colon \Om \times \TT \to \left[ - \i, \i \right]$ is called infimally measurable iff, for every open set $O \subset \TT$ and every open interval $I \subset \R$, the pre-image set
		$$
		S^-_f \left( O \times I \right) = \left\{ \om \in \Om \st \epi f_\om \cap \left( O \times I \right) \ne \emptyset \right\}
		$$
		is measurable.
	\end{definition}
	
	\begin{lemma} \label{lem: equivalence infimal measurability}
		For an integrand $f \colon \Om \times \TT \to \left[ - \i, \i \right]$, the following are equivalent:
		\begin{enumerate}
			\item $f$ is infimally measurable;
			\item For $O \subset \TT$ open, the functions $\inf_O f_\om$ are measurable;
			\item For $\alpha \in \left[ -\i, \i \right]$, the strict sublevel multimaps
			$$
			L_\alpha \colon \Om \to \mathcal{P} \left( \TT \right) \colon \om \mapsto \lev_{< \alpha} f_\om = \left\{ x \st f_\om (x) < \alpha \right\}
			$$
			are Effros measurable.
		\end{enumerate}
	\end{lemma}
	
	\begin{proof}
		For an open subset $O \subset \TT$ and an interval $I = (\alpha, \beta)$, there holds
		\begin{equation*}
			S^-_f\left( O \times I \right)
			= \left\{ \epi f_\om \cap O \times I \ne \emptyset \right\}
			= \left\{ \inf_O f_\om < \beta \right\}
			= L^-_\beta(O). \qedhere
		\end{equation*}
	\end{proof}
	
	\begin{lemma} \label{lem: inf meas normality}
		If the integrand $f \colon \Om \times \TT \to \left[ - \i, \i \right]$ is pre-normal, then $f$ is infimally measurable. If $\TT$ is second countable, the converse is true as well.
	\end{lemma}
	
	\begin{proof}
		$\implies$: Recall Definition \ref{def: infimal measurability}. $\impliedby$: Let $O_n \subset \TT$ be a base sequence of open sets. By definition of the product topology, every open set $U \subset \TT \times \R$ may be written as $U = \bigcup_{k \in \N} O_{n_k} \times \left( \alpha_k, \beta_k \right)$ with $\alpha_k, \beta_k \in \Q$, whence there follows measurability of the set
		\begin{equation*}
			S^-_f \left( U \right) = \bigcup_{k \in \N} S^-_f \left( O_{n_k} \times \left( \alpha_k, \beta_k \right) \right). \qedhere
		\end{equation*}
	\end{proof}
	
	We call a map $F \colon \TT \to \left[ - \i, \i \right]$ upper semicontinuous if its hypograph is closed. When $F$ is $\left[ -\i, \i \right)$-valued, this coincides with other known characterizations of upper semicontinuity such as open sublevel sets and the $\limsup$ inequality.
	
	\begin{lemma} \label{lem: normal sums}
		If $f \colon \Om \times \TT \to \left[ -\i, \i \right]$ is pre-normal and $g \colon \TT \to \R$ is upper semicontinuous, then $h_\om (x) = f_\om (x) + g \left( x \right)$ is infimally measurable.
	\end{lemma}
	
	\begin{proof}
		By upper semicontinuity of $g$, the set $V = \left\{ \left( x, r \right) \st x \in O, \, r < \beta - g \left( x \right) \right\}$ is open. The claim obtains if we show that
		$$
		S^-_h \left[ O \times \left( \alpha, \beta \right) \right] = \left\{ \epi h_\om \cap O \times \left( \alpha, \beta \right) \ne \emptyset \right\} = \left\{ \epi f_\om \cap V \ne \emptyset \right\} = S^-_f \left( V \right).
		$$
		We check the set identity: Let $\left( x, r \right) \in \epi h_\om \cap O \times \left( \alpha, \beta \right)$ so that $f_\om (x) + g \left( x \right) \le r$ while $x \in O$ and $r \in \left( \alpha, \beta \right)$. Then $f_\om (x) < \beta - g \left( x \right)$ so that $\epi f_\om \cap V $ is non-empty.
		
		Conversely, if $\left( x, r \right) \in \epi f_\om \cap V$, then $f_\om (x) + g \left( x \right) < r + g \left( x \right) < \beta$ so that the intersection $\epi h_\om \cap O \times \left( \alpha, \beta \right)$ is non-empty.
	\end{proof}
	
	\begin{lemma} \label{lem: upper semi normal}
		Let $M$ be separable. Suppose $f \colon \Om \times M \to \left[ - \i, \i \right]$ is such that
		\begin{enumerate}
			\item For all $\om \in \Om$, $x \mapsto f_\om (x)$ is upper semicontinuous; \label{en. it. normality via partial properties 1}
			\item For all $x \in M$, $\om \mapsto f_\om (x)$ is measurable. \label{en. it. normality via partial properties 2}
		\end{enumerate}
		Then $f$ is a pre-normal integrand.
	\end{lemma}
	
	\begin{proof}
		By \ref{en. it. normality via partial properties 1}, there holds $\cl \epi f_\om = \cl \interior \epi f_\om$ so that Lemma \ref{lem: solid multimaps simpler measurability} makes it sufficient to observe that, for all $\left( x, \alpha \right) \in M \times \R$, the set $S_f^- \left( \left\{ \left( x, \alpha \right) \right\} \right) = \left\{ \om \st f_\om (x) \le \alpha \right\}$ is measurable by \ref{en. it. normality via partial properties 2}.
	\end{proof}
	
	\begin{lemma} \label{lem: Lipschitz regularization properties}
		Let $f \colon M \to \left[ - \i, \i \right]$ be a function and $\lambda > 0$. If, for some $x_0 \in M$, the Lipschitz regularization
		$$
		f_\lambda \left( x \right) = \inf_{y \in M} f \left( y \right) + \lambda d \left( x, y \right)
		$$
		is finite, then $f_\lambda$ is finite-valued and Lipschitz continuous with constant $\lambda$.
	\end{lemma}
	
	\begin{proof}
		For $x_1, x_2, y \in M$, there holds $f \left( y \right) + \lambda d \left( x_1, y \right) \le f \left( y \right) + \lambda d \left( x_2, y \right) + \lambda d \left( x_1, x_2 \right)$ so that $f_\lambda \left( x_1 \right) \le f_\lambda \left( x_2 \right) + \lambda d \left( x_1, x_2 \right)$. In particular, if $f_\lambda \left( x_1 \right) \in \R$ for some $x_1 \in M$, then $f_\lambda$ is Lipschitz continuous with constant $\lambda$. Also, if $f_\lambda \left( x_1 \right)$ is infinite, then $f_\lambda \equiv - \i$ or $f_\lambda \equiv \i$.
	\end{proof}
	
	\begin{lemma} \label{lem: Lipschitz regularization pre-normal}
		Let $M$ be separable. If $f \colon \Om \times M \to \left[ -\i, \i \right]$ is a pre-normal integrand, the Lipschitz regularization
		$$
		f_{\om, \lambda} (x) = \inf_{y \in M} f_\om (y) + \lambda d \left( x, y \right), \quad \lambda > 0
		$$
		also is a pre-normal integrand. Moreover, for all $\om \in \Om$, the partial map $x \mapsto f_{\om, \lambda} (x)$ is upper semicontinuous.
	\end{lemma}
	
	\begin{proof}
		By Lemma \ref{lem: normal sums}, the integrand $h_\om (y) + \lambda d \left( x, y \right)$ is infimally measurable for $x \in M$ and $\lambda > 0$ so that $f_{\om, \lambda} (x)$ is measurable in $\om $ and Lipschitz continuous or assumes a constant value $\left\{ - \i, \i \right\}$ in $x$. Either way, the partial map is upper semicontinuous, hence Proposition \ref{lem: upper semi normal} implies the claim.
	\end{proof}
	
	\begin{lemma} \label{lem: joint measurability}
		Let $M$ be separable and $f \colon \Om \times M \to \left( - \i, \i \right]$ an integrand.
		\begin{enumerate}
			\item If $f$ is normal, then it is $\AA$-$\BB \left( M \right)$-measurable. \label{en. it. joint measurability 1}
			\item If $f$ is $\AA_\mu$-$\BB \left( M \right)$-measurable, then it is pre-normal w.r.t. $\AA_\mu$. \label{en. it. joint measurability 2}
		\end{enumerate}
	\end{lemma}
	
	\begin{proof}
		\ref{en. it. joint measurability 1}: Lemmas \ref{lem: equivalence infimal measurability} and \ref{lem: inf meas normality} guarantee that truncation of an integrand retains pre-normality, hence we may reduce to the case when $f$ is bounded below. Since $f = \lim_{\lambda \to \i} f_\lambda$ pointwise as a monotone limit for the Lipschitz regularization $f_\lambda$ of $f$ according to \cite[Prop. 1.33]{Br}, we may reduce to considering $f_\lambda$. Lemma \ref{lem: Lipschitz regularization pre-normal} shows that, for all $\om \in \Om$, the partial map $x \mapsto f_{\om, \lambda} (x)$ is upper semicontinuous. Consequently, for $\alpha \in \left[ - \i, \i \right]$, the strict sublevel multimap
		$$
		\Om \to \mathcal{P} \left( \TT \right) \colon \om \mapsto L_{f, \alpha} \left( \om \right) := \left\{ x \in \TT \st f_\om (x) < \alpha \right\}
		$$
		is open and measurable. Hence, by Corollary \ref{cor: open graph measurable}, its graph
		$$
		\graph L_{f, \alpha} = \left\{ \left( \om, x \right) \st f_\om (x) < \alpha \right\}
		$$
		belongs to $\AA$-$\BB \left( M \right)$ so that $f$ is $\AA$-$\BB \left( M \right)$-measurable.
		
		\ref{en. it. joint measurability 2}: The Aumann theorem \cite[Thm. 6.10]{FoLe} yields an $\AA_\mu$-measurable Castaing representation for the closure of the epigraphical multimap $\om \mapsto \epi f_\om$, hence its measurability follows from \cite[Thm. III.9]{CaVa}. As the Effros measurability of $\cl S_f$ and $S_f$ is equivalent, the claim obtains.
	\end{proof}
	
	\begin{lemma}[Semicontinuous approximation of normal integrands] \label{lem: unif Lus}
		Let $E \subset \R$ be a Lebesgue measurable set, $\left( M, d \right)$ a separable metric space and $X$ a Banach space. Let $Y \subset X^*$ be a closed separable subspace such that $\left| x \right| = \sup_{y \in B_Y} \langle y, x \rangle$ for all $x \in X$. We denote the space $X$ in the weak topology $\ss \left( X, Y \right)$ by $X_\ss$. Let $f \colon E \times M \times X_\ss$ be a normal integrand. Then, for every $\e > 0$, there exists a closed set $C_\e \subset E$ with $\lambda \left( E \setminus C_\e \right) < \e$ such that the restriction of $f$ to $C_\e \times M \times X_\ss$ is sequentially lower semicontinuous.
	\end{lemma}
	
	\begin{proof}
		We observe that the case without any dependence of $f$ on the component $X$ can be proved in the same way as \cite[Thm. 6.28]{FoLe}, where this has been done for $M = \R^d$. Using the separability of $M$, one need only replace $\Q^d$ by a dense sequence in $M$.
		
		For the general case, we note that the closed balls of the $X$-norm are metrizable in $\ss \left( X, Y \right)$ since $Y$ is separable. Hence, for every $\e > 0$ and $n \in \N$, we can apply the first case to the restriction of $f$ to the set $E \times M \times n B_X$, obtaining closed sets $C_{\e, n}$ such that (i) $\lambda \left( E \setminus C_{\e, n} \right) < \frac{\e}{2^n}$ and (ii) the restriction of $f$ to the set $C_{\e, n} \times M \times n B_X$ is lower semicontinuous. Setting
		$$
		C_\e = \bigcap_{n \in \N} C_{\e, n},
		$$
		we find that (j) the set $C_\e$ is closed (jj) $\lambda \left( E \setminus C_\e \right) < \e$ and (jjj) the restriction of $f$ to $C_\e \times M \times X_\ss$ is sequentially lower semicontinuous. The third conclusion combines (ii) with the boundedness of convergent sequences in $X_\ss$ by the Banach-Steinhaus uniform boundedness principle.
	\end{proof}
	
	\section{Hyperspace topologies}
	
	We collect and prove here results about the Attouch-Wets $\tau_{AW}$ topology introduced in Section \ref{sec: inf-int}.
	
	\begin{lemma} \label{lem: bounded inf covergence}
		Let $f_n \in \LS \left( M \right)$ be a sequence with $\tau_{AW}$-$\lim f_n = f$. For each bounded set $B \subset M$, there holds
		\begin{equation} \label{eq: bounded inf}
			\liminf_n \inf_B f_n \ge \inf_B f.
		\end{equation}
		For each open set $O \subset M$, there holds
		\begin{equation} \label{eq: open inf}
			\limsup_n \inf_O f_n \le \inf_O f.
		\end{equation}
	\end{lemma}
	
	\begin{proof}
		(\ref{eq: bounded inf}): Suppose $n_k$ were a subsequence with $\lim_k \inf_B f_{n_k} < \alpha < \inf_B f$.	Pick $x_k \in B$ with $\lim_k f_{n_k} \left( x_k \right) < \alpha$ and a sequence $\alpha_k \in \left[ f_{n_k} (x_k), \alpha \right)$ bounded below by $\beta \in \R$. We have $y_k = \left( x_k, \alpha_k \right) \in \epi f_{n_k}$ so that $\tau_{AW}$-$\lim f_n = f$ by definition of the box metric on $M \times \R$ yields the contradiction
		\begin{align*}
			0 < \left| \inf_B f - \alpha \right| & \le \lim_k \left| \inf_B f - \alpha_k \right|
			\le \lim_k \left| d_{y_k} \left( \epi f \right) - 0 \right| \\
			& \le \lim_k \sup_{y \in B \times \left[ \beta, \alpha \right] } \left| d_y \left( \epi f \right) - d_y \left( \epi f_{n_k} \right) \right| = 0.
		\end{align*}
		(\ref{eq: open inf}): As $\tau_{AW}$ implies epi-convergence, this obtains by \cite[Prop. 1.18]{Br}.
	\end{proof}
	
	\begin{lemma} \label{lem: AW affine stable}
		For $\tau_{AW}$ on $\LS(X)$, the mapping
		$$
		\LS \left( X \right) \times X^* \times \R \to \LS \left( X \right) \colon \left( f, x', \alpha \right) \mapsto f + x' + \alpha
		$$
		is continuous.
	\end{lemma}
	
	\begin{proof}
		If $\lim_n x'_n = x'$ and $\lim_n \alpha_n = \alpha$, then $\lim_n x'_n + \alpha_n = x' + \alpha$ uniformly on bounded subsets, whence the claim follows by \cite[Thm. 7.1.5]{B}.
	\end{proof}
	
	\begin{proposition} \label{pr: separable reduction}
		Given a sequence of subsets $S_k \subset M$ and $W_0 \in \SS \left( M \right)$, there exists a closed set $W \in \SS \left( M \right)$ such that $W_0 \subset W$ and
		$$
		d_x \left( S_k \right) = d_x \left( S_k \cap W \right) \quad \forall x \in W \quad \forall k \in \N.
		$$
		Moreover, if $M = U \times V$ with $V$ separable, then $W$ may be chosen of the form $Y \times V$.
	\end{proposition}
	
	\begin{proof}
		Following \cite[Lem. 7.2]{Pe}, we inductively define an increasing sequence $W_\ell \in \SS \left( M \right)$ with
		\begin{equation} \label{eq: approximate distance persistance}
			d_x \left( S_k \right) = d_x \left( S_k \cap W_\ell \right) \quad \forall x \in W_{\ell - 1} \quad \, k \in \N.
		\end{equation}
		If $W_1, \ldots, W_\ell$ satisfying (\ref{eq: approximate distance persistance}) have been defined, take a dense sequence $w_m$ in $W_\ell$, pick $x_n(k, m)$ a sequence in $S_k$ with $d \left( w_m, x \right) \le d \left( w_m, S_k \right) + \frac{1}{n}$ and set
		\begin{equation} \label{eq: definition next subset}
			W_{\ell + 1} = \cl \bigcup_{k, m, n} \left\{ x \right\} \cup \left\{ w \right\}.
		\end{equation}
		Then $W_\ell \subset W_{\ell + 1} \in \SS \left( M \right)$ and $d_w \left( S_k \right) = \inf_n d_w(x) = d_w \left( S_k \cap W_{\ell + 1} \right)$. The functions $d \left( S_k \right)$ and $d \left( S_k \cap W_{\ell + 1} \right)$ coincide on $W_\ell$ as they are continuous and equal on the dense sequence $w_m$. This completes the inductive construction. Finally, we set
		$$
		W = \cl \bigcup_{\ell \in \N} W_\ell
		$$
		and observe that $d_x \left( S \right) = d_x \left( S_k \cap W_{\ell + 1} \right) \ge d_x \left( S_k \cap W \right) \ge d_x \left( S \right)$ for $x \in W_\ell$ so that the first part follows by density. Regarding the addendum, note that the construction still works if we increase $W_{\ell + 1}$ to be any separable superset of the right-hand side in (\ref{eq: definition next subset}). In particular, we may take $\tilde{W}_{\ell + 1} = P_U \left( W_{\ell + 1} \right) \times V$ if $P_U$ denotes the projection of $M$ onto $U$. By definition, $W$ will then be of the required form.
	\end{proof}
	
	\begin{lemma} \label{lem: strongly tau_AW sep subspace}
		If $f \colon \Om \to \LS(M)$ is strongly $\tau_{AW}$-measurable, then there exists $W = \cl W \in \SS(M)$ such that $f$ is normal on $W$ and $\inf_W f_\om = \inf_M f_\om$ for all $\om \in \Om$.
	\end{lemma}
	
	\begin{proof}
		Pick $B_m$ an isotonic sequence of bounded open balls with $M = \bigcup_{m \ge 1} B_m$. As $f_n$ takes finitely many values
		$$
		\exists W_0 \in \SS(M) \colon \inf_{B_m} f_{\om, n} = \inf_{W_0 \cap B_m} f_{\om, n} \quad \forall m \in \N, \quad \om \in \Om.
		$$
		Invoking Proposition \ref{pr: separable reduction}, we find a closed set $W \in \SS \left( M \right)$ with $W_0 \subset W$ and
		$$
		d_x \left( \epi f_{\om, n} \right) = d_x \left( \epi f_{\om, n} \cap W \times \R \right) \quad \forall x \in W \times \R, \quad \om \in \Om, \quad n \in \N.
		$$
		Thus, the restriction $f_{\om, n, W}$ is a Cauchy sequence in $\left( \LS \left( W \right), \tau_{AW} \right)$. Hence, there exist $G_\om = \tau_{AW}$-$\lim_n \epi f_{\om, n, W} \in \CL\left( W \times \R \right)$ which agrees with the (relative) epigraph $\epi f_\om$ since for $x \in W \times \R$ there holds
		$$
		d_x \left( \epi f_\om \right)
		= \lim_n d_x \left(\epi f_{\om, n} \right)
		= \lim_n d_x \left( \epi f_{\om, n} \cap W \times \R \right)
		= d_x \left( G_\om \right).
		$$
		Lemma \ref{lem: bounded inf covergence} implies $\lim_n \inf_B f_{\om, n} = \inf_B f_\om$ for any bounded relatively open set $B \subset W$ and so
		\begin{align*}
			\inf_M f_\om
			= \inf_m \inf_{B_m} f_\om
			= \inf_m \lim_n \inf_{B_m} f_{\om, n}
			& = \inf_m \lim_n \inf_{W \cap B_m} f_{\om, n} \\
			& = \inf_m \inf_{W \cap B_m} f_\om \\
			& = \inf_W f_\om.
		\end{align*}
		Finally, since $\tau_{AW}$ is finer than $\tau_W$, we conclude that the limiting integrand is $\tau_W$-measurable hence normal by the Hess theorem \cite[Thm. 6.5.14]{B}.
	\end{proof}
	
	\section{Young measures} \label{sec: Young measures}
	
	\subsection{Young measures and weak topologies}
	
	In this section, we extend some results on Young measures to a setting where they take values in a Banach space $U$ that carries an auxiliary weak type topology induced by the pairing $\langle \cdot, \cdot \rangle$ with a separable Banach space $Z$. Let $\tt = \ss(U, Z)$. The space $U$ is required to carry a norm such that
	\begin{equation} \label{eq: dual nrm}
		\| u \|_U = \sup_{z \in B_Z} \langle u, z \rangle.
	\end{equation}
	Therefore, by means of the isometric embedding $U \to Z^* \colon u \mapsto \langle u, \cdot \rangle$, the Banach space $U$ is a closed subspace of the dual space $Z^*$. It is clear that $\tt$ agrees with the relative weak* topology under this identification. We require $U$ to be a weak*-closed subspace in this way. Let $z_n \in S_Z$ be a sequence that is dense in the unit sphere $S_Z$. Then
	\begin{subequations}
		\begin{equation} \label{eq: cvg bnd}
			u_n \ttto u \implies \sup_n \| u_n \|_U < + \i;
		\end{equation}
		\begin{equation} \label{eq: nrm U lsc}
			u_i \ttto u \implies \| u \|_V \le \liminf \| u_i \|_U;
		\end{equation}
		\begin{equation} \label{eq: B U cpt}
			\text{The unit ball } B_U \text{ is sequentially compact in } \ss(U; Z);
		\end{equation}
		\begin{equation} \label{eq: tt metric}
			\text{The metric } d_\tt(u_1, u_2) = \sum_{n = 1}^\i 2^{- n} \left| \langle z_n, u_1 - u_2 \rangle \right|^2 \text{ induces $\tt$ on } B_U;
		\end{equation}
		\begin{equation} \label{eq: tt Suslin}
			\text{The space } \left( U, \tt \right) \text{ is a Suslin locally convex vector space.}
		\end{equation}
	\end{subequations}
	We denote $\left( U, \tt \right)$ by $U_\tt$ for short. We have (\ref{eq: tt Suslin}) according to \cite[§4, Def. 17]{CaVa} since (\ref{eq: tt metric}) renders $U_\tt$ a continuous image of a Polish space. An unbounded sequence in $U$ does not permit a $\tt$-convergent subnet due to the Banach-Steinhaus Theorem. Therefore, every $\tt$-compact set is bounded so that sequential and topological compactness coincide for $\tt$ by (\ref{eq: tt metric}). In particular, a subset of $U_\tt$ is relatively compact if and only if it is bounded in $U$. Let $\left( \Om, \AA, \mu \right)$ be a finite measure space. An integrand $h \colon \Om \times U \to \left( - \i, \i \right]$ that is measurable with respect to the product $\sigma$-algebra $\AA \otimes \BB \left( U_\tt \right)$ is called sequentially $\tt$-normal if
	\begin{equation} \label{eq: h sqt lsc}
		u \mapsto h(\om, u) \text{ is sequentially } \tt \text{-lsc. for } \om \in \Om.
	\end{equation}
	A family of $\FF$ of measurable functions $u \colon \Om \to U_{\tt}$ is said to be $\tt$-tight if there exists a non-negative $\tt$-normal integrand $h$ such that
	$$
	\left\{ u \in U \st h(\om, u) \le \alpha \right\} \text{ is } \tt \text{-compact for } \om \in \Om \text{ and } \sup_{u \in \FF} \int_\Om h(\om, u(\om)) \, d \mu(\om) < \i.
	$$
	By (\ref{eq: h sqt lsc}), the compactness of the sublevel sets equivalently means
	$$
	\lim_{\| u \|_U \uparrow \i} h(\om, u) = +\i \text{ for } \om \in \Om.
	$$ 
	A Young measure on $U_\tt$ is a family $\nu = \left( \nu_\om \right)_{\om \in \Om}$ of Borel probability measure on $U_\tt$ such that
	$$
	\om \mapsto \nu_\om(A) \text{ is } \AA \text{-measurable for all } A \in \BB \left( U_\tt \right).
	$$
	We denote the set of all Young measures by $\YY \left( \Om; U_\tt \right)$.
	
	\paragraph{Remark.} Note that if the Banach space $U$ is separable, then, since closed balls are $\tt$-closed by (\ref{eq: B U cpt}), we have $\BB(U) = \BB \left( U_\tt \right)$ for the Borel $\sigma$-algebrae so that $\AA \otimes \BB \left( U_\tt \right)$-measurability reduces to $\AA \otimes \BB \left( U \right)$-measurability.
	
	\begin{theorem} \label{thm: Y meas}
		Let $f_n, f_\i \colon \Om \times U \to \left( - \i, \i \right]$ be sequentially $\tt$-normal integrands such that
		\begin{equation} \label{eq: G-liminf assu}
			f_\i(\om, u) \le \inf \left\{ \liminf_{n \uparrow \i} f_n(\om, u_n) \colon u_n \ttto u \right\} \quad \forall (\om, u) \in \Om \times U.
		\end{equation}
		Let $u_n \colon \Om \to U_\tt$ be a $\tt$-tight sequence. Then there exists a subsequence $n_k$ and $\nu \in \YY(\Om; U_\tt)$ such that, for a.e. $\om \in \Om$,
		\begin{equation} \label{eq: cnc inc}
			\nu_\om \text{ is concentrated on the set } \bigcap_{p = 1}^\i \overline{\left\{ u_{n_k}(\om) \colon k \ge p \right\}}^\tt
		\end{equation}
		and, if the sequence $\om \mapsto f_{n_k} \left( \om, u_{n_k}(\om) \right)^-$ of negative parts is equi-integrable, then
		\begin{equation} \label{eq: G-liminf}
			\int_U \int f_\i(\om, v) \, d \nu_\om(v) \, d \mu(\om) \le \liminf_{k \uparrow \i} \int f_{n_k}(\om, u_{n_k}(\om)) \, d \mu(\om).
		\end{equation}
	\end{theorem}
	
	\begin{proof}
		Setting
		$$
		\vvvert u \vvvert^2 = d_\tt \left( u, 0 \right) \quad \forall u \in U
		$$
		for $d_\tt$ defined in (\ref{eq: tt metric}), the proof is analogous to the one of \cite[Thm. 4.3]{St} with our function $\vvvert \cdot \vvvert$ in place of the eponymous one therein. To prove the preparatory result \cite[Lem. 4.2]{St} in the current setting, note that $\BB(U)$ needs to be replaced by $\BB \left( U_\tt \right)$ in the statement and proof if $U$ is not separable, which causes no problems.
	\end{proof}
	
	We denote by $V_1 \left( \mu; U_\tt \right)$ the normed space of $\AA$-$\BB \left( U_\tt \right)$-measurable functions $v \colon \Om \to U_\tt$ with finite norm
	\begin{equation} \label{eq: fnt nrm 2}
		\| v \|_{V_1 \left( \mu; U_\tt \right)} = \int_\Om \| v(\om) \|_U \, d \mu(\om) < + \i.
	\end{equation}
	Here, the measurability of $v$ equivalently means that, for every $z \in Z$, the scalar function $\om \mapsto \langle v(\om), z \rangle$ is measurable.	Note that the integrand in (\ref{eq: fnt nrm 2}) is measurable in view of (\ref{eq: dual nrm}) and the separability of $Z$.
	
	\begin{corollary} \label{cor: f thm y mes}
		Let $v_n \in V_1 \left( \mu; U_\tt \right)$ be a bounded sequence. There exists a subsequence $n_k$ and a Young measure $\nu \in \YY \left( \Om; U_\tt \right)$ such that, for a.e. $\om \in \Om$, the relation (\ref{eq: cnc inc}) holds and, setting
		$$
		v(\om) = \int_U u \, d \mu_\om(u) \quad \text{for } \om \in \Om,
		$$
		we have
		\begin{equation} \label{eq: exp b cvg 2}
			v_{n_k} \to v \text{ in the biting sense of } V_1 \left( \mu; U_\tt \right),
		\end{equation}
		i.e., there exists a decreasing sequence of measurable sets $A_j$ with $\mu \left( \Om \setminus A_j \right) \to 0$ such that
		$$
		v_{n_k} \to v \text{ in } \ss \left( V_1 \left( \Om \setminus A_j; U_\tt \right) ; L_\i \left( \Om \setminus A_j ; Z \right) \right) \text{ for every } j \in \N.
		$$
		Moreover, if $f_n, f_\i$ are convex integrands as in Theorem \ref{thm: Y meas} for which the sequence $\om \mapsto f_{n_k}(\om, u_{n_k}(\om) )^{-}$ of negative parts is equi-integrable, then
		\begin{equation} \label{eq: G-liminf, cnvx}
			\int f_\i(\om, u(\om)) \, d \mu(\om) \le \liminf_{k \uparrow \i} \int f_{n_k}(\om, u_{n_k}(\om)) \, d \mu(\om).
		\end{equation}
	\end{corollary}
	
	\begin{proof}
		Let $\e > 0$. Invoking \cite[Lem. 2.31]{FoLe}, we may extract from $v_n$ a subsequence (not relabeled) such that there exist sets $M_{n, \e} \in \AA$ with $\mu \left( \Om \setminus M_{n, \e} \right) \le 2^{- n} \e$ and $v_n \chi_{M_{n, \e} }$ is equi-integrable. Setting
		$$
		M_\e = \bigcap_{n \ge 1} M_{n, \e},
		$$
		we have $\mu \left( \Om \setminus M_\e \right) \le \e$ and $v_n \chi_{M_\e}$ is equi-integrable. As $\e$ may be arbitrarily small, we can iterate this construction to find suitable sets $M_{\e_i}$ and extract a diagonal sequence $n_k$ such that setting
		$$
		A_j = \bigcup_{i = 1}^j M_{\e_i}
		$$
		obtains an increasing sequence of measurable sets such that $\mu \left( \Om \setminus A_j \right) \to 0$ and, for each fixed $j \in \N$, the sequence $v_n \chi_{A_j}$ is equi-integrable. Applying Theorem \ref{thm: Y meas} to the integrands $\langle \cdot, z(\om) \rangle$ for $z \in L_\i \left( A_j ; Z \right)$ and the sequence $v_n$ yields a further subsequence $v_{n_k}$ such that (\ref{eq: exp b cvg 2}) obtains. To prove the addendum, we combine (\ref{eq: G-liminf}) with the Jensen inequality. Note in this regard that $U_\tt$ is a locally convex Hausdorff space since $Z$ separates its points by (\ref{eq: dual nrm}) so that every convex, lower semicontinuous, proper function on $U_\tt$ is the supremum of its affine, continuous minorants, hence the Jensen inequality holds.
	\end{proof}
	
	\paragraph{Remark.} The lower semicontinuity (\ref{eq: G-liminf, cnvx}) continues to hold if $v_n$ is not bounded in $V_1(\mu; U_\tt)$ but there exists a sequence $B_j \subset B_{j + 1} \subset \Om$ of measurable sets such that $\chi_{B_j} v_n$ is bounded for every $j \in \N$. This follows by an easy limiting argument using the equi-integrability assumption on the negative parts. If $v_n$ is strongly measurable with values in a Banach space, then a sufficient condition for such a bound is if $v_n$ is bounded in some generalized Orlicz space. This follows by Lemma \ref{lem: a emb}.
	
	\subsection{A chain rule for Young measures}
	
	In this appendix, we prove a Young measure version of the chain rule inequality (\ref{eq: ch ru}). We shall repeatedly use (\ref{eq: dd Suslin}) in the following.
	
	\begin{lemma} \label{lem: YM chain-rule ineq}
		Suppose that (\ref{eq: phi mb}), (\ref{eq: phi* mb}), (\ref{eq: phi.2.2.1}), (\ref{eq: phi.2.2.2}) hold and $\EE$ satisfies (\ref{eq: E_0}) and either (\ref{eq: ch ru}) or (\ref{eq: ch ru wkr}). Let $T_0 \in \cl \left[ 0, T \right)$ and $u \colon \left[ 0, T_0 \right] \to V_\dd$ be a $\LL(0, T_0)$-$\BB \left( V_\dd \right)$-measurable curve such that
		\begin{equation} \label{eq: assu 1}
			\left( t, u(t) \right) \in \dom(F) \text{ for a.e. t } \in \left( 0, T_0 \right) \text{ and } \sup_{0 \le t \le T_0} \EE_t \left( u(t) \right) < + \i.
		\end{equation}
		Let $\mu = \left( \mu_t \right)_{t \in (0, T_0)}$ be a Young measure in $V_\dd \times X_{\dd^*} \times \R$ such that
		\begin{equation} \label{eq: YM ch rl ass 2}
			\begin{gathered}
				\int_0^{T_0} \int \phi_{t, u(t)}(v) + \phi^*_{t, u(t)}(- \zeta) + | p | \, d \mu_t(v, \zeta, p) \, dt < \i;
				\\
				u'(t) = \int v \, d \mu_t(v, \zeta, p) \quad \text{for a.e } t \in \left( 0, T_0 \right);
				\\
				\xi \in F_t \left( u(t) \right), \, p \le P_t \left( u(t), \xi \right)
				\\
				\text{ for a.e. } t \in \left( 0, T_0 \right) \text{ and all } \left( v, \xi, p \right) \in \supp \left( \mu_t \right).
			\end{gathered}
		\end{equation}
		Then the map $t \mapsto \EE_t(u(t) )$ is of bounded variation on $\left( 0, T_0 \right)$. In addition, if (\ref{eq: ch ru}), then
		\begin{subequations}
			\begin{equation} \label{eq: YM chain-rule ineq}
				\frac{d}{dt} \EE_t(u(t) ) \ge \int \langle \zeta, u'(t) \rangle + p \, d \mu_t \left( v, \zeta, p \right) \quad \text{ in } \DD' \left( 0, T_0 \right)
			\end{equation}
			and if (\ref{eq: ch ru wkr}), then
			\begin{equation} \label{eq: YM chain-rule ineq 2}
				\frac{d}{dt} \EE_t(u(t) ) \ge \int \langle \zeta, u'(t) \rangle + p \, d \mu_t \left( v, \zeta, p \right) \quad \text{ for a.e. } t \in \left( 0, T_0 \right).
			\end{equation}
		\end{subequations}
	\end{lemma}
	
	\begin{proof}
		We indicate the necessary adaptions and additions to the proof of \cite[Prop. B.1]{MRS}. Throughout the proof, consider $X_{\dd^*}$ in place of $V^*$ with the corresponding Borel $\sigma$-algebra $\BB \left( X_{\dd^*} \right)$. Accordingly, the functions taking values in $V^*$ like $\xi_n, \xi_{n, k}$ are no longer strongly measurable but $\LL(0, T_0)$-$\BB \left( X_{\dd^*} \right)$-measurable. We work through the assertions of the old proof consecutively.
		
		Claim 1: Observe that $X_{\dd^*}$, by the separability of $Y$, is a Suslin locally convex space, thereby we still may apply \cite[Thm. III.22]{CaVa} to obtain a Castaing representation.
		
		Claim 2: Let $V_1 \left( 0, T_0; X_{\dd^*} \right)$ be the space of $\BB \left( X_{\dd^*} \right)$-measurable functions with the norm
		$$
		\| v \|_{V_1 \left( 0, T; X_{\dd^*} \right)} = \int_0^T \| v(t) \|_X \, dt.
		$$
		Replace the measurable map $\xi_n$ by the maps $\left( \xi_n , p_n \right) \colon \left( 0, T_0 \right) \to X_{\dd^*} \times \R$ and arrange that
		\begin{equation} \label{eq: int repr}
			\begin{gathered}
				\left( \xi_n, p_n \right) \in V_1 \left( 0, T_0; X_{\dd^*} \right) \times L_1 \left( 0, T_0; \R \right) \text{ for every } n \in \N \\
				\text{ and } \sup_n \int_0^{T_0} \phi^*_{t, u(t) } \left( - \xi_n(t) \right) + | p_n(t) | \, dt < + \i
			\end{gathered}
		\end{equation}
		instead of
		\begin{equation}
			\begin{gathered}
				\xi_n \in L_1 \left( 0, T_0; V^* \right) \text{ for every } n \in \N \\
				\text{ and } \sup_n \int_0^{T_0} \phi^*_{t, u(t) } \left( - \xi_n(t) \right) \, dt < + \i
			\end{gathered}
		\end{equation}
		by replacing the integrand $\xi \mapsto \phi^*_{t, u(t) } \left( - \xi \right)$ with $(\xi, p) \mapsto \phi^*_{t, u(t) } \left( - \xi \right) + | p |$ throughout the argument. To be on the safe side, replace the min operator in the definition of $\MM_{*}$ by an inf and note that (B.11) in the original proof is a result of (B.2) and (B.4). To conclude that $\overline{\xi} \in V_1\left(0, T_0; X_{\dd^*} \right)$, it suffices to have (\ref{eq: phi.2.2.2}) instead of the stronger superlinear growth assumption from \cite{MRS}.
		
		Claim 3: This part requires only trivial changes for (\ref{eq: YM chain-rule ineq 2}). To obtain (\ref{eq: YM chain-rule ineq}), let $\gamma \in C^1_c \left( 0, T_0 \right)$ with $\gamma \ge 0$ and consider the integrand $f \colon \left[ 0, T_0 \right] \times X_{\dd^*} \times \R \to \R \colon (t, \eta, p) \mapsto \left( \langle \eta, u'(t) \rangle_{X, V} + p \right) \cdot \gamma(t)$. Recall from \cite[Prop. B.1]{MRS} the measurable multimap
		$$
		\KK(t, u(t) ) = \left\{ (\xi, p) \in X_{\dd^*} \times \R \st \xi \in F_t(u(t) ), \, p \le P_t(u(t), \xi(t) ) \right\}
		$$
		for which there exists a Castaing representation, i.e., a sequence of measurable maps $\left( \xi_n, p_n \right) \colon \left[ 0, T_0 \right] \to X_{\dd^*} \times \R$ such that
		\begin{equation} \label{eq: Ca repr}
			\left\{ \left( \xi_n(t), p_n(t) \right) \st n \in \N \right\} \subset \KK(t, u(t) ) \subset \overline{\left\{ \left( \xi_n(t), p_n(t) \right) \st n \in \N \right\} }.
		\end{equation}
		Moreover, we arranged (\ref{eq: int repr}) for the representation. Let $K$ denote those measurable selections $t \mapsto \left( \xi(t), p(t) \right)$ from $K(t, u(t) )$ satisfying the integrability condition
		$$
		\int_0^{T_0} \phi^*_{t, u(t) } \left( - \xi(t) \right) + < p(t) | \, dt < + \i.
		$$
		We claim that
		\begin{equation} \label{eq: sup-int swtch}
			\sup_{(\xi, p) \in K} \int_0^{T_0} f \left( t, \xi(t), p(t) \right) \, dt = \int_0^{T_0} \sup_{(\xi, p) \in \KK(t, u(t) )} f(t, \xi, p) \, dt.
		\end{equation}
		To see this, let $g(t) \coloneqq \sup_{(\xi, p) \in \KK(t, u(t) ) } f(t, \xi, p)$, a measurable function by \cite[Lem. III.39]{CaVa}. Due to (\ref{eq: Ca repr}) and (\ref{eq: int repr}), for every $\e > 0$, the multimap
		$$
		\Gamma_\e(t) = \left\{ (\xi, p) \in \KK(t, u(t) ) \st \phi^*_{t, u(t) } \left( - \xi \right) < \i \text{ and } f(t, \xi, p) \ge g(t) - \e \right\}
		$$
		has non-empty values. It retains a measurable graph as the intersection of two such multimaps, hence it admits a measurable selection $t \mapsto (\xi(t), p(t) )$ by \cite[Lem. III.22]{CaVa}. We modify $t \mapsto (\xi(t), p(t) )$ by setting
		\begin{equation}
			\left( \overline{\xi}_m(t), \overline{p}_m(t) \right) =
			\begin{cases}
				(\xi(t), p(t) ) & \text{ if } \phi^*_{t, u(t) } \left( - \xi(t) \right) + | p(t) | \le m, \\
				(\overline{\xi}(t), \overline{p}(t) ) & \text{ else.}
			\end{cases}
		\end{equation}
		Then
		$$
		\int_0^{T_0} \phi^*_{t, u(t) } \left( - \overline{\xi}_m(t) \right) + | p_m(t) | \, dt < + \i \quad \forall m \in \N
		$$
		so that $\left( \overline{\xi}_m, \overline{p}_m \right)$ is an admissible competitor on the left side of (\ref{eq: sup-int swtch}). Since $\left( \overline{\xi}_m(t), \overline{p}_m(t) \right) = (\xi(t), p(t) )$ eventually for every $t \in \left[ 0, T_0 \right]$ as $m \to + \i$, we find
		\begin{equation}
			\liminf_m \int_0^{T_0} f \left( t, \overline{\xi}_m(t), \overline{p}_m(t) \right) \, dt \ge \int_0^{T_0} g(t) \, dt - \e T_0,
		\end{equation}
		whence (\ref{eq: sup-int swtch}) follows by the arbitrariness of $\e > 0$. Using (\ref{eq: sup-int swtch}) and (\ref{eq: ch ru}), we find
		\begin{align*}
			- \int_0^{T_0} \EE_t \left( u(t) \right) \cdot \gamma'(t) \, dt
			& \ge \sup_{(\xi, p) \in K} \int_0^{T_0} f \left( t, \xi(t), p(t) \right) \, dt \\
			& = \int_0^{T_0} \sup_{(\xi, p) \in \KK(t, u(t) )} f(t, \xi, p) \, dt \\
			& \ge \int_0^{T_0} \int f(t, \eta, p) \, \mu_t(v, \eta, p) \, dt \\
			& = \int_0^{T_0} \int \langle \eta, u'(t) \rangle + p \, \mu_t(v, \eta, p) \cdot \gamma(t) \, dt,
		\end{align*}
		which proves (\ref{eq: YM chain-rule ineq}) since $\gamma \in C^1_c \left( 0, T_0 \right)$ with $\gamma \ge 0$ is arbitrary.
	\end{proof}
	
	\begin{lemma} \label{lem: mes sel}
		Suppose (\ref{eq: phi mb}), (\ref{eq: phi* mb}), (\ref{eq: phi.2.2.1}), (\ref{eq: phi.2.2.2}) and let $\EE$ satisfy (\ref{eq: E_0}), (\ref{eq: absolute  continuity}), (\ref{eq: radially time differentiable}), and (\ref{eq: cc cnclsn}). Let $u \colon \left[ 0, T \right] \to V_\dd$ be a measurable curve satisfying (\ref{eq: assu 1}) and suppose that the set
		\begin{gather*}
			\SS \left( t, u(t), u'(t) \right) \coloneqq \\
			\left\{ (\zeta, p) \in X \times \R \st \zeta \in - \p \phi_{t, u(t) } \left( u'(t) \right) \cap F_t(u(t) ) , \, p \le P_t \left( u(t), \zeta \right) \right\}
		\end{gather*}
		is non-empty for almost every $t \in \left( 0, T_0 \right)$. Then there exist measurable functions $\xi \colon \left( 0, T_0 \right) \to X_{\dd^*}$ and $p \colon \left( 0, T_0 \right) \to \R$ such that
		\begin{equation} \label{eq: Argmin selection}
			\begin{gathered}
				\left( \xi(t), p(t) \right) \in \Argmin \left\{ \phi^*_{t, u(t) } (-\zeta) - p \st (\zeta, p) \in \SS \left( t, u(t), u'(t) \right) \right\} \\
				\text{ for a.e. } t \in \left( 0, T_0 \right).
			\end{gathered}
		\end{equation}
	\end{lemma}
	
	\begin{proof}
		We indicate the necessary adaptions to the proof of \cite[Lem. B.2]{MRS}. To see that
		$$
		\Argmin \left\{ \phi^*_{t, u(t) }(-\zeta) - p \st (\zeta, p) \in \SS \left( t, u(t), u'(t) \right) \right\} \not = \emptyset \text{ for a.e. } t \in \left( 0, T_0 \right),
		$$
		we can use (\ref{eq: G controlled}), (\ref{eq: radially time differentiable}), and (\ref{eq: cc cnclsn}) as in the original proof up to now permitting every constant to depend on $t \in \left( 0, T \right)$. To bound a sequence $\xi_n \in F_t \left( u(t) \right)$ such that
		$$
		\phi^*_{t, u(t) } \left( - \xi_n \right) < + \i,
		$$
		it suffices to have (\ref{eq: phi.2.2.2}). Finally, we need to replace weak convergence in $V^*$ by the one in $X_{\dd^*}$. As the ball $B_X$ is sequentially compact in $\dd^*$, this is possible.
		
		Now, to prove that $\SS \left( t, u(t), u'(t) \right) \not = \emptyset$, we can modify the old argument as follows: Consider $X_{\dd^*}$ and $\BB \left( X_{\dd^*} \right)$ instead of $V^*$ everywhere and observe that, by \cite[Lem. III.39]{CaVa}, the Argmin multimap in (\ref{eq: Argmin selection}) is measurable thus permits a Lebesgue measurable selection by \cite[Thm. III.22]{CaVa}.
	\end{proof}

	\section{Subdifferential calculus}
	
	In this section, we introduce and give a short account of the theory of functions that are semiconvex with respect to a general non-negative extended real-valued modulus function and collect subdifferential calculus results used in the main part.
	
	Throughout the section, let $X$ be a (real) locally convex Hausdorff space in duality with another such space $X^*$, $p \colon X \to \left[ 0, \i \right)$ a continuous seminorm, $\om \colon X \times X \to \left[ 0, \i \right]$ any function and $f \colon X \to \left( - \i, \i \right]$ a proper function, $x_0 \in X$ a point with $f(x_0) < +\i$. We say that $p$ norms a subspace $Y \subset X$ if its restriction to $Y$ is a norm. We denote by $\mm$ the Mackey topology and by $\ss$ the weak topology. We write $x_\alpha \to_f x$ to mean $x_\alpha \to x$ and $f(x_\alpha) \to f(x)$. For $v \in X$ and $h \in \R$, we set
	$$
	\delta_h f(x_0; v) \coloneqq \frac{f(x_0 + h v) - f(x_0)}{h}.
	$$
	The (lower) Dini derivative of $f$ at $x_0$ towards $v$ is given by
	$$
	f'_D \left(x_0; v \right) \coloneqq \liminf_{h \downarrow 0} \delta_h f(x_0; v) \in \left[ - \i, \i \right].
	$$
	The upper Dini derivative $f'_{D^+} \left( x_0; v \right)$ arises by replacing the lower limit with the upper one. If $f'_D \left(x_0; v \right) = f'_{D^+} \left( x_0; v \right)$, then we call this common value the radial derivative of $f$ at $x_0$ towards $v$ and denote it by $f'_r \left( x_0; v \right)$. If the radial derivative exists in all directions, then we say that $f$ is radially differentiable at $x_0$. The Dini-Hadamard subderivative of $f$ at $x_0$ towards $v$ is given by
	$$
	f'_H \left(x_0; v \right) \coloneqq \liminf_{h \downarrow 0, u \to v} \delta_h f(x_0; u) \in \left[ - \i, \i \right].
	$$
	The above subderivatives are positively homogeneous with respect to the direction. The subdifferential belonging to such a subderivative is given by
	$$
	\p_i f(x_0) = \left\{ x' \in X^* \st x' \le f'_i \left( x_0; \cdot \right) \right\} = \dom f'_i \left( x_0; \cdot \right)^*, \quad i \in \left\{ D, D^+, H, r \right\}.
	$$
	
	\subsection{Semiconvexity}
	
	In this subsection, we investigate a class of functions exhibiting favorable subdifferential calculus rules. The class consists of those functions satisfying the Jensen inequality up to an error term, a property we call semiconvexity. Thus, convexity serves as our base line from which deviation is interpreted as a perturbation. The precise strength of the results available for a given semiconvex function will, of course, depend on the particular error term or modulus of semiconvexity, our approach being that we seek to characterize the worst acceptable behavior of an error that allows to recover a result that is analogous to the convex case. We are not the first to consider functions that may be fruitfully interpreted as being convex up to a perturbation. For example, various authors considered a proper subclass of the functions we investigate under the name approximately convex functions, cf., e.g., \cite{DaJuLa, NgLT, NgPe} and the references therein. In contrast to earlier works, we do not presuppose a certain behavior of the perturbation.
	
	\begin{definition} \label{def: smcnvx}
		The function $f$ is $p$-$\om$-semiconvex at $x_0$ if there exists a continuous seminorm $p$ such that, for all $x_1 \in X$ and $x_\lambda = \lambda x_1 + (1 - \lambda) x_0$ with $\lambda \in (0, 1)$, there holds
		\begin{equation} \label{eq: f smcnvx}
			f \left( x_\lambda \right) \le \lambda f(x_1) + (1 - \lambda) f(x_0) + \lambda (1 - \lambda) \om (x_1, x_0) p(x_1 - x_0).
		\end{equation}
		We say that $f$ is midpoint $p$-$\om$-semiconvex at $x_0$ if, for all $x_1 \in X$, there holds (\ref{eq: f smcnvx}) whenever $\lambda = \frac{1}{2}$.
	\end{definition}
	
	To clarify the scope of our theory, facilitating its application, we want to equivalently describe semiconvexity in different terms. For this, we generalize \cite[Thm. 10]{NgPe}, where a subdifferential characterization for approximately convex functions is proved. In contrast to \cite{NgPe}, we do not assume that the subdifferential to be valuable on $X$. We start with an auxiliary result and a slight extension of the Zagrodny mean value theorem needed in the following.
	
	\begin{definition} \label{def: preSD}
		Let $X$ be a Banach space and
		$$
		\p \colon \left( - \i, \i \right]^X \times X \to \PP \left( X^* \right) \colon (f, x) \mapsto \p f(x)
		$$
		a multimap. We call $\p$ a presubdifferential for a function $f \in \left( - \i, \i \right]^X$ if
		\begin{enumerate}
			\item $\p f(x) = \emptyset$ if $f(x) = \i$;
			\item $\p f(x) = \p g(x)$ if $f$ and $g$ agree around $x$;
			\item $\p f(x) = \p_{FM} f(x) = \left\{ \xi \in X^* \st f(y) \ge f(x) + \langle \xi, y - x \rangle \quad \forall y \in X \right\}$ for all $x \in X$ whenever $f \in \Gamma(X)$;
			\item $0 \in f(x)$ if $x$ is a local minimizer of $f$;
			\item $0 \in \limsup_{y \underset{f}{\to} x} \p f(y) + \p_{FM} g(x)$ whenever $g \in \Gamma(X)$ is continuous.
		\end{enumerate}
		Here, $\p_{FM}$ denotes the Fenchel-Moreau subdifferential of convex analysis.
	\end{definition}
	
	The upper set limit is the sequential strong-weak*-limit
	$$
	\limsup_{y \to_f x} \p f(y) = \left\{ \xi \in X^* \st \exists y_n \to_f x \land \xi_n \in \p f(y_n) \colon \xi_n \weakast \xi \right\}.
	$$
	
	\begin{proposition} \label{pr: seminorm SD}
		Let $p \colon X \to \left( - \i, \i \right]$ be a sublinear, lower semicontinuous function. Then the subdifferential $K = \p p(0)$ is non-empty and there holds $p(x) = \sup_{x' \in K} \langle x', x \rangle$ for all $x \in X$.
	\end{proposition}
	
	\begin{proof}
		There holds $p(0) = 0$ by definition of a sublinear function so that we may express the proper function $p$ as a supremum of affine continuous functions by \cite[§3.3, Cor. 1]{IT}. If $x' + \alpha \le p$ for $x' \in X^*$ and $\alpha \in \R$, then $x' \le p$ by positive homogeneity. Therefore, we may express $p$ as a supremum of linear continuous functions $x'$ with $x' \le p$ thus as the supremum over $x' \in \p p(0)$. In particular, the subdifferential is non-empty.
	\end{proof}
	
	\begin{theorem} \label{thm: zgrdn mvth}
		Let $\left( X, p \right)$ be a Banach space, $f \colon X \to \left( - \i, \i \right]$ be a lower semicontinuous function, $\p$ be a presubdifferential for $f$. If $a, b \in \dom(f)$ with $a \ne b$, then there exist $x_n \to_f c \in \left[ a, b \right)$, $x^*_n \in \p f(x_n)$ such that
		\begin{enumerate}
			\item $f(b) - f(a) \le \lim_n \langle x^*_n, b - a \rangle$;
			\item $\frac{p(x - c)}{p(b - a)} \left( f(b) - f(a) \right) \le \lim_n \langle x^*_n, x - x_n \rangle \quad \forall x \in \bigcup_{m \ge 1} \left[ c, m(b - c) \right]$;
			\item $p(b - a) \left( f(c) - f(a) \right) \le p(c - a) \left( f(b) - f(a) \right)$.
		\end{enumerate}
	\end{theorem}
	
	\begin{proof}
		The proof extends that of the Zagrodny mean value theorem presented in \cite{Thi}. We give full details for convenience of the reader. To simplify our proof, we may assume that $f(a) = f(b)$ through an affine transformation that does not impact the statement's validity. Let $c \in [a,b)$ be a point where $f$ attains its minimum value on $[a,b]$. Also, choose $r > 0$ such that $f$ is bounded below on $U = [a,b] + rB_X$ by $\gamma \in \R$. We define a function $f_U(x) = f(x)$ if $x \in U$, and $f(x) = \i$ otherwise. For each $n \ge 1$, we pick a real number $r_n \in (0,r)$ such that $f(x) \ge f(c) - \tfrac{1}{n^2}$ for all $x \in [a,b] + r_n B_X$. We also choose $t_n \geq n$ such that $\gamma + t_n r_n \ge f(c) - \tfrac{1}{n^2}$. Then, we obtain the following inequality:
		\begin{equation} \label{eq: mvth 1}
			f(c) \le \inf_{x \in X} f_U(x) + t_n \dist_{[a,b]}(x) + \frac{1}{n^2}.
		\end{equation}
		By the Ekeland variational principle \cite[Thm. 1.88]{Pe}, we find for the lower semicontinuous function $F_n = f_U + t_n \dist_{[a,b]}$ a point $x_n \in X$ satisfying the following conditions:
		\begin{gather}
			p(x_n - c) \le \frac{1}{n}; \\
			F_n(x_n) \le F_n(c) = f(c); \\
			F_n(x_n) \le F_n(x) + \frac{1}{n} p(x - x_n) \quad \forall x \in X. \label{eq: mvth 2}
		\end{gather}
		By (\ref{eq: mvth 1}), we may suppose $x_n \in \interior U$ for all $n \ge 1$. Hence, from (\ref{eq: mvth 2}) and the above properties, there exist $u_n^* \in \limsup_{x \to_f x_n} \p f(x)$, $v_n^* \in \p d_{[a,b]}(x_n)$, and $b_n^* \in B_X$ such that
		\begin{equation} \label{eq: mvth 3}
			- u^*_n = t_n v^*_n + \frac{1}{n} b^*_n.
		\end{equation}
		We define $L \coloneqq \bigcup_{m \geq 1} [c, m(b - c)]$. Since $x_n \to c \ne b$ and $\dist_{[a,b]}$ agrees with $\dist_L$ on a neighborhood of $c$, we eventually have $v_n^* \in \dist_L(x_n)$ by locality of the subdifferential. Therefore, we can choose $y_n \in [a,b]$ such that $p(x_n - y_n) = \dist_{[a,b]}(x_n)$. From (\ref{eq: mvth 1}), we have $y_N \to c$ and
		\begin{equation} \label{eq: mvth 6}
			\langle v^*_n, x - x_n \rangle \le \dist_L(x) - \dist_L(x_n) = - \dist_L(x_n) \le 0 \quad \forall x \in L,
		\end{equation}
		and, since $\| v^*_n \| \le 1$,
		\begin{equation} \label{eq: mvth 7}
			\begin{aligned}
				\langle v^*_n, b - y_n \rangle
				& = \langle v^*_n, b - x_n \rangle + \langle v^*_n, x_n - y_n \rangle \\
				& \le \dist_{[a, b]}(b) - \dist_{[a, b]}(x_n) + \| v^*_n \| p(x_n - y_n) \\
				& \le - \dist_L(x_n) + \dist_{[a, b]}(x_n)
				\le 0.
			\end{aligned}
		\end{equation}
		As $y_n \to c \ne b$, for $n$ sufficiently large, there holds $y_n \in \left[ a, b \right)$, which, by (\ref{eq: mvth 7}), implies
		\begin{equation} \label{eq: mvth 8}
			\langle v^*_n, b - a \rangle \le 0.
		\end{equation}
		Hence, using (\ref{eq: mvth 3}), (\ref{eq: mvth 6}), and (\ref{eq: mvth 7}), we obtain $\liminf_n \langle u^*_n, x - x_n \rangle \ge 0$ and $\liminf_n \langle u^*_n, b - a \rangle \ge 0$. Since $f$ is lower semicontinuous and $f(x_n) \le f(c)$, we have $f(x_n) \to f(c)$. To conclude, note that $u^*_n \in \limsup{x \to_f x_n} \p f(x)$ and consider subsequences.
	\end{proof}
	
	\begin{theorem} \label{thm: smcnvx}
		Each of the following properties implies the subsequent one:
		\begin{enumerate}[label=(\roman*)]
			\item\label{it. thm: smcnvx 1.1} $f$ is $p$-$\om$-semiconvex at $x_0$; or equivalently, if $0 < s < t$ and $v \in X$, then
			\begin{equation} \label{eq: slp nqlt}
				\delta_s f(x_0; v) \le \delta_t f(x_0; v) + \left( 1 - s / t \right)\om(x_0 + tv, x_0) p(v);
			\end{equation}
			\item\label{it. thm: smcnvx 1.2} $f$ is $p$-$\om$-subderivable at $x_0$, i.e., if $0 < t$ and $v \in X$, then
			$$
			f'_D(x_0; v) \le \delta_t f(x_0; v) + \om(x_0 + tv, x_0) p(v);
			$$
			\item\label{it. thm: smcnvx 1.3} $f$ is $p$-$\om$-Dini-subdifferentiable at $x_0$, i.e., if $0 < t$ and $v \in X$, then
			$$
			\liminf_{u \to v} f'_D(x_0; u) \le \delta_t f(x_0; v) + \om(x_0 + tv, x_0) p(v);
			$$
			Or equivalently
			$$
			\forall \xi \in \p_D f(x_0) \quad \xi(v) \le \delta_t f(x_0; v) + \om(x_0 + tv, x_0) p(v).
			$$
			\item\label{it. thm: smcnvx 1.4} $f$ is $p$-$\om$-Dini-Hadamard-subdifferentiable at $x_0$.
		\end{enumerate}
		Moreover, if $f$ is $p$-$\om$-$\p$-subdifferentiable on a relatively open subset $U \subset \dom(f)$ for a multimap
		$$
		\p \colon \left( -\i, \i \right]^X \times X \rightrightarrows X^*,
		$$
		then $\p f$ is $p$-$\bar{\om}$-semimonotone on $U$ with $\bar{\om}(x_0, x_1) = \om(x_0, x_1) + \om(x_1, x_0)$, i.e., if $x_i \in U$ and $x^*_i \in \p f(x_i)$, then
		\begin{equation} \label{it. thm: smcnvx 2.3}
			\langle x^*_0 - x^*_1, x_0 - x_1 \rangle \ge - \bar{\om}(x_0, x_1) p(x_0 - x_1).
		\end{equation}
		Finally, if $(X, p)$ is a Banach space and $\p$ is a presubdifferential for $f$ such that $\p f$ is $p$-$\om$-semimonotone on $U$, then $f$ is $p$-$\om$-semiconvex on $U$.
	\end{theorem}
	
	\begin{proof}
		\ref{it. thm: smcnvx 1.1} $\implies$ \ref{it. thm: smcnvx 1.2}: Taking the lower limit $s \to 0^+$ in \ref{it. thm: smcnvx 1.2} yields the claim. Regarding equivalence of the stated conditions, choosing $x_1 = x_0 + tv$ and setting $\lambda = s / t$ in the definition of $p$-$\om$-semiconvexity, we may rewrite it by subtracting $f(x_0)$ from both sides and dividing the resulting inequality by $s$ to equivalently find (\ref{eq: slp nqlt}).
		\ref{it. thm: smcnvx 1.2} $\implies$ \ref{it. thm: smcnvx 1.3}: This is immediate since $\liminf_{u \to v} f'_D(x_0; u) \le f'_D(x_0; v)$. Regarding equivalence of the conditions in \ref{it. thm: smcnvx 1.3}, by definition of $\p_D f(x_0)$, we have $\xi(v) \le \liminf_{u \to v} f'_D(x_0; u)$ for all $v \in X$ so that the first conditions implies the second. Conversely, the lower closure $\cl_v f'_D(x_0; v) = \liminf_{u \to v} f'_D(x_0; u)$ agrees with the supremum over $\p_D f(x_0)$ by Proposition \ref{pr: seminorm SD}.
		\ref{it. thm: smcnvx 1.3} $\implies$ \ref{it. thm: smcnvx 1.4}: This follows since $\p_H f$ is contained in $\p_D f$.
		
		Now, let $\xi_i \in \p f(x_i)$. Choosing $tv = x_1 - x_0$ in the definition of subdifferentiability yields
		$$
		\langle x^*_0, x_1 - x_0 \rangle \le f(x_1) - f(x_0) + \om(x_1, x_0) p(x_1 - x_0)
		$$
		and analogously
		$$
		\langle x^*_1, x_0 - x_1 \rangle \le f(x_0) - f(x_1) + \om(x_0, x_1) p(x_1 - x_0).
		$$
		Adding both inequalities yields the semimonotonicity
		\begin{equation*}
			\langle \xi_0 - \xi_1, x_0 - x_1 \rangle \ge - \bar{\om}(x_1, x_0) p(x_1 - x_0).
		\end{equation*}
		
		We adapt an idea from \cite[Thm. 10]{NgPe} to prove the final implication. For points $x_0, x_1 \in U$, we set $x_\lambda \coloneqq (1 - \lambda) x_0 + \lambda x_1$, $\lambda \in (0, 1)$. Invoking Theorem \ref{thm: zgrdn mvth} for $f$ on the interval $\left[ x_0, x_\lambda \right]$ and the Dini-subdifferential $\p_D$, we find for any real number $r < f(x_\lambda)$ a point $y \in \left[ x_0, x_\lambda \right)$ and sequences $y_n \to y$, $y^*_n \in \p_D f(y_n)$ such that
		\begin{equation} \label{eq: sd nqlt}
			\frac{r - f(x_0)}{p(x_\lambda - x_0)} < \lim_{n \uparrow \i} \langle y^*_n , \frac{x_1 - y_n}{p(x_1 - y_n)} \rangle.
		\end{equation}
		Let $s \in (0, 1)$ be such that $x_\lambda = x_1 + s(y - x_1)$ and set $x^n_\lambda = x_1 + s(y_n - x_1)$. Since $y_n \to y$, there holds $x^n_\lambda \to x_\lambda$ and, for all $n$ sufficiently large, $r < f \left( x^n_\lambda \right)$ by lower semicontinuity. Moreover $p(x^n_\lambda - x_1) = (1 - \lambda_n) p(x_0 - x_1)$ for a convergent sequence $\lambda_n \to \lambda$. Invoking again Theorem \ref{thm: zgrdn mvth} for $f$ on $\left[ x_1, x^n_\lambda \right]$, we find $z_n \in \left[ x_1, x^n_\lambda \right)$ and sequences $z_{n, m} \to z_n$, $z^*_{n, m} \in \p f(z_{n, m} )$ such that
		\begin{equation} \label{eq: sd nqlt 2}
			\lim_{m \uparrow \i} \langle z^*_{n, m}, \frac{y_n - z_{n, m} }{p(y_n - z_{n, m})} \rangle > \frac{r - f(x_1)}{p(x^n_\lambda - x_1)} = \frac{r - f(x_1)}{(1 - \lambda_n) p(x_0 - x_1)}.
		\end{equation}
		By (\ref{eq: sd nqlt}), we know that eventually $\left[ x_1, y_n \right] \subset U$ and
		\begin{equation} \label{eq: sd nqlt 3}
			\langle y^*_n , \frac{z_n - y_n}{p(z_n - y_n)} \rangle
			= \langle y^*_n , \frac{x_1 - y_n}{p(x_1 - y_n)} \rangle
			> \frac{r - f(x_0)}{p(x_\lambda - x_0)} = \frac{r - f(x_0)}{\lambda p(x_0 - x_1)}.
		\end{equation}
		On the other hand, by (\ref{eq: sd nqlt 2}) and (\ref{eq: sd nqlt 3}), for every fixed $n$ there eventually hold the inequalities
		\begin{equation} \label{eq: sd nqlt 4}
			\begin{aligned}
				\langle y^*_n, \frac{z_{n, m} - y_n}{p(z_{n, m} - y_n)} \rangle & > \frac{r - f(x_0)}{\lambda p(x_0 - x_1)}, \\
				\langle z^*_{n, m}, \frac{y_n - z_{n, m}}{p(y_n - z_{n, m})} \rangle & > \frac{r - f(x_1)}{(1 - \lambda_n) p(x_0 - x_1)},
			\end{aligned}
		\end{equation}
		as $m \uparrow \i$ since $\lim_{m \uparrow \i} z_{n, m} = z_n$. Adding the inequalities (\ref{eq: sd nqlt 4}), we may invoke (\ref{eq: sd nqlt}) to find
		$$
		\om(x_1, x_0) \ge \frac{r - f(x_0)}{\lambda p(x_0 - x_1)} + \frac{r - f(x_1)}{(1 - \lambda_n) p(x_0 - x_1)}.
		$$
		Sending $n \uparrow \i$ yields
		$$
		\om(x_1, x_0) \lambda(1 - \lambda) p(x_0 - x_1) \ge (1 - \lambda) (r - f(x_0) ) + \lambda (r - f(x_0) )
		$$
		so that sending $r \uparrow f(x_\lambda)$, we recognize $f$ as $p$-$\om$-semiconvex on $U$.
	\end{proof}
	
	The proof remains valid if $\om$ is valued in $\left( - \i, \i \right]$. It is in general possible that the semiconvexity error term $\om_f$ of a function is strictly better than the semimonotonicity error term of its subdifferential. Consider the (continuously differentiable) example $f \colon \R \to \R \colon x \mapsto - x^2$ with
	\begin{gather*}
		p(x_1 - x_0) = \om_f(x_1, x_0) = \left| x_1 - x_0 \right|; \quad \\
		\langle f'(x_1) - f'(x_0), x_1 - x_0 \rangle = - 2 \om_f(x_1, x_0)p(x_1 - x_0).
	\end{gather*}
	
	It is useful to know if a function is radially differentiable and if the derivative is subadditive. For it is easier to compute a radial derivative than a Fréchet or Hadamard one, whereas a subadditive derivative is amenable to convex analysis. Therefore, the next lemma characterizes both properties in terms of $p$-$\om$-semiconvexity, thus instructing how to check them.
	
	\begin{lemma} \label{lem: dffblty}
		If $f$ is $p$-$\om$-semiconvex at $x_0$ with $\lim_{h \downarrow 0} \om(x_0 + hv, x_0) = 0$ for $v \in X$, then the radial derivative $f'_r(x_0; v) \in \left[ -\i, \i \right]$ exists. If $\left| f'_D(x_0; v ) \right| < \i$, then this condition is also necessary for every $p$ norming $\lin \left( x_0, v \right)$.
		Second, if $f$ is midpoint $p$-$\om$-semiconvex around $x_0$ with
		\begin{equation} \label{eq: mdls ssmptn}
			\begin{gathered}
				\liminf_{h \downarrow 0} \left\{ \delta_h f(x_0; v) + \delta_h f(x_0; w) + \om(x_0 + hv, x_0 + hw) p(v - w) \right\} \\
				\le f'_D(x_0; v) + f'_D(x_0; w),
			\end{gathered}
		\end{equation}
		then the Dini derivative $f'_D(x_0; \cdot)$ is subadditive for $v, w \in X$, i.e.,
		\begin{equation} \label{eq: sbddtvty}
			f'_D(x_0; v + w) \le f'_D(x_0; v) + f'_D(x_0; w).
		\end{equation}
		Third, if there exist $f'_r(x_0; v), f'_r(x_0; w), f'_r(x_0; v) + f'_r(x_0, w) \in \left[ -\i, \i \right]$ and $\left| f'_D(x_0; v + w) \right| < \i$, then (\ref{eq: sbddtvty}) holds if and only if $f$ is midpoint $p$-$\om$-semiconvex around $x_0$ for every $p$ that norms $\lin\left( x_0, v, w \right)$ and some $\om$ that satisfies $\liminf_{h \downarrow 0} \om(x_0 + hv, x_0 + hw) = 0$.
	\end{lemma}
	
	\begin{proof}
		We pick for $\e > 0$ an $h_\e \in (0, \e)$ such that $\delta_{h_\e} f(x_0; v) - \e \le f'_D(x_0; v)$, which together with Property \ref{it. thm: smcnvx 1.2} of Theorem \ref{thm: smcnvx} implies
		\begin{equation} \label{eq: dq ineq cnsq}
			\begin{aligned}
				\delta_h f(x_0; v)
				& \le \delta_{h_\e} f(x_0; v) + \om \left( x_0 + h_\e v, x_0 \right) p(v) \\
				& \le f'_D(x_0; v) + \e + \om \left( x_0 + h_\e v, x_0 \right) p(v) \quad \forall h \in \left( 0, h_\e \right).
			\end{aligned}
		\end{equation}
		Consequently, taking the upper limit $h \downarrow 0$ yields
		$$
		f'_{D^+}(x_0; v) \le f'_D(x_0; v) + \e + \om \left( x_0 + h_\e v, x_0 \right) p(v),
		$$
		whence, sending $\e \downarrow 0$, we conclude that $f'_{D^+}(x_0; v) \le f'_D(x_0; v)$, i.e., $f'_r(x_0; v)$ exists. Regarding necessity, we set
		\begin{align*}
			\bar{\om}(x_0 + hv, x_0)
			& \coloneqq \sup_{x \in \left[ x_0, x_0 + hv \right]} \left| \tfrac{f(x) - f(x_0) - f'_D(x_0; x - x_0)}{p(x - x_0)} \right| \\
			& = \sup_{0 \le k \le h} \left| \tfrac{\delta_k f(x_0; v) - f'_D(x_0; v)}{p(v)} \right|
		\end{align*}
		for any $p$ norming $\lin\left( x_0, v \right)$ - at least one such $p$ exists since $X$ is Hausdorff - so that on the one hand $\lim_{h \downarrow 0} \bar{\om}(x_0 + hv, x_0) = 0$ as
		$$
		\left| f'_D(x_0; v) \right| < \i,
		$$
		on the other hand
		\begin{align*}
			\delta_h f(x_0; v) + \bar{\om}(x_0 + hv, x_0) p(v)
			& \ge f'_D(x_0; v) + \bar{\om}(x_0 + hv, x_0) p(v) \\
			& \ge \delta_k f(x_0; v) \quad \forall k \in \left( 0, h \right].
		\end{align*}
		(\ref{eq: sbddtvty}): By positive homogeneity, midpoint semiconvexity, and (\ref{eq: mdls ssmptn}), we find
		\begin{equation} \label{eq: brdil cntrl}
			\begin{aligned}
				& f'_D(x_0; v + w) = \liminf_{h \downarrow 0} \delta_{h / 2} f \left( x_0; \tfrac{v + w}{2} \right) \\
				& \le \liminf_{h \downarrow 0} \delta_h f(x_0; v) + \delta_h f(x_0; w) + \om \left( x_0 + h v, x_0 + h w \right) p(v - w) \\
				& = f'_D(x_0; v) + f'_D(x_0; w),
			\end{aligned}
		\end{equation}
		whence (\ref{eq: sbddtvty}) obtains. Regarding the addendum, we have already seen that the condition is sufficient. Necessity: Setting
		$$
		\tilde{\om}(x, y) p(x - y) = 4 \left\{ 2 f \left( \tfrac{x + y}{2} \right) - f(x) - f(y) \right\}^+,
		$$
		for any $p$ norming $\lin\left( x_0, v, w \right)$, we have
		$$
		\tilde{\om}(x_0 + hv, x_0 + hw) p(v - w) = 4 \left\{ 2 \delta_h f \left( x_0; \tfrac{v + w}{2} \right) - \delta_h f(x_0; v) - \delta_h f(x_0; w) \right\}^+
		$$
		so that, on the one hand, $\liminf_{h \downarrow 0} \tilde{\om}(x_0 + hv, x_0 + hw) = 0$ since $\left| f'_D(x_0; v + w ) \right| < \i$, on the other hand,
		\begin{align*}
			f \left( x_0 + \tfrac{v + w}{2} \right) \le \tfrac{1}{2} f(x_0 + v) + \tfrac{1}{2} f(x_0 + w) + \tfrac{1}{4} \tilde{\om}(x_0 + hv, x_0 + hw) p(v - w),
		\end{align*}
		i.e., $f$ is midpoint $p$-$\om$-semiconvex.
	\end{proof}
	
	The existence of the modulus $\om$ in Lemma \ref{lem: dffblty} is no longer necessary for $f'_r(x_0; v)$ to exist or for $f'_D(x_0; \cdot)$ to be subadditive if $\left| f'_D(x_0; v) \right| = \i$. It suffices to provide counterexamples for the first case as choosing the trivial direction $w = 0$ then covers the second. Let $q \colon \R \setminus \left\{ 0 \right\} \to \R$ be a continuous function such that there exists a sequence $a_n \downarrow 0$ satisfying
	$$
	0 < \lim_n q(2 a_n) = \liminf_{x \to 0} < \limsup_{x \to 0} q(x) = \lim_n q(a_n) < 1.
	$$
	We set $f(x) \coloneqq \left| x \right|^{q(x) }$ and $g(x) \coloneqq - f(x)$. Since $\delta_h f(0; 1) < \i$ but $\delta_h f(0, 1) \uparrow \i$, no semiconvexity modulus satisfying $\lim_{h \downarrow 0} \om(h, 0) = 0$ exists for $f$. Setting $\underline{\theta} = \liminf_{x \to 0} q(x)$ and $\overline{\theta} = \limsup_{x \to 0} q(x)$, there holds
	$$
	\delta_{a_n} g(0, 1) - \delta_{2 a_n} g(0, 1) \thickapprox \left| a_n \right|^{\underline{\theta} - 1} \left( 2^{\underline{\theta} - 1} - \left| a_n \right|^{\overline{\theta} - \underline{\theta} } \right) \to \i \text{ as } a_n \downarrow 0
	$$
	so that indeed there exists no vanishing modulus satisfying
	$$
	\delta_{a_n} g(0, 1) - \delta_{2 a_n} g(0, 1) \le \om(2 a_n, 0) \quad \forall n \in \N.
	$$
	
	It facilitates the application of non-smooth analysis if different subdifferentials coincide for a given function, as results requiring these may then be combined, leading to stronger conclusions. Motivated by this, the next lemma characterizes when $\p_D f(x_0)$ and $\p_H f(x_0)$ coincide, using the notion of $p$-$\om$-subdifferentiability.
	
	\begin{theorem} \label{thm: SD cncdnc}
		If the function $f$ is $p$-$\om$-Dini-subdifferentiable at $x_0$ with
		\begin{equation} \label{eq: drly sbdrvbl}
			\liminf_{u \to v, h \downarrow 0} \delta_h f(x_0; u) + \om(x_0 + h u, x_0) p(u) \le f'_H(x_0; v),
		\end{equation}
		then
		\begin{equation} \label{eq: D H sbdrvtv gr}
			f'_H(x_0; v) = \cl f'_D(x_0; v) \coloneqq \liminf_{u \to v} f'_D(x_0; u).
		\end{equation}
		Conversely, if (\ref{eq: D H sbdrvtv gr}) and $\left| \cl f'_D(x_0; v) \right| < \i$, then there exists an error term $\om$ such that, for every $p$ norming $\lin\left( x_0, v \right)$, the function $f$ is $p$-$\om$-Dini-subdifferentiable at $x_0$ and (\ref{eq: drly sbdrvbl}) holds. In particular, if (\ref{eq: drly sbdrvbl}) for all $v \in E$ with
		$$
		E = \left\{ v \in X \st f'_D(x_0; v) > -\i \right\},
		$$
		then
		\begin{equation} \label{eq: D H SD gr}
			\p_H f(x_0) = \p_D f(x_0).
		\end{equation}
		Conversely, if (\ref{eq: D H SD gr}) as an identity of non-empty sets and $f'_D(x_0; v) = \i$ for $v \in X$, then (\ref{eq: drly sbdrvbl}) for every continuous seminorm $p$ norming $\lin\left( x_0, v \right)$.
	\end{theorem}
	
	\begin{proof}
		$\implies$: Trivially, $f'_H \le f'_D$, hence $f'_H \le \cl f'_D$ since the map $v \mapsto f'_H \left( x_0; v \right)$ is lower semicontinuous as a $\Gamma$-$\liminf$. For the converse inequality, let $h_n > 0$ and $u_n$ with $\lim_n h_n = 0$, $\lim_n u_n = v$ such that
		$$
		\lim_n \delta_{h_n} f(x; u_n) + \om(x_0 + h_n u_n, x_0) p(u_n) = f'_H(x; v).
		$$
		Then, since $f$ is $p$-$\om$-Dini-subdifferentiable, there holds
		$$
		\cl f'_D \left( x; u_n \right) \le \delta_{h_n} f(x; u_n) + \om \left( x + h_n u_n, x \right) p(u_n).
		$$
		Hence, by (\ref{eq: drly sbdrvbl}), we find
		\begin{align*}
			\cl f'_D \left( x; v \right)
			& \le \liminf_n \cl f'_D \left( x; u_n \right) \\
			& \le \lim_n \delta_{h_n} f(x; u_n) + \om(x_0 + h_n u_n, x_0) p(u_n)
			= f'_H \left( x; v \right).
		\end{align*}
		$\impliedby$: If $\left| \cl f'_D(x_0; v) \right| < \i$, we set
		$$
		\om(x_0 + hu, x_0) p(u) = \left\{ \cl f'_D(x_0; u) - \delta_h f(x; u) \right\}^+
		$$
		so that $\cl f'_D(x_0; u) \le \delta_h f(x_0; u) + \om(x_0 + hu, x_0) p(u)$ for all $h > 0$ and $u \in X$, i.e., the claimed $p$-$\om$-subdifferentiability at $x_0$. By our assumption, there holds $\liminf_{u \to v, h \downarrow 0} \delta_h f(x_0; u) + \om(x_0 + hu, x_0) p(u) = f'_H(x_0; v)$, whence (\ref{eq: D H sbdrvtv gr}) follows by the lower semicontinuity of $\cl f'_D(x_0; \cdot)$.
		
		Addendum, $\implies$: If (\ref{eq: drly sbdrvbl}) for all $v \in E$, then (\ref{eq: D H sbdrvtv gr}) for all $v \in E$ by what has been proved. The equality (\ref{eq: D H sbdrvtv gr}) being trivial for $v \in X \setminus E$ by $f'_H \le f'_D$, we conclude (\ref{eq: D H sbdrvtv gr}) for all $v \in X$, hence (\ref{eq: D H SD gr}).
		
		$\impliedby$: From (\ref{eq: D H SD gr}) follows $f'_H(x_0; v) = \sup_{\xi \in \p_H f(x_0)} \langle \xi, v \rangle = \sup_{\xi \in \p_D f(x_0)} \langle \xi, v \rangle = \cl f'_D(x_0; v)$ for all $v \in X$ by Proposition \ref{pr: seminorm SD}. The function $f'_D(x_0; \cdot)$ being proper because $\p_D f(x_0) \ne \emptyset$, the claim follows by what has been proved before.
	\end{proof}
	
	A convex function is continuous at a point iff it is bounded above around the point. This easy criterion contributes greatly to the theory of convex functions, cf, e.g., all instances where \cite[§3.2, Thm. 1]{IT} is invoked by the authors. Inspired by this, we generalize this classical result to a class of semiconvex functions. We sharpen the statement if $X$ is normed, treating situations where $x_0$ belongs to the boundary of $\dom(f)$. This refines the classical continuity theorem even for convex functions.
	
	\begin{theorem} \label{thm: bd abv cnt}
		Let $f$ be $p$-$\om$-semiconvex at $x_0$ with $\limsup_{x \to x_0} \om(x, x_0) < \i$. If $f$ is bounded above around $x_0$, then $f$ is upper semicontinuous at $x_0$. If moreover $\limsup_{x, y \to x_0} \om(x, y) < \i$, then $f$ is continuous at $x_0$. Finally, if $X$ is a normed space and $f$ is merely bounded above around $x_0$ on $\dom(f)$, then the first conclusion continues to hold if merely $\limsup_{\dom(f) \ni x \to x_0} \om(x, x_0) < \i$.
	\end{theorem}
	
	\begin{proof}
		Upper semicontinuity: By translation, we may assume $f(x_0) = 0$ and $x_0 = 0$. Let $q \ge p$ be a seminorm such that $f$ and $\om(\cdot, x_0)$ are bounded above on $U = \left\{ q \le 1 \right\}$. By scaling, we may assume $f \le 1/2$ and $\om(\cdot, x_0) \le 1/2$ on $U$. By $p$-$\om$-semiconvexity and $q \ge p$, we find, for $h > 0$ fixed,
		\begin{equation} \label{eq: ppr stmt}
			f(x_0 + h u)
			\le f(x_0) + h \left( \delta_1 f(x_0; u) + \om(x_0 + u, x_0) p(u) \right)
			\le f(x_0) + 2 h.
		\end{equation}
		Consequently, upper semicontinuity obtains by
		$$
		\limsup_{x \to x_0} f(x) = \inf_{V \in N(x_0) } \sup_{u \in V} f(x_0 + u) \le \lim_{h \downarrow 0} \sup_{u \in U} f(x_0 + hu) \le f(x_0).
		$$
		Lower semicontinuity: By further scaling, we may suppose $\sup_{x, y \in U} \om(x, y) \le 1$. For $\e \in (0, 1)$, let $V_\e \coloneqq \tfrac{\e}{2} U$. By $0 = (1 + \e)^{-1} x + \e (1 + \e)^{-1} ( - \e^{-1} x)$ and $p$-$\om$-semiconvexity, there holds
		$$
		0 = f(0) \le (1 + \e)^{-1} f(x) + \e (1 + \e)^{-1} f \left( - \e^{-1} x \right) + \frac{1}{4} \om(x, - \e^{-1} x) p(x)
		$$
		i.e., $f(x) \ge - \e$ if $x \in V_\e$. Addendum: Let $x_n \in \dom(f)$ be a sequences converging to $x_0$. Setting $x_n = x_0 + h u$ with $h = \| x_n - x \|$ in (\ref{eq: ppr stmt}) yields
		$$
		\limsup_n f(x_n) \le f(x_0)
		$$
		so that $f$ is upper semicontinuous at $x_0$ on $\dom(f)$.
	\end{proof}
	
	The addendum in Theorem \ref{thm: bd abv cnt} may fail in a non-norm topology. For example, on a Banach space $X$, the convex function $f(x) = \| x \| + I_{B_X}(x)$ is bounded above on $B_X = \dom(f)$, but lacks weak continuity in general.
	
	Points of continuity are a qualification condition for the application of calculus rules, cf., e.g., \cite[Thm. 2]{Roc 2}. For this reason, the next lemma gives a sufficient condition for them.
	
	\begin{lemma} \label{lem: cnt rdr}
		Let $f$ be finite continuous at all points of a set $U \subset X$ and locally bounded above on $\interior \dom(f)$. If $\dom(f)$ is convex, $f'_D \left( x_0; \cdot \right)$ is subadditive, there exists $\bar{x} \in X$ such that $x_0 + \bar{x} \in U$, $f'_D \left( x_0; \bar{x} \right) > -\i$, and $f$ is $p$-$\om$-subderivable at $x_0$ with
		\begin{equation} \label{eq: fin limsup}
			\limsup_{h \downarrow 0, x \to \bar{x}} \om \left( x_0 + h x , x_0 \right) < \i,
		\end{equation}
		then $f_D \left(x_0; \cdot \right)$ is proper and continuous at all points of the cone $K$ generated by the set $U - x_0$ except, possibly, at the origin.
	\end{lemma}
	
	\begin{proof}
		The function $f'_D (x_0; \cdot)$ being positively homogeneous, we are to verify that it is continuous on $U - x_0$ by \cite[§4.2, Prop. 2]{IT}. To prove that it is proper, we suppose that there exists $x_1 \in X$ such that $f'_D(x_0; x_1) = - \i$. By continuity, the point $x_0 + x_2 \coloneqq x_0 + \bar{x} + \e \left( \bar{x} - x_1 \right)$ belongs to $\interior \dom(f)$ if $\e > 0$ is sufficiently small. Since $x_0 + x_2$ belongs to the convex set $\dom(f)$, there holds $\delta_h f(x_0; x_2) < \i$. Therefore, by further decreasing $\e > 0$ if necessary and choosing $h > 0$ sufficiently small, we may arrange
		\begin{equation} \label{eq: f om sbdbl}
			f'_D(x_0; x_2) \le \delta_h f(x_0; x_2) + \om(x_0 + hx_2, x_0) p(x_2) < \i
		\end{equation}
		by (\ref{eq: fin limsup}). Consequently, subadditivity yields
		$$
		f'_D(x_0; \bar{x} ) \le \left( 1 + \e \right)^{-1} f'_D(x_0; x_2) + \e \left( 1 + \e \right)^{-1} f'_D(x_0; x_1) = - \i,
		$$
		which contradicts the assumption $f'_D(x_0; \bar{x} ) > - \i$. Consequently, $f'_D(x; \cdot)$ is proper. If $x_1 \in U - x_0$, then there exists $h > 0$ and a sufficiently small neighborhood $V$ of the point $x_1$ such that
		\begin{equation} \label{eq: ppr bnd om}
			\sup_{v \in V} \om (x_0 + h v, x_0 ) < \i
		\end{equation}
		by (\ref{eq: fin limsup}). Since $\dom(f)$ is convex and $x_0 + x_1 \in \interior \dom(f)$, we may decrease $V$ if necessary to arrange $x_0 + h V \subset \interior \dom(f)$. As $f$ is locally bounded above on $\interior \dom(f)$, we may, by possibly further decreasing $V$, also arrange that
		\begin{equation} \label{eq: ppr bnd f}
			\sup_{v \in V} f(x_0 + h v) < \i.
		\end{equation}
		Putting (\ref{eq: ppr bnd om}) and (\ref{eq: ppr bnd f}) together, we arrive at
		$$
		\sup_{v \in V} f'_D(x_0; v) \le \sup_{v \in V} \delta_h f(x_0; v) + \sup_{v \in V} \om(x_0 + hv, x_0) p(v) < \i
		$$
		by (\ref{eq: f om sbdbl}), i.e., $f'_D(x_0; \cdot)$ is bounded above around $x_1$, whence Theorem \ref{thm: bd abv cnt} implies that it is continuous at $x_1$.
	\end{proof}
	
	The next theorem generalizes the classical sum rule of convex analysis to the Dini subdifferential. This will lead to a sum rule for $p$-$\om$-subderivable functions as a corollary.
	
	\begin{theorem} \label{thm: sum rule}
		Let $f, g \colon X \to \left( -\i, \i \right]$ be functions. Then
		\begin{equation} \label{eq: trvl sd nclsn}
			\p_D f(x_0) + \p_D g(x_0) \subset \p_D \left( f + g \right)(x_0) \quad \forall x_0 \in \dom(f) \cap \dom(g).
		\end{equation}
		Moreover, let $f'_D(x_0; \cdot), g'_D(x_0; \cdot)$ be proper and subadditive satisfying
		\begin{equation} \label{eq: ddtv drvtv}
			\left( f + g \right)'_D \left( x_0; \cdot \right) \le f'_D \left( x_0; \cdot \right) + g'_D \left( x_0; \cdot \right).
		\end{equation}
		For example, (\ref{eq: ddtv drvtv}) is true if $f$ or $g$ is radially differentiable at $x_0$. Then
		\begin{equation} \label{eq: ddtv sd}
			\p_D \left( f + g \right)(x_0) = \overline{\p_D f(x_0) + \p_D g(x_0)}
		\end{equation}
		if and only if
		\begin{equation} \label{eq: ddtv clsr}
			\cl \left( f + g \right)'_D(x_0; \cdot) = \cl f'_D(x_0; \cdot) + \cl g'_D(x_0; \cdot).
		\end{equation}
		Furthermore, if the upper bound in (\ref{eq: trvl sd nclsn}) is non-empty and there exists $\bar{x} \in X$ such that $f'_D(x_0; \bar{x} ) < \i$ and $g'_D(x_0; \cdot)$ is continuous at $\bar{x}$, then the requirement that $f'_D(x_0; \cdot)$ be proper is redundant and there holds
		\begin{equation} \label{eq: ddtv sd 2}
			\p_D \left( f + g \right)(x_0) = \p_D f(x_0) + \p_D g(x_0).
		\end{equation}
	\end{theorem}
	
	\begin{proof}
		(\ref{eq: trvl sd nclsn}): If there exists $v \in X$ such that $f'_D(x_0; v) = - \i$ or $g'_D(x_0; v) = -\i$, then $\p_D f(x_0)$ or $\p_D g(x_0)$ is empty so nothing remains to prove. Else
		$$
		f'_D(x; v) + g'_D(x; v) \le \left( f + g \right)'_D (x_0; v) \quad \forall v \in X
		$$
		so that the claim obtains by definition of $\p_D$. Equivalence of (\ref{eq: ddtv sd}) and (\ref{eq: ddtv clsr}): Apply \cite[Thm. 3.3]{Ro} to the sum of $f'_D(x_0; \cdot)$ and $g'_D(x_0; \cdot)$, which yields the subdifferential $\p_D \left( f + g \right)(x_0)$ by (\ref{eq: ddtv drvtv}). Note in this regard that $f'_D(x_0; \cdot)$ and $g'_D(x_0; \cdot)$ agree with their own Dini-subderivative at zero by homogeneity. We used that a sublinear function is Dini-subdifferentiable at the origin iff it is proper. (\ref{eq: ddtv sd 2}): We start with the addendum, claiming that $f'_D(x_0; \cdot)$ is proper. Arguing by contradiction, we assume that there exists $y \in X$ such that
		$$
		f'_D (x_0; y) = - \i.
		$$
		We set $x_\lambda = (1 - \lambda) \bar{x} + \lambda y$ for $\lambda \in \left( 0, 1 \right)$. By continuity,
		$$
		\lim_{\lambda \downarrow 0} g'_D (x_0; x_\lambda) = g'_D (x_0; \bar{x}) < \i
		$$
		so that, for $\lambda$ sufficiently close to zero, we conclude
		\begin{align*}
			- \i < \langle x', x_\lambda \rangle
			& \le \left( f + g \right)'_D (x_0; x_\lambda )
			= f'_D (x_0; x_\lambda ) + g'_D (x_0; x_\lambda ) \\
			& \le \left( 1 - \lambda \right) f'_D (x_0; \bar{x}) + \lambda f'_D (x_0; y)
			+ g'_D (x_0; x_\lambda)
			= - \i.
		\end{align*}
		Now, since $f'_D(x_0; \cdot)$, $g'_D(x_0, \cdot)$ are homogeneous, subadditive hence convex, \cite[§3.4, Thm. 1]{IT} and (\ref{eq: ddtv drvtv}) yield
		\begin{align*}
			\p_D f(x_0) + \p_D g(x_0)
			& = \dom f'_D (x_0; \cdot )^* + \dom g'_D (x_0; \cdot )^* \\
			& = \dom \left( f'_D + g'_D \right) (x; \cdot )^* \\
			& \supset \dom \left( f + g \right)'_D (x; \cdot )^* \\
			& = \p_D \left( f + g \right)(x_0). \tag*{\qedhere} 
		\end{align*}
	\end{proof}
	
	The assumptions on $g$ for (\ref{eq: ddtv sd 2}) to hold are satisfied, for example, if it is convex and has a suitable point of continuity; or if $g$ is locally Lipschitz continuous and $g'_D(x_0; \cdot)$ is subadditive.
	
	\begin{corollary} \label{cor: sum rule}
		Let $f, g \colon X \to \left( - \i, \i \right]$ be functions such that $\p \left( f + g \right)(x_0)$ is non-empty, and $f'_D(x_0; \cdot), g'_D(x_0; \cdot)$ are subadditive. If $\dom(g)$ is convex, $g$ is bounded above on $\interior \dom(g)$ and continuous at a point $\bar{x} \in \dom(f) \cap \dom(g)$ such that $\left[ x_0, \bar{x} \right] \subset \dom(f) \cap \dom(g)$, and $f, g$ are $p$-$\om$-subderivable at $x_0$ with
		\begin{equation} \label{eq: sbdrvblts}
			\liminf_{h \downarrow 0} \om_f \left( x_0 + h \bar{x}, x_0 \right) < \i, \quad \limsup_{h \downarrow 0, u \to \bar{x}} \om_g \left( x_0 + hu, x_0 \right) < \i.
		\end{equation}
		Then $f'_D(x_0; \cdot)$ and $g'_D(x_0; \cdot)$ are proper. Moreover, if (\ref{eq: ddtv drvtv}) is satisfied, then (\ref{eq: ddtv sd 2}) is true.
	\end{corollary}
	
	\begin{proof}
		Let $v = \bar{x} - x_0$. By Theorem \ref{thm: sum rule}, it suffices to show that (i) $f'_D(x_0; v) < \i$ (ii) $g'_D(x_0; \cdot)$ is proper and continuous at $v$. (i): By the first of (\ref{eq: sbdrvblts}), there exists a sequence $h_n \downarrow 0$ such that
		$$
		f'_D (x; v)
		\le \delta_{h_n} f(x; v) + \om_f \left( x_0 + h_n v, x_0 \right) p(v) < \i.
		$$
		(ii): We claim that $g'_D(x_0; \cdot)$ is proper and continuous at $v$ by Lemma \ref{lem: cnt rdr} and check its assumptions: $g$ is continuous at $\bar{x}$; it is locally bounded above on $\interior \dom(g)$; $\dom(g)$ is convex; $g'_D(x_0; \cdot)$ is subadditive; $g$ is $p$-$\om$-subderivable with $\om_g$ satisfying (\ref{eq: sbdrvblts}).
	\end{proof}
	
	Frequently, results for evolution equations rely on the fact that the subdifferential of a convex potential has good closedness properties, i.e., if $f$ is convex, lower semicontinuous, then
	\begin{equation} \label{ga. clsd grph}
		\begin{gathered}
			x_n \ssto x, \, f(x_n) \to \alpha, \, x^*_n \in \p f(x_n), \, x^*_n \ssto x^*, \, \langle x^*_n, x_n \rangle \to \langle x^*, x \rangle \implies \\
			\alpha = f(x), \, x^* \in \p f(x).
		\end{gathered}
	\end{equation}
	Often in such investigations, convexity is non-essential, whereas (\ref{ga. clsd grph}) is crucial. Therefore, we study subdifferential closedness of $p$-$\om$-Dini-subdifferentiable functions. We consider a situation where the topology of $X$ is not necessarily induced by the duality for which we consider the subdifferential, as exemplified by our abstract result on generalized gradient flows and its application in Section \ref{sec: Application}.
	
	\begin{lemma} \label{lem: sd clsdnss}
		Let $X, Y$ be a dual pair of locally convex Hausdorff spaces. Let $q \colon X \to \left[ 0, \i \right]$ be a sublinear function, such that $f$ is $q$-$\om$-Dini-subdifferentiable at $x$, i.e., for every $v \in X$ and $h > 0$, there holds
		\begin{equation} \label{eq: q-om-sbdffrntbl}
			\langle \xi, v \rangle \le \delta_h f(x; v) + \om(x + hv, x) q(v) \quad \forall \xi \in \p_D f(x).
		\end{equation}
		Let there be given nets $x_\alpha \in X$ and $\xi_\alpha \in \p_D f(x_\alpha)$ such that
		\begin{equation} \label{eq: sqncs cnvrg}
			\begin{gathered}
				x_\alpha \to x, \quad f(x) \le \liminf f(x_\alpha),
				\\
				\xi_\alpha \to \xi \text{ in } \ss(Y, X), \quad \limsup \langle \xi_\alpha; x_\alpha \rangle \le \langle \xi, x \rangle.
			\end{gathered}
		\end{equation}
		If
		\begin{equation} \label{eq: cl cnd 1.1}
			\lim \om \left( x, x_\alpha \right) q(x_\alpha - x) = 0,
		\end{equation}
		then
		\begin{equation} \label{eq: f cnt cvg}
			\lim f(x_\alpha) = f(x).
		\end{equation}
		Let $E = \left\{ v \in X \st f'_D(x; v) < \i \right\}$. If (\ref{eq: f cnt cvg}) and
		\begin{equation} \label{eq: cl cnd 1.2}
			\begin{gathered}
				\forall v \in E \quad \liminf_{h \downarrow 0} \left\{ \delta_h f(x; v) + \liminf_\alpha \om(x + hv, x_\alpha) q \left( \tfrac{x - x_\alpha}{h} + v \right) \right\} \\
				\le f'_D(x; v),
			\end{gathered}
		\end{equation}
		then $\xi \in \p_D f(x)$.
	\end{lemma}
	
	\begin{proof}
		For every $v \in X$, by (\ref{eq: sqncs cnvrg}) and (\ref{eq: q-om-sbdffrntbl})
		\begin{equation} \label{eq: lng}
			\begin{aligned}
				\langle \xi, v \rangle & \le \liminf_\alpha \langle \xi_\alpha, \tfrac{x - x_\alpha}{h} + v \rangle
				\le \liminf_\alpha f'_D \left( x_\alpha ; \tfrac{x - x_\alpha}{h} + v \right) \\
				& \le \liminf_\alpha \left\{ \tfrac{f(x + h v ) - f(x_\alpha)}{h} + \om \left( x + hv, x_\alpha \right) q \left( \tfrac{x - x_\alpha}{h} + v \right) \right\} \eqqcolon T_1 \\
				& \le \delta_h f(x; v) + \tfrac{f(x) - \limsup_\alpha f(x_\alpha)}{h} \\
				& + \limsup_\alpha \om \left( x + hv, x_\alpha \right) q \left( \tfrac{x - x_\alpha}{h} + v \right).
			\end{aligned}
		\end{equation}
		Setting $v = 0$ and sending $h \downarrow 0$ minding (\ref{eq: cl cnd 1.1}), we find
		$$
		0 \le \liminf_{h \downarrow 0} \tfrac{f(x) - \limsup_\alpha f(x_\alpha)}{h}
		$$
		so that (\ref{eq: f cnt cvg}) as $f(x) \le \liminf_\alpha f(x_\alpha)$ by (\ref{eq: sqncs cnvrg}). Inserting (\ref{eq: f cnt cvg}) into (\ref{eq: lng}), we find
		$$
		\langle \xi, v \rangle \le T_1 = \delta_h f(x; v) + \liminf_\alpha \om \left( x + hv, x_\alpha \right) q \left( \tfrac{x - x_\alpha}{h} + v \right),
		$$
		into which we insert (\ref{eq: cl cnd 1.2}) to conclude $\xi \in \p_D f(x)$ by sending $h \downarrow 0$.
	\end{proof}
	
	If $X$ is quasi-complete, i.e., every bounded, closed set is complete, and the nets are sequences, then \cite[Thm. 2.1]{Za} guarantees that a sufficient condition for $\langle \xi_\alpha, x_\alpha \rangle \to \langle \xi, x \rangle$ is
	\begin{equation} \label{eq: sffcnt}
		\begin{aligned}
			& x_n \to x \text{ in } \mm \left( X, Y \right), \quad && \xi_n \to \xi \text{ in } \ss \left( Y, X \right) \text{ or;} \\
			& x_n \to x \text{ in } \ss \left( X, Y \right), \quad && \xi_n \to \xi \text{ in } \mm \left( Y, X \right).
		\end{aligned}
	\end{equation}
	The conditions of Lemma \ref{lem: sd clsdnss} simplify if additional structure is present. For example, if $q$ is continuous such that $\lim q(x - x_\alpha) = 0$, then it suffices for (\ref{eq: cl cnd 1.1}) if $\om$ is locally bounded above. To check (\ref{eq: cl cnd 1.2}), one could try to obtain an upper semicontinuity estimate such as
	$$
	\liminf_\alpha \om \left( x + hv, x_\alpha \right) q \left( \tfrac{x - x_\alpha}{h} + v \right) \le \om \left( x + hv, x \right) q(v)
	$$
	and then find a sequence $h_n \downarrow 0$ along which
	$$
	\lim_n \delta_n f(x; v) + \om \left( x + h_n v, x \right) q(v) = f'_D(x; v).
	$$
	
	\subsection{Miscellaneous subdifferential calculus}
	
	In this subsection, we collect various useful non-smooth calculus results. An important feature of the Dini subdifferential is the chain rule we present in Lemma \ref{lem: smplst cr}. It is free of qualification condition on the outer function $f$, contrasting with chain rules for subdifferentials involving limits in $\dom(f)$ like the Fréchet or Dini-Hadamard subdifferential. Even when such a topological subdifferential is of interest, one may frequently use the Dini subdifferential together with Lemma \ref{lem: smplst cr} to obtain an upper estimate for them.
	
	\begin{lemma} \label{lem: smplst cr}
		Let $Y$ be a locally convex Hausdorff space with continuous inclusion $i \colon X \to Y$ and $f \colon Y \to \left( - \i, \i \right]$ a function with $f_X$ the restriction of $f$ to $X$ such that $\dom(f) = \dom(f_X)$. Then
		$$
		i^{-*} \p_D f_X(x) = \p_D f \circ i(x) \quad \forall x \in \dom(f).
		$$
	\end{lemma}
	
	\begin{proof}
		Let $x \in \dom(f), v \in X$. Setting $y \coloneqq i(x), u \coloneqq i(v)$, we have $f'_D(y; u) = f'_{X, D}(x; v)$ since $f_X = f \circ i$. If $w \in Y \setminus X$, then $\left( y, y + w \right] \subset Y \setminus X$ so that $f'_D(y; w) = + \infty$. Hence, if $y' \in \p_D f(y)$, then $i^* y'(v) = y'(u) \le f'_D(y; u) = f'_{X, D}(x; v)$ so that $i^* y' \in \p f_X(x)$. Conversely, if $x' \in \p_D f_X(x)$ such that $y' \in i^{-*}(x)$, then $y'(u) = x'(v) \le f'_{X, D}(x; v) = f'_D(y; u)$ so that $y' \in \p_D f(y)$.
	\end{proof}
	
	Energies arising in evolution equations usually are inf-compact, rendering them discontinuous on infinite-dimensional stat spaces. Frequently, however, there exists an energy topology in which a point of continuity is present, allowing to obtain useful estimates on the subdifferential, as the next lemma demonstrates.
	
	\begin{lemma}\label{lem: ppr SD stmt}
		Let $X, X_0, Y$ be locally convex Hausdorff spaces such that $Y \to X_0$ continuously, $F \colon X_0 \to \left( - \i, + \i \right]$ be a proper convex function and $A \colon X \supset \dom(A) \to X_0$ be a closed linear operator. We equip $W_A Y = \left\{ x \in X \st Ax \in Y \right\}$ with the graph topology so that $\AA \colon W_A Y \to Y \colon y \mapsto Ay$ is continuous. If the restriction $F_Y$ is continuous at $y_0 = Ax_0 \in Y$, then
		\begin{equation} \label{eq: ppr SD stmt}
			\left. \p \left( F A \right) (u) \right|_{W_A Y} \subset \AA^* \p F_{Au + Y} (Au) \quad \forall u \in \dom( F A ).
		\end{equation}
	\end{lemma}
	
	\begin{proof}
		Let $u \in X_0$ be a point such that $F(Au) < + \i$. We introduce the affine hull $Y_0$ of $Au \cup Y$ along with its pre-image $W_A Y_0 = A^{-1} Y_0$. The restriction $F_{Y_0 }$ is locally bounded above around $(y_0, Au)$ since $F_Y$ is continuous at $y_0$ and convex; See \cite[§3.2, Thm. 1]{IT}. Therefore, the chain rule \cite[§4.2, Thm. 2]{IT} implies
		\begin{equation} \label{eq: SD dntt}
			\p \left( F \AA \right) (u) = \AA^* \p F_{Y_0 } (Au) \quad \forall u \in W_A Y_0.
		\end{equation}
		We claim that
		\begin{equation} \label{eq: SD dntt 2}
			\left. \p F_{Y_0 } ( Au ) \right|_Y = \p F_{Au + Y}(A u) \quad \forall u \in W_A Y_0.
		\end{equation}
		Observe that $\subset$ is trivial, while $\supset$ follows since subgradients may be extended from $Y$ to $Y_0$ respecting domination by the (sublinear) radial derivative according to the Hahn-Banach theorem. Therefore, the extensions remain subgradients, minding that extension by a single dimension preserves continuity. We note
		$$
		\left. \AA^* \p F_{Y_0}(Au) \right|_{W_A Y} = \AA^* \left( \p F_{Y_0 } (Au)_Y \right).
		$$
		Therefore, we may combine (\ref{eq: SD dntt}) and (\ref{eq: SD dntt 2}) to find (\ref{eq: ppr SD stmt}) by
		\begin{equation*}
			\left. \p \left( F A \right) (u) \right|_{W_A Y} \subset \p \left( F \AA \right) (u) = \AA^* \p F_{Au + Y}(Au) \quad \forall u \in W_A Y_0. \qedhere
		\end{equation*}
	\end{proof}
	
	\section{The negligible sets of weak differentiability}
	
	In this section, we prove that a closed subset $F \subset U$ of an open set $U \subset \R^d$ is negligible for the existence of weak derivatives on $U$ if $\HH^{d - 1}(F) = 0$, where $\HH^s$ denotes the $s$-dimensional Hausdorff measure. We give a simple example indicating that the criterion is sharp with respect to the Hausdorff dimension. The result is instrumental in recognizing differential operators as closed with respect to the weak* topology of certain Orlicz spaces whose weak* compact sets need not be $L^1_{\text{loc} }(U)$-equi-integrable. We are inspired by \cite[Satz 8.1.1]{Str}, which we improve by reducing the assumptions on the integrability of $u$ to the bare minimum.
	
	\begin{lemma} \label{lem: nglgbl sts}
		For $d \ge 2$, let $U \subset \R^d$ be open, $u, v \in L^1_{\text{loc} }(U)$, and $F \subset U$ be a closed set such that $u$ has the weak first order partial derivative $D_i u = v$ on $U \setminus F$. If $\HH^{d - 1}(F) = 0$, then $u$ has the weak partial derivative $D_i u = v$ on $U$.
	\end{lemma}
	
	\begin{proof}
		Without restriction, we may take $F$ to be compact because it suffices to prove the claim for any compact subset of a general $F$. For $k \in \N$, we define the truncated function $u_k \coloneqq \max\{ - k, \min\{ k, u \} \}$ such that $u_k \in W^{1, 1}_{\text{loc} } \left( U \setminus F \right)$ with $D_i u_k = D_i u \chi_{\left\{ \left| u \right| \le k \right\} }$ a.e. by \cite[Thm. 4.4(iii)]{EG}. Fix $\psi \in C^\i_c \left( U \right)$. We find a sequence $\varphi_n \in C^\i_c \left( U ; \left[ 0, 1 \right] \right)$ such that $\| \varphi_n \|_{W^{1, 1}(U) } \to 0$ and $\varphi_n \ge 1$ in a neighborhood of $F$ by \cite[Def. 4.10, Rem. (ii), Thm. 4.13(ii)]{EG} because $\HH^{d - 1}(F) = 0 \iff \capa_1(F) = 0$ by \cite[Thm. 5.12]{EG}. In particular, $\psi \left( 1 - \varphi_n \right) \in C^\i_c \left( U \setminus F \right)$ for all $n \in \N$. Consequently,
		\begin{align*}
			\int_U u_k D_i \psi \, dx
			& = \lim_n \int_U u_k D_i \left( \psi \left( 1 - \varphi_n \right) + \psi \varphi_n \right) \, dx \\
			& = - \int_U D_i u_k \psi \, dx + \lim_n \int_U u_k \left( D_i \psi \varphi_n + \psi D_i \varphi_n \right) \, dx.
		\end{align*}
		The first integral in the last step is meaningful since $D_i u_k$ agrees with the locally integrable function $v$ a.e. The last limit vanishes since $u_k D_i \psi, u_k \psi \in L_\i \left( U \right)$ and $\| \varphi_n \|_{W^{1, 1}(U) } \to 0$. In conclusion, $u_k$ has a weak partial derivative equal to $D_i u_k$ a.e. Sending $k \uparrow \i$ yields that $u$ itself has a weak partial derivative equal to $D_i u$ a.e.
	\end{proof}
	
	Lemma \ref{lem: nglgbl sts} is vacuous if $d = 1$ as $\HH^0$ measures cardinality. A point indeed suffices for it to fail: the jump function $F(x) = \chi_{\left\{ x \ge 0 \right\}}$ has vanishing weak derivative on $\R \setminus \left\{ 0 \right\}$ but is not Sobolev regular by discontinuity. This demonstrates that the dimensional threshold in Lemma \ref{lem: nglgbl sts} is sharp: the example extends to arbitrary dimensions by considering $F$ as constant with respect to the remaining co-ordinate directions so that the exceptional set is the hyperplane $\left\{ 0 \right\} \times \R^{d - 1}$ of Hausdorff dimension $d - 1$.
	
	\section{Inequalities}
	
	In this section, we collect inequalities used in the main part. The first result implies various Poincaré type inequalities on Orlicz-Sobolev spaces once a compact embedding into an Orlicz space is known.
	
	\begin{lemma} \label{lem: bstrct pncr}
		Let $X, Y$ be Banach spaces, $p \colon X \to \left[ 0, \i \right)$ a semi-norm such that
		\begin{enumerate}[label=(\roman*)]
			\item $X \to Y$ is a compact embedding; \label{en. bstrct pncr 1.1}
			\item $\| \cdot \|_X \coloneqq p + \| \cdot \|_Y$ is an equivalent norm on $X$; \label{en. bstrct pncr 1.2}
			\item If $x_n \in X$ with $p(x_n) \to 0$, $x_n \to x$ in $Y$, then $x \in X$, $p(x) = 0$. \label{en. bstrct pncr 1.3}
		\end{enumerate}
		If $h \colon X \to \left[ 0, \i \right]$ is a proper, lower semicontinuous function such that
		\begin{enumerate}[label=(\alph*)]
			\item $h(\lambda x) = \lambda h(x)$ for all $\lambda > 0$ and $x \in X$, \label{en. bstrct pncr 2.1}
			\item $\inf_{\| x \|_X = 1, p(x) = 0} h(x) > 0$, \label{en. bstrct pncr 2.2}
		\end{enumerate}
		then there exists a constant $C > 0$ such that
		\begin{equation} \label{eq: bstrct pncr cnclsn}
			\| x \|_Y \le C \left[ p(x) + h(x) \right] \quad \forall x \in X.
		\end{equation}
	\end{lemma}
	
	\begin{proof}
		Arguing by contradiction, we suppose that
		\begin{equation} \label{eq: cntrdctn ssmptn}
			\forall n \in \N \, \exists x_n \in X \colon 1 = \| x_n \|_Y \ge n \left[ p(x_n) + h(x_n) \right].
		\end{equation}
		The homogeneity of (\ref{eq: bstrct pncr cnclsn}) by \ref{en. bstrct pncr 2.1} entered to arrange $\| x_n \|_Y = 1$. From (\ref{eq: cntrdctn ssmptn}) follows $p(x_n) + h(x_n) \to 0$. In particular, $x_n$ is bounded in $X$ by \ref{en. bstrct pncr 1.2}, hence $x_n \to x$ in $Y$ for a subsequence (not relabeled) by \ref{en. bstrct pncr 1.1}. We have $p(x_n - x) \le p(x_n) + p(x) = p(x_n) \to 0$ by \ref{en. bstrct pncr 1.3}, hence $x_n \to x$ in $X$ by \ref{en. bstrct pncr 1.2}. Consequently, lower semicontinuity yields
		$$
		0 = \liminf_n p(x_n) + h(x_n) \ge p(x) + h(x) \ge \inf_{\| x \|_X = 1, p(x) = 0} h(x) > 0
		$$
		by \ref{en. bstrct pncr 2.2} and since $\| x \| = 1$ by (\ref{eq: cntrdctn ssmptn}). We have arrived at a contradiction.
	\end{proof}
	
	\begin{corollary} \label{cor: cncrt pncr}
		Let $U \subset \R^d$ be a bounded, connected, open set on which the Rellich-Kondrachov compact embedding theorem holds, e.g., $\p U \in C^{0, 1}$, and let $\phi \colon U \times \R \to \left[ 0, \i \right]$ be an integrand generating an Orlicz space $L_\phi(U)$. Then, there exists for every $\e > 0$ a measurable set $A_\e \subset U$ with $\left| U \setminus A_\e \right| < \e$ and a constant $C > 0$ such that $\chi_{A_\e} L_\phi(U) \to \chi_{A_\e} L_1(U)$ is a continuous embedding and
		\begin{equation} \label{eq: cncrt pncr}
			\| u \|_p \le C \left[ \| \nabla u \|_1 + \left| \int_{A_\e} u \, dx \right| \right] \quad \forall u \in W^{1, 1}(U), \quad 1 \le p < \frac{d}{d - 1}.
		\end{equation}
	\end{corollary}
	
	\begin{proof}
		We apply Lemma \ref{lem: bstrct pncr}: Take $X = W^{1, 1}(U)$ and $Y = L_p(U)$. We have \ref{en. bstrct pncr 1.1} by Rellich-Kondrachov, \ref{en. bstrct pncr 1.2} by the Sobolev inequality for $p(u) \coloneq \| \nabla u \|_1$, and \ref{en. bstrct pncr 1.3} follows since $\| \nabla u_n \|_1 \to 0$ and $u_n \to u$ in $L_p(U)$ imply that $u \in W^{1, 1}(U)$ and $\| \nabla u \|_1 = 0$, hence $u$ is constant. Finally, a set with $\chi_{A_\e} L_\phi \to \chi_{A_\e} L_1$ exists by Lemma \ref{lem: a emb} while, for
		$$
		h(u) \coloneqq \left| \int_{A_\e} u \, d x \right|,
		$$
		there hold \ref{en. bstrct pncr 2.1} and \ref{en. bstrct pncr 2.2} since
		\begin{equation*}
			\inf_{\| u \|_X = 1} \inf_{p(u) = 0} h(u) = h \left( \frac{1}{\left| U \right|} \right) = \frac{\left| A_\e \right|}{\left| U \right|} > 0. \qedhere
		\end{equation*}
	\end{proof}


\begin{thebibliography}{widest entry}
\bibitem{Lu} S. Luckhaus, \textit{Solutions for the two-phase Stefan problem with the Gibbs-Thomson Law for the melting temperature}, Euro. J. Appl. Math. 1, 101-111, 1990.
\end{thebibliography}
\end{document}